%% file: DHL.tex
\documentclass[a4paper,11pt, english]{smfart}
\usepackage{vmargin,graphicx,amsmath,amssymb,latexsym, enumerate, color,setspace}
\usepackage{microtype}
\usepackage[french]{babel} 
\usepackage{hyperref}
\usepackage{pdflscape}
\usepackage{graphicx,subcaption}
\usepackage{tikz,pgfplots}
\pgfplotsset{compat=1.14}
\pgfplotsset{/pgf/number format/.cd,fixed,precision=4}
\newlength\fwidth
\newlength\fheight

\input{macros.tex}

\author{ Guillaume Dujardin } \address{\'Equipe MEPHYSTO, Inria, 40 avenue
  Halley, 59650 Villeneuve d'Ascq} \email{guillaume.dujardin@inria.fr}
\author{Fr\'ed\'eric H\'erau} \address{Laboratoire de Math\'ematiques J. Leray,
  UMR 6629 du CNRS, Universit\'e de Nantes, 2, rue de la Houssini\`ere, 44322
  Nantes Cedex 03, France} \email{frederic.herau@univ-nantes.fr} 
\author{Pauline Lafitte} \address{F\'ed\'eration de Math\'ematiques FR3487, 
CentraleSup\'elec, 3 rue Joliot-Curie, 91190
  Gif-sur-Yvette} \email{pauline.lafitte@centralesupelec.fr}

\title[Trend to equilibrium for Fokker--Planck equations]{Coercivity,
  hypocoercivity, exponential time decay and simulations for discrete
  Fokker-Planck equations}

\begin{document}

\frontmatter

 \begin{abstract} In this article, we propose and study several
discrete versions of homogeneous and inhomogeneous one-dimensional
Fokker-Planck equations.
In particular, for these discretizations of velocity and space,
we prove the exponential convergence to the equilibrium of the solutions,
for time-continuous equations as well as for time-discrete equations.
Our method uses new types of discrete Poincar\'e inequalities for a
``two-direction'' discretization of the derivative in velocity.
For the inhomogeneous problem, we adapt  hypocoercive methods 
to the discrete cases.
\end{abstract}

\subjclass{35Q83; 35Q84;35B40}
\keywords{return to equilibrium}
\thanks{G.D. is supported by the Inria project-team MEPHYSTO
and the Labex CEMPI (ANR-11-LABX-0007-01).
F.H. is supported by the grant "NOSEVOL" ANR-2011-BS01019-01.}
\maketitle
\mainmatter

\tableofcontents

\section{Introduction}

In this article we study the long time behavior of the solutions of discrete versions of the
following {\em inhomogeneous}\footnote{{\em ie} involving the space variable $x$ and the
velocity variable $v$} Fokker--Planck equation
\begin{equation} \label{eq:IHFPF} {\D_t} F + v {\D_x} F - {\D_v}({\D_v}+v) F = 0, \qquad
  F|_{t=0} = F^0,
\end{equation}
where $F= F(t,x,v)$ with $t\geq 0$, $x$ in the one-dimensional torus $\T$, and
$v \in \R$.  In general, this problem is set with
$F^0\in L^1 (\T \times \R, \dx\,\dv)$ with norm $1$, non-negative, and one looks
for solutions of \eqref{eq:IHFPF} with values in the same set at all time
$t\geq 0$.

To begin with, we study discretizations of the much simpler {\em
  homogeneous}\footnote{{\em ie} involving the variable $v$ but not the variable
  $x$} Fokker--Planck equation, set {\it a priori} in $L^1 (\dv)$
\begin{equation} \label{eq:HFPF} {\D_t} F - {\D_v} ({\D_v}+v) F = 0, \qquad F|_{t=0} =
  F^0,
\end{equation}
where $F= F(t,v)$ is unknown for $t>0$ and $ v \in \R$.  In particular, we use
this equation to introduce a first discretization of the operator ${\D_v}$ in
Section \ref{sec:homogeneous}, that we later generalize to the inhomogeneous
case in Section \ref{sec:eqinhomo}.

We include in this paper the theoretical study of these discretizations of the two
equations above when the one-dimensional velocity variable $v$ stays in a
bounded symmetric interval of the form $(-\vmax,\vmax)$ for some $\vmax>0$. In
this case, these equations are supplemented with homogeneous boundary conditions
at $v=\pm\vmax$ in the form $({\D_v} + v)F(\cdot,\cdot,\pm\vmax)=0$.  As in the
unbounded velocity case, we first introduce a discretization of the operator
${\D_v}$ in Section \ref{sec:homobounded} that we later generalize to the
inhomogeneous case in Section \ref{sec:eqinhomoboundedvelocity}.

All sections but the Introduction share the same structure. We first recall the
statements for the continuous solutions of the continuous equation, as well as
the continuous tools that allow to prove the results in the continuous setting:
one usually works in a Hilbertian subspace of $L^1$, uses the equilibrium of the
equation to write a rescaled equation, and derives the exponential convergence
of the continuous solutions to equilibrium using estimates on well-adapted
entropies.  Then, we introduce discretized operators together with a functional
framework dedicated to the equation at hand and we introduce the analogous tools
that allow to mimic the continuous setting and prove the exponential convergence
to equilibrium for the {\it discretized} equations, in space, time and velocity.
The main goal of this article is to introduce and analyze these discretizations
to obtain full proofs of exponential convergences to equilibrium for
discretizations of homogeneous as well as inhomogeneous Fokker--Planck
equations.  At the end of Sections \ref{sec:homobounded} and
\ref{sec:eqinhomoboundedvelocity}, we provide the reader with numerical results
that illustrate our theoretical analysis.

As in the continuous cases, our analysis starts with discrete equilibrium for
the discretized equations, that are analogous to the continuous Maxwellian
\begin{equation} \label{eq:gaussianintro} \mu(v) = c \exp^{-v^2/2},
\end{equation}
(where $c$ is a positive normalization constant) which is an equilibrium state
for the continuous equations \eqref{eq:IHFPF} and \eqref{eq:HFPF}.  Part of the
discretization and, more importantly, the functional framework, use deeply the
discrete equilibrium.  This allows in particular to obtain fundamental
functional inequalities as the discrete level, such as the Poincar\'e--Wirtinger
inequality which reads for the homogeneous unbounded continuous case
$$
\int_\R g^2 \mu \dv \leq \int_\R ({\D_v} g)^2 \mu \dv, \qquad \textrm{ when } \qquad
\int_\R g \mu \dv = 0.
$$
In all cases, this type of inequalities, together with adapted commutation
relations for the discretized operators, and mass-preservation properties,
allows for entropy dissipation control, which in the end yields exponential
convergence to equilibrium.

We propose and analyze several schemes in this paper but we present in this
introduction the two main ones and the corresponding results.  We postpone
to the end of this introduction the references to the other schemes
and results.

\bigskip The first scheme is an implicit Euler method in time for discretization
of the inhomogeneous Fokker--Planck equation \eqref{eq:IHFPF} set on the unbounded
velocity domain $\R$.  We consider the following discretization of
$\R^+ \times \T \times \R$.  For a fixed $\dth>0$ we discretize the half line
$\R^+$ by setting for all $n\in\N$, $t_n = n\dth$.  For a sequence
$(G^n)_{n\in \N}$, the discretization $D_t$ of the time-derivation operator
${\D_t}$ is defined by
$$
(D_t G)^n = \frac{G^{n+1}- G^n}{\dth}, \qquad n \in \N.
$$
For a small fixed $\dvh >0$, we discretize the real line $\R$ by setting for all
$i\in\Z$, $v_i = i\dvh$ and we work (concerning velocity only) in the set
$$
\ell^1(\Z,\dvh) = \set{ G \in \R^{\Z} \ | \ \sum_{ i \in \Z} \abs{G_{i}} \dvh <
  \infty },
$$
with the naturally associated norm. We consider the following ``two-direction''
discretization of the derivation operator in velocity: For
$G \in \ell^1(\Z,\dvh)$, we define $\Dv G\in \ell^1(\Z^*,\dvh)$ by the following
formulas
\begin{equation} \label{eq:defderivintro} (\Dv G)_i = \frac{G_{i+1}-G_{i}}{\dvh}
  \textrm{ for } i<0, \qquad (\Dv G)_i = \frac{G_i-G_{i-1}}{\dvh} \textrm{ for }
  i>0.
\end{equation} 
For $G\in \ell^1(\Z,\dvh)$ or $G\in \ell^1(\Z^*,\dvh)$ we define also $v G$ by
$ (v G)_i = v_i G_i $ (either for $i\in \Z$ or $i \in \Z^*$ depending on the
framework we work in)\footnote{Note that, in these definitions, the range of
  indices of the image $\Dv G$ is $\Z^*$ and not $\Z$, in order to keep into
  account the natural shift induced by the ``two-direction'' definition of
  $\Dv$.}.  The discretized Maxwellian ${\muh = (\muh)_{i \in \Z}}$, analogous of
the continuous one \eqref{eq:gaussianintro} is defined by
$$
\muh_i = \frac{c_\dvh}{ \prod_{\ell=0}^{|i|} (1+ v_\ell\dvh )}, \qquad i \in \Z.
$$
It satisfies $(\Dv + v) \muh = 0$, just as $\mu$ solves $({\D_v}+v)\mu=0$.  Since
we shall later work in a Hilbertian framework, we introduce the formal adjoint
$\Dvs$ of the velocity derivation operator $\Dv$.  For
$G \in \ell^1(\Z^*,\dvh)$, we define $\Dvs G \in \ell^1(\Z, \dvh)$ by the
following formulas\footnote{We emphasize the fact that there is no mistake in
  the denominator of $(\Dvs G)_0$.}
\begin{equation}
  \begin{split}
    & (\Dvs G)_i = \frac{G_{i}-G_{i-1}}{\dvh} \textrm{ for } i<0, \qquad \qquad
    (\Dvs G)_i = \frac{G_{i+1}-G_i}{\dvh} \textrm{ for } i>0,
    \\
    & \qquad \qquad \textrm{ and } \qquad \qquad (\Dvs G)_0 =
    \frac{G_1-G_{-1}}{\dvh}.
  \end{split}
\end{equation}
In order to discretize the one dimensional torus $\T$, we denote by $\dxh >0$
the step of the uniform discretization of $\T$ into $N\in\N^*$ sub-intervals,
and we denote by $\jjj = \Z/ N\Z$ the corresponding finite set of indices.  In
what follows, the index $i \in \Z$ will always refer to the velocity variable
and the index $j \in \jjj$ to the space variable.  The discretized
derivation-in-space operator $\Dx$ is defined by the following centered scheme :
for $G = (G_{j})_{j\in \jjj}$ we set
$$
(D_x G)_{j} = \frac{G_{j+1} - G_{j-1}}{2\dxh}, \qquad j\in \jjj.
$$
We now extend the definitions above to sequences with indices in $\jjj\times\Z$,
in the sense that the velocity index $j$ plays no role in the definition of
$\Dx$ and the space index $i$ plays no role in the definition of $v,\Dv,\Dvs$
and $\muh$.  The discrete mass of a sequence $G \in \ell^1(\jjj \times \Z)$ is
defined by
$$
m(G) = \dxh\dvh \sum_{j \in \jjj, i \in \Z} G_{j,i} .
$$
The first discretized version of \eqref{eq:IHFPF} that we consider in this
Introduction is the following implicit Euler scheme with unknown
$(F_n)_{n\in\N} \in (\ell^1(\jjj \times \Z))^\N$:
\begin{equation} \label{eq:eulerimplicite} F^{n+1} = F^n - \dth \sep{ v \Dx
    F^{n+1} + \Dvs (\Dv + v) F^{n+1}} = 0, \qquad F^0 \in \ell^1(\jjj \times
  \Z).
\end{equation}
Before stating our main result for the solutions of this last equation, we introduce
two adapted Hilbertian spaces and an adapted entropy functional.  First, we
define using the discretized equilibrium $\muh$ the two spaces
$$
\Ldeuxmudvdxh = \set{ g \in \R^{\jjj \times \Z} \ | \ \dxh\dvh\sum_{j \in \jjj,
    i \in \Z} \sep{g_{j,i}}^2 \muh_i < \infty },
$$
and
$$
\Ldeuxmudvdxs = \set{ h \in \R^{\jjj \times \Z^*} \ | \ \dxh\dvh\sum_{j \in
    \jjj, i \in \Z^*} \sep{g_{j,i}}^2 \mus_i < \infty },
$$
where $\mus$ is a ``two-direction'' translation of $\muh$ to be precised later.
We denote the naturally associated norms respectively by $\norm{\cdot}$ and
$\norms{\cdot}$.  Note that there is a natural injection
$\mu \Ldeuxmudvdxh \hookrightarrow \ell^1(\jjj\times \Z)$.  Second, we define
the following modified Fisher information, for all doubly indexed sequence $G$,
$$
\eeed(G) = \norm{ \frac{G}{\muh}}^2 + \norms{ \Dv \sep{\frac{G}{\muh}}}^2 +
\norm{ \Dx \sep{\frac{G}{\muh}}}^2.
$$
The main result concerning the scheme \eqref{eq:eulerimplicite} is the
following.

\begin{thm} \label{thm:eulerimplicite} For all $\dvh>0$, $\dxh>0$ and $\dth>0$,
  the problem \eqref{eq:eulerimplicite} is well-posed in the space of finite
  Fisher information and the scheme preserves the mass.  Besides, there exists
  explicit positive constants $\kappa_\delta$, $C_\delta$ and $\dvh_0$ such that
  for all $\dvh <\dvh_0$, $\dxh>0$ and $\dth>0$, for all $F^0$ of mass $1$ such
  that $\eeed(F^0) <\infty$, the corresponding solution $(F^n)_{n\in\N}$ of
  \eqref{eq:eulerimplicite} satisfies for all $n\geq 0$,
  \begin{equation*}
\eeed(F^n-\muh) \leq C_\delta
    (1+2\dth \kappa_\delta)^{-n} \eeed(F^0-\muh).
  \end{equation*}
\end{thm}

In the theorem, well-posedness means that the corresponding discrete semi-group
is well defined in the space of finite Fisher information.  Note that there is
no Courant-Friedrichs-Lewy (CFL) stability condition linking the numerical
parameters $\dth$, $\dvh$ and $\dxh$ (the scheme is implicit).  The whole
theorem is proved in Section \ref{subsec:eqinhomototalementdiscretisee} using
tools developed in the preceding sections and briefly introduced above.  Note
that, as a direct corollary,
we straightforwardly get the exponential trend of a solution $(F^n)_{n \in \N}$
to the equilibrium $\muh$:

\begin{cor}
  \label{cor:decrexpdiscretnonborne}
  Consider the constants $\kappa_\delta$, $C_\delta$ and $\dvh_0$ given by
  Theorem \ref{thm:eulerimplicite}. Then for all $\dth >0$ there exists
  $\kappa_\dth>0$ explicit with
  $\lim_{\dth \rightarrow 0} \kappa_\dth = \kappa_\delta$ such that for all
  $\dvh <\dvh_0$, all $\dxh >0$, all $F^0$ of mass $1$ such that
  $\eeed(F^0) <\infty$, the solution $(F^n)_{n\in\N}$ of
  \eqref{eq:eulerimplicite} satisfies for all $n\geq 0$,
  \begin{equation}
    \eeed(F^n-\muh) \leq C_\delta  \exp^{-2 \kappa_\dth n\dth} \eeed(F^0-\muh).
  \end{equation}
\end{cor}

\bigskip

The second discretization scheme we emphasize in this introduction is explicit
and deals with Equation \eqref{eq:IHFPF} set on a finite velocity domain
$(-\vmax, \vmax)$.  The main reason for proposing this scheme is that numerical
simulations we will present in Sections \ref{sec:homobounded} and
\ref{sec:eqinhomoboundedvelocity} are only possible with a finite set of indices
in all variables.

Our aim is now to discretize the following equation
\begin{equation*} 
  \begin{split}
    & {\D_t} F + v {\D_x} F - {\D_v}({\D_v}+v) F = 0, \qquad F|_{t=0} = F^0, \\
    & \qquad \qquad ({\D_v} + v)F|_{\pm \vmax} = 0,
  \end{split}
\end{equation*}
where $F= F(t,x,v)$ with $t\geq 0$, $x \in \T$ and $v \in I = (-\vmax, \vmax)$,
and $F^0\in L^1 (\T \times I, \dx\,\dv)$ is fixed.  For all $t>0$, the unknown
$F(t,\cdot, \cdot)$ is in $L^1 (\T \times I, \dx\,\dv)$.  We keep the notations
and definitions for the time and space discrete derivatives and we change to a
finite setting the definition of the velocity one.  The discretization in
velocity is the following: For a positive integer $\imax$, we define the set of
indices
$$
\iii = \set{-\imax + 1, -\imax + 2, \cdots, -1, 0, 1, \cdots, \imax-2, \imax
  -1}.
$$
Note for further use that the boundary indices $\pm\imax$ do not belong to the
full set $\iii$ of indices.  We set $\dvh=\vmax/\imax$ and for all $i\in\iii$,
$v_i = i\dvh$.  We also set $v_{\pm\imax} = \pm \vmax$.  The new discrete
Maxwellian $\muh \in \R^\iii$ is defined by
\begin{equation*}
  \muh_i = \frac{c_\dvh}{ \prod_{\ell=0}^{|i|} (1+ v_\ell\dvh )}, \qquad i \in \iii,
\end{equation*}
where the normalization constant $c_\dvh$ is defined such that
$\dvh \sum_{i\in \iii} \muh_i= 1$.  For the sake of simplicity, we will keep the
same notation $\muh$ as in the unbounded velocity case.  Note also that we do
not need to define the discrete Maxwellian $\muh$ at the boundary indices
$\pm \imax$.  We work in the following in the space ${\ell^1(\iii,\dvh)}$ of all
finite real sequences $g = (g_i)_{i\in \iii}$ with the norm
$\dvh\sum_{i \in \iii} \abs{g_i}$.  As we did above in the infinite velocity
case, we introduce another set of shifted indices and another discrete
Maxwellian. We set
$$
\iiis = \set{-\imax , -\imax + 1, \cdots, -2, -1, 1, 2, \cdots, \imax-1, \imax},
$$
and define $\mus \in {\ell^1(\iiis, \dvh)}$ by for all $ i\in\iiis$,
$$
\mus_i = \muh_{i+1} \textrm{ for } i<0, \qquad \mus_i = \muh_{i-1} \textrm{ for
} i>0.
$$
We consider the discrete derivation operators $\Dv$ and $\Dvs$ that are the same
as is the unbounded case except at the boundary where we impose a discrete
Neumann condition.  A good framework is the following: we define
$\Dv : \ell^1(\iii, \dvh) \longrightarrow \ell^1(\iiis, \dvh) $ for all
$G \in \ell^1(\iii, \dvh)$ by
\begin{equation} \label{eq:defderivn}
  \begin{split}
    & (\Dv G)_i = \frac{G_{i+1}-G_{i}}{\dvh} \textrm{ when } -\imax+1\leq i \leq -1, \\
    &
    (\Dv G)_i = \frac{G_i-G_{i-1}}{\dvh} \textrm{ when } 1 \leq i \leq \imax-1, \\
    & ((\Dv + v)G)_{\pm \imax} \defegal \muh \Dv \sep{ \frac{G}{\muh}}_{\pm
      \imax} = 0.
  \end{split}
\end{equation}
The last condition defines only implicitly both the derivation and the
multiplication at index $\pm \imax$.
For $G\in \ell^1(\iii)$ or $G\in \ell^1(\iiis)$ we define also $v G$ by
$ (v G)_i = v_i G_i $ (either for $i\in \iii$ or $i \in \iiis$ depending on the
framework we work in, and without ambiguity).  Similarly, we define
$\Dvs : \ell^1(\iiis, \dvh) \longrightarrow \ell^1(\iii, \dvh) $ for all
$H \in \ell^1(\iii, \dvh)$ by\footnote{Once again, there is no typo in the
  formula defining $(\Dvs H)_0$.}
\begin{equation}
  \begin{split}
    &  (\Dvs H)_i = \frac{H_{i}-H_{i-1}}{\dvh} \textrm{ when } -\imax+1 \leq i< -1, \\
    & (\Dvs H)_i = \frac{H_{i+1}-H_i}{\dvh} \textrm{ when } 1 \leq i \leq \imax
    -1,
    \\
    & (\Dvs H)_0 = \frac{H_1-H_{-1}}{\dvh}.
  \end{split}
\end{equation}
As in the unbounded case, we define the mass of a sequence
$G \in \ell^1(\jjj \times \iii)$ by
$$
m(G) = \dxh\dvh \sum_{j \in \jjj, i \in \iii} G_{j,i}.
$$
The second discretized version of \eqref{eq:IHFPF} is the following explicit Euler
scheme with unknown $F \in (\ell^1(\jjj \times \iii))^\N$:
\begin{equation} \label{eq:eulerexplicite}
  \begin{split}
    & F^{n+1} = F^n - \dth \sep{ v \Dx F^{n} + \Dvs (\Dv + v) F^{n}} = 0, 
\qquad F^0 \in \ell^1(\jjj \times \iii), \\
  \end{split}
\end{equation}
where we note that the Neumann type boundary condition is now included in the
definition of the derivation operator $\Dv$ in \eqref{eq:defderivn}.  We work with
the following Hilbertian structures on $\R^{\jjj \times \iii}$ and
$\R^{\jjj \times \iiis}$:
$$
\Ldeuxmudvdxh = \set{ g \in \R^{\jjj \times \iii} \ | \ \dxh\dvh\sum_{j \in
    \jjj, i \in \iii} \sep{g_{j,i}}^2 \muh_i < \infty },
$$
and
$$
\Ldeuxmudvdxs = \set{ h \in \R^{\jjj \times \iiis} \ | \ \dxh\dvh\sum_{j \in
    \jjj, i \in \iiis} \sep{g_{j,i}}^2 \mus_i < \infty },
$$
with the naturally associated norms again denoted respectively by $\norm{\cdot}$
and $\norms{\cdot}$.  There is again a natural injection
$\mu \Ldeuxmudvdxh \hookrightarrow \ell^1(\jjj\times \iii)$.  We define the same
modified Fisher information as in the unbounded case but in this new framework
\begin{equation}
  \eeed(G) = \norm{ \frac{G}{\muh}}^2 +  \norms{ \Dv \sep{\frac{G}{\muh}}}^2 +
  \norm{ \Dx \sep{\frac{G}{\muh}}}^2\label{eq:mfi}.
\end{equation}

For the scheme \eqref{eq:eulerexplicite}, the well-posedness for all $\dth>0$ is
granted since we are in a finite dimensional setting.  Since the scheme is
explicit, a CFL type condition is needed.  For that purpose, we introduce the
following CFL constant
$$
\bcfl = \max \set{ 1, 4 \frac{1+ \dvh \vmax}{\dvh^2},4 \frac{1+ \dvh
    \vmax}{\dxh^2},4\frac{\vmax^2}{\dxh^2}}.
$$
The main result in this explicit in time and bounded in velocity inhomogeneous
setting is the following

\begin{thm} \label{thm:eulerexplicite} The scheme \eqref{eq:eulerexplicite}
  preserves the mass.  Besides, there exists explicit positive constants
  $\kappa_\delta$, $C_\delta$, $\dvh_0$ and $C_{\rm CFL}$ such that for all
  $\dvh\in(0,\dvh_0)$ and $\dxh>0$, for all $F^0$ of mass $1$ such that
  $\eeed(F^0) <\infty$, for all $\dth>0$ satisfying the CFL condition
  $C_{\rm CFL}\bcfl\dth<1$, the solution $(F^n)_{n\in\N}$ of the scheme
  \eqref{eq:eulerexplicite} satisfies for all $n\in\N$,
  \begin{equation*} 
\eeed(F^n-\muh) \leq C_\delta
    (1-2\dth \kappa_\delta)^{n} \eeed(F^0-\muh).
  \end{equation*}
\end{thm}

The values of the explicit constants are given in Theorem
\ref{thm:decrexpeulerexpldiscr} in Section \ref{sec:eqinhomoboundedvelocity}.
Note that, as a direct corollary, using an asymptotic development of the
logarithm, we straightforwardly get the exponential trend of a solution
$(F^n)_{n \in \N}$ to the equilibrium $\muh$:

\begin{cor}
  Consider the constants $\kappa_\delta$, $C_\delta$, $\dvh_0$ and $C_{CFL}$
  given by Theorem \ref{thm:eulerexplicite}.  For all $\dvh\in(0,\dvh_0)$ and
  $\dxh>0$, for all $\dth >0$ satisfying the CFL condition
  $C_{\rm CFL}\bcfl\dth<1$, there exists $\kappa_\dth>0$ explicit with
  $\lim_{\dth \rightarrow 0} \kappa_\dth = \kappa_\delta$ such that for all
  $F^0$ of mass $1$ such that $\eeed(F^0) <\infty$, the solution
  $(F^n)_{n\in\N}$ of \eqref{eq:eulerexplicite} satisfies for all $n\in\N$,
  \begin{equation}
    \eeed(F^n-\muh) \leq C_\delta  \exp^{-2 \kappa_\dth n\dth} \eeed(F^0-\muh).
  \end{equation}
\end{cor}

\bigskip

As was already stated, the main goal of our paper is to propose and analyze
hypocoercive numerical schemes for inhomogeneous kinetic equations, for which
one can prove exponential in time return to the equilibrium.  In the literature,
one can find theoretical results either about numerical schemes for homogeneous
kinetic equations, built upon coercivity for discrete models, or about exact
solutions of inhomogeneous equations, built upon hypocoercivity techniques.  In
this paper, we want to tackle both problems at the same time and prove
theoretical results on exponential time return to equilibrium for discrete {\it
  and} inhomogeneous kinetic equations.  Up to our knowledge, these are the
first theoretical results dealing with the two difficulties at the same time.

Concerning the simpler homogeneous kinetic equations, the question of finding
efficient schemes has a long story and deep recent developments. Let us mention
a few results that are already known in these directions.  One can find this
kind of problems for example in \cite{CC70} for the linear homogeneous
Fokker-Planck equation in a fully discrete setting.  More recently, schemes have
been proposed for nonlinear degenerate parabolic equations that numerically
preserve the exponential trend to equilibrium (see for example \cite{BCF12} for
a finite volume scheme which works numerically even for nonlinear problems).
This question has also been addressed numerically together with that of the
order of the schemes, for nonlinear diffusion and kinetic equations {\it e.g.}
in \cite{PR16}.  In particular, it is known that, even for the {\it linear}
Fokker-Planck equation, "wrong" discretizations lead to "wrong" qualitative
behaviour of the schemes in long time.  So-called spectral methods are also
proposed (see for example recent developments for the Boltzmann equation in
\cite{AGT16}), with the drawback that they do not ensure the non-negativity
of the solutions. Let
us also mention the recent paper \cite{Filbet2017}, where a finite volume scheme
is introduced for a class of boundary-driven convection-diffusion equations on
bounded domains.  The question of the long-time behaviour of the scheme is
addressed using the relative entropy structure.

Concerning inhomogeneous kinetic (continuous) equations, the so-called
hypocoercive theory is now rather well understood with various results
concerning many models.  In this direction, first results on linear models were
obtained in \cite{Her06}, \cite{MN06} \cite{Vil09} or \cite{DMS15}. They were in
fact adapted on the very abstract theory of hypoellipticity of Kohn or (type II
hypoelliptic operators) of H\"ormander that explain in particular the
regularization of such degenerate parabolic equations. The cornerstone of the
theory is that, although the drift $v.\nabla_x$ is degenerate (at $v=0$ in
particular), one commutator with the velocity gradient erases the degeneracy :
$[\nabla_v, v.\nabla_x] = \nabla_x$.  The main feature of the hypocoercive
theory is that this commutation miracle leads also to exponential return to the
equilibrium (independantly of the regularization property).
One other feature is that
it can be enlarged to collision kernels even without diffusive velocity kernel
and to many other inhomogeneous kinetic models systems (see e.g.  \cite{Vil09,BDMMS17} or the introduction course \cite{Her17a}).

Concerning the numerical analysis of inhomogeneous kinetic equations, we mention
the paper \cite{PZ16} where the Kolmogorov equation is discretized in order to
get short time estimates, following the short time continuous "hypocoercive"
strategy proposed in \cite{Her07}. However, the corresponding scheme is not
asymptotically stable and no notion of equilibrium or long-time behaviour is
proposed there. This paper was anyway a source of inspiration of the present
work (see also point 4 in Section \ref{sec:ccl} here for further interactions
between the two articles).  We also mention the work on the
Kolmogorov--Fokker--Planck equation carried out in \cite{FosterLoheacTran2017},
where a time-splitting technique based on self-similarity properties is used for
solutions that decay like inverse powers of the time.

In this article we show that the hypocoercive theory is sufficiently robust to
indeed give exponential time decay of partially or fully discretized
inhomogeneous equations. This is done here in the case of the Fokker-Planck
equation in one dimension.  We cover fully discretized as well as
semi-discretized situations.  We propose, for each setting, for the first
time up to our knowledge, a full proof of exponential convergence towards
equilibrium for the corresponding solutions. Once again these proofs use
discrete analogues to the continuous tools, such as the Poincar\'e inequality
and the hypocoercive techniques. Even for the simple homogeneous setting, to our
knowledge, the (optimal) discrete Poincar\'e inequality with a weight is
new (see Proposition \ref{prop:poindiscrete}) in both bounded and
unbounded cases.

We hope that this approach can be generalized to various multi-dimensional
kinetic models of the form ${\D_t} u + Pu=0$, with $P$ hypocoercive. One aim would
be to write a systematic "black box scheme" theorem with $P= X_0 - L$ where $L$
is the collision kernel (independently studied in velocity variable only) and
$X_0$ the drift, as proposed in e.g. \cite{DMS15} in the continuous case.  In
this sense a lot of work has to be done.  Of course we also hope that our scheme
approach can be used to predict some results for more complex situations
including non-linear inhomogeneous ones.

\bigskip

The outline of this article is the following.  In the second section, we deal
with the homogeneous equation \eqref{eq:HFPF} in time and velocity only, with
velocity varying in the full real line.  We first recall the continuous
framework in a very simplified and concise way. Then, we adapt it to
semi-discrete and fully discrete cases.  In particular, we focus on the
homogeneous case and we state a new discrete Poincar\'e inequality with the
discrete Gaussian weight $\muh$.

In the third section, we deal with the full inhomogeneous case \eqref{eq:IHFPF},
and propose a concise version of the continuous results.  Then, we adapt these
results to several discretized versions of the equation: the semi-discrete in
time case, the implicit semi-discrete in space and velocity case, ending with
the full implicit discrete case corresponding to Theorem
\ref{thm:eulerimplicite}.  In particular, we develop discrete versions of the
commutation Lemmas at the core of the (continuous) hypocoercive method.

In the fourth section, we focus on the homogeneous case \eqref{eq:HFPF} set on a
bounded velocity domain.  We only deal with the continuous and the explicit
fully discrete case.  Once again, a new Poincar\'e inequality is proposed.
Moreover, a CFL condition appears.

In the fifth section, we consider the inhomogeneous problem \eqref{eq:IHFPF} set on
a bounded velocity domain.  We first present the continuous case.  Then, we
propose the study of the fully discrete case with an Euler explicit scheme
leading to Theorem \ref{thm:eulerexplicite}.

In the appendix, we propose some comments and possible generalizations, as well
as a table summarizing the main results concerning discrete commutators.

\section{The homogeneous equation}
\label{sec:homogeneous}
\subsection{The continuous time-velocity setting}
\label{subsec:homogeneouscontinuous}

We start by recalling the main features of the continuous equation \eqref{eq:HFPF}
set on the unbounded domain $\R$.  These features will have discrete analogues
described in the next subsection.

Since we are interested in the long time behavior and the trend to the
equilibrium, we start by checking what the good equilibrium states are.  We
first look at the continuous homogeneous equation \eqref{eq:HFPF}.  We say that a
function $\mu (v)$ is an equilibrium if $-{\D_v}({\D_v}+v) \mu(v) = 0$.  The first
idea is to suppose only that $({\D_v} + v) \mu(v) = 0$ which leads to
\begin{equation}
  \label{eq:Maxwellian}
  \mu(v) = \frac{1}{\sqrt{2\pi}} \exp^{-v^2/2}, 
\end{equation}
if we impose in addition that $\mu\geq 0$ is $L^1(\dv)$-normalized.

A standard strategy in statistical mechanics is then to build an adapted
functional framework (a subspace of $L^1(\dv)$) where non-negativity of the
collision operator $-{\D_v}({\D_v}+v) $ is conserved.  A standard choice is then to
take $F (t,\cdot)\in \mu \Ldeuxmudv \hookrightarrow L^1(\dv)$ where
$\mudv = \mu(v) \dv$. We check then that operator $-{\D_v}({\D_v}+v) $ is
self-adjoint in $\mu \Ldeuxmudv$, with compact resolvent.  Therefore it has
discrete spectrum and $0$ is a single eigenvalue associated with the
eigenfunction $\mu$.  In fact, this result can be easily checked using the
following change of unknown, which will be of deep and constant use through out
this article.

We pose for the following $F= \mu+\mu f$ and call $f$ the rescaled density. With
this new unknown function, and in the new adapted framework, the equation
\eqref{eq:HFPF} writes
\begin{equation} \label{eq:HFPf} {\D_t} f +(- {\D_v}+v) {\D_v} f = 0 , \qquad f|_{t=0} =
  f^0,
\end{equation}
where $f= f(t,\cdot) \in \Ldeuxmudv \hookrightarrow L^1(\mu \dv)$.  The
non-negativity of the collision kernel is then direct to verify: in $\Ldeuxmudv$
with the associated scalar product we have ${\D_v}^* = (-{\D_v}+ v)$ and therefore
for all $g\in\Hunmudv$ with $(- {\D_v}+v){\D_v} g\in\Ldeuxmudv$,
 $$
 \seq{(- {\D_v}+v){\D_v} g, g}_{\Ldeuxmudv} = \norm{{\D_v} g}^2_{\Ldeuxmudv} = \int_\R
 |{\D_v} g|^2 \mudv.
 $$
 it is easy to check that operator $P=(- {\D_v}+v){\D_v}$ is maximal accretive
 (\cite{HN04}) with domain
 $D(P) = \set{ g \in \Ldeuxmudv \ | \ (- {\D_v}+v){\D_v} g\in\Ldeuxmudv}$ and using
 the Hille--Yosida Theorem, one obtains at once the existence and uniqueness of
 the solution $f$ of \eqref{eq:HFPf} in
 $\ccc^1(\R^+, \Ldeuxmudv) \cap \ccc^0(\R^+, D(P))$ for all $f^0\in D(P)$, and
 that the problem is also well-posed in $\ccc^0(\R^+, \Ldeuxmudv)$ in the sense
 of distributions.  From the preceding equality, for $g\in \Ldeuxmudv$,
 $$
 (- {\D_v}+v) {\D_v} g = 0 \Longleftrightarrow {\D_v} g = 0 \Longleftrightarrow
 g\text{ is constant,}
 $$
 and therefore the constants are the only equilibria of the equation
 \eqref{eq:HFPf}.  Note that in this $L^2$ framework, the conservation of mass is
 obtained by integrating equation \eqref{eq:HFPf} against the constant function $1$
 in $\Ldeuxmudv$ to obtain for all $t\geq 0$,
 \begin{equation}
   \label{eq:consmass}
   \seq{f(t)} \defegal \int_\R f(t,v) \mu (v)\dv = \seq{f(t),1}_{\Ldeuxmudv} = \seq{f^0}.
 \end{equation}
 In that case a system with null mass corresponds to a rescaled density $f$ such
 that $f \perp 1$ in $\Ldeuxmudv$.  Note that Equation \eqref{eq:HFPf} is also
 well posed in $\Hunmudv$ thanks to the Hille--Yosida Theorem again, and that it
 yields a unique solution in
 $\ccc^1(\R^+, \Hunmudv) \cap \ccc^0(\R^+, D_{\Hunmudv}(P))$ for all
 $f^0\in \Hunmudv$, where $D_\Hunmudv(P)$ is the domain of $P= (-{\D_v} + v) {\D_v}$
 in $\Hunmudv$.  Of course, this solution coincides with the one with values in
 $\Ldeuxmudv$ when $f^0\in \Hunmudv$.

 \bigskip One of the main tools in the study of the return to equilibrium for
 Fokker--Planck equations is the Poincar\'e inequality.  There are many ways of
 proving it (including the compact resolvent property) but one direct way, well
 adapted to a coming discretization, can be inspired by the original proof by
 Poincar\'e in the flat case.

 \begin{lem}[homogeneous Poincar\'e inequality]
   \label{lem:Poincarecontinu}
   For all $g \in \Hunmudv$, we have
   \begin{equation*}
     \norm{g-\seq{g}}^2_\Ldeuxmudv  \leq  \norm{{\D_v} g}_\Ldeuxmudv^2.
   \end{equation*}
 \end{lem}
 \begin{proof}
   Replacing if necessary $g$ by $g-\seq{g}$, it is sufficient to prove the
   result for $\seq{g}=0$.  In the following, we denote for simplicity
   $g(v) = g$, $g(v') = g'$, $\mu(v) = \mu $ and $\mu(v') = \mu'$.  We first
   note that
   \begin{equation*}
     \begin{split}
       \int_\R g^2\mudv & = \frac{1}{2} \iint_{\R^2} (g'-g)^2 \mudv \mudvp,
     \end{split}
   \end{equation*}  since
   $2 \iint g g' \mudv \mudvp = 2 \int g\mudv \int g' \mudvp = 0$.
   Using that $g'-g= \int_v^{v'} {\D_v} g (w) \dw$ we can write
   \begin{multline*}
     \int_\R g^2\mudv
     = \frac{1}{2} \iint_{\R^2} \sep{ \int_v^{v'} {\D_v} g (w) \dw}^2 \mudv \mudvp \\
     \leq \frac{1}{2} \iint_{\R^2} \sep{ \int_{v}^{v'} \abs{{\D_v} g (w)}^2
       \dw}(v'-v) \mudv \mudvp
   \end{multline*}
   where we used the Cauchy--Schwarz inequality in the flat space.  Let us
   denote by $\G$ an anti-derivative of $\abs{{\D_v} g}^2$, for example this one :
   $\G(v) = \int_{0}^v \abs{{\D_v} g (w)}^2 \dw$.  We have then
   \begin{equation} \label{eq:eqintermedpoincarecontinu}
     \begin{split}
       &  \int_\R g^2\mudv \\
       & \leq \frac{1}{2} \iint_{\R^2} \sep{ \G'-\G}(v'-v) \mudv \mudvp
       =  \frac{1}{2} \iint_{\R^2} \sep{ \G'-\G}(v'-v)\mu \mu' \dv \dv' \\
       & = \frac{1}{2} \sep{ \iint_{\R^2} \G'v'\mudv \mudvp + \iint_{\R^2} \G v
         \mudv \mudvp -\iint_{\R^2} \G v' \mudv\mudvp
         -\iint_{\R^2} \G' v \mudv \mudvp  } \\
       & = \int_{\R} \G v \mudv,
     \end{split}
   \end{equation}
   where we used the Fubini Theorem and the fact that $\int v \mudv = 0$ and
   $\int \mudv = 1$ (and their counterparts in variable $v'$).  At this point,
   it is sufficient to note that ${\D_v} \mu = -v \mu$ and perform an
   integration by parts to obtain with the inequality above,
   \begin{equation*}
     \begin{split}
       \int_\R g^2\mudv & \leq \int_\R (\G v \mu) \dv = - \int_\R \G ({\D_v} \mu) \dv
       =\int_\R ({\D_v} \G) \mudv = \int_\R |{\D_v} g|^2 \mudv.
     \end{split}
   \end{equation*}
   The proof is complete. \end{proof}

 A direct consequence of this Poincar\'e inequality is the exponential
 convergence to the equilibrium in the space $\Ldeuxmudv$ of the solution $f$ of
 \eqref{eq:HFPf}, that we prove below.
 In Section \ref{sec:inhomc}, we will use an entropy formulation to prove the
 exponential convergence to the equilibrium of the solutions of the {\it
   inhomogeneous} Fokker--Planck equation.
 For this reason, we decide to adopt the same framework in this section, devoted
 to the (simpler) homogeneous case.
 We define the two following entropies for $g\in \Ldeuxmudv$ and $g\in\Hunmudv$
 respectively :
$$
\fff(g) = \norm{g}_\Ldeuxmudv^2, \qquad \ggg(g) = \norm{g}_\Ldeuxmudv^2 +
\norm{{\D_v} g}_\Ldeuxmudv^2.
$$
Note that these entropies are exactly the squared norms of $g$ in $\Ldeuxmudv$
and $\Hunmudv$ respectively.  To keep notations short, in the remaining of this
section, we denote by $\norm{\cdot}$ the $\Ldeuxmudv$ norm.  The exponential
convergence to the equilibrium of the solutions of \eqref{eq:HFPf} is stated in the
following easy Theorem.

\begin{thm}
  \label{thm:exponentialtrendtoequilibrium}
  Let $f^0 \in \Ldeuxmudv$ such that $\seq{f^0} = 0$ and let $f$ be the solution
  in $\ccc^0(\R^+,\Ldeuxmudv)$ of \eqref{eq:HFPf} (in the semi-group sense). Then
  $\seq{f(t)} = 0$ for all $t\geq 0$, and we have
  \begin{equation}
    \label{eq:decrFhomo}
    \forall t\geq 0,\qquad
    \fff(f(t)) \leq \exp^{-2t}\fff(f^0).
  \end{equation}

  If in addition $f^0 \in \Hunmudv$, then $f \in \ccc^0(\R^+, \Hunmudv)$ and we
  have
  \begin{equation}
    \label{eq:decrGhomo}
    \forall t\geq 0,\qquad 
    \ggg(f(t)) \leq \exp^{-t}\ggg(f^0).
  \end{equation}
\end{thm}
\begin{proof}
  We first recall that operator $P = (-{\D_v} + v) {\D_v}$ is the generator of a
  semi-group of contractions in both $\Ldeuxmudv$ and $\Hunmudv$.  This is
  direct to check that $\Hunmudv$ is dense in $\Ldeuxmudv$ and that when both
  defined, the solutions of the heat problem ${\D_t} f + P f = 0$ coincide.  In
  the following, we therefore focus on the $\Hunmudv$ case corresponding to
  solutions with finite modified entropy $\ggg$.

  We denote by $D_{\Hunmudv}(P) $ the domain of $P$ in $\Hunmudv$.  We note
  again that $D_{\Hunmudv}(P) $ is dense in $\Hunmudv $, and we consider a
  solution $f$ of \eqref{eq:HFPf} which satisfies
$$
f \in \ccc^1(\R^+, \Hunmudv) \cap \ccc^0(\R^+, D_{\Hunmudv}(P)).
$$
All the computations below are therefore authorized.  The main inequalities
\eqref{eq:decrFhomo} and \eqref{eq:decrGhomo} are then consequences of the above
mentioned density properties and of the definition of a bounded semi-group.

We compute the time derivative of the corresponding entropies along the exact
solution $f$ of \eqref{eq:HFPf}.  Using \eqref{eq:consmass}, we have for all
$t\geq 0$, $\seq{f(t)} = \seq{f^0} = 0$.  For the first entropy, we have
$$
\ddt \fff(f) = -2 \seq{ (-{\D_v}+v){\D_v} f , f} = -2 \norm{{\D_v} f}^2 \leq -2
\norm{f}^2 = -2 \fff(f),
$$
where we used the Poincar\'e Lemma \ref{lem:Poincarecontinu}. This directly
gives \eqref{eq:decrFhomo}.  For the second entropy $\ggg$, we do the same:
\begin{equation*}
  \begin{split}
    \ddt \ggg(f) & = -2 \seq{ (-{\D_v}+v){\D_v} f , f} -2 
\seq{ {\D_v}(-{\D_v}+v){\D_v} f , {\D_v} f} \\
    & = -2 \norm{{\D_v} f}^2 - 2 \norm{(-{\D_v} +v){\D_v} f}^2 \\
    & \leq - \norm{f}^2 - \norm{{\D_v} f}^2 - 2 \norm{(-{\D_v} +v){\D_v} f}^2 \leq -
    \ggg(f),
  \end{split}
\end{equation*}
where we used the following splitting~:
$2 \norm{{\D_v} f}^2 \geq \norm{{\D_v} f}^2 + \norm{f}^2$, obtained again with Lemma
\ref{lem:Poincarecontinu}.  We therefore get the result \eqref{eq:decrGhomo}.
The proof is complete.\end{proof}

The following corollary is then straightforward, as a reformulation of the
preceding Theorem.
\begin{cor} \label{cor:exponentialtrendtoequilibrium} Let $f^0 \in \Ldeuxmudv$
  and let $f$ be the solution in $\ccc^0(\R^+, \Ldeuxmudv)$ of
  \eqref{eq:HFPf}. Then for all $t\geq 0$,
$$\norm{f(t)- \seq{f^0}}_\Ldeuxmudv \leq \exp^{-t}\norm{f^0- \seq{f^0}}_\Ldeuxmudv.
 $$
 If in addition $f^0 \in \Hunmudv$ then $f \in \ccc^0(\R^+, \Hunmudv)$ and we
 have for all $t\geq 0$,
 $$
 \norm{f(t)- \seq{f^0}}_\Hunmudv \leq \exp^{-\frac{t}{2}}\norm{f^0-
   \seq{f^0}}_\Hunmudv.
 $$
\end{cor}

\subsection{Discretizing the velocity variable}
\label{subsec:homsd}
 
In the discrete and semi-discrete cases, the main difficulty is to find a
suitable discretization of the equation that will mimic the qualitative
asymptotic properties of the continuous equation, see {\it e.g} Theorem
\ref{thm:exponentialtrendtoequilibrium}.  In particular, one has to decide how
to discretize the differential operators in $v$.  For a small fixed $\dvh >0$,
we discretize the real line $\R_v$ by setting for all $i\in\Z$, $v_i = i\dvh$.

We work now step by step, and look first at what could be a suitable equilibrium
state $\muh$ replacing $\mu$ in the continuous case.  As in the continuous case,
$\muh$ has to satisfy elementary structural properties.  The first ones are to
be positive and to be normalized in the (discrete) probability space
$\ell^1(\Z, \dvh)$ which means 
\begin{equation*}
  \label{eq:normalizationdemudv}
  \norm{\muh}_{\ell^1(\Z, \dvh)} = \dvh \sum_i \muh_i = 1.
\end{equation*}
Mimicking the continuous case, we also require $\muh$ to be even and to satisfy
the equation $(\Dv + v)\muh = 0$ where $\Dv$ is a discretization of ${\D_v}$ and
$v$ stands for the sequence $( v_i)_{i\in \Z}$ or by extension the
multiplication term by term by it.  A good choice for $\Dv$ leading to this
property is the following :

\begin{defn} \label{def:dv} Let $G \in \ell^1(\Z,\dvh)$, we define
  $\Dv G\in \ell^1(\Z^*,\dvh)$ by the following formulas
  $$
  (\Dv G)_i = \frac{G_{i+1}-G_{i}}{\dvh} \textrm{ for } i<0, \qquad (\Dv G)_i =
  \frac{G_i-G_{i-1}}{\dvh} \textrm{ for } i>0,
  $$
  and $v G\in \ell^1(\Z^*,\dvh)$ by
  $$
  (v G)_i = v_i G_i \textrm{ for } i\neq 0,
  $$
  when this series is absolutely convergent.
\end{defn}

\noindent
With this definition, solving the equation $(\Dv + v)\muh = 0$ leads to the
following proposition.

\begin{lem} \label{lem:muh} Assume $\dvh>0$ is fixed.  Then there exists a unique
  positive, $\ell^1(\Z, \dvh)$~-~normalized, solution $\nu$ of
  $(\Dv + v)\nu = 0$.  We denote this solution by $\muh$.  There exists a unique
  positive constant $c_\dvh$ such that
$$
\muh_i = \frac{c_\dvh}{ \prod_{\ell=0}^{|i|} (1+ v_\ell\dvh )}, \qquad i \in \Z.
$$
Moreover, $\muh$ is even.
\end{lem}

\begin{remark}
  Note that the discrete Maxwellian $\muh$ converges to the continuous
  Maxwellian $\mu$ defined in \eqref{eq:Maxwellian} when $\dvh$ tends to 0 in
  the following sense :
  \begin{equation*}
    {\sup_{i\in\Z}} |\muh_i-\mu(v_i)| \underset{\dvh\to 0}{\longrightarrow} 0.
  \end{equation*}
\end{remark}

\begin{proof} The proof is a direct computation.  The fundamental equations term
  by term solved by $\muh$ are indeed
  \begin{equation}
    \label{eq:relationsmudv}
    \left\{
      \begin{array}{ll}
        \dfrac{\muh_i  -\muh_{i-1} }{\dvh} + v_i \muh_i = 0 & \quad \textrm{ for  } i>0 \\
        \dfrac{\muh_{i+1}  -\muh_{i} }{\dvh} + v_i \muh_i =0 & \quad \textrm{ for  } i<0,
      \end{array}
    \right.
  \end{equation}
  which give the expression of $\muh$ up to a normalization constant.
\end{proof}

With the discretization $\Dv+v$ of the operator ${\D_v}+v$ above, we propose the
following discretization $-\Dvs$ of $-{\D_v}$ so that the discretized version of
\eqref{eq:HFPF}, with operator $\Pd= -\Dvs(\Dv+v)$, has a non-negative collision
kernel.

\begin{defn} \label{def:dvs} Let $G \in \ell^1(\Z^*, \dvh)$, we define
  $\Dvs G \in \ell^1(\Z, \dvh)$ by the following formulas
  \begin{equation}
    \begin{split}
      & (\Dvs G)_i = \frac{G_{i}-G_{i-1}}{\dvh} \textrm{ for } i<0, \qquad
      \qquad (\Dvs G)_i = \frac{G_{i+1}-G_i}{\dvh} \textrm{ for } i>0
      \\
      & \qquad \qquad \textrm{ and } \qquad \qquad (\Dvs G)_0 =
      \frac{G_1-G_{-1}}{\dvh},
    \end{split}
  \end{equation}
  (be careful, there is no mistake in the denominator of $(\Dvs G)_0$).  We also
  define the operator $\vs$ from $\ell^1(\Z^*, \dvh)$ to $\ell^1(\Z, \dvh)$ by
  setting for $G\in \ell^1(\Z^*, \dvh)$,
  \begin{equation*}
    \forall i\neq 0,\quad (\vs G)_i=v_i G_i \qquad \text{and} \qquad
    (\vs G)_0=0.
  \end{equation*}
\end{defn}

We are now in position to define a good discretization of the main equation
\eqref{eq:HFPF} and the adapted discretized framework.

\begin{defn} For a given $F^0\in\ell^1(\Z,\dvh)$, we shall say that a function
  $F \in \ccc^0(\R^+, \ell^1(\Z, \dvh))$ satisfies the (flat) semi-discrete
  homogeneous Fokker--Planck equation if
  \begin{equation} \label{eq:DHFPF} {\D_t} F - \Dvs(\Dv+v) F = 0, \qquad F|_{t=0} =
    F^0,
  \end{equation}   
  in the sense of distributions.
\end{defn}

\noindent
As in the continuous case, we perform the change of unknown, thanks to the
discrete equilibrium state $\muh$: $G = \muh g$ so that
$$
G \in \ell^1(\Z, \dvh) \Longleftrightarrow g \in \ell^1(\Z, \muh \dvh).
$$
Let us perform this change of unknown in the differential operator
$-\Dvs(\Dv+v)$. For $i>0$, we have
\begin{equation*}
  \begin{split}
    ((\Dv+v)G)_i & = ((\Dv+v)\muh g)_i
    = \frac{\muh_i g_i -\muh_{i-1} g_{i-1}}{\dvh} + v_i \muh_i g_i \\
    & = \Bigg(\underbrace{ \frac{\muh_i -\muh_{i-1} }{\dvh} + v_i \muh_i
    }_{=0}\Bigg) g_i +\muh_{i-1} \frac{ g_i - g_{i-1}}{\dvh} = \muh_{i-1} (\Dv
    g)_i.
  \end{split}
\end{equation*}
Similarly, we find for $i<0$,
\begin{equation*}
  \begin{split}
    ((\Dv+v)G)_i & = ((\Dv+v)\muh g)_i
    = \frac{\muh_{i+1} g_{i+1} -\muh_{i} g_{i}}{\dvh} + v_i \muh_i g_i \\
    & = \Bigg( \underbrace{\frac{\muh_{i+1} -\muh_{i} }{\dvh} + v_i \muh_i
    }_{=0}\Bigg) g_i +\muh_{i+1} \frac{ g_{i+1} - g_{i}}{\dvh} = \muh_{i+1} (\Dv
    g)_i.
  \end{split}
\end{equation*}
From the computation above, we get that
\begin{equation} \label{eq:double}
  \begin{split}
    -\Dvs((\Dv+v)G) & = \muh (-\Dvs + \vs) \Dv g,
  \end{split}
\end{equation}
Therefore, for any $F\in\ccc^0(\R^+, \ell^1(\Z, \dvh))$, setting for all
$t\geq 0$, $f(t,\cdot)=(F(t,\cdot)-\muh)/\muh$, we have
$$
{\D_t} F - \Dvs(\Dv+v) F = \muh ({\D_t} f + ( - \Dvs +\vs) \Dv f),
$$
where we recall that the multiplication is done term by term.  This computation
motivates the definition of the following rescaled equation.

\begin{defn} For a given $f^0\in\ell^1(\Z,\muh\dvh)$, we shall say that a
  function $f \in \ccc^0(\R^+, \ell^1(\Z, \muh\dvh))$ satisfies the (scaled)
  semi-discrete homogeneous Fokker--Planck equation if
  \begin{equation} \label{eq:DHFPf} {\D_t} f + ( - \Dvs +\vs) \Dv f = 0, \qquad
    f|_{t=0} = f^0,
  \end{equation}
  in the sense of distributions.
\end{defn}

\noindent
With the definitions and computations above, $F$ is a solution of the flat
semi-discrete Fokker--Planck equation \eqref{eq:DHFPF} if and only if $f$ defined
by $F = \muh + \muh f$ is a solution of the scaled semi-discrete Fokker--Planck
equation \eqref{eq:DHFPf}.

Just as we recalled in the continuous velocity setting in Section
\ref{sec:homogeneous}, the next step in the discrete velocity setting is to find
a suitable subspace of $\ell^1(\Z, \muh\dvh)$, with a Hilbertian structure, in
which the non-negativity property of the collision operator is satisfied.  We
mimic the continuous case and choose the space
$ \ell^2(\Z, \muh \dvh) \hookrightarrow \ell^1(\Z, \mu^h \dvh)$ denoted for
short $\Ldeuxmudvh$.

\begin{defn} \label{def:homd} We define the space $\Ldeuxmudvh$ to be the Hilbertian
  subspace of $\R^\Z$ of sequences $g$ such that
$$
\norm{ g}_{\Ldeuxmudvh }^2 \defegal \dvh \sum_{i\in \Z} (g_i)^2 \muh_i <\infty.
$$
This defines a Hilbertian norm, and the related scalar product will be denoted
by $\seq{ \cdot, \cdot}$. For $g \in \Ldeuxmudvh$, we also define
$$
\seq{g} \defegal \sum_{i\in \Z} g_i \mu_i^h \dvh = \seq{g, 1}_{\Ldeuxmudvh },
$$
the mean of $g$ (with respect to this weighted scalar product).
\end{defn}

\noindent
In order to give achieve a useful functional framework for the (scaled)
homogeneous Fokker-Planck equation \eqref{eq:DHFPf} in this discrete velocity
setting, we introduce now a shifted Maxwellian $\mus \in \ell^1(\Z^*, \dvh)$ and
a new suitable Hilbert subspace that appears naturally in the computations:

\begin{defn} \label{def:homds} Let us define $\mus \in \ell^1(\Z^*, \dvh)$ by
 $$
 \mus_i = \muh_{i+1} \textrm{ for } i<0, \qquad \mus_i = \muh_{i-1} \textrm{ for
 } i>0.
$$
We define the space $\Ldeuxmudvs$ to be the subspace of $\R^{\Z^*}$ of sequences
$g \in \ell^1(\Z^*, \mus\dvh)$ such that
$$
\norm{g}_{\Ldeuxmudvs }^2 \defegal \dvh \sum_{i \in \Z^*} (g_i)^2 \mus_i
<\infty.
$$
This defines a Hilbertian norm, and the related scalar product will be denoted
by $\seq{ \cdot, \cdot}_\sharp$.  Eventually, we define
$$
\Hunmudvh = \set{ g \in \Ldeuxmudvh, \textrm{ s.t. } \Dv g \in \Ldeuxmudvs }.
$$
\end{defn}

\begin{remark}
  In contrast to the classical finite differences setting where the
  discretizations of ${\D_v}$ give rise to {\it bounded} linear operators (with
  continuity constants of size $1/\dvh$), the above definition makes $D_v$ an
  {\it unbounded} linear operator from $\Ldeuxmudvh$ to $\Ldeuxmudvs$, with
  domain $\Hunmudvh$.  Moreover, the multiplication operator $\vs$ is a {\it
    bounded} linear operator from $\Ldeuxmudvs$ to $\Ldeuxmudvh$, with constant
  of size $1/\dvh$.
\end{remark}

We now summarize the structural properties of Equation \eqref{eq:DHFPf} and the
involved operator in the following Proposition:

\begin{prop}\label{prop:hstruct} The following properties hold true for all $\dvh>0$.
  \begin{enumerate}
  \item Let us consider $ \Pd = ( - \Dvs +\vs) \Dv$ with domain
$$
D(\Pd) =\set{g \in \Ldeuxmudvh, \ | \ ( - \Dvs +\vs) \Dv f \in \Ldeuxmudvh}.
$$
Then $\Pd$ is self-adjoint non-negative with dense domain and is maximal
accretive in $\Ldeuxmudvh $. Moreover, for all $h \in \Ldeuxmudvs$,
$g \in \Ldeuxmudvh $ for which it makes sense
\begin{equation} \label{eq:of} \seq{ ( - \Dvs +\vs) h, g} = \seq{h, \Dv g}_\sharp,
  \quad \textrm{ and } \quad
  \seq{ ( - \Dvs +\vs) \Dv g, g} = \norm{\Dv g}_\Ldeuxmudvs^2.
\end{equation}
\item For an initial data $f^0
  \in D(\Pd)$, there exists a unique solution of \eqref{eq:DHFPf} in $\ccc^1(\R^+,
  \Ldeuxmudvh) \cap \ccc^0(\R^+,
  D(\Pd))$, and the associated semi-group naturally defines a solution in
  $\ccc^0(\R^+, \Ldeuxmudvh)$ when $f^0\in\Ldeuxmudvh$.
\item The preceding properties remain true if we consider operator $\Pd$
  in $\Hunmudvh$
  with domain $D_\Hunmudvh(\Pd)$.
  In particular it defines a unique solution of \eqref{eq:DHFPf} in $\ccc^1(\R^+,
  \Hunmudvh) \cap \ccc^0(\R^+,D_\Hunmudvh(\Pd)) $ if $f^0 \in
  D_\Hunmudvh(\Pd)$ and a semi-group solution $f \in \ccc^0(\R^+,
  \Hunmudvh)$ if $f^0 \in \Hunmudvh$.
\item Constant sequences are the only equilibrium states of equation
  \eqref{eq:DHFPf} and the evolution preserves the mass $\seq{f(t)}
  = \seq{f^0}$ for all $t\geq 0$.
\end{enumerate}
\end{prop}

\begin{proof}
  The proof of the second equality in \eqref{eq:of} is a direct consequence of the
  first equality there, and leads directly to the self-adjointness and the
  non-negativity of $( - \Dvs +\vs) \Dv$.

  The proof of the first equality in \eqref{eq:of} is very similar to the one of
  \eqref{eq:double} but we propose it for completeness.  We write for
  $h \in \Ldeuxmudvs$ and $g \in \Ldeuxmudvh$ with finite supports
  \begin{equation} \label{eq:calcof}
    \begin{split}
      \dvh^{-1} \seq{ ( - \Dvs +\vs)  h, g} & =  \sum_i ((-\Dvs +\vs) h)_i g_i \mu_i  \\
      & = \sum_{i>0} ((-\Dvs +\vs) h)_i g_i \mu_i -(\Dvs h)_0 g_0 \mu_0
      +\sum_{i<0} ((-\Dvs +\vs) h)_i g_i \mu_i
    \end{split}
  \end{equation}
  The first term in the last right hand side of \eqref{eq:calcof} reads
  \begin{equation*}
    \begin{split}
      & \sum_{i>0} ((-\Dvs +\vs) h)_i g_i \mu_i \\
      & = \sum_{i>0} \sep{ -\frac{h_{i+1} - h_i}{\dvh} + v_i h_i} g_i \mu_i \\
      &  = \sum_{i>0} h_i \sep{ \frac{-g_{i-1}\mu_{i-1} + g_i \mu_i }{\dvh} 
+ v_i g_i \mu_i} + \frac{h_1 g_0}{\dvh} \mu_0 \\
      & = \sum_{i>0} h_i g_i \sep{ \frac{-\mu_{i-1} +  \mu_i }{\dvh} + v_i
        \mu_i}  
+ \sum_{i>0} h_i \sep{ - \frac{g_{i-1}- g_i}{\dvh}}\mu_{i-1}   + \frac{h_1 g_0}{\dvh} \mu_0 \\
      & = \sum_{i>0} h_i (\Dv g)_i \mu_{i-1} + \frac{h_1 g_0}{\dvh} \mu_0,
    \end{split}
  \end{equation*}
  where for the last equality we used the fact that $(\Dv+ v)\muh = 0$.
  Similarly for the third term in the last right hand side of \eqref{eq:calcof}, we
  get
  \begin{equation*}
    \begin{split}
      & \sum_{i<0} ((-\Dvs +\vs) h)_i g_i \mu_i \\
      & = \sum_{i<0} \sep{ -\frac{h_{i} - h_{i-1}}{\dvh} + v_i h_i} g_i \mu_i \\
      &  = \sum_{i<0} h_i \sep{ \frac{-g_i \mu_i +g_{i+1}\mu_{i+1} }{\dvh} 
+ v_i g_i \mu_i} - \frac{h_{-1} g_0}{\dvh} \mu_0 \\
      & = \sum_{i<0} h_i g_i \sep{ \frac{-\mu_{i} +  \mu_{i+1} }{\dvh} 
+ v_i  \mu_i}  + \sum_{i<0} h_i \sep{ - \frac{g_{i+1}- g_i}{\dvh}}\mu_{i+1}   
- \frac{h_{-1} g_0}{\dvh} \mu_0 \\
      & = \sum_{i<0} h_i (\Dv g)_i \mu_{i+1} - \frac{h_{-1} g_0}{\dvh} \mu_0.
    \end{split}
  \end{equation*}
  The center term in \eqref{eq:calcof} is then
  \begin{equation*}
    -(\Dvs h) g_0 \mu_0 = -\frac{ h_1- h_{-1}}{\dvh} g_0 \mu_0.
  \end{equation*}
  Therefore the sum of the 3 terms in the last right hand side of \eqref{eq:calcof}
  reads
$$
\dvh^{-1} \seq{ ( - \Dvs +\vs) h, g} = \sum_{i>0} h_i (\Dv g)_i \mu_{i-1} +
\sum_{i<0} h_i (\Dv g)_i \mu_{i+1} = \dvh^{-1} \seqs{h, \Dv g},
$$
since the boundary terms disappear. This is the first equality in \eqref{eq:of}.

Concerning the functional analysis and existence of solutions, we observe that
the maximal accretivity of $( - \Dvs +\vs) \Dv$ in both $\Ldeuxmudvh$ and
$\Hunmudvh$ is then direct to get.  In particular, the non-negativity in
$\Hunmudvh$ follows from the following identity for $g \in D_\Hunmudvh(\Pd)$:
$$
\seq{ \Dv( - \Dvs +\vs) \Dv g, \Dv g} = \norm{( - \Dvs +\vs)\Dv g}_\Ldeuxmudv^2
\geq 0.
$$
The fact that the equation is well-posed is then a direct consequence of the
Hille--Yosida Theorem.  The fact that constant sequences are the only
equilibrium solutions comes from the fact that for any solution
$f\in \ccc^1(\R^+, \Hunmudvh)$,
$$
\ddt \norm{f}^2 = - \norms{\Dv f}^2,
$$
and the preservation of mass comes from the fact that
$$
{\D_t} \seq{f} = \seq{(-\Dv +v)\Dv f, 1} = \seq{\Dv f, \Dv 1} = 0,
$$
for any solution $f$ such that $f^0 \in D(\Pd)$, and then in general by density
of $D(\Pd)$ in $\Ldeuxmudvh$.  The proof is complete. \end{proof}

\bigskip As in the continuous case, the Poincar\'e inequality is a fundamental
tool to prove the exponential convergence of the solution.  It appears that such
an inequality is true with $\norm{ \cdot }_{\Ldeuxmudvs }^2$ in the right-hand
side, even though the index $0$ is missing in the definition of this norm.

\begin{prop}[Discrete Poincar\'e inequality] \label{prop:poindiscrete} Let
  $g\in\Hunmudvh$.  Then,
  $$\norm{g-\seq{g}}^2_{\Ldeuxmudvh } \leq \norm{ \Dv g}_{\Ldeuxmudvs }^2.$$
\end{prop}

\begin{proof}
  We essentially follow the proof of the continuous case done before in Section
  \ref{subsec:homogeneouscontinuous}.  Let us take $g\in \Hunmudvh$.  Replacing if
  necessary $g$ by $g-\seq{g}$, it is sufficient to prove the result for
  $\seq{g}=0$. We first note that, with the normalization
  \eqref{eq:normalizationdemudv} of $\muh$, we have
  \begin{equation*}
    \begin{split}
      \dvh^{-1} \norm{g}^2 = \sum_i g_i^2 \muh_i & = \frac{\dvh}{2} \sum_{i,j} (
      g_j-g_i)^2 \muh_i \muh_j = \dvh \sum_{i<j} ( g_j-g_i)^2 \muh_i \muh_j,
    \end{split}
  \end{equation*}
  since
  $2 \sum_{i,j} g_i g_j \muh_i \muh_j = 2 \sum_i g_i \muh_i \sum_j g_j \muh_j =
  0$
  implies that the diagonal terms are zero. Now for $i<j$, we can write the
  telescopic sum
$$
g_j-g_i = \sum_{\ell=i+1}^j (g_\ell-g_{\ell-1}),
$$
so that
\begin{equation} \label{eq:interm}
  \begin{split}
    \dvh^{-1}\sum_i g_i^2 \muh_i & = \sum_{i < j} \sep{ \sum_{\ell=i+1}^j
      (g_\ell-g_{\ell-1}) }^2 \muh_i \muh_j \leq \sum_{i < j} \sep{
      \sum_{\ell=i+1}^j (g_\ell-g_{\ell-1})^2 } (j-i) \muh_i \muh_j,
  \end{split}
\end{equation}
where we used the discrete flat Cauchy--Schwarz inequality.  Let us now
introduce $\G$ a discrete anti-derivative of $(g_\ell-g_{\ell-1})^2$, for
example this one:
$$
\G_j = - \sum_{\ell=j+1}^{-1} (g_\ell-g_{\ell-1})^2 \textrm{ for } j\leq -1,
\qquad \G_j = \sum_{\ell=0}^j (g_\ell-g_{\ell-1})^2 \textrm{ for } j\geq 0,
$$
so that for all $i<j$ we have
$\G_j - \G_i = \sum_{\ell=i+1}^{j} (g_\ell-g_{\ell-1})^2$.  We infer from
\eqref{eq:interm}
\begin{equation*}
  \begin{split}
    \dvh^{-1}\sum_i g_i^2 \muh_i & \leq \sum_{i < j} \sep{ \G_j-\G_i } (j-i)
    \muh_i \muh_j = \frac{1}{2} \sum_{i,j} \sep{ \G_j-\G_i } (j-i) \muh_i
    \muh_j,
  \end{split}
\end{equation*}
where in the last equality we used that
$\sep{\G_j-\G_i } (j-i) = \sep{\G_i-\G_j } (i-j)$ and the fact that the diagonal
terms vanish.  We can now split the last sum into four parts:
\begin{equation*}
  \begin{split}
    \dvh^{-1}\sum_i g_i^2 \muh_i & \leq \frac{1}{2} \sep{
      \sum_{i,j} \G_j j  \muh_i \muh_j +\sum_{i,j} \G_i i \muh_i \muh_j
      -\sum_{i,j} 
\G_i j \muh_i \muh_j - \sum_{i,j} \G_j i \muh_i \muh_j} \\
    & \leq \dvh^{-1} \sum_{i} \G_i i \muh_i = \dvh^{-1}\sum_{i \neq 0} \G_i i
    \muh_i,
  \end{split}
\end{equation*}
where we used the discrete Fubini Theorem and the fact that
$\sum_{j} j\muh_j = 0$ and $\dvh\sum_j \muh_j = 1$
(and their counterparts in variable $i$), by parity and normalization of $\muh$.
The last step is to perform a discrete integration by part (Abel transform)
using deeply the functional equation \eqref{eq:relationsmudv} satisfied by
$\muh$ that we recall now :
$$
i\muh_i = -\frac{\muh_i- \muh_{i-1}}{\dvh^2} \textrm{ for } i>0, \qquad i\muh_i
= -\frac{\muh_{i+1} - \muh_{i}}{\dvh^2} \textrm{ for } i<0.
$$
We therefore get
\begin{equation*}
  \begin{split}
    \sum_{i \neq 0} \G_i i \muh_i& = \sum_{i > 0} \G_i i \muh_i
    +  \sum_{i < 0} \G_i i \muh_i\\
    & = -  \sum_{i > 0} \G_i \frac{\muh_i- \muh_{i-1}}{\dvh^2}  
-  \sum_{i <0 } \G_i \frac{\muh_{i+1} - \muh_{i}}{\dvh^2} \\
    & = - \sum_{i > 0} \frac{ \G_i - \G_{i+1}}{\dvh^2} \muh_i +
    \frac{\G_1}{\dvh^2} \muh_0 - \sum_{i < 0} \frac{ \G_{i-1} - \G_{i}}{\dvh^2}
    \muh_i - \frac{\G_{-1}}{\dvh^2} \muh_0.
  \end{split}
\end{equation*}
Now, using the definition of $\G$ and in particular the fact that
$$
\G_1 - \G_{-1} = (g_{1}- g_0)^2 + (g_{0}- g_{-1})^2,
$$
we obtain
\begin{equation} \label{eq:sommemirac}
  \begin{split}
    \sum_{i \neq 0} \G_i i \muh_i
    & = \sum_{i > 0}  \sep{ \frac{ g_{i+1}- g_i}{\dvh}}^2 \muh_i 
+  \sum_{i < 0}  \sep{ \frac{ g_{i}- g_{i-1}}{\dvh}}^2 \muh_i \\
    & \qquad \qquad \qquad \qquad + \sep{\frac{ g_{1}- g_0}{\dvh}}^2 \muh_0 +
    \sep{\frac{ g_{0}- g_{-1}}{\dvh}}^2
    \muh_0 \\
    & = \dvh^{-1} \norm{ \Dv g}_{\Ldeuxmudvs }^2
  \end{split}
\end{equation}
and therefore $ \norm{g}_{\Ldeuxmudvh }^2 \leq \norm{\Dv g}_{\Ldeuxmudvs }^2$.
The proof is complete. \end{proof}

We can now study the exponential convergence to the equilibrium in the spaces
$\Ldeuxmudvh$ and $\Hunmudvh$ of the solution $f$ of \eqref{eq:DHFPf}, for
$f^0\in \Ldeuxmudvh$ and $f^0\in \Hunmudvh$ respectively.
As in the continuous case of Section \ref{subsec:homogeneouscontinuous}, we
propose two different entropies well-adapted to the coming discretization case:
$$
\fffd(g) = \norm{g}_{\Ldeuxmudvh }^2, \qquad \gggd(g) = \norm{g}_{\Ldeuxmudvh
}^2 + \norm{\Dv g}_{\Ldeuxmudvs}^2,
$$
defined for $g\in\Ldeuxmudvh$ and $g\in\Hunmudvh$ respectively.

Our result for the exponential convergence to equilibrium of the exact solution
of the discrete evolution equation \eqref{eq:DHFPf} is the following.

\begin{thm} Let $f^0 \in \Ldeuxmudvh$ such that $\seq{f^0} = 0$ and let $f$ be
  the solution of \eqref{eq:DHFPf} (in the semi-group sense) in
  $\ccc^0(\R^+, \Ldeuxmudvh)$ with initial data $f^0$.  Then for all $t\geq0$,
  \begin{equation*}
    \fffd(f(t)) \leq \exp^{-2t}\fffd(f^0).
  \end{equation*}
  If in addition $f^0 \in \Hunmudvh$ and $f$ is the semi-group solution in
  $f \in \ccc^0(\R^+, \Hunmudvh)$, then for all $t\geq 0$
$$
\gggd(f(t)) \leq \exp^{-t}\gggd(f^0).
$$
\end{thm}

\begin{proof}
  We follow the steps of the proof of Theorem
  \ref{thm:exponentialtrendtoequilibrium}. In particular we take
  $f^0 \in D_\Hunmudvh(\Pd)$ in all the computations below, so that the
  computations and differentiations below are authorized, and the Theorem is
  then a consequence of the density of $D_\Hunmudvh(\Pd)$ in $\Ldeuxmudvh$ or
  $\Hunmudvh$.

  For the first entropy, we have, using \eqref{eq:DHFPf}, \eqref{eq:of}, and
  Proposition \ref{prop:poindiscrete},
$$
\ddt \fffd(f) = -2 \seq{ (-\Dv^\sharp+\vs)D_v f , f}
= -2 \norm{\Dv f}_{\Ldeuxmudvs }^2 \leq -2 \norm{f}^2
= -2 \fffd(f).
$$
Now we deal with the second entropy $\gggd$.  We use the discrete Poincar\'e
inequality of Proposition \ref{prop:poindiscrete} and the same splitting
$$
2 \norm{\Dv f}_{\Ldeuxmudvs }^2 = \norm{\Dv f}_{\Ldeuxmudvs }^2 + \norm{\Dv
  f}_{\Ldeuxmudvs }^2 \geq \norm{\Dv f}_{\Ldeuxmudvs }^2 + \norm{f}^2,
$$
as in the proof of Theorem \ref{thm:exponentialtrendtoequilibrium}.  We get next
from equations \eqref{eq:DHFPf} and \eqref{eq:of}
\begin{equation*}
  \begin{split}
    \ddt \gggd(f) & = -2 \seq{ (-\Dvs+\vs) \Dv f , f}_{\Ldeuxmudvh }
    -2 \seq{ \Dv(-\Dvs+\vs) \Dv f , \Dv f}_{\Ldeuxmudvs } \\
    & = -2 \norm{\Dv f}_{\Ldeuxmudvs }^2 - 2 \norm{(-\Dvs+\vs) \Dv f}_{\Ldeuxmudvh }^2\\
    & \leq - \norm{f}_{\Ldeuxmudvh }^2 - \norm{\Dv f}_{\Ldeuxmudvs }^2 - 2
    \norm{(-\Dvs+\vs) \Dv f}_{\Ldeuxmudvh }^2 \leq - \gggd(f).
  \end{split}
\end{equation*}
The proof is complete. \end{proof}

As in the Corollary \ref{cor:exponentialtrendtoequilibrium} we therefore
immediately get

\begin{cor} Let $f^0 \in \Ldeuxmudvh$ and let $f$ be the solution of
  \eqref{eq:DHFPf} in $\ccc^0(\R^+, \Ldeuxmudvh)$ with initial data $f^0$.  Then
  for all $t\geq 0$,
  \begin{equation*}
    \norm{f(t) - \seq{f^0}}_\Ldeuxmudvh \leq \exp^{-t}\norm{f^0 - \seq{f^0}}_\Ldeuxmudvh.
  \end{equation*}
  If in addition $f^0 \in \Hunmudvh$ then $f \in \ccc^0(\R^+, \Hunmudvh)$ and we
  have
$$
\norm{f(t) - \seq{f^0}}_\Hunmudvh \leq \exp^{-\frac{t}{2}}\norm{f^0 -
  \seq{f^0}}_\Hunmudvh.
$$
\end{cor}

\subsection{Remark on the full discretization}\label{subsec:fulld}

A full discretization of the preceding equation \eqref{eq:HFPf}
is of course possible, using the velocity discretization introduced in
this section, and, for example the implicit Euler scheme
$$
f^{n} = f^{n+1}  -\dth ( - \Dvs +\vs) \Dv f^{n+1}.
$$
In order to describe the long time behavior of such a fully discretized
scheme, the functional framework introduced in this Section can be used,
and similar arguments work to obtain exponential convergence
to equilibrium\footnote{Note anyway that the explicit Euler scheme
$$ f^{n+1} = f^n  -\dth ( - \Dvs +\vs) \Dv f^n ,$$
is not well posed due to the fact that the discretized operator
$( - \Dvs +\vs) \Dv $ is not bounded.}.  We do not present in this paper the
corresponding statements and results since they are actually not difficult to
obtain, and may be thought as very simple versions of the results of the
following sections.  Indeed, we shall focus on the discretization on the full
inhomogeneous equation \eqref{eq:IHFPF} in Section \ref{sec:eqinhomo} and on the
discretization of the homogeneous and inhomogeneous equations \eqref{eq:HFPF}
and \eqref{eq:IHFPF} on a bounded velocity domain with Neumann conditions (in
velocity) in Sections \ref{sec:homobounded} and
\ref{sec:eqinhomoboundedvelocity}.

\section{The inhomogeneous equation in space,
velocity and time}
\label{sec:eqinhomo}

In this Section, we deal with the inhomogeneous equation \eqref{eq:IHFPF}
with velocity domain $\R$ and its discretized versions.
We present the fully continuous analysis in the first subsection.
Then, we study in Subsection \ref{subsec:discretentcontinuenxv}
the semi-discretization in time by the implicit Euler scheme.
Afterwards, we focus in Subsection \ref{sec:inhomsd}
on the semi-discretization in space and velocity only.
In particular, we introduce part of the material that will be needed
in the final study of the fully-discretized implicit Euler scheme
which is considered in Subsection \ref{subsec:eqinhomototalementdiscretisee},
where we prove Theorem \ref{thm:eulerimplicite}.

\subsection{The fully continuous analysis}\label{sec:inhomc}
In this subsection we recall briefly now standard results about the original
inhomogeneous Fokker-Planck equation with unknown $F(t,x,v)$ with $(t,x,v) \in\R^+\times\T\times\R$ and where $\T=[0,1]_{\rm per}$. The equation reads
\begin{equation*}
  {\D_t} F+v{\D_x}F-{\D_v} ({\D_v}+v)F=0, \qquad F|_{t=0} = F^0,
\end{equation*}
We assume that the initial density $F^0$ is non-negative, in $L^1(\T\times\R)$,
and satisfies $\int_{\T\times\R}F^0 \dx\dv=1$.
We directly check that
$(x,v) \longmapsto \mu(v) $ is an equilibrium of the equation,
and we shall continue to denote this function $\mu$
(in the sense that it is now a constant function w.r.t. the variable $x$).
As in the homogeneous case, it is convenient to work in the subspace
$\mu \Ldeuxmudvdx \hookrightarrow L^1(\dv\dx)$ and take benefit
of the associated Hilbertian structure.
We therefore pose for the following  $f=(F-\mu)/\mu$,
and we perform here the analysis for $f\in \Ldeuxmudvdx$
as we did in $\Ldeuxmudv$ in the homogeneous case in Section
\ref{sec:homogeneous}.
The rescaled equation writes
\begin{equation}
  \label{eq:rescaled}
   {\D_t} f+v{\D_x}f+(-{\D_v}+v){\D_v} f=0, \qquad \qquad f|_{t=0} = f^0.
\end{equation}
The non-negativity of the associated operator $P = v{\D_x}+(-{\D_v}+v){\D_v} $
is straightforward since $v{\D_x}$ is skew-adjoint in $\Ldeuxmudvdx$.  The
maximal accretivity of this operator in $\Ldeuxmudvdx$ or $\Hunmudvdx$ is not so
easy and we refer for example to \cite{HelN04}.  As in the homogeneous case,
using the Hille--Yosida Theorem, this implies that for an initial datum
$f^0 \in D(P)$ (resp. $D_\Hunmudvdx(P)$) there exists a unique solution in
$\ccc^1( \R^+, \Ldeuxmudvdx) \cap \ccc^0( \R^+, D(P)) $ (resp.
$\ccc^1(\R^+, \Hunmudvdx) \cap \ccc^0( \R^+, D_\Hunmudvdx(P) )$. As before we
will call semi-group solution the function in $\ccc^0(\R^+, \Ldeuxmudvdx)$
(resp. $\ccc^0(\R^+, \Hunmudvdx)$) given by the semi-group associated to $P$
with the suitable domain.

From now on, the norms and scalar products without subscript are taken
in $\Ldeuxmudvdx$.

As in the homogeneous case, we shall define an entropy
adapted to the $\Hunmudvdx$ framework.
Its exponential decay, however, is a bit more difficult to prove
in the inhomogeneous case.
As consequence of the maximal accretivity, we first note that,
for $f^0\in D_\Hunmudvdx(P)$,
along the corresponding solution of \eqref{eq:rescaled}, we have
$$
\ddt \norm{f}^2 = -2 \seq{ v{\D_x} + (-{\D_v}+v){\D_v} f, f} = -2 \norm{D_v f}^2 \leq 0,
$$
so that $g\mapsto \norm{g}^2$ is an entropy of the system.
Such an inequality is nevertheless not strong or precise enough to get
an exponential decay.
In order to prepare for the discrete cases in the next sections,
we again introduce and recall a particularly simple entropy
leading to the result.

For $C>D>E>1$ to be precised later,
the  modified entropy is
defined for $g\in \Hunmudvdx$ by
\begin{equation}
  \label{eq:entropyfunc}
  \hhh(g) \defegal  C\norm{g}^2+D\norm{{\D_v} g}^2+E\seq{{\D_v} g,{\D_x} g}+\norm{{\D_x}g}^2.
\end{equation}
We will show later that for well chosen $C,D,E$, $t\mapsto \hhh(f(t))$ is
exponentially decreasing when $f$ solves the rescaled equation
\eqref{eq:rescaled} with initial datum $f^0\in \Hunmudvdx$. As a norm in
$\Hunmudvdx$ we choose the standard one defined for $g\in \Hunmudvdx$ by
$$
\norm{g}_\Hunmudvdx \defegal \left(\norm{g}^2 + \norm{{\D_v} g}^2
+ \norm{{\D_x} g}^2\right)^\frac12.
$$
\bigskip
We first prove that $\sqrt{\hhh}$ is equivalent to the $\Hunmudvdx$-norm.
\begin{lem}\label{lem:equiv}
Assume $C>D>E>1$ are given such that $E^2<D$.
For all $g\in \Hunmudvdx$, one has
  \begin{equation*}
    \dfrac{1}{2}\norm{g}_{\Hunmudvdx}^2\leq\hhh(g)\leq 2C\norm{g}_{\Hunmudvdx}^2.
  \end{equation*}
\end{lem}
\begin{proof}
Using a standard Cauchy--Schwarz--Young
inequality, we observe that
  \begin{equation*}
    2\abs{E\seq{{\D_v} g,{\D_x} g}}\leq E^2\norm{{\D_v} g}^2+\norm{{\D_x} g}^2,
  \end{equation*}
which implies for all $g\in \Hunmudvdx$
\begin{multline*}
  \underbrace{C}_{1/2\leq}\norm{g}^2
  +\underbrace{(D-E^2/2)}_{1/2\leq D/2\leq}\norm{{\D_v} g}^2
  +\dfrac{1}{2}\norm{{\D_x}g}^2 \\
  \leq \hhh(g)\leq
  C\norm{g}^2+\underbrace{(D+E^2/2)}_{\leq D+D/2\leq 3C/2\leq
    2C}\norm{{\D_v} g}^2+\underbrace{3/2}_{\leq 3C/2\leq 2C} \norm{{\D_x}g}^2,
\end{multline*}
which in turn implies the result since $E^2<D$.
\end{proof}

As in the homogeneous case, one of the main ingredients to prove the
exponential decay is again a Poincar\'e inequality,
which is essentially obtained by tensorizing
the one in velocity with the one in space.
In the following, we denote the mean of  $g\in \Ldeuxmudvdx$ with respect
to all variables by
 $$
 \seq{g} \defegal \iint g(x,v) \mudv\dx.
 $$

\begin{lem}[Inhomogeneous Poincar\'e inequality]
\label{lem:fullPoincarecontinu}
For all $g \in \Hunmudvdx$, we have
 \begin{equation*}
  \norm{g-\seq{g}}^2  \leq  \norm{{\D_v} g}^2 + \norm{{\D_x} g}^2.
\end{equation*}
\end{lem}
\begin{proof}
  Replacing if necessary $g$ by $g-\seq{g}$, it is sufficient to prove the
  result for $\seq{g}=0$.  For convenience, we introduce
  $\rho: x\mapsto \int g(x,\cdot) \mudv$, the macroscopic density of
  probability. Recall the standard Poincar\'e inequality in space only
$$
\norm{\rho}^2 \leq \frac{1}{4\pi^2}\norm{{\D_x} \rho}^2\leq \norm{{\D_x} \rho}^2,
$$
which is a consequence of the the fact that the torus $\T$ is compact
and the fact that $\int \rho \dx = \iint g \mu \dv\dx = 0$
(note that the proof of this last Poincar\'e inequality is very standard
and could be done following the method employed in the proof of
Lemma \ref{lem:Poincarecontinu}).
Now we observe that orthogonal projection properties and Fubini Theorem imply
$$
\norm{\rho}^2_\Ldeuxdx \leq \norm{g}^2
\qquad \text{and}\qquad
\norm{{\D_x}\rho}^2_\Ldeuxdx \leq \norm{{\D_x} g}^2,
$$
since $(x,v) \mapsto \rho(x)$ (resp. $(x,v) \mapsto {\D_x}\rho(x)$)
is the orthogonal projection of $g$ (resp. ${\D_x} g$) onto the closed space
$$
\set{ (x,v) \longmapsto  \phi(x) \ | \ \phi \in \Ldeuxdx},
$$
and
$\norm{\phi \otimes 1 }_{} = \norm{\phi}_\Ldeuxdx$
for all $\phi \in \Ldeuxdx$ since we are in probability spaces
(there is a natural injection $\Ldeuxdx \hookrightarrow \Ldeuxmudvdx$
of norm $1$).
Using the Fubini Theorem again, we also directly get from
Lemma \ref{lem:Poincarecontinu} that
$$
\norm{g- \rho\otimes 1}_{}^2 \leq  \norm{{\D_v} g}_{}^2.
$$
We therefore can write, using orthogonal projection properties again, that
\begin{equation}
\begin{split}
\norm{g}_{}^2 & = \norm{g- \rho\otimes 1}_{}^2 + \norm{ \rho\otimes 1}_{}^2 \\
& = \norm{g- \rho\otimes 1}_{}^2 + \norm{\rho}_\Ldeuxdx^2 \\
& \leq \norm{{\D_v} g}_{}^2 + \norm{{\D_x} \rho}_\Ldeuxdx^2 \\
& \leq \norm{{\D_v} g}_{}^2 + \norm{{\D_x} g}_{}^2.
\end{split}
\end{equation}
The proof is complete. \end{proof}

\noindent
For convenience, we will sometimes denote in the following
$$
  \Xzero = v {\D_x},
$$
so that the equation \eqref{eq:rescaled} satisfied by $f$ is
${\D_t} f = -\Xzero f -(-{\D_v} +v) {\D_v} f$.
We shall use intensively the fact that $\Xzero$ is skew-adjoint
and the formal adjoint of $(-{\D_v} + v)$ is ${\D_v}$,
together with the commutation relations
\begin{equation}
\label{eq:commrel}
  \adf{{\D_v}, X_0} = {\D_x}, \qquad \adf{{\D_x}, X_0} = 0,
  \qquad\text{and}\qquad
  \adf{{\D_v},(-{\D_v}+v)}=1.
\end{equation}

\begin{thm}\label{thm:decrexp}
  Assume that $C>D>E>1$ satisfy $E^2<D$ and $(2D+E)^2<2C$.
Let $ f^0 \in \Hunmudvdx$ such that $\seq{f^0} = 0$ and let $f$ be the solution
 (in the semi-group sense)  in $\ccc^0(\R^+, \Hunmudvdx)$ of Equation \eqref{eq:rescaled}. Then for all $t\geq 0$,
  \begin{equation*}
     \hhh(f(t))\leq \hhh(f^0)\exp^{-2\kappa t}.
  \end{equation*}
  with $2\kappa = \frac{E}{8C}$.
\end{thm}
\begin{proof}
We suppose that $f^0 \in D_\Hunmudvdx(P)$ and we consider
the corresponding solution $f$ of \eqref{eq:rescaled}
in $\ccc^1(\R^+, \Hunmudvdx) \cap \ccc^0(\R^+, D_\Hunmudvdx(P))$ with initial
datum $f^0$.
The theorem for a general $f^0\in\Hunmudvdx$
is then a consequence of the density
of $D_\Hunmudvdx(P)$ in $\Hunmudvdx$.

We compute separately the time derivatives of the four terms defining
$\hhh(f(t))$.
Omitting the dependence on $t$,
the time derivative of the first term in $\hhh(f(t))$ reads
  \begin{align*}
    \ddt \norm{f}^2_{\Ldeuxmudvdx}&
=2\seq{{\D_t}f,f}=-2\underbrace{\seq{\Xzero f,f}}_{=0}-2\seq{(-{\D_v}+v){\D_v} f,f}\\
&=-2\norm{{\D_v} f}_{\Ldeuxmudvdx}^2.
  \end{align*}
The second term writes
\begin{align*}
  \ddt \norm{{\D_v} f}^2_{\Ldeuxmudvdx}&=2\seq{{\D_v}({\D_t}f),{\D_v} f}\\
&=-2\seq{{\D_v}(\Xzero f+(-{\D_v}+v){\D_v} f),{\D_v} f}\\
&=-2\underbrace{\seq{\Xzero {\D_v} f,{\D_v} f}}_{=0}
-2\seq{\adf{{\D_v},\Xzero }f,{\D_v} f}-2\seq{{\D_v}(-{\D_v}+v){\D_v} f,{\D_v} f}.
\end{align*}
We again use the fact that $\Xzero $ is a skew-adjoint operator in
$\Ldeuxmudvdx$ and the fundamental relation
$\adf{{\D_v},\Xzero }={\D_x}$ and we get
\begin{equation*}
  \ddt \norm{{\D_v} f}^2_{\Ldeuxmudvdx}=-2 \seq{{\D_x}f,{\D_v} f}-2\norm{(-{\D_v}+v){\D_v} f}^2.
\end{equation*}
The time derivative of the third term can be computed as follows
\begin{align*}
   \ddt \seq{{\D_x}f,{\D_v} f}
   = & -\seq{{\D_x}(\Xzero f+(-{\D_v}+v){\D_v} f),{\D_v} f}
-\seq{{\D_x}f,{\D_v}(\Xzero f+(-{\D_v}+v){\D_v} f)}\\
    = &  -\seq{ {\D_x} \Xzero f,{\D_v} f}
              -\seq{ {\D_x} f,{\D_v} \Xzero f} \qquad (I) \\
     &      -\seq{{\D_x} (-{\D_v}+v){\D_v} f, {\D_v} f}
          -\seq{{\D_x} f, {\D_v} (-{\D_v}+v){\D_v} f}. \qquad (II)
\end{align*}
For the term $(I)$ we use the fact that $\Xzero$ is skew-adjoint
and the commutation relations \eqref{eq:commrel} to obtain
\begin{align*}
(I) =  & \underbrace{-\seq{ \Xzero {\D_x} f,{\D_v} f}
              -\seq{ {\D_x} f, \Xzero {\D_v} f}}_{0} -\seq{ {\D_x} f, 
\adf{ {\D_v}, \Xzero} f}  = - \norm{{\D_x} f}^2.
\end{align*}
For the term $(II)$ we use that the adjoint of ${\D_v}$ is $-({\D_v} +v)$ and
 the one of ${\D_x}$ is $-{\D_x}$  and we get
\begin{align*}
(II) = &  \seq{ (-{\D_v}+v){\D_v} f,{\D_x} {\D_v} f}
          + \seq{{\D_v} (-{\D_v}+v) f,{\D_x}{\D_v} f} \\
       =    &  2\seq{ (-{\D_v}+v){\D_v} f,{\D_x} {\D_v} f} +
             \seq{\adf{ {\D_v} , (-{\D_v}+v)}  f,{\D_x}{\D_v} f}.
\end{align*}
Now  the commutation relation \eqref{eq:commrel} yields
\begin{align*}
(II) =           &  2\seq{ (-{\D_v}+v){\D_v} f,{\D_x} {\D_v} f} +
             \seq{  f,{\D_x}{\D_v} f} \\
          =   &  2\seq{ (-{\D_v}+v){\D_v} f,{\D_x} {\D_v} f} -
             \seq{ {\D_x} f,{\D_v} f}.
\end{align*}
Form the preceding estimates on $(I)$ and $(II)$ we therefore have
\begin{equation*}
   \ddt \seq{{\D_x}f,{\D_v} f}=-\norm{{\D_x}f}^2+2
\seq{(-{\D_v}+v){\D_v} f,{\D_x}{\D_v} f}-\seq{{\D_x}f,{\D_v} f}.
\end{equation*}
Finally, observing that ${\D_x} f$ also solves \eqref{eq:rescaled}, we
obtain for the last term of $\hhh(f(t))$ the same estimate as the one we
obtained for the first term:
\begin{equation*}
    \ddt \norm{{\D_x}f}^2_{\Ldeuxmudvdx}=-2\norm{{\D_v} {\D_x} f}_{\Ldeuxmudvdx}^2.
\end{equation*}
Eventually, we obtain
\begin{align}
   \ddt 
  \hhh(f)&=-2C\norm{{\D_v} f}^2-2D\norm{(-{\D_v}+v){\D_v} f}^2
-E\norm{{\D_x}f}^2-2\norm{{\D_x}{\D_v} f}^2 \nonumber\\
\label{eq:decrentropycontinuenonborne}
&\qquad -(2D+E)\seq{{\D_x}f,{\D_v} f}+2E\seq{(-{\D_v}+v){\D_v} f,{\D_x}{\D_v} f}.
\end{align}

\noindent
Only the last two terms above do not have a sign {\it a priori}.
Using the Cauchy--Schwarz--Young inequality,
we observe that
\begin{equation*}
  \abs{(2D+E)\seq{{\D_x}f,{\D_v} f}}\leq \dfrac{1}{2}\norm{{\D_x}f}^2
+\frac{(2D+E)^2}{2}\norm{{\D_v} f}^2,
\end{equation*}
and
\begin{equation*}
  \abs{2E\seq{(-{\D_v}+v){\D_v} f,{\D_x}{\D_v} f}}\leq
  \norm{{\D_x}{\D_v} f}^2 +E^2 \norm{(-{\D_v}+v){\D_v} f}^2.
\end{equation*}
Therefore, assuming again that $1<E<D<C$, $E^2<D$ and $(2D+E)^2<2C$, we get
\begin{equation*}
   \ddt \hhh(f)\leq
   -C\norm{{\D_v} f}^2-(E-1/2)\norm{{\D_x}f}^2\leq -\dfrac{E}{2}(\norm{{\D_v} f}^2+\norm{{\D_x}f}^2).
\end{equation*}
Using the Poincar\'e inequality in space-velocity proven in Lemma \ref{lem:fullPoincarecontinu} with constant $1$, we
derive
\begin{align*}
  -\dfrac{E}{2}(\norm{{\D_v} f}^2+\norm{{\D_x}f}^2)
  &\leq    -\dfrac{E}{4}(\norm{{\D_v} f}^2+\norm{{\D_x}f}^2)-\dfrac{E}{4}\norm{f}^2
  \leq-\dfrac{E}{4}\dfrac{1}{2C}\hhh(f),
\end{align*}
using eventually the equivalence property proven in Lemma \ref{lem:equiv}.
We therefore have with $2\kappa = E/8C$
$$
 \ddt \hhh(f)\leq - 2\kappa \hhh(f),
 $$
 and Theorem \ref{thm:decrexp} is a consequence of the Gronwall Lemma.
The proof is complete.
\end{proof}

\begin{cor}
  Let $C>D>E>1$ be chosen as in Theorem \ref{thm:decrexp}, and pose
  $\kappa = E/(16C)$. Let $f^0\in \Hunmudvdx$ such that $\seq{f^0} = 0$ and let
  $f$ be the semi-group solution in $\ccc^0(\R^+, \Hunmudvdx)$ of equation
  \eqref{eq:rescaled}. Then for all $t\geq 0$, we have
  \begin{equation*}
   \norm{f(t)}_{\Hunmudvdx}\leq 2\sqrt{C} \exp^{-\kappa t} \norm{f^0 }_{\Hunmudvdx}.
  \end{equation*}
\end{cor}

\begin{proof}
  Choose $C>D>E>1$ as in Theorem \ref{thm:decrexp} and set $\kappa =E/(16C)$. We
  apply Theorem \ref{thm:decrexp} and Proposition \ref{lem:equiv} to $f$ and we
  obtain for all $t\geq 0$,
\begin{equation*}
  \norm{f(t)}^2_{\Hunmudvdx}  \leq  2  \hhh(f(t))
   \leq  2  \exp^{-2\kappa t} \hhh(f^0)
   \leq  4 C \exp^{-2\kappa t}\norm{f^0}^2_\Hunmudvdx.
\end{equation*}
The proof is complete.
\end{proof}

\subsection{The semi-discretization in time}
\label{subsec:discretentcontinuenxv}
In order to solve Equation \eqref{eq:rescaled} numerically, we consider the
one-step implicit Euler method.  We introduce the time step $\dth>0$ supposed to
be small.

\begin{defn} We shall say that a sequence $(f^n)_{n\in\N} \in (\Ldeuxmudvdx)^\N$
  (resp. $(\Hunmudvdx)^\N$) satisfies the (scaled) time-discrete inhomogeneous
  Fokker-Planck equation if for a given $f^0$ in $\Ldeuxmudvdx$
  (resp. $\Hunmudvdx$), for all $n \in \N$,
  \begin{equation}
    \label{eq:eulerimpl}
    f^{n+1} = f^n - \dth (\Xzero  f^{n+1}+(-{\D_v}+v){\D_v} f^{n+1}),
  \end{equation}
  for some $\dth>0$.
\end{defn}

The main goal of this section is to prove that this numerical scheme has the
same asymptotic behavior as that of the exact flow, in the sense that it
satisfies a discrete analogue of Theorem \ref{thm:decrexp} (see Theorem
\ref{thm:decrexpeulerimpl}).

We first check that this implicit scheme is well posed.

\begin{prop}
  \label{prop:stabiliteetmoyennepreservee}
  For all given initial condition $f^0$ in $\Ldeuxmudvdx$ (resp. $\Hunmudvdx$),
  and all $\dth>0$, there exists a unique solution $f \in (\Ldeuxmudvdx)^\N$
  (resp. $(\Hunmudvdx)^\N$) of the time-discrete evolution equation
  \eqref{eq:eulerimpl}. Moreover it satisfies for all $n\in \N$,
$$
\norm{f^n} \leq \norm{f^0}, \qquad \seq{f^n} = \seq{f^0}.
$$
\end{prop}

\begin{proof} Let us denote $P = \Xzero +(-{\D_v}+v) {\D_v}$. Then equation
  \eqref{eq:eulerimpl} writes
$$
(\Id + \dth P) f^{n+1} = f^n.
$$
The linear operator $P$ is maximal accretive in $\Ldeuxmudvdx$
(resp. $\Hunmudvdx$, see \cite{HelN04}), so that the resolvent
$(\Id + \dth P)^{-1}$ is a well defined operator in $\Ldeuxmudvdx$
(resp. $\Hunmudvdx$) of norm $1$.  This implies the well-posedness and the
uniform boundedness of the norms of the functions $f^n$ with respect to $n$.
Similarly to the continuous case, we have in addition
$$
\seq{f^{n+1}} = \seq{f^n} + \dth \iint \sep{ \Xzero f^{n+1} + (-{\D_v}+v) {\D_v}
  f^{n+1}} \mudvdx = \seq{f^n} + 0 = \seq{f^0},
$$
by integration by parts.  The proof is complete.
\end{proof}

In order to prove the exponential (discrete-)time decay of the solutions in
Theorem \ref{thm:decrexpeulerimpl}, similar to the exponential decay of the
continuous solutions (Theorem \ref{thm:decrexp}), we shall examine the behaviour
of the same entropy $\hhh$ defined in \eqref{eq:entropyfunc} along numerical
solutions of \eqref{eq:eulerimpl}.

\begin{lem}
  \label{lem:CSplus}
  Assume $C>D>E>1$ and $E^2<D$. Let us introduce the bilinear map $\varphi$
  defined for all $g,\gtilde\in \Hunmudvdx$ by
  \begin{equation*}
    \varphi(g,\gtilde)=C\seq{g,\gtilde}+D\seq{{\D_v} g,{\D_v} \gtilde} 
+ \frac{E}{2}\seq{{\D_x}g,{\D_v}\gtilde}+\frac{E}{2}\seq{{\D_v} g,{\D_x}
  \gtilde} 
+ \seq{{\D_x} g,{\D_x} \gtilde}.
  \end{equation*}
  Then $\phi$ defines a scalar product in $\Hunmudvdx$ and the associated norm
  is $\sqrt{\hhh(\cdot)}$. In particular one has
  \begin{equation*}
    |\varphi(g,\gtilde)| \leq \sqrt{\hhh(g)}\sqrt{\hhh(\gtilde)}
    \leq \frac{1}{2} \hhh(g) + \frac{1}{2} \hhh(\gtilde).
  \end{equation*}
\end{lem}

\begin{proof}
  The map $\varphi$ is bilinear and symmetric on $\Hunmudvdx$. It is positive
  definite on $\Hunmudvdx$ provided $E^2<D$ using Proposition \ref{lem:equiv}.
  In particular, it is non-negative and one has the Cauchy--Schwarz' inequality
  \begin{equation*}
    \forall g,\gtilde\in \Hunmudvdx,\qquad |\varphi(g,\gtilde)| 
\leq\sqrt{\mathcal H(g)}\sqrt{\mathcal H(\gtilde)}.
  \end{equation*}
  The last inequality is just another Young's inequality.
\end{proof}

We now state the main Theorem of this section.

\begin{thm}
  \label{thm:decrexpeulerimpl}
  Assume that $C>D>E>1$ satisfy $E^2<D$ and $(2D+E)^2<2C$.  For all $\dth>0$ and
  $f^0\in \Hunmudvdx$, we denote by $(f^n)_{n\in\N}$ the sequence solution of
  the implicit Euler scheme \eqref{eq:eulerimpl}.  If $\seq{f^0}=0$, then
  \begin{equation*}
    \forall n\in\N,\qquad \hhh(f^n)\leq (1+ 2 \kappa \dth)^{-n} \hhh(f^0).
  \end{equation*}
  with $\kappa = E/(16C)$.

  In addition, for all $\dth >0$ there exists $k>0$ (explicit) with
  $\lim_{\dth \rightarrow 0} k = \kappa$ such that
  \begin{equation*}
    \forall n\in\N,\qquad \hhh(f^n)\leq \hhh(f^0)\exp^{-2 k n\dth}.
  \end{equation*}
\end{thm}

\begin{proof}
  Using Proposition \ref{prop:stabiliteetmoyennepreservee}, the sequence
  $(f^n)_{n\in\N}$ satisfies for all $n\in\N$ $\seq{f^n} = \seq{f^0} =0$.  Fix
  $n\in\N$.  We evaluate the four terms in the definition of
  $\mathcal H(f^{n+1})$ as follows.  Taking the $\Ldeuxmudvdx$-scalar product of
  relation \eqref{eq:eulerimpl} with $f^{n+1}$ yields
  \begin{equation*}
    \norm{f^{n+1}}^2 = \seq{f^n,f^{n+1}}-\dth\seq{\Xzero  f^{n+1},f^{n+1}}
    - \dth \seq{(-{\D_v}+v){\D_v} f^{n+1},f^{n+1}}.
  \end{equation*}
  The first term in $\dth$ above vanishes by skew-adjointness of the operator
  $\Xzero $. The second term in $\dth$ above can be rewritten to obtain
  \begin{equation}
    \label{eq:estim1eulerimpl}
    \norm{f^{n+1}}^2 = \seq{f^n,f^{n+1}}
    - \dth \norm{{\D_v} f^{n+1}}^2,
  \end{equation}
  since $-{\D_v}+v$ is the formal adjoint of ${\D_v}$.  Differentiating relation
  \eqref{eq:eulerimpl} with respect to $v$ and taking the $\Ldeuxmudvdx$-scalar
  product with ${\D_v} f^{n+1}$ allows to write
  \begin{multline*}
    \norm{{\D_v} f^{n+1}}^2 = \seq{{\D_v} f^n,{\D_v} f^{n+1}}
    -\dth\seq{\Xzero  {\D_v} f^{n+1},{\D_v} f^{n+1}} \\
    - \dth \seq{{\D_x} f^{n+1},{\D_v} f^{n+1}} - \dth
    \seq{{\D_v}(-{\D_v}+v){\D_v} f^{n+1},{\D_v} f^{n+1}}.
  \end{multline*}

\noindent
As before, the skew-adjointness of $\Xzero $ makes the first term in $\dth$
vanish. The third term in $\dth$ can be rewritten as before so that
\begin{equation}
  \label{eq:estim2eulerimpl}
  \norm{{\D_v} f^{n+1}}^2 = \seq{{\D_v} f^n,{\D_v} f^{n+1}}
  - \dth \seq{{\D_x} f^{n+1},{\D_v} f^{n+1}}
  - \dth \norm{(-{\D_v}+v){\D_v} f^{n+1}}^2.
\end{equation}

\noindent
For the third term in $\mathcal H(f^{n+1})$, we first compute ${\D_v} f^{n+1}$
with \eqref{eq:eulerimpl} and take its $\Ldeuxmudvdx$-scalar product with
${\D_x} f^{n+1}$ to write
\begin{eqnarray*}
  \lefteqn{\seq{{\D_v} f^{n+1},{\D_x} f^{n+1}} = } \\
& &  \seq{{\D_v} f^n,{\D_x} f^{n+1}}
    - \dth \seq{\Xzero {\D_v} f^{n+1},{\D_x} f^{n+1}}
    - \dth \seq{{\D_x}f^{n+1},{\D_x} f^{n+1}}\\
& &  - \dth \seq{{\D_v} (-{\D_v}+v){\D_v} f^{n+1},{\D_x} f^{n+1}}.
\end{eqnarray*}
Using that $[{\D_v},(-{\D_v}+v)]=1$, we obtain
\begin{eqnarray*}
  \lefteqn{\seq{{\D_v} f^{n+1},{\D_x} f^{n+1}} =}\\
& &  \seq{{\D_v} f^n,{\D_x} f^{n+1}}
    - \dth \seq{\Xzero {\D_v} f^{n+1},{\D_x} f^{n+1}}
    - \dth \norm{{\D_x}f^{n+1}}^2\\
& &  - \dth \seq{{\D_v} f^{n+1},{\D_x} f^{n+1}}
    - \dth \seq{(-{\D_v}+v)\partial^2_v f^{n+1},{\D_x} f^{n+1}}.
\end{eqnarray*}
Then, we compute ${\D_x} f^{n+1}$ with \eqref{eq:eulerimpl} and take its
$\Ldeuxmudvdx$-scalar product with ${\D_v} f^{n+1}$ to write
\begin{equation*}
  \seq{{\D_x} f^{n+1},{\D_v} f^{n+1}} =
  \seq{{\D_x} f^n,{\D_v} f^{n+1}}
  - \dth \seq{v{\D_x} ^2 f^{n+1},{\D_v} f^{n+1}}
  - \dth \seq{{\D_x}(-{\D_v}+v){\D_v} f^{n+1},{\D_v} f^{n+1}}.
\end{equation*}

\noindent
Summing up the last two identities yields
\begin{eqnarray*}
  \lefteqn{\seq{{\D_v} f^{n+1},{\D_x} f^{n+1}}
  + \seq{{\D_x} f^{n+1},{\D_v} f^{n+1}}=}\\
& & \seq{{\D_v} f^n,{\D_x} f^{n+1}} + \seq{{\D_x} f^n,{\D_v} f^{n+1}} \\
& & - \dth \norm{{\D_x}f^{n+1}}^2 - \dth \seq{{\D_v} f^{n+1},{\D_x} f^{n+1}}\\
& & - \dth \seq{\Xzero {\D_v} f^{n+1},{\D_x} f^{n+1}}
    - \dth \seq{(-{\D_v}+v)\partial^2_v f^{n+1},{\D_x} f^{n+1}}\\
& &  + \dth \seq{{\D_x} f^{n+1},\Xzero {\D_v} f^{n+1}}
    - \dth \seq{{\D_x}(-{\D_v}+v){\D_v} f^{n+1},{\D_v} f^{n+1}}.
\end{eqnarray*}
Using the skew-adjointness of ${\D_x}$ and the fact that $(-{\D_v}+v)^\star={\D_v}$
twice, we obtain
\begin{eqnarray}
  \lefteqn{\seq{{\D_v} f^{n+1},{\D_x} f^{n+1}}
  + \seq{{\D_x} f^{n+1},{\D_v} f^{n+1}}=\seq{{\D_v} f^n,{\D_x} f^{n+1}} 
+ \seq{{\D_x} f^n,{\D_v} f^{n+1}} }\label{eq:estim3eulerimpl}\\
& & - \dth \norm{{\D_x}f^{n+1}}^2 - \dth \seq{{\D_v} f^{n+1},{\D_x} f^{n+1}}
    +2 \dth \seq{(-{\D_v}+v){\D_v} f^{n+1},{\D_x} {\D_v} f^{n+1}}
    \nonumber.
\end{eqnarray}
For the last term in $\mathcal H(f^{n+1})$, we compute the $\Ldeuxmudvdx$-scalar
product of ${\D_x} f^{n+1}$ computed with relation \eqref{eq:eulerimpl} with
${\D_x} f^{n+1}$. This yields directly using the skew-adjointness of ${\D_x}$ and
the fact that $(-{\D_v}+v)^\star={\D_v}$,
\begin{equation}
  \label{eq:estim4eulerimpl}
  \norm{{\D_x} f^{n+1}}^2 = \seq{{\D_x} f^{n},{\D_x} f^{n+1}}
  - \dth \norm{{\D_v} {\D_x} f^{n+1}}^2.
\end{equation}
Summing up relations \eqref{eq:estim1eulerimpl}, \eqref{eq:estim2eulerimpl},
\eqref{eq:estim3eulerimpl} and \eqref{eq:estim4eulerimpl} with respective
coefficients $C$, $D$, $E/2$ and $1$, we obtain
\begin{eqnarray*}
  \lefteqn{\mathcal H(f^{n+1}) = \varphi(f^{n},f^{n+1})
  -\dth \left( C \norm{{\D_v} f^{n+1}}^2
  + \left(D+\frac{E}{2}\right)\seq{{\D_x} f^{n+1},{\D_v} f^{n+1}}\right.}\\
& & \left.
    + D \norm{(-{\D_v} + v ){\D_v} f^{n+1}}^2
    + \frac{E}{2}\norm{{\D_x} f^{n+1}}^2 - E\seq{(-{\D_v}+v){\D_v} f^{n+1},{\D_v}{\D_x}  
f^{n+1}} + \norm{{\D_v}{\D_x}  f^{n+1}}^2\right).
\end{eqnarray*}

\noindent
Using Lemma \ref{lem:CSplus}, we may write
\begin{eqnarray*}
  \lefteqn{\mathcal H(f^{n+1}) \leq \frac{1}{2}\mathcal H(f^n)
  +\frac{1}{2}\mathcal H(f^{n+1})
  -\dth \left( C \norm{{\D_v} f^{n+1}}^2
  + \left(D+\frac{E}{2}\right)\seq{{\D_x} f^{n+1},{\D_v} f^{n+1}}\right.}\\
& & \left.
    + D \norm{(-{\D_v} + v ){\D_v} f^{n+1}}^2
    + \frac{E}{2}\norm{{\D_x} f^{n+1}}^2 - E\seq{(-{\D_v}+v){\D_v} f^{n+1},{\D_v}{\D_x}
    f^{n+1}} 
    + \norm{{\D_v}{\D_x}  f^{n+1}}^2\right).
\end{eqnarray*}

\noindent
This relation is to be compared with the (time)-continuous one
\eqref{eq:decrentropycontinuenonborne}.  The very same estimates as that used in
the end of the proof of Theorem \ref{thm:decrexp}, with $f$ replaced with
$f^{n+1}$, ensure that
\begin{equation*}
  \mathcal H(f^{n+1}) \leq \mathcal H(f^n) - \dth\frac{E}{4}\frac{1}{2C} \mathcal H(f^{n+1}).
\end{equation*}
This gives by induction
$$
\forall n\in\N,\qquad \mathcal H(f^{n}) \leq (1+ 2\kappa \dth)^{-n}\mathcal
H(f^0).
$$
Using a Taylor development of the exponential function we get Theorem
\ref{thm:decrexpeulerimpl}.
\end{proof}

As in the continuous case, we can state as a corollary of the preceding Theorem
the exponential decay in $\Hunmudvdx$ norm, which is a direct consequence of the
equivalence of the norms $\sqrt{\hhh(\cdot)}$ and $\norm{\cdot}_\Hunmudvdx$
stated in Lemma \ref{lem:equiv}.

\begin{cor}
  Let $C>D>E>1$ be chosen as in Theorem \ref{thm:decrexpeulerimpl}.  Let $\kappa$
  be defined as in the same Theorem.  For all $\dth >0$ there exists
  $\kappa_\dth >0$ (explicit) with
  $\lim_{\dth \rightarrow 0} \kappa_\dth = \kappa$ such that for all
  $f^0\in \Hunmudvdx$ with $\seq{f^0} = 0$, the sequence solution
  $(f^n)_{n\in\N}$ of the implicit Euler scheme \eqref{eq:eulerimpl} satisfies
  for all $n\in \N$,
  \begin{equation*}
    \norm{f^n}_\Hunmudvdx \leq 2 \sqrt{C} \exp^{- \kappa_\dth n\dth}  \norm{f^0}_\Hunmudvdx.
  \end{equation*}
\end{cor}

\subsection{The semi-discretization in space and velocity}\label{sec:inhomsd}
In this subsection we are interested in the semi-discretized equation in space
and velocity.  The time is a continuous variable again.

We denote by $\dxh >0$ the step of the uniform discretization of the torus $\T$
into $N$ subintervals, and denote $\jjj = \Z/ N\Z$ the finite set of indices of
the discretization in $x \in \T$. In what follows, the index $i \in \Z$ will
always refer to the velocity variable and the index $j \in \jjj$ to the space
variable.  As mentioned in the introduction, the derivation-in-space discretized
operator is defined by the following centered scheme

\begin{defn}
  For a sequence $G = (G_{i,j})_{i\in \Z, j\in \jjj}$ we define $\Dx G$ by
$$
\forall i \in \Z, j\in \jjj, \qquad (\Dx G)_{j,i} = \frac{G_{j+1,i} -
  G_{j-1,i}}{2\dxh}.
$$
For a sequence $G = (G_{i,j})_{i\in \Z^*, j\in \jjj}$ we define $\Dx G$ by
$$
\forall i \in \Z^*, j\in \jjj, \qquad (\Dx G)_{j,i} = \frac{G_{j+1,i} -
  G_{j-1,i}}{2\dxh}.
$$
Depending on the context, we will use the first definition or the other.
Similarly, we will keep on writing $v$ the pointwise multiplication by $v_i$
from the set of sequences indexed by $\jjj\times\Z$ to itself {\bf and} from the
set of sequences indexed by $\jjj\times\Z^*$ to itself depending on the
context. However, we use the notation $\vs$ from Subsection \ref{subsec:homsd} (see
Definition \ref{def:dvs}, and add $j\in\jjj$ as a parameter) of the pointwise
multiplication operator by $v_i$ from the set of sequences indexed by
$\jjj\times\Z^*$ to the set of sequences indexed by $\jjj\times\Z$.
\end{defn}

Concerning the velocity discretization, we stick on the one corresponding to the
homogeneous case introduced in Subsection \ref{subsec:homsd}.  The definition of
$\Dv$, $\Dvs$, $\muh$ and $\mus$ are the same (with the space index $j$ playing
the role of a parameter) as in Definitions \ref{def:dv}, \ref{def:dvs}, \ref{def:homd}
and \ref{def:homds}.

\noindent
The original semi-discretized equation that we consider is
\begin{equation*}
  {\D_t}F+v \Dx F-\Dvs(\Dv+v)F=0, \qquad F|_{t=0} = F^0,
\end{equation*}
where $F^0\in \ell^1(\jjj\times\Z)$ is a non-negative function with
$\norm{F^0}_{\ell^1(\jjj\times\Z)}=1$ and the unknown $F$ is such that for all
$t>0$, $F(t) \in \ell^1(\jjj\times\Z)$.  As in Section \ref{subsec:homsd}, we
rather work with the rescaled function $f$ defined by
$$
F = \muh + \muh f,
$$ where $\muh$ is the Maxwellian introduced in Lemma \ref{lem:muh},
now considered as a function of both $i$ and $j$.  In that case for all $t>0$ we
have the equivalence
$$
F\in \ell^1(\jjj\times\Z, \dvh \dxh) \Longleftrightarrow f \in
\ell^1(\jjj\times\Z, \muh \dvh \dxh).
 $$

\noindent
Referring again to the homogeneous setting studied in Section
\ref{sec:homogeneous}, we introduce the definition of a solution of the (scaled)
semi-discretized equation that we will study in this subsection.

\begin{defn} We shall say that a function
  $f \in \ccc^0( \R^+, \ell^1(\jjj\times\Z, \muh \dvh \dxh))$ satisfies the
  (scaled) semi-discrete inhomogeneous Fokker-Planck equation if
  \begin{equation}
    \label{eq:FPISD}
    {\D_t}f+v \Dx f+(-\Dvs+\vs)\Dv f=0,
  \end{equation}
  in the sense of distributions
\end{defn}

As in the homogeneous case of Section \ref{sec:homogeneous}, we work in
Hilbertian subspaces of $\ell^1(\jjj\times\Z, \muh \dvh \dxh)$ that we introduce
below.

\begin{defn}
  We define the space $\Ldeuxmudvdxh$ to be the Hilbertian subspace of
  $\R^{\jjj\times\Z}$ made of sequences $f $ such that
$$
\norm{ f}_{\Ldeuxmudvdxh}^2 \defegal \sum_{j\in \jjj, i\in \Z} (f_{j,i})^2
\muh_i \dvh \dxh <\infty.
$$
This defines a Hilbertian norm, and the related scalar product will be denoted
by $\seq{ \cdot, \cdot}$. For $f \in \Ldeuxmudvdxh$, we define the mean of $f$
(with respect to this weighted scalar product in both velocity and space) as
$$
\seq{f} \defegal \sum_{j\in \jjj, i\in \Z} f_{j,i} \muh_i \dvh \dxh = \seq{f,
  1}.
$$
We define the space $\Ldeuxmudvdxs$ to be the Hilbertian subspace of
$\R^{\jjj\times\Z^*}$ made of sequences $g $ such that
$$
\norm{ g}_{\Ldeuxmudvdxs }^2 \defegal \sum_{j\in \jjj, i\in \Z^*} (g_{j,i})^2
\mus_i \dvh \dxh <\infty.
$$
This defines also a Hilbertian norm, and the related scalar product will be
denoted by $\seqs{ \cdot, \cdot}$.  Eventually we define
$$
\Hunmudvdxh = \set{ f \in \Ldeuxmudvdxh, \textrm{ s.t. } \Dv f \in
  \Ldeuxmudvdxs, \ \Dx f \in \Ldeuxmudvdxh },
$$
with the norm
$$
\norm{f}^2_\Hunmudvdxh = \norm{f}^2_\Ldeuxmudvdxh + \norm{\Dv f}^2_\Ldeuxmudvdxs
+ \norm{\Dx f}^2_\Ldeuxmudvdxh.
$$
\end{defn}

We define the operator $\Pd$ involved in Equation \eqref{eq:FPISD} by
\begin{equation*}
  \Pd= \Xzerod + (-\Dvs+\vs)\Dv,
\end{equation*}
with $\Xzerod = v \Dx : \Ldeuxmudvdxh \hookrightarrow \Ldeuxmudvdxh$ defined
by
$$ ( \Xzerod f)_{j,i} = (v \Dx f)_{j,i} \textrm{ when } i\neq 0, \qquad ( \Xzerod
f)_{j,0} =0.
 $$

 This way, Equation \eqref{eq:FPISD} reads ${\D_t} f+ \Pd f = 0$.  We
 summarize the structural properties of \eqref{eq:FPISD} and of
 the operator $\Pd$ in the following Proposition.  From now on and for the rest
 of this subsection, we work in $\Ldeuxmudvdxh$ and denote (when no ambiguity
 happens) the corresponding norm $\norm{\cdot}$ without subscript. Similarly
 $\norms{\cdot}$ stands for the norm in $\Ldeuxmudvdxs$.

 \begin{prop} 
 We have
   \begin{enumerate}
   \item The operator $\Pd = \Xzerod + ( - \Dvs +\vs) \Dv$ on $\Ldeuxmudvdxh$
     equipped with its graph domain $D(\Pd)$ is maximal accretive in
     $\Ldeuxmudvdxh$.
   \item The Operator $( - \Dvs +\vs) \Dv$ is formally self-adjoint and the
     operator $\Xzerod$ is formally skew-adjoint in $\Ldeuxmudvdxh$.  Moreover,
     for all $g \in \Ldeuxmudvdxh$, $h \in \Ldeuxmudvdxs $ for which it makes
     sense
     \begin{align} \label{eq:ofin}
       \seq{ ( - \Dvs +\vs)  h, g}  = \seq{h, \Dv g}_\sharp, \\
       \label{eq:ofin2}
       \seq{\Pd g,g} = \seq{ ( - \Dvs +\vs) \Dv g, g} = \norms{\Dv g}^2.
     \end{align}
   \item For an initial data $f^0
     \in
     D(\Pd)$, there exists a unique solution of \eqref{eq:FPISD} in
     $\ccc^1(\R^+,
     \Ldeuxmudvdxh) \cap \ccc^0(\R^+,
     D(\Pd))$, and the associated semi-group naturally defines a solution in
     $\ccc^0(\R^+, \Ldeuxmudvdxh)$ for all $f^0\in \Ldeuxmudvdxh$.
   \item The preceding properties remain true if we consider the operator
     $\Pd$
     in $\Hunmudvdxh$
     with domain $D_\Hunmudvdxh(\Pd)$.
     In particular it defines a unique solution of \eqref{eq:FPISD} in
     $\ccc^1(\R^+,
     \Hunmudvdxh) \cap \ccc^0(\R^+,D_\Hunmudvdxh(\Pd)) $ if $f^0 \in
     D_\Hunmudvdxh(\Pd)$ and a semi-group solution $f \in \ccc^0(\R^+,
     \Hunmudvdxh)$ if $f^0 \in \Hunmudvdxh$.
   \item Constant sequences are the only equilibrium states of equation
     \eqref{eq:FPISD} and the evolution preserves the mass $\seq{f(t)}
     = \seq{f^0}$ for all $t\geq 0$.
   \end{enumerate}
 \end{prop}

\begin{proof}
  The maximal accretivity can be proven using the same scheme of proof as in the
  continuous case and we won't do it here, referring to \cite{HelN04}.  The
  skew-adjointedness of $\Xzerod$
  is clear since we chose a centered scheme in space.  The properties stated in
  \eqref{eq:ofin} and \eqref{eq:ofin2} are direct consequences of the homogeneous
  analysis (see Proposition \ref{prop:hstruct}).  The well-posedness is then a direct
  consequence of Hille--Yosida's Theorem.  In particular, we can check that if
  $f$ is a solution in $\ccc^1(\R^+, \Ldeuxmudvdxh)$
$$
\ddt \norm{f}^2 = -2 \seq{\Pd f, f} = -2 \norm{\Dv f}_\sharp^2 \leq 0,
$$
so that the $\Ldeuxmudvdxh$
norm is non-increasing.  For the last point, we first infer that if $f$
is a stationary solution then
$$
\ddt \norm{f}^2 = -2 \norm{\Dv f}_\sharp^2 =0 \Longrightarrow \Dv f = 0.
$$
Introducing the macroscopic density $\rho$
defined for all $j\in\jjj$
by $\rho_j
= \dvh \sum_{i \in \Z} f_{j,i}\muh_i$, the fact that $\Dv
f=0$ yields that for all $(j,i)\in\jjj\times\Z$, $f_{j,i} =
\rho_j$.  Then, the equation $\Xzerod f = 0$ implies that
$\rho_j$ does not depend on $j\in\jjj$ and we summarize this with
$$
f_{j,i} = \rho_j = \seq{f} = \seq{f^0}, \qquad \forall (j,i) \in \jjj \times \Z,
$$
so that constant sequences are the only equilibrium states of the equation.  The
remaining parts of the proof follow the ones of the continuous case.  The proof
is complete.
\end{proof}

For later use, we introduce the operator $S=\Dv v - v \Dv$ from $\Ldeuxmudvdxh$
to $\Ldeuxmudvdxs$, where the first operator $v$ is the pointwise multiplication
by $v_i$ at each $(j,i)\in\jjj\times\Z$ and the second one is the pointwise
multiplication by $v_i$ at each $(j,i)\in\jjj\times\Z^*$.  The operator $S$ will
essentially play the role of $\adf{{\\Dv},v}$ in the continuous case.  We
observe that $S$ is a shift operator in the velocity variable and we have the
following lemma:

\begin{lem}
  \label{lem:constcontS}
  Operator $S : \Ldeuxmudvdxh \hookrightarrow \Ldeuxmudvdxs$ satisfies the
  following: for all $g\in \Ldeuxmudvdxh$, we have for all $j\in \jjj$,
  $$
  (S g)_{j,i} = g_{j,i+1} \textrm{ for } i \leq -1, \qquad (S g)_{j,i} = g_{j,
    i-1} \textrm{ for } i \geq 1,
  $$
  and
  \begin{equation*}
    \norm{g}^2\leq \norms{S g}^2\leq 2\norm{g}^2.
  \end{equation*}
\end{lem}

\begin{proof}
  Let $g\in\Ldeuxmudvdxh$.  We first compute $\Dv v g$ (where the multiplication
  operator $v$ is supposed to be defined from $\Ldeuxmudvdxh $ to
  $\Ldeuxmudvdxh$).  We omit for convenience the index $j\in\jjj$ in the
  computations. We have
$$
(\Dv(v g))_{i} = \dfrac{v_{i+1}g_{i+1}-v_{i}g_{i}}{\dvh} \textrm{ for } i \leq
-1, \qquad (\Dv(v g))_{i} = \dfrac{v_i g_{i}-v_{i-1}g_{i-1}}{\dvh} \textrm{ for }
i \geq 1.
$$
Similarly we compute $v \Dv g$ (where the multiplication operator $v$ is now
supposed to be defined from $ \Ldeuxmudvdxs $ to $ \Ldeuxmudvdxs$):
  $$
  (v \Dv g)_{i} = v_i\dfrac{g_{i+1}-g_{i}}{\dvh} \textrm{ for } i \leq -1, \qquad
  (v \Dv g)_{i} = v_i\dfrac{g_{i}-g_{i-1}}{\dvh} \textrm{ for } i \geq 1.
    $$
    Comparing the two preceding results gives the expression of $S g$.  We now
    compute the norms using the definition of $\mus$ and get
    \begin{align*}
      (\dvh\dxh)^{-1} \norms{S g}^2&
           =\sum_{j\in\jjj,i\leq -1}g_{j,i+1}^2\mus_i + \sum_{j\in\jjj,i\geq 1}g_{j,i-1}^2\mus_i\\
          &=\sum_{j\in\jjj,i\leq 0}g_{j,i}^2\muh_i + \sum_{j\in\jjj,i\geq 0}g_{j,i}^2\muh_i\\
          &=(\dvh\dxh)^{-1} \norm{g}^2+\muh_0\sum_{j\in\jjj}g_{j,0}^2.
    \end{align*}
    This last term is one of the terms (the centered one) in the definition of
    the norm in $\Ldeuxmudvdxh$, and we therefore get
    \begin{equation*}
      \norm{g}^2\leq  \norms{S g}^2 \leq 2  \norm{g}^2.
    \end{equation*}
    The proof is complete.
  \end{proof}

  We define the operator $S^\sharp:\Ldeuxmudvdxs\rightarrow\Ldeuxmudvdxh$ to be
  the adjoint of the operator $S$, {\it i.e.} satisfying the relation
  \begin{equation*}
    \forall (g,h)\in \Ldeuxmudvdxh\times\Ldeuxmudvdxs,\qquad
    \seq{S g,h}_{\sharp}=\seq{g,S^\sharp h}.
  \end{equation*}

  This is again a shift operator in the velocity variable, but it is not
  injective, and we have the following lemma
  \begin{lem}
    \label{lem:constcontSsharp}
    Operator $\Ss : \Ldeuxmudvdxs \hookrightarrow \Ldeuxmudvdxh$ satisfies the
    following: For $h\in\Ldeuxmudvdxs$, we have for all $j\in\jjj$,
$$
(\Ss h)_{j,i} = h_{j,i-1} \textrm{ for } i\leq -1, \qquad (\Ss h)_{j,0} =
h_{j,-1}+h_{j,1},\qquad (\Ss h)_{j,i} = h_{j,i+1} \textrm{ for } i\geq 1.
$$
Moreover, for all $h \in\Ldeuxmudvdxs$, we have
\begin{equation*}
  \norm{S^\sharp h}^2 \leq 4 \norm{h}_\sharp^2.
\end{equation*}
\end{lem}

\begin{proof}
  The proof is straightforward, using similar tools as in the one of Lemma
  \ref{lem:constcontS}.
\end{proof}

In order to apply a procedure similar to the one we used in the continuous
inhomogeneous case in Section \ref{sec:inhomc}, we introduce the following
modified entropy defined for $g\in \Hunmudvdxh$ by
\begin{equation}
  \hhhd(g) \defegal  C\norm{g}^2+D\norms{\Dv g}^2+E\seqs{\Dv g,S\Dx g}+\norm{\Dx g}^2,
\end{equation}
for well chosen $C>D>E>1$ to be defined later.  We will show in a moment that
for these parameters, $t\mapsto \hhhd(f(t))$ is exponentially decreasing in time
when $f$ is the semi-group solution of the scaled inhomogeneous Fokker-Planck
equation \eqref{eq:FPISD} with initial datum $f^0\in\Hunmudvdxh$ of zero mean.
Before doing this, we compare this entropy $\hhhd$ with the usual $\Hunmudvdxh$
norm.

\begin{lem}\label{lem:equivd}
  If $2E^2<D$ then for all $g\in \Hunmudvdxh$,
  \begin{equation*}
    \dfrac{1}{2}\norm{g}_{\Hunmudvdxh}^2\leq\hhhd(g)\leq 2C\norm{g}_{\Hunmudvdxh}^2.
  \end{equation*}
\end{lem}
\begin{proof}
  We stick to the proof of Lemma \ref{lem:equiv}. Let $g\in \Hunmudvdxh$.  We
  use the Cauchy--Schwarz--Young inequality and observe that
  \begin{equation*}
    2\abs{E\seq{\Dv g,S\Dx g}_\sharp}\leq 2E^2\norm{\Dv g}_\sharp^2 
    +\frac{1}{2}\norms{S\Dx g}^2 \leq 2E^2\norm{\Dv g}_\sharp^2+\norm{\Dx g}^2,
  \end{equation*}
  where we used Lemma \ref{lem:constcontS} for the last inequality.  This
  implies
  \begin{multline*}
    \underbrace{C}_{1/2\leq}\norm{g}^2 +\underbrace{(D-E^2)}_{1/2\leq
      E^2\leq}\norm{\Dv g}_\sharp^2
    +\dfrac{1}{2}\norm{\Dx g}^2 \\
    \leq \hhhd(g)\leq C\norm{g}^2+\underbrace{(D+E^2)}_{\leq D+D/2\leq 3C/2\leq
      2C}\norm{\Dv g}_\sharp^2+\underbrace{3C/2}_{\leq 2C} \norm{\Dx g}^2,
  \end{multline*}
  which implies the result since $2E^2<D$.
\end{proof}

As in the continuous case, we need a full Poincar\'e inequality in space and
velocity.  We first note that, for functions $\rho$ of the space variable
$j\in\jjj$ only, provided that $N=\#\jjj$ is odd (which is assumed from now on
in this paper), the Poincar\'e inequality
\begin{equation}
  \label{eq:Poincarediscretxseulement}
  \norm{\rho-\seq{\rho}}^2_\Ldeuxdxh \leq \norm{\Dx \rho}^2_\Ldeuxdxh,
\end{equation}
is standard (and easy to reproduce following the proof of Lemma
\ref{lem:Poincarecontinu}), where
\begin{equation}
  \label{eq:defL2dxh}
  \norm{\rho}^2_\Ldeuxdxh = \sum_{j\in \jjj} \rho_j^2 \dxh,
\end{equation}
is the standard norm on the discretized torus,
$$
\seq{\rho} = \sum_{j\in \jjj} \rho_j \dxh,
$$
is the mean of $\rho$ and $\Dx$ is the centered finite difference derivation
operator defined above. In particular, for $g\in \Ldeuxmudvdxh$, one can apply
\eqref{eq:Poincarediscretxseulement} to the macroscopic density $\rho$ of $g$
defined of $j\in\jjj$ by $\rho_j = \dvh \sum_{i} g_{j,i}\muh_i$.  The fully
discrete Poincar\'e inequality of Lemma \ref{lem:fullPoincarediscret} is then a
consequence of Proposition \ref{prop:poindiscrete} in velocity only (following the
proof of the continuous case stated in Lemma \ref{lem:Poincarecontinu}).

\begin{lem}[Full Discrete Poincar\'e inequality]
  \label{lem:fullPoincarediscret}
  For all $g \in \Hunmudvdxh$, we have
  \begin{equation*}
    \norm{g-\seq{g}}^2_\Ldeuxmudvdxh  \leq  \norm{\Dv g}_\Ldeuxmudvdxs^2 + 
    \norm{\Dx g}_\Ldeuxmudvdxh^2.  \end{equation*}
\end{lem}
\begin{proof}
  Replacing if necessary $g$ by $g-\seq{g}$, it is sufficient to prove the
  result for $\seq{g}=0$.  We observe that Parseval's formula and discrete
  Fubini's theorem imply
$$
\norm{\rho}^2_\Ldeuxdxh \leq \norm{g}^2_\Ldeuxmudvdxh \qquad \text{and} \qquad
\norm{\Dx\rho}^2_\Ldeuxdxh \leq \norm{\Dx g}^2_\Ldeuxmudvdxh,
$$
since $(j,i) \rightarrow \rho_j$ (resp. $(j,i) \rightarrow (\Dx\rho)_j$) is the
orthogonal projection of $g$ (resp. $\Dx g$) onto the closed space
$$
\set{ (j,i) \longmapsto \phi_j \ | \ \phi \in \Ldeuxdxh},
$$
and $\norm{\phi \otimes 1 }_\Ldeuxmudvdxh = \norm{\phi}_\Ldeuxdxh$ for all
$\phi \in \Ldeuxdxh$ since we are in probability spaces. We note here the
natural injection $\Ldeuxdxh \hookrightarrow \Ldeuxmudvdxh$ of norm $1$.  Using
the discrete Fubini theorem again, we also directly get from Proposition
\ref{prop:poindiscrete} that 
$$
\norm{g- \rho\otimes 1}_\Ldeuxmudvdxh^2 \leq \norm{\Dv g}_\Ldeuxmudvdxs^2.
$$
Using Parseval's formula again yields
\begin{equation}
  \begin{split}
    \norm{g}_\Ldeuxmudvdxh^2 & = \norm{g- \rho\otimes 1}_\Ldeuxmudvdxh^2 + \norm{ \rho\otimes 1}_\Ldeuxmudvdxh^2 \\
    & = \norm{g- \rho\otimes 1}_\Ldeuxmudvdxh^2 + \norm{\rho}_\Ldeuxdxh^2 \\
    & \leq \norm{\Dv g}_\Ldeuxmudvdxs^2 + \norm{\Dx \rho}_\Ldeuxdxh^2 \\
    & \leq \norm{\Dv g}_\Ldeuxmudvdxs^2 + \norm{\Dx g}_\Ldeuxmudvdxh^2.
  \end{split}
\end{equation}
The proof is complete. \end{proof}

We can now state the main Theorem of this subsection concerning the exponential
return to equilibrium of solutions of Equation \eqref{eq:FPISD}.

\begin{thm}\label{thm:decrexpd}
  There exists $C>D>E>1$, $\dvh_0>0$ and $\kappa_d>0$ explicit such that the
  following holds: For all $ f^0 \in \Hunmudvdxh$ such that $\seq{f^0} = 0$, the
  solution $f$ (in the semi-group sense) in $\ccc^0(\R^+, \Hunmudvdxh)$ of
  Equation \eqref{eq:FPISD} with initial data $f^0$ satisfies
  \begin{equation*}
    \hhhd(f(t))\leq \hhhd(f^0)\exp^{-2\kappa_d t},
  \end{equation*}
  for all $t\geq 0$, $\dvh\in(0,\dvh_0)$ and $\dxh>0$.
\end{thm}

\begin{proof}(of Theorem \ref{thm:decrexpd} -- 1/4) We divide the proof in four
  parts, and we insert technical lemmas in between those parts, so that the
  reader may understand why new discrete operators are introduced and studied,
  as the computations go.  Let us consider $f$ the solution in
  $\ccc^1(\R^+, D_\Hunmudvdxh(\Pd))$ with initial data
  $f^0 \in D_\Hunmudvdxh(\Pd)$.  This choice allows all the computations done
  below, and Theorem \ref{thm:decrexpd} will be a direct consequence of the
  density of $D_\Hunmudvdxh(\Pd)$ in $\Hunmudvdxh$ and the boundedness of the
  associated semi-group.

  As in the continuous case, we shall differentiate w.r.t. time the four terms
  appearing in the definition of $\hhhd$.  The derivatives of the 1st, 2nd and
  4th term are fairly easy to estimate, as we will see below. The more intricate
  estimate of the derivative of the 3rd term will require Lemmas \ref{lem:triple},
  \ref{lem:sb} and \ref{lem:estimdelta}.

  For the derivative of the first term in $\hhhd$, we compute
  \begin{align*}
    \ddt\norm{f}^2&=2\seq{f,-v \Dx f-(-\Dvs+\vs)\Dv f}\\
                  &=-2\seq{f,-v \Dx f}-2\seq{f,-(-\Dvs+\vs)\Dv f}.
  \end{align*}
  Using the fact that $v \Dx$ is skew-adjoint in $\Ldeuxmudvdxh$ and the identity
  derived from \eqref{eq:ofin2}, we obtain
  \begin{equation} \label{eq:premier} \ddt \norm{f}^2=-2\norms{\Dv f}^2.
  \end{equation}
  The second term of the time derivative can be computed as follows:
  \begin{align}
    \ddt\norms{\Dv f}^2&=2\seqs{\Dv\left(-v \Dx-(-\Dvs+\vs)\Dv\right) f,\Dv f} \nonumber \\
                       &=-2\seqs{\Dv(v \Dx)f,\Dv f}-2\seqs{\Dv(-\Dvs+\vs)\Dv f,\Dv f} \nonumber \\
                       &=-2\seqs{\underbrace{\adf{\Dv,v \Dx}}_{=\adf{\Dv,v}\Dx=S\Dx}f,\Dv f}
                         -2\underbrace{\seqs{v \Dx \Dv f,\Dv f}}_{=0}-2\underbrace{\norm{(-\Dvs+\vs)\Dv f}^2}_{\mbox{using
                         } \eqref{eq:ofin2}} \nonumber \\
                       &=-2\seqs{S\Dx f,\Dv f}-2\norm{(-\Dvs+\vs)\Dv f}^2.
                         \label{eq:second}
  \end{align}
The time derivative of the last term in $\hhh(f)$ is
\begin{equation} \label{eq:quatrieme}
 \ddt \norm{\Dx f}^2=-2\norms{\Dv \Dx f}^2.
\end{equation}
since $\Dx$ commutes with the full operator.

\bigskip
All the difficulties are concentrated in the third term.
We are going to need a few lemmas in order to be able to write
the time-derivative of that third term in \eqref{eq:troisieme}.
After that, we will get back to the proof of the Theorem
by expressing the time-derivative of $t\mapsto \hhhd(f(t))$ in \eqref{eq:defdcde}
using an entropy-dissipation term.
We will need a last lemma (Lemma \ref{lem:dissdiss}) to estimate
the entropy-dissipation term before getting to the end of the proof of
Theorem \ref{thm:decrexpd}.
\end{proof}

In order to prepare the computations, we state and prove two lemmas concerning
discrete commutators.

\begin{lem} \label{lem:triple}
 We have
$$
\Dv (-\Dv^\sharp + \vs) S - S(-\Dv^\sharp+\vs) \Dv = S+ \sigma,
$$
where $\sigma$ is the singular operator from $\Ldeuxmudvdxh$ to $\Ldeuxmudvdxs$
defined for $g\in \Ldeuxmudvdxh$ by
$$
\qquad (\sigma g)_{j, -1} = \frac{g_{j,1}-g_{j,0}}{\dvh^2}, \qquad (\sigma
g)_{j, 1} = -\frac{g_{j,0}-g_{j,-1}}{\dvh^2} \qquad \textrm{ and } (\sigma
g)_{j,i}=0 \ \textrm{ for } \ |i| \geq 2,
$$
for all $ j \in \jjj$.
\end{lem}

\begin{proof}
We postpone the proof of this computational lemma to the end
of the paper, where Table \ref{tab:comm} summarizes all the computations
of commutators.
\end{proof}

The second lemma of commutation type is the following
\begin{lem}
\label{lem:sb}
We define the operator $S^b : \Ldeuxmudvdxh \longrightarrow \Ldeuxmudvdxs$ by
\begin{align*}
 & (S^b g)_{j,i} = g_{j,i+1} \textrm{ for } i\leq -1  \qquad \qquad
 (S^b g)_{j,i} = -g_{j,i-1} \textrm{ for } i\geq 1,
\end{align*}
for all $g \in
\Ldeuxmudvdxh$ and $ j \in \jjj$. Then we have
$$
S\Dx v \Dx g -v \Dx S\Dx g= \dvh S^b \Dx^2 g.
$$
Moreover
\begin{equation*}
  \norm{g}^2 \leq \norm{S^\flat g}_\sharp^2 \leq 2 \norm{g}^2.
\end{equation*}
\end{lem}

\begin{proof} We postpone the proof to the table at the end of the paper (see
  Table \ref{tab:comm}). \end{proof}

\begin{proof}(of Theorem \ref{thm:decrexpd} -- 2/4)
We go on with the proof of Theorem \ref{thm:decrexpd}, and we recall
that we consider a solution $f \in \ccc^1(\R^+, D_\Hunmudvdxh(P))$.
We want to estimate the derivative of the third term defining $\hhhd(f(t))$.
Let us compute

\begin{align*}
  &   \ddt\seq{S\Dx f,\Dv f}_\sharp \\
  & = -\seqs{S\Dx (\Xzerod  f+(-\Dvs +\vs)\Dv f),\Dv f}-
    \seqs{S\Dx f,\Dv (\Xzerod  f+(-\Dvs +\vs)\Dv f)}\\
  & = -\seqs{S\Dx  \Xzerod  f,\Dv f} -\seqs{S\Dx  f,\Dv  \Xzerod  f} \qquad (I) \\
  &\qquad - \seqs{S\Dx (-\Dvs +\vs)\Dv f,\Dv f)} - \seqs{S\Dx f,\Dv (-\Dvs
    +\vs)\Dv  f)}. 
    \qquad (II)
   \end{align*}

We first deal with the sum (I) in the previous equality.
Using Lemma \ref{lem:sb} and $[\Dv, \Xzerod ] = S\Dx$ we get
\begin{align*}
(I) =&   -\seqs{\Xzerod  S\Dx  f,\Dv f} -\seqs{S\Dx  f,\Xzerod  \Dv  f} \\
 & - \dvh \seqs{S^b \Dx^2 f, \Dv f} - \norms{S\Dx f}^2 \\
 =&   - \dvh \seqs{S^b \Dx^2 f, \Dv f} - \norms{S\Dx f}^2,
   \end{align*}
where we used that the first two terms compensate by skew-adjunction of $\Xzerod $.  Using that
$\Dx$ is skew-adjoint and commutes with $S^b$ we get
$$
(I) = \dvh \seqs{S^b \Dx f, \Dx \Dv f} - \norms{S\Dx f}^2.
$$

Now we deal with the term (II).
We first use that the adjoint of $\Dv$ is $(-\Dv^\sharp+\vs)$
two times and we get
\begin{align*}
(II) &= -\seqs{S\Dx (-\Dvs+\vs)\Dv f,\Dv f}  -\seqs{\Dv(-\Dvs+\vs) S\Dx f,\Dv f}.
\end{align*}
Now from Lemma \ref{lem:triple} applied to the second term we get
\begin{align*}
(II) &= -2\seqs{S\Dx (-\Dvs+\vs)\Dv f,\Dv f}  -\seqs{ S\Dx f,\Dv f}- \seqs{\sigma \Dx f, \Dv f}.
\end{align*}
We used also that $\Dx$ commutes with all operators. This yields
\begin{align*}
(II) &= 2 \seq{(-\Dvs+\vs)\Dv f,S^\sharp \Dx  \Dv f}  -\seqs{ S\Dx f,\Dv f}
- \seqs{\sigma \Dx f, \Dv f}.
\end{align*}

Using the relations above for $(I)$ and $(II)$, 
we get eventually for the derivative of the third term:
\begin{align} \label{eq:troisieme}
 &   \ddt\seq{S\Dx f,\Dv f}_\sharp \\
 &=   - \norms{S\Dx f}^2+ \dvh \seqs{S^b \Dx f, \Dv \Dx f} \nonumber\\
  & \qquad \qquad +  2 \seq{(-\Dvs+\vs)\Dv f,S^\sharp \Dx  \Dv f}  -\seqs{ S\Dx
    f,\Dv f}
- \seqs{\sigma \Dx f, \Dv f}.\nonumber
\end{align}

The first term in this sum has a sign.
All the other terms except the last one are easy to deal with,
as in the continuous case.
The last one involving $\sigma$ is more involved since it seems to be singular.
Anyway, it can also be controlled as shown in this last lemma.
\end{proof}

\begin{lem}
\label{lem:estimdelta}
For all $\varepsilon>0$ and $g \in \Ldeuxmudvdxh$ we have
\begin{equation}
\label{eq:estimsigma}
\abs{ \seqs{\sigma \Dx  g, \Dv  g} } \leq \frac{1}{\varepsilon}
\norm{(-\Dvs+\vs)\Dv g}^2 + \varepsilon
\norms{ \Dv  \Dx  g}^2.
\end{equation}
\end{lem}

\begin{proof}
For all $j\in\jjj$, the contribution to the scalar product in the right-hand
side of \eqref{eq:estimsigma} reduces to two terms according
to the expression of $\sigma$ (see Lemma \ref{lem:triple}).
We denote by $\seq{ ., .}_\Ldeuxdxh $ the scalar product in the variable
$j$ only, associated to the norm defined in \eqref{eq:defL2dxh}.
In the computations below, we omit for convenience the subscript $j$
corresponding to the space discretization.
We have
\begin{align*}
\seqs{\sigma \Dx  g, \Dv  g} =
& \seq{\frac{\Dx  g_1-\Dx  g_0}{\dvh^2}, \frac{g_0-g_{-1}}{\dvh}}_\Ldeuxdxh
  \mu_0 \dvh 
- \seq{\frac{\Dx  g_0-\Dx  g_{-1}}{\dvh^2},  \frac{g_1-g_{0}}{\dvh}}_\Ldeuxdxh \mu_0 \dvh.
\end{align*}
Using that $\Dx $ is skew-adjoint (or using an Abel transform in $j$),
we get
\begin{align*}
\seqs{\sigma \Dx  g, \Dv  g} =
& 2 \seq{\frac{\Dx  g_1-\Dx  g_0}{\dvh^2}, \frac{g_0-g_{-1}}{\dvh}}_\Ldeuxdxh \mu_0 \dvh.
\end{align*}
For convenience, we set $\GPLUS= \frac{g_1- g_0}{\dvh}=(\Dv g)_{1}$
and $\GMOINS = \frac{g_0-g_{-1}}{\dvh} = (\Dv g)_{-1}$.
We have then
\begin{align*}
\seqs{\sigma \Dx  g, \Dv  g} =
& 2 \seq{ \frac{\Dx  \GPLUS}{\dvh}, \GMOINS}_\Ldeuxdxh \mu_0 \dvh \\
= & 2 \seq{ \Dx  \frac{ \GPLUS-\GMOINS}{\dvh}, \GMOINS}_\Ldeuxdxh \mu_0 \dvh 
+ \frac{2}{\dvh} \seq{ \Dx  \GMOINS, \GMOINS}_\Ldeuxdxh \mu_0 \dvh.
\end{align*}
The last term is zero and we therefore get, using a last integration by part for the first term
\begin{align*}
\seqs{\sigma \Dx  g, \Dv  g}
= & - 2 \seq{  \frac{ \GPLUS-\GMOINS}{\dvh}, \Dx  \GMOINS}_\Ldeuxdxh \mu_0 \dvh.
\end{align*}
We observe that
$$
 \frac{ \GPLUS-\GMOINS}{\dvh} = \frac{ \frac{g_1-g_0}{\dvh} 
- \frac{g_0-g_{-1}}{\dvh}}{\dvh} = (\Dvs \Dv g)_0 = -((-\Dvs +\vs) \Dv  g)_0.
 $$
Hence, for all $\varepsilon>0$
\begin{align*}
\abs{\seqs{\sigma \Dx  g, \Dv  g}} \leq
 &  2 \norm{ ((-\Dvs +\vs) \Dv  g)_0}_\Ldeuxdxh  \norm{({\Dx}\Dv g)_{-1}}_\Ldeuxdxh \mu_0 \dvh \\
 \leq
 &  \frac{1}{\varepsilon}
 \norm{ ((-\Dvs +\vs) \Dv  g)_0}^2_\Ldeuxdxh \mu_0 \dvh + \varepsilon
 \norm{({\Dx}\Dv g)_{-1}}^2_\Ldeuxdxh \mu_0 \dvh \\
 \leq & \frac{1}{\varepsilon}
 \norm{(-\Dvs+\vs)\Dv g}^2 +
\varepsilon
\norms{ \Dv  \Dx  g}^2.
\end{align*}
 The proof of the Lemma is complete.
\end{proof}

\begin{proof}(of Theorem \ref{thm:decrexpd} -- 3/4)
Now we come back to the proof of Theorem \ref{thm:decrexpd}.
We consider all the four relations \eqref{eq:premier}, \eqref{eq:second},
\eqref{eq:quatrieme} and \eqref{eq:troisieme}, and we multiply the first one by $C$,
the second one by $D$, the third more involved one by $E$, and we get by
addition
\begin{eqnarray}
  \ddt \hhhd(f(t))  & = & -2 \left( C \norm{\Dv f}_\sharp^2
+ D \seq{S\Dx f,\Dv f}_\sharp
+ D \norm{(-\Dvs+\vs)\Dv f}^2 \right.\nonumber\\
&&  \left.  +\frac{E}{2} \norm{S\Dx f}_\sharp^2
- \frac{E}{2} \dvh \seq{S^b \Dx f,\Dx \Dv f}_\sharp \right. \nonumber\\
&& \left. -E \seq{(-\Dvs+\vs)\Dv f,S^\sharp \Dx \Dv f}\right.\nonumber\\
&&   \left.
+\frac{E}{2} \seq{S\Dx f,\Dv f}_\sharp
+\frac{E}{2}\seq{\sigma \Dx f,\Dv f}_\sharp
+\norm{\Dx \Dv f}_\sharp^2 \right) \nonumber \\
&& \qquad  \defegal -2 \dddd(f). \label{eq:defdcde}
\end{eqnarray}

The term $2\dddd(f) $ is the discrete entropy-dissipation term and
we prove that it can be bounded below (so that, in particular, it has a sign)
for well chosen parameters $C$, $D$, and $E$.
This is the goal of the following lemma.
\end{proof}

\begin{lem} \label{lem:dissdiss}
There exists constants $C>D>E>1$ and $\dvh_0>0$
such that for all $g \in \Hunmudvdxh$, $\dvh\leq \dvh_0$ and $\dxh>0$,
\begin{equation}
\label{eq:entropydissipationineq}
\dddd(g) \geq \kappa_d \hhhd(g),
\end{equation}
with $\kappa_d = 1/(4C)$.
Moreover, it is sufficient for the constants above to satisfy
relations \eqref{eq:choixE}-\eqref{eq:choixC} to come to ensure that
the result above hold.
\end{lem}

\begin{proof}
Grouping terms and estimating the big parentheses in \eqref{eq:defdcde},
we obtain first for all $\theta>0$,
\begin{eqnarray*}
 & \dddd(g) \geq  &
C \norm{\Dv g}_\sharp^2
+ \left(D+\frac{E}{2}\right) \seq{S\Dx g,\Dv g}_\sharp
+ D \norm{(-\Dvs+v)\Dv g}^2 \\
&&  +\frac{E}{2} \norm{S\Dx g}_\sharp^2
- \frac{1}{2}\norm{S^b \Dx g}_\sharp^2
- \frac{1}{2}\frac{E^2\dvh^2}{4} \norm{\Dx \Dv g}_\sharp^2 \\
&&  - \frac{1}{2} \frac{1}{\theta}E^2\norm{(-\Dvs+v)\Dv g}^2
- \frac{1}{2} \theta \norm{S^\sharp \Dx \Dv g}^2\\
&&
-\frac{E}{2}\left|\seq{\sigma \Dx g,\Dv g}_\sharp\right|
+\norm{\Dx \Dv g}_\sharp^2.
\end{eqnarray*}
Using the continuity constants of $S$, $S^\sharp$ and $S^\flat$ (see Lemmas
\ref{lem:constcontS}, \ref{lem:constcontSsharp} and \ref{lem:sb}),
as well as Lemma \ref{lem:estimdelta}, we obtain for all $\varepsilon>0$,

\begin{eqnarray*}
 &  \dddd(g) \geq  &
C \norm{\Dv g}_\sharp^2
- \frac{1}{2}\norm{S\Dx g}_\sharp^2
- \frac{\left(D+\frac{E}{2}\right)^2}{2} \norm{\Dv g}_\sharp^2 \\
&& + D \norm{(-\Dvs+v)\Dv g}^2 \\
&&   +\frac{E}{2} \norm{\Dx g}^2
- \norm{\Dx g}^2
- \frac{1}{2}\frac{E^2\dvh^2}{4} \norm{\Dx \Dv g}_\sharp^2 \\
&&  - \frac{1}{2} \frac{1}{\theta} E^2 \norm{(-\Dvs+v)\Dv g}^2
- 2 \theta \norm{\Dx \Dv g}_\sharp^2\\
&&
-\varepsilon\frac{E}{2}\norm{\Dx \Dv g}_\sharp^2
- \frac{1}{\varepsilon}\frac{E}{2}
\norm{(-\Dvs+\vs)\Dv g}^2
+\norm{\Dx \Dv g}_\sharp^2.
\end{eqnarray*}

Using again the continuity constant of $S$ from Lemma \ref{lem:constcontS}
and grouping terms, we find
\begin{eqnarray*}
 & \dddd(g) \geq  &
\left(C- \frac{\left(D+\frac{E}{2}\right)^2}{2}\right) \norm{\Dv g}_\sharp^2
+\left(\frac{E}{2}-2\right) \norm{\Dx g}^2
\nonumber\\
& &
+ \left( D - \frac{1}{2} \frac{1}{\theta} E^2 - \frac{1}{\varepsilon}\frac{E}{2} \right)
\norm{(-\Dvs+v)\Dv g}^2
\nonumber\\
& &
+\left( 1
-\varepsilon\frac{E}{2}
- \frac{1}{2}\frac{E^2\dvh^2}{4}
- 2 \theta
\right)
 \norm{\Dx \Dv g}_\sharp^2.
\end{eqnarray*}

Let us now discuss the existence of a set of constants that achieve
the functional inequality \eqref{eq:entropydissipationineq}.
First, we fix
\begin{equation} \label{eq:choixE}
E\geq 6.
\end{equation}
Then, we can choose $\theta$, $\eps$ and $\dvh_0>0$ such that
\begin{equation*} 
\theta=1/8, \qquad \varepsilon=1/(4E), \qquad \dvh_0^2E^2/8\leq 1/8,
\end{equation*}
so that we obtain that for all $\dvh\leq \dvh_0$
\begin{equation*}
   1
-\varepsilon\frac{E}{2}
- \frac{1}{2}\frac{E^2\dvh^2}{4}
- 2 \theta\geq 1/2.
\end{equation*}
Then, we can choose $D$ big enough to ensure that
\begin{equation} \label{eq:choixD}
  D - \frac{1}{2} \frac{1}{\theta} E^2 - \frac{1}{\varepsilon}\frac{E}{2} \geq 1 \qquad {\rm and} \qquad D>2E^2.
\end{equation}
Eventually, we choose $C$ big enough to ensure that
\begin{equation} \label{eq:choixC}
  C- \frac{\left(D+\frac{E}{2}\right)^2}{2}\geq 1.
\end{equation}

When all these constraints are fulfilled, we get that
  \begin{equation}  \label{eq:pourplustard}
\dddd(g) \geq \norm{\Dv g}_\sharp^2 + \norm{\Dx g}^2+\norm{(-\Dvs+v)\Dv g}^2
+\frac{1}{2} \norm{ \Dv \Dx g}_\sharp^2.
\end{equation}
Using now the Poincar\'e estimate from Lemma \ref{lem:fullPoincarediscret}
applied to half of the right-hand-side of the last inequality, we get
  \begin{equation*}
\dddd(g) \geq \frac{1}{2} \norm{\Dv g}_\sharp^2 + \frac{1}{2}\norm{\Dx g}^2 
+  \frac{1}{2}\norm{g}^2.
\end{equation*}
Since $D>2E^2$ by \eqref{eq:choixD}, Lemma \ref{lem:equivd}
about the equivalence of the $\Hunmudvdxh$
and the $\hhhd$ norms ensures that
 $$
 \dddd(g) \geq  \frac{1}{4C}\hhhd(g).
 $$
\end{proof}

\begin{proof}(of Theorem \ref{thm:decrexpd} -- 4/4)
Provided $C>D>E>1$ are chosen as above, we have along
the solution $f$ of the discrete scaled Fokker-Planck equation \eqref{eq:FPISD}
with zero mean, with the estimates above and in particular
\eqref{eq:defdcde}
 \begin{equation*}
  \ddt \hhhd(f(t))  \leq  -2 \dddd(f)  \leq -2 \kappa_d \hhhd(f(t)).
\end{equation*}
Gronwall's lemma gives directly the result of
Theorem \ref{thm:decrexpd}.
This completes the proof.
\end{proof}

\subsection{The full discretization and proof of Theorem \ref{thm:eulerimplicite} }
\label{subsec:eqinhomototalementdiscretisee}

In this subsection we prove Theorem \ref{thm:eulerimplicite},
which will be a direct consequence of Theorem \ref{thm:decrexpeulerimpldiscr}
below.
We directly work on the scaled sequence $f$ defined by $F = \muh + \muh f$
where $F$ satisfies \eqref{eq:eulerimplicite}.

\begin{defn}
We shall say that a sequence
$f = (f^n)_{n \in \N} \in (\ell^1(\jjj\times \Z, \muh\dvh\dxh))^\N$ satisfies
the scaled fully discrete implicit inhomogeneous Fokker-Planck equation if,
for some $\dth>0$,
\begin{equation}
  \label{eq:eulerimpldiscr}
  \forall n\in\N,\qquad f^{n+1} = f^n
- \dth (v \Dx f^{n+1}+(-\Dvs+\vs) \Dv f^{n+1}).
\end{equation}
\end{defn}

As in all the previous cases, we can check that constant sequences are the only
equilibrium states of this equation, and that the mass conservation property is
satisfied:
$$
\forall n\in\N,\qquad \seq{f^n} = \seq{f^0},
$$
where we use all the notations and definitions of Subsection
\ref{subsec:discretentcontinuenxv}, and in particular work in $\Ldeuxmudvdxh$ or
$\Hunmudvdxh$.

In Subsection \ref{subsec:discretentcontinuenxv}, we proved
a time-discrete result (Theorem \ref{thm:decrexpeulerimpl})
for the solutions in the continuous (in space and velocity)
setting \eqref{eq:rescaled}, in accordance with the behaviour of
the exact solutions (Theorem \ref{thm:decrexp}).
The goal of this section is to prove a similar time-discrete result
for the solutions of the implicit Euler scheme \eqref{eq:eulerimpldiscr},
in accordance with the result (Theorem \ref{thm:decrexpd})
for the exact solutions of \eqref{eq:FPISD} in the discrete (in velocity
and  space) setting.

As in the semi-discrete case, we shall work with the modified entropy defined by
$$
\hhhd(g) = C \norm{g}^2 + D \norms{\Dv g}^2+ E \seqs{\Dv g, S\Dx g} + \norm{\Dx g}^2,
$$
for well chosen $C > D > E >1$ to be defined later.
Under the condition $2 E^2 <D$, Lemma \ref{lem:equivd} holds.
We denote by $\phid$ the polar form associated to $\hhhd$ defined for
$g,\gtilde\in \Hunmudvdxh$ by
\begin{equation*}
\phid(g,\gtilde)=C\seq{g,\gtilde}+D\seq{\Dv g,\Dv \gtilde}_\sharp
+ \frac{E}{2}\left(\seq{S\Dx g,\Dv\gtilde}_\sharp
+\seq{\Dv g,S \Dx \gtilde}_\sharp\right)
+ \seq{\Dx g,\Dx \gtilde},
\end{equation*}
and recall that the Cauchy--Schwarz--Young inequality holds and reads
\begin{equation}\label{eq:CSY}
  |\phid(g,\gtilde)| \leq \sqrt{\hhhd(g)}\sqrt{\hhhd(\gtilde)}
\leq \frac{1}{2} \hhhd(g) + \frac{1}{2} \hhhd(\gtilde),
\end{equation}
just as in the continuous (in space and velocity) case
(see Lemma \ref{lem:CSplus}).

The main result of this section (leading directly to Theorem
\ref{thm:eulerimplicite} in the introduction) is the following theorem.

\begin{thm}   \label{thm:decrexpeulerimpldiscr}
  Assume $C>D>E>1$, $\dvh_0>0$ and $\kappa_d$ are chosen as in Theorem \ref{thm:decrexpd}.
Then for all $ f^0 \in \Hunmudvdxh$, for all $\dth>0$, $\dvh\in(0,\dvh_0)$, and
$\dxh>0$,
the problem \eqref{eq:eulerimpldiscr} with initial datum $f^0$ is well-posed
in $\Hunmudvdxh$.
Suppose in addition  that $\seq{f^0} = 0$ and  let $(f^n)_{n\in\N}$ denote
the sequence solution of Equation \eqref{eq:eulerimpldiscr} with initial
datum $f^0$, we have in this case for all $n\geq 0$,
  \begin{equation*}
     \hhhd(f^n)\leq (1+ 2\kappa_d \dth)^{-n} \hhhd(f^0).
  \end{equation*}

\end{thm}

\begin{remark}
  Doing just as we did at the end of the proof of Theorem
\ref{thm:decrexpeulerimpl} for continuous space and velocity variables,
the result above implies first,
exponential convergence to $0$ with respect to the discrete time
of $(\hhhd(f^n)_{n\in\N})$ and second, exponential convergence of $(f^n)_{n\in\N}$
to its mean in $\Hunmudvdxh$ for all $f^0\in\Hunmudvdxh$.
This allows to prove Corollary \eqref{cor:decrexpdiscretnonborne}
from Theorem \ref{thm:eulerimplicite}.
\end{remark}

\begin{proof} 
Let $f^0 \in \Hunmudvdxh$ and consider in this space the unbounded operator
$\Pd = v \Dx +(-\Dvs+v) \Dv $ with domain $D_{\Hunmudvdxh} (\Pd)$.
It was mentioned in the preceding section that this operator is maximal
accretive. Let us fix $\dth>0$.
Equation \eqref{eq:eulerimpldiscr} reads for all $n\in\N$,
$$
f^{n+1} = (Id + \dth \Pd)^{-1} f^n.
$$
This relation gives sense to the a unique sequence solution
$f= (f^n)_{n\in \N} \in \Hunmudvdxh$ by induction since
$(Id + \dth \Pd)^{-1} : \Hunmudvdxh \longrightarrow
D_{\Hunmudvdxh} (P) \hookrightarrow \Hunmudvdxh$.

Assume now that $\seq{f^0}=0$.
By induction, we directly get  that for all $n\in\N$, $\seq{f^n}=0$.
We fix now $n\in\N$ and compute the four terms appearing in the definition of
$\hhhd(f^{n+1})$ before estimating their sum.
We start by computing the $\Ldeuxmudvdxh$-scalar product of $f^{n+1}$
with itself using relation \eqref{eq:eulerimpldiscr} on the left to obtain
\begin{eqnarray}
  \norm{f^{n+1}}^2 & = & \seq{f^n,f^{n+1}} - \dth \underbrace{\seq{v \Dx f^{n+1},f^{n+1}}}_{=0}
- \dth \seq{(-\Dvs+\vs) \Dv f^{n+1},f^{n+1}}\nonumber\\
 \label{eq:estim1discr}
& = & \seq{f^n,f^{n+1}} - \dth \norm{\Dv f^{n+1}}_\sharp^2,
\end{eqnarray}
using \eqref{eq:ofin}.

Next, we compute $\Ldeuxmudvdxs$-scalar product of $\Dv f^{n+1}$
with itself using relation \eqref{eq:eulerimpldiscr} on the left to obtain
\begin{eqnarray*}
  \lefteqn{\norm{\Dv f^{n+1}}^2_\sharp  = }\\
& & \seq{\Dv f^n,\Dv f^{n+1}}_\sharp - \dth \seq{\Dv v \Dx f^{n+1},\Dv f^{n+1}}_\sharp
- \dth \seq{\Dv(-\Dvs+\vs) \Dv f^{n+1},\Dv f^{n+1}}_\sharp.\\
\end{eqnarray*}
The first term in $\dth$ can be rewritten as
\begin{eqnarray*}
  - \dth \seq{\Dv v \Dx f^{n+1},\Dv f^{n+1}}_\sharp & = &
-\dth \seq{\adf{\Dv,v \Dx}f^{n+1},\Dv f^{n+1}}_\sharp
-\dth \underbrace{\seq{v \Dx \Dv f^{n+1},\Dv f^{n+1}}_\sharp}_{=0}\\
& = & -\dth \seq{S\Dx f^{n+1},\Dv f^{n+1}}_\sharp,
\end{eqnarray*}
thanks to the definition of $S$.
The second term in $\dth$ becomes, using \ref{eq:ofin},
\begin{equation*}
  - \dth \seq{\Dv(-\Dvs+\vs) \Dv f^{n+1},\Dv f^{n+1}}_\sharp
  = - \dth \norm{(-\Dvs+\vs)\Dv f^{n+1}}^2.
\end{equation*}
We infer, for the second term in $\hhhd (f^{n+1})$,
\begin{equation}
\label{eq:estim2discr}
  \norm{\Dv f^{n+1}}^2_\sharp  =
 \seq{\Dv f^n,\Dv f^{n+1}}_\sharp - \dth \seq{S\Dx f^{n+1},\Dv f^{n+1}}_\sharp
- \dth \norm{(-\Dvs+\vs)\Dv f^{n+1}}^2.
\end{equation}

For the third term in $\hhhd(f^{n+1})$, we compute
$2\seq{S\Dx f^{n+1},\Dv f^{n+1}}_\sharp$ using relation \eqref{eq:eulerimpldiscr}
once on the left and once on the right to obtain
\begin{eqnarray*}
  \lefteqn{2\seq{S\Dx f^{n+1},\Dv f^{n+1}}_\sharp = }\\
 & &  \seq{S\Dx f^{n},\Dv f^{n+1}}_\sharp + \seq{S\Dx f^{n+1},\Dv f^{n}}_\sharp\\
& & -\dth\left(
\seq{S\Dx v \Dx f^{n+1},\Dv f^{n+1}}_\sharp + \seq{S\Dx f^{n+1},\Dv v \Dx f^{n+1}}_\sharp
\right)\\
& & -\dth\left(
\seq{S\Dx(-\Dvs+\vs)\Dv f^{n+1},\Dv f^{n+1}}_\sharp
+ \seq{S\Dx f^{n+1},\Dv(-\Dvs+\vs)\Dv f^{n+1}}_\sharp
\right).
\end{eqnarray*}
The two terms in $\dth$ above can be computed just as terms $(I)$ and $(II)$
in the proof of Theorem \ref{thm:decrexpd} (with $f$ there replaced
by $f^{n+1}$ here) to get as in \eqref{eq:troisieme}
\begin{eqnarray}
  \lefteqn{2\seq{S\Dx f^{n+1},\Dv f^{n+1}}_\sharp = } \nonumber\\
 & &  \seq{S\Dx f^{n},\Dv f^{n+1}}_\sharp + \seq{S\Dx f^{n+1},\Dv f^{n+1}}_\sharp\nonumber\\
& & -\dth\left(
\norm{S\Dx f^{n+1}}_\sharp^2 - \dvh\seq{S^b \Dx f^{n+1},\Dx \Dv f^{n+1}}_\sharp
\right)\nonumber\\
& & -\dth\left(
-2 \seq{(-\Dvs+\vs)\Dv f^{n+1},S^\sharp \Dx \Dv f^{n+1}} \right. \nonumber\\
& & \left. \qquad \qquad \qquad + \seq{S\Dx f^{n+1},\Dv f^{n+1}}_\sharp
+\seq{\sigma \Dx f^{n+1},\Dv f^{n+1}}_\sharp
\right),\label{eq:estim3discr}
\end{eqnarray}
where we used Lemmas \ref{lem:triple} and \ref{lem:sb}.

For the last term in $\hhhd(f^{n+1})$, we compute as for
\eqref{eq:estim1discr},
\begin{equation}
\label{eq:estim4discr}
  \norm{\Dx f^{n+1}}^2 = \seq{\Dx f^n,\Dx f^{n+1}} - \dth \norm{\Dx \Dv f^{n+1}}_\sharp^2.
\end{equation}

Summing up the four identities \eqref{eq:estim1discr}, \eqref{eq:estim2discr},
\eqref{eq:estim3discr} and \eqref{eq:estim4discr}, multiplied
respectively by $C$, $D$, $E/2$ and $1$, we infer that
\begin{eqnarray*}
  \lefteqn{\hhhd(f^{n+1}) = \phid(f^n,f^{n+1})}\\
&  & -\dth\left[
C \norm{\Dv f^{n+1}}_\sharp^2
+ D \seq{S\Dx f^{n+1},\Dv f^{n+1}}_\sharp
+ D \norm{(-\Dvs+\vs)\Dv f^{n+1}}^2 \right. \\
& & \qquad \qquad \left.  +\frac{E}{2} \norm{S\Dx f^{n+1}}_\sharp^2 -
    \frac{E}{2} \dvh\seq{S^b \Dx f^{n+1},\Dx \Dv f^{n+1}}_\sharp  \right.\\
& & \qquad \qquad \left.
-E \seq{(-\Dvs+\vs)\Dv f^{n+1},S^\sharp \Dx \Dv f^{n+1}}
+\frac{E}{2} \seq{S\Dx f^{n+1},\Dv f^{n+1}}_\sharp  \right.\\
& & \qquad \qquad \qquad \qquad \left.
+\frac{E}{2}\seq{\sigma \Dx f^{n+1},\Dv f^{n+1}}_\sharp
+\norm{\Dx \Dv f^{n+1}}_\sharp^2
\right].
\end{eqnarray*}

We recognize here inside square brackets exactly the same term as the one in
parentheses defining $\dddd(f)$ in \eqref{eq:defdcde} with $f^{n+1}$ here instead
of $f$ there, so that the preceding identity reads
\begin{equation*}
  \hhhd(f^{n+1}) = \phid(f^n,f^{n+1}) -\dth \dddd(f^{n+1}).
\end{equation*}
Using Lemma \ref{lem:dissdiss} we therefore get that for $C$, $D$, $E$ and
$\dvh_0$ be chosen as in \eqref{eq:choixE}-\eqref{eq:choixC}, we have
\begin{equation*}
  \hhhd(f^{n+1}) = \phid(f^n,f^{n+1}) -\dth \kappa_d \hhhd(f^{n+1}),
\end{equation*}
with $\kappa_d = 1/(4C)$.

Using Cauchy--Schwarz--Young with the scalar product $\phid$
(see \eqref{eq:CSY}),
we obtain for all $n\in\N$,
\begin{equation*}
\hhhd(f^{n+1}) \leq \frac{1}{2} \hhhd(f^{n+1}) + \frac{1}{2} \hhhd (f^n)
-\dth \kappa_d \hhhd(f^{n+1}),
\end{equation*}
which yields for all $n\in\N$,
\begin{equation*}
\hhhd(f^{n+1}) \leq \hhhd (f^n)
-2\kappa_d \dth \hhhd(f^{n+1}),
\end{equation*}
which implies
\begin{equation*}
\hhhd(f^{n}) \leq (1+ 2\dth \kappa_d)^{-n} \hhhd (f^0).
\end{equation*}
This concludes the proof of the theorem.
\end{proof}

\section{The homogeneous equation on bounded velocity domains}
\label{sec:homobounded}

In this Section, we study a discretization of the homogeneous Fokker--Planck
equation \eqref{eq:HFPF} with velocity variable confined in the interval
$I = (-\vmax,\vmax)$, where $\vmax>0$ is given.
We first briefly treat the fully continuous case,
and then we focus on the fully discrete explicit case :
this is possible since only a finite number of points of discretization
are needed (in contrast to the case where $v$ was on the whole real line in the
preceding sections).
The choice of discretization is again made to ensure exponential
convergence to the equilibrium and the functional framework is built
using the natural Maxwellian (stationary solution of the problem,
again denoted $\muh$ below).

In this section, we also prepare the study of the inhomogeneous equation
in Section \ref{sec:eqinhomoboundedvelocity}.
Part of the material is very similar to the one developed in Section
\ref{sec:homogeneous} and we will sometimes refer to there.

Note that the functional spaces in space and velocity introduced and used
in Sections \ref{sec:homobounded} and \ref{sec:eqinhomoboundedvelocity}
are finite dimensional. We will however specify norms on these spaces
and constants for (continuous) linear operators between such spaces,
to emphasize the behaviour of those norms and constants when
the discretization parameters $\dvh$ and $\dxh$ tend to $0$.

\subsection{The fully continuous case}
\label{sec:homoboundedcontinuous}
We consider here the case where the velocity domain is an interval
$$I = (-\vmax,\vmax), \qquad \vmax>0,
$$
and focus on the fully continuous case.
We thus need a boundary condition and choose a homogeneous Neumann one, to
ensure total mass conservation.
Our new problem is thus
 \begin{equation} \label{eq:IHFPFb}
 {\D_t} F  - {\D_v}({\D_v}+v) F = 0, \qquad F|_{t=0} = F^0, \qquad (({\D_v}+v)F)(\pm\vmax)=0.
 \end{equation}

\noindent
The initial density $F^0$ is a non-negative function from $I$ to $\R^+$
such that $\int_I F^0(v) \dv = 1$.
The function
$$
I \ni v \mapsto \frac{1}{\sqrt{2\pi}} \exp^{-v^2/2},
$$
is a continuous
equilibrium of \eqref{eq:IHFPFb}, but we need to renormalize it in $L^1(I,\dv)$.
We keep the same notation as in the first sections of this paper
and we define this normalized equilibrium
\begin{equation*}
  \mu(v)=\dfrac{\exp^{-v^2/2}}{\displaystyle\int_{I}\exp^{-w^2/2}\dw}.
\end{equation*}

In the same way as in the unbounded velocity domain cases,
we pose $F=\mu+ \mu f$, and the rescaled density $f$
solves equivalently
\begin{equation}
  \label{eq:HFPfb}
  {\D_t} f +(- {\D_v}+v) {\D_v} f = 0 , \qquad f|_{t=0} = f^0, \qquad {\D_v} f(\pm\vmax)=0.
\end{equation}

We work with the following adapted functional spaces:
We introduce the space $\Ldeuxmudvb$ and it subspace
$\Hunmudvb = \set{ g \in \Ldeuxmudvb, \ {\D_v} g \in \Ldeuxmudvb}$.
We again denote $\int_I g(v) \mudv$ by $\seq{g}$.

As in the continuous homogeneous case (see Section \ref{sec:homogeneous} for
example), the main ingredient in the proof of the convergence to the equilibrium
is the Poincar\'e inequality, that we prove now.

\begin{lem}[Homogeneous Poincar\'e inequality on a bounded velocity domain]
\label{lem:Poincarecontinub}
  For all $g \in \Hunmudvb$ with , we have
\begin{equation*}
  \norm{g-\seq{g}}^2_\Ldeuxmudvb  \leq  \norm{{\D_v} g}_\Ldeuxmudvb^2.
\end{equation*}
\end{lem}
\begin{proof}
  The proof follows exactly the same lines as in the full space case described
  in Lemma \ref{lem:Poincarecontinu}. We take $g \in \Ldeuxmudvb$ and assume
  that $\seq{g}=0$.  The first steps of the proof are exactly the same as that
  of the proof of Lemma \ref{lem:Poincarecontinu}, changing $\R$ in $I$ until
  relation \eqref{eq:eqintermedpoincarecontinu} there.  Note that we again use
  strongly Fubini Theorem and the fact that $\int_I v \mudv = 0$ and
  $\int_I \mudv = 1$ (and their counterparts in variable $v'$).  We therefore
  have
\begin{equation*} 
\begin{split}
 &  \int_I g^2\mudv  = \int_{I} \G v \mudv,
\end{split}
\end{equation*}
where we have set as before
$\G(v) = \int_{0}^v \abs{{\D_v} g (w)}^2 \dw$ for $|v| \leq \vmax$.
Using that ${\D_v} \mu = -v \mu$ and an integration by part, we get
\begin{equation*}
  \begin{split}
\norm{g}_\Ldeuxmudvb^2 & =   \int_{(-\vmax,\vmax)} \G\, v\, \mudv \\
 &= - \int_{(-\vmax,\vmax)} \G\, ({\D_v} \mu)\, \dv
    \\
& =- [\G\mu]_{-\vmax}^{\vmax} + \int_{(-\vmax,\vmax)} {\D_v} \G\, \mudv\\
&=-\mu(\vmax)\int_{-\vmax}^{\vmax}\abs{{\D_v} g}^2+ \int_{(-\vmax,\vmax)}
\abs{{\D_v} g}^2 \mudv\\
& \leq \norm{{\D_v} g}_\Ldeuxmudvb^2.
  \end{split}
\end{equation*}
The proof is complete.
\end{proof}

Now we can state the main result concerning the convergence to the equilibrium
for Equation \eqref{eq:HFPfb}.
We  consider the operator $P = (- {\D_v}+v) {\D_v}$ with domain
$$
D(P) = \set{ g \in \Ldeuxmudvb, \ (- {\D_v}+v) {\D_v} g \in \Ldeuxmudvb,
\ {\D_v} g (\pm \vmax) =0},
$$
which corresponds to the operator with Neumann conditions.
Note that constant functions are in $D(P)$.
Equation \eqref{eq:HFPfb} reads ${\D_t} f+ P f = 0$ and we
define the two following entropies for $g\in \Ldeuxmudvb$
and $g\in\Hunmudvb$ respectively :

\begin{equation}
\fff(g) = \norm{g}_\Ldeuxmudvb^2, \qquad \ggg(g) = \norm{g}_\Ldeuxmudvb^2
+ \norm{{\D_v} g}_\Ldeuxmudvb^2.\label{eq:fgh}
\end{equation}

The following result holds
 \begin{thm}
   \label{thm:exponentialtrendtoequilibriumb}
Let $f^0 \in \Ldeuxmudvb$.
The Cauchy problem \eqref{eq:HFPfb} has a unique solution $f$ in
$\ccc^0(\R^+, \Ldeuxmudvb)$.
If $f^0$ is such that $\seq{f^0} = 0$, then $\seq{f(t)} = 0$ for all $t\geq 0$
and we have
\begin{equation*}
\forall t\geq 0,\qquad
\fff(f(t)) \leq \exp^{-2t}\fff(f^0).
\end{equation*}

If in addition $f^0 \in \Hunmudvb$, then $f \in \ccc^0(\R^+, \Hunmudvb)$
and we have
\begin{equation*}
\forall t\geq 0,\qquad
 \ggg(f(t)) \leq \exp^{-t}\ggg(f^0).
\end{equation*}
\end{thm}

\begin{proof}
  The proof follows exactly the lines of the proof of Theorem
  \ref{thm:exponentialtrendtoequilibrium}.  The existence part is insured by the
  Hille--Yosida theorem again (either in $\Ldeuxmudvb$ or in $\Hunmudvb$).  As
  in the unbounded case, the key points are the fact that the operator
  $P = (-{\D_v}+v){\D_v}$ is self-adjoint on $\Ldeuxmudvb$ with Neumann boundary
  condition and the Poincar\'e inequality (Lemma \ref{lem:Poincarecontinub}).
\end{proof}

\subsection{The full discretization with discrete Neumann conditions}
\label{subsec:homsdb}

As in the unbounded case,
we discretize the interval of velocities $I = (-\vmax, \vmax)$
and the equation with boundary condition \eqref{eq:IHFPFb}
by introducing an operator $\Dv$.
This indeed yields a discretization of the rescaled equation
\eqref{eq:HFPfb}.

For a fixed positive integer $\imax$, we set
\begin{equation}
\label{eq:defdv}
  \dvh=\frac{\vmax}{\imax},
\end{equation}
and
$$
\iii = \set{-\imax + 1, -\imax + 2,  \cdots, -1, 0, 1, \cdots, \imax-2, \imax -1}.
$$
Moreover, we define
\begin{equation}
\label{eq:defvib}
\forall i\in\iii, \quad v_i = i\dvh, \qquad v_{\pm\imax} = \pm \vmax.
\end{equation}
Note for further use that the boundary indices $\pm\imax$ do not belong to the
full set $\iii$ of indices.  The new discrete Maxwellian $\muh \in \R^\iii$, is
defined by
\begin{equation}
\label{eq:defMaxwellboundedvelocity}
\muh_i = \frac{c_\dvh}{ \prod_{\ell=0}^{|i|} (1+ v_\ell\dvh )}, \qquad i \in \iii,
\end{equation}
where the normalization constant $c_\dvh>0$
is defined such that $\dvh\sum_{i\in \iii}  \muh_i= 1$.
This definition is consistent
with the Definition \ref{def:dvb} of the operator $\Dv$
in the sense that it satisfies \eqref{eq:okmuh}.
For the sake of simplicity,
we will keep the same notation as in the unbounded velocity case.
Note again that we do not need to define the Maxwellian
at the boundary indices $\pm \imax$.

\noindent
We work in the following in the space ${\ell^1(\iii, \muh \dvh)}$ of all finite
sequences $g = (g_i)_{i\in \iii}$ with the norm
$\dvh\sum_{i \in \iii} \abs{g_i}\muh_i$.  We note that
\begin{equation}
\label{eq:normalizationmuhboundeddomain}
\norm{1}_{\ell^1(\iii, \muh \dvh)} =  \norm{\muh}_{\ell^1(\iii, \dvh)} = 1.
\end{equation}

\noindent
For the analysis to come, we introduce another
set of indices and a new Maxwellian $\mus$. We set
$$
\iiis = \set{-\imax , -\imax + 1,  \cdots, -2, -1, 1, 2, \cdots, \imax-1, \imax}
=\left(\iii\setminus\{0\}\right)\cup \{\pm\imax\},
$$
and define $\mus \in {\ell^1(\iiis, \dvh)}$
for all $ i\in\iiis$ by,
\begin{equation}
\label{eq:defmusb}
  \mus_i = \muh_{i+1} \textrm{ for } i<0, \qquad \mus_i = \muh_{i-1} \textrm{ for } i>0.
\end{equation}
We now adapt to this finite case of indices the definitions of the discrete
derivation given in the unbounded velocity case
(see there Definitions \ref{def:dv} and \ref{def:dvs}).

\begin{defn} \label{def:dvb} Let $g \in \ell^1(\iii,\muh \dvh)$, we define
  $\Dv g\in \ell^1(\iiis, \mus \dvh)$ by the following formulas for
  $i \in \iiis$,
  \begin{equation}
  \begin{split}
 & (\Dv g)_i = \frac{g_{i+1}-g_{i}}{\dvh} \textrm{ when } -\imax+1\leq i \leq -1, \qquad
  (\Dv g)_i = \frac{g_i-g_{i-1}}{\dvh} \textrm{ when } 1 \leq i \leq \imax-1, \\
  & \qquad \qquad \qquad \qquad \textrm{ and } \qquad \qquad (\Dv g)_{\pm \imax} = 0,
  \end{split}
  \end{equation}
  and $v g\in \ell^1(\iiis,\muh \dvh)$ by
  $$
  (v g)_i = v_i g_i \ \textrm{ for }\  1\leq |i| \leq \imax-1
  \qquad
  \textrm{ and } \qquad (v g)_{\pm \imax} = v_{\pm \imax} g_{\pm(\imax-1)}.
  $$
  Similarly for $h \in \ell^1(\iiis, \mus \dvh)$, we define
  $\Dvs h \in \ell^1(\iii, \muh \dvh)$ by the following formulas for all
  $i \in \iii$,
  \begin{equation}
  \begin{split}
    & (\Dvs h)_i = \frac{h_{i}-h_{i-1}}{\dvh} \textrm{ when } -\imax+1 \leq i<
    -1, \qquad (\Dvs h)_i = \frac{h_{i+1}-h_i}{\dvh} \textrm{ when } 1 \leq i
    \leq \imax -1
    \\
    & \qquad \qquad \qquad \textrm{ and } \qquad \qquad (\Dvs h)_0 =
    \frac{h_1-h_{-1}}{\dvh}.
  \end{split}
  \end{equation}
For $h\in \ell^1(\iiis, \muh \dvh)$, we also define $\vs g \in \ell^1(\iii, \dvh)$ by
\begin{equation*}
  \forall i \in \iii \setminus \set{0} ,\quad (\vs h)_i=v_i h_i \qquad \text{and} \qquad
  (\vs g)_0=0.
\end{equation*}
\end{defn}

\noindent
Looking  at the proof of Lemma \ref{lem:muh},
we directly check that with this definition we have
\begin{equation} \label{eq:okmuh}
\forall i\in\iii\setminus\{0\},\qquad
[(\Dv + v)\muh]_i = 0,
\end{equation}
The definition of the derivative at the boundary points (always $0$)
is nevertheless adapted to the scaled equation.
This is not in contradiction with the preceding equality which occurs
only in $\iii\setminus\{0\}$. 
We write below the (rescaled) fully discrete homogeneous Fokker--Planck
equation, noting that the discrete Neumann conditions are included in the
definition of $\Dv$.

\begin{defn} We shall say that a  sequence
$f = (f^n)_{n\in \N} \in (\ell^1(\iii, \muh\dvh))^\N$ satisfies the
(scaled) full discrete explicit homogeneous Fokker--Planck equation
with initial data $f^0$ if
\begin{equation} \label{eq:DHFPfb}
\forall n\in\N,\qquad
f^{n+1} = f^n  -\dth ( - \Dvs +\vs) \Dv f^n,
\end{equation}
for some $\dth>0$.
\end{defn}

In order to solve this equation, we build Hilbertian norms on $\R^\iii$
and $\R^\iiis$, taking into account the conservation of mass
and insuring the non-negativity of the associated operator.

\begin{defn} 
We denote by $\Ldeuxmudvh$ the space $\R^\iii$ endowed with the Hilbertian norm
$$
\norm{ g}_{\Ldeuxmudvh }^2  \defegal  \dvh \sum_{i\in \iii} (g_i)^2 \muh_i.
$$
The related scalar product is denoted by $\seq{ \cdot, \cdot}$.
For $g \in \Ldeuxmudvh$, we also define
$ \seq{g} \defegal \sum_{i\in \iii} g_i \mu_i^h \dvh = \seq{g, 1}_{\Ldeuxmudvh }, $
 the mean of $g$. Similarly, we denote by $\Ldeuxmudvs$ the space
 $$
 \Ldeuxmudvs = \set{ g \in \R^\iiis, g_{\pm \imax} = 0},
 $$
endowed with the Hilbertian norm
$$
\norm{g}_{\Ldeuxmudvs }^2  \defegal  \dvh \sum_{i \in \iiis} (g_i)^2 \mus_i,
$$
and the related scalar product is denoted by $\seq{ \cdot, \cdot}_\sharp$.
We denote by $\Hunmudvh$ the space $\R^\iii$ endowed with the norm
\begin{equation*}
  \norm{g}_{\Hunmudvh}^2=\norm{g}_{\Ldeuxmudvh}^2+\norm{\Dv g}_{\Ldeuxmudvs}^2.
\end{equation*}
\end{defn}

\noindent
We introduce the associated operator with discrete Neumann conditions and its
functional and structural properties.

\begin{prop} 
Let $\dvh$ be defined by \eqref{eq:defdv}
and $\dth>0$ be given and sufficiently small.
\begin{enumerate}
\item We have $\Dv : \Ldeuxmudvhb \rightarrow \Ldeuxmudvsb$ and
  $\Dvs : \Ldeuxmudvsb \rightarrow \Ldeuxmudvhb$ and $P = ( - \Dvs +\vs) \Dv$ is
  a bounded operator on $\Ldeuxmudvhb $.
\item For all
$h \in \Ldeuxmudvsb$, $g \in \Ldeuxmudvhb $ we have
\begin{equation} \label{eq:ofb}
\seq{ ( - \Dvs +\vs)  h, g} = \seq{h, \Dv g}_\sharp,
\qquad \textrm{ and } \qquad
\seq{ ( - \Dvs +\vs) \Dv h, h} = \norm{\Dv h}_\Ldeuxmudvs^2.
\end{equation}
\item For an initial data $f^0 \in \Ldeuxmudvh$,
there exists a unique solution of \eqref{eq:DHFPfb} in $ (\Ldeuxmudvh)^\N$.
\item  Constant sequences are the only equilibrium states of Equation \eqref{eq:DHFPfb}.
\item The mass is conserved by the discrete evolution,
{\it i.e.} for all $n\in \N$,  $\seq{f^n} = \seq{f^0}$.
\end{enumerate}
\end{prop}

\begin{proof}
  The linear operator $P$
  is a mapping from the finite dimensional linear space $\Ldeuxmudvhb$
  to itself. Hence it is bounded.  The proof of the second equality in
  \eqref{eq:ofb} is a direct consequence of the first equality, and leads directly
  to the self-adjointness and the non-negativity of $(
  - \Dvs +\vs) \Dv$. The (maximal) accretivity of $( - \Dvs +\vs)
  \Dv$ in both $\Ldeuxmudvh$ and
  $\Hunmudvh$
  is easy to get (perhaps adding a constant to the operator). The fact that the
  equation is well-posed is a direct consequence of the fact that the scheme is
  explicit.  The fact that constant sequences are the only equilibrium solutions
  is an easy consequence of the second identity in \eqref{eq:ofb}.

Due to its importance in the functional framework we give a complete
proof of the first equality in \eqref{eq:ofb}
although it is very similar to the one of \eqref{eq:of}.
We write for $h \in \Ldeuxmudvsb$ and $g \in \Ldeuxmudvhb$
\begin{equation} \label{eq:calcofb}
\begin{split}
  & \dvh^{-1} \seq{ ( - \Dvs +\vs)  h, g} \\
  & =  \sum_{i \in \iii} ((-\Dvs +\vs) h)_i g_i \mu_i  \\
  & = \sum_{1 \leq i \leq \imax -1} ((-\Dvs +\vs) h)_i g_i \mu_i -(\Dvs h)_0 g_0
  \mu_0 +\sum_{-\imax + 1 \leq i \leq -1} ((-\Dvs +\vs) h)_i g_i \mu_i.
\end{split}
\end{equation}
For the first term in the right-hand side of \eqref{eq:calcofb}, we have
\begin{equation}
\begin{split}
  & \sum_{1 \leq i \leq \imax -1} ((-\Dvs +\vs) h)_i g_i \mu_i \\
  & = \sum_{1 \leq i \leq \imax -1} \sep{ -\frac{h_{i+1} - h_i}{\dvh} + v_i h_i} g_i \mu_i \\
  & = \sum_{1 \leq i \leq \imax -1} h_i \sep{ \frac{-g_{i-1}\mu_{i-1} + g_i
      \mu_i }{\dvh} + v_i g_i \mu_i} + \frac{h_1 g_0}{\dvh} \mu_0 -
  \frac{h_{\imax} g_{\imax-1}}{\dvh} \mu_{\imax-1}.
\end{split}
\end{equation}
Since $h \in \Ldeuxmudvsb$ we have $h_{\imax}=0$.
Therefore we have
\begin{equation}
\begin{split}
  & \sum_{1 \leq i \leq \imax -1} ((-\Dvs +\vs) h)_i g_i \mu_i \\
  & = \sum_{1 \leq i \leq \imax -1} h_i g_i \sep{ \frac{-\mu_{i-1} + \mu_i
    }{\dvh} + v_i \mu_i} + \sum_{1 \leq i \leq \imax -1} h_i
  \sep{ - \frac{g_{i-1}- g_i}{\dvh}}\mu_{i-1}   + \frac{h_1 g_0}{\dvh} \mu_0 \\
  & = \sum_{1 \leq i \leq \imax -1} h_i (\Dv g)_i \mu_{i-1} + \frac{h_1 g_0}{\dvh} \mu_0 \\
  & = \sum_{1 \leq i \leq \imax} h_i (\Dv g)_i \mus_{i} + \frac{h_1 g_0}{\dvh}
  \mu_0,
\end{split}
\end{equation}
where we used \eqref{eq:okmuh}, the definition of $\mus$,
and again the fact that $h_{\imax} = 0$.
Similarly we get
\begin{equation}
\begin{split}
& \sum_{-\imax + 1 \leq i \leq -1} ((-\Dvs +\vs) h)_i g_i \mu_i
 = \sum_{-\imax  \leq i \leq -1} h_i (\Dv g)_i \mus_{i} - \frac{h_{-1} g_0}{\dvh} \mu_0.
\end{split}
\end{equation}
The center term in the right-hand side of \eqref{eq:calcofb} is
$-(\Dvs h) g_0 \mu_0 = -\frac{ h_1- h_{-1}}{\dvh} g_0 \mu_0$, 
so that we have
$$
\dvh^{-1} \seq{ ( - \Dvs +\vs) h, g} = \sum_{i \in \iiis} h_i (\Dv g)_i \mus_{i}
= \dvh^{-1} \seqs{h, \Dv g},
$$
since the boundary terms around $0$ disappear.  This is the first equality in
\eqref{eq:ofb} and the proof is complete. \end{proof}

\bigskip As in the cases with unbounded velocity domains (see Sections
\ref{sec:homogeneous} and \ref{sec:eqinhomo}), in continuous or discretized
settings, and as in the case with bounded velocity domain in the continuous
setting (see Lemma \ref{lem:Poincarecontinub}), the Poincar\'e inequality is a
fundamental tool to obtain the convergence of the solution, and we give below a
version for the bounded velocity case adapted to the velocity discretization
above.

\begin{prop}[Discrete Poincar\'e inequality on bounded velocity domain]
\label{prop:poindiscreteb}
Let $\dvh>0$ be defined as in \eqref{eq:defdv}, and let $g\in\Ldeuxmudvh$.
Then, $$\norm{g-\seq{g}}^2_{\Ldeuxmudvh } \leq \norm{ \Dv g}_{\Ldeuxmudvs }^2.$$
\end{prop}

\begin{proof} Although part of the proof is similar to the proofs of previous
  Poincar\'e inequalities in this paper,
  we give a complete proof, following the lines of the one of Proposition
  \ref{prop:poindiscrete}.  This is to illustrate how our choice of discretization of
  the bounded velocity domain allows to obtain this fundamental inequality.  We
  take $g\in \Ldeuxmudvh$ with $\seq{g}=0$ (note that the boundary conditions
  are preserved by addition of a constant).  We have with the normalization
  convention \eqref{eq:normalizationmuhboundeddomain}
\begin{equation*}
\begin{split}
\dvh^{-1} \norm{g}^2_\Ldeuxmudvh  = \sum_{-\imax < i <\imax} g_i^2 \muh_i
&= \frac{\dvh}{2} \sum_{-\imax < i,j < \imax} ( g_j-g_i)^2 \muh_i \muh_j \\
& = \dvh\sum_{-\imax < i<j <\imax} ( g_j-g_i)^2 \muh_i \muh_j,
\end{split}
\end{equation*}
since $2 \sum_{-\imax < i,j < \imax} g_i g_j \muh_i \muh_j = 2 \sum_{-\imax < i < \imax} g_i \muh_i \sum_{-\imax < j < \imax} g_j \muh_j
= 0$. For $i<j$, we can write the telescopic sum
$$
 g_j-g_i = \sum_{\ell=i+1}^j (g_\ell-g_{\ell-1}),
$$
so that
\begin{equation*} 
\begin{split}
  \dvh^{-1}\sum_{-\imax < i <\imax} g_i^2 \muh_i &
=  \sum_{{-\imax < i <j<\imax}} \sep{ \sum_{\ell=i+1}^j (g_\ell-g_{\ell-1}) }^2 \muh_i \muh_j \\
& \leq \sum_{{-\imax < i <j<\imax}}  \sep{ \sum_{\ell=i+1}^j (g_\ell-g_{\ell-1})^2 } (j-i) \muh_i \muh_j,
\end{split}
\end{equation*}
where we used the discrete flat Cauchy--Schwarz inequality.
Let us  now introduce $\G$ the discrete anti-derivative of $(g_\ell-g_{\ell-1})^2$,
given by
$$
\G_j = - \sum_{\ell=j+1}^{-1} (g_\ell-g_{\ell-1})^2 \textrm{ for } j\leq -1,
\qquad   \G_j = \sum_{\ell=0}^j (g_\ell-g_{\ell-1})^2  \textrm{ for } j\geq 0,
$$
we get (exactly as after \eqref{eq:interm}) that
\begin{equation*}
\begin{split}
\dvh^{-1}\sum_{-\imax < i <\imax} g_i^2 \muh_i &   
= \dvh^{-1}\sum_{{-\imax < i <\imax}} \G_i i  \muh_i 
= \dvh^{-1}\sum_{{-\imax < i <\imax}, i\neq 0} \G_i i \muh_i,
\end{split}
\end{equation*}
where we used the fact that $\sum_{{-\imax < j <\imax}} j\muh_j = 0$ and
$\sum_{-\imax < i <\imax} \muh_j = \dvh^{-1}$.  The last step is to perform a
discrete integration by part using deeply the functional equation
\eqref{eq:okmuh} satisfied by $\muh$ and taking here the boundary terms. We write
using that functional property of $\muh$,
\begin{equation*}
\begin{split}
  \sum_{{-\imax+1 \leq  i  \leq \imax-1 }, \ i \neq 0} \G_i i \muh_i
& =  \sum_{1 \leq i \leq \imax-1 } \G_i i \muh_i
+  \sum_{-\imax+1 \leq  i \leq -1 } \G_i i \muh_i\\
  & = -  \sum_{1 \leq i \leq \imax-1} \G_i \frac{\muh_i- \muh_{i-1}}{\dvh^2}  
-  \sum_{-\imax+1 \leq  i \leq -1 } \G_i \frac{\muh_{i+1} - \muh_{i}}{\dvh^2} \\
  & = - \sum_{1 \leq i \leq \imax-2}  \frac{ \G_i - \G_{i+1}}{\dvh^2} \muh_i 
+ \frac{\G_1}{\dvh^2} \muh_0 - \frac{\G_{\imax-1}}{\dvh^2} \muh_{\imax-1} \\
 & \qquad  -
   \sum_{-\imax+2 \leq  i \leq -1}  \frac{ \G_{i-1} - \G_{i}}{\dvh^2} \muh_i 
- \frac{\G_{-1}}{\dvh^2} \muh_0 + \frac{\G_{-\imax+1}}{\dvh^2} \muh_{\imax-1}.
\end{split}
\end{equation*}
Now, using the definition of $\G$ and in particular the fact that
$$
  \G_1 - \G_{-1} =  (g_{1}- g_0)^2 + (g_{0}- g_{-1})^2,
$$
we obtain as in \eqref{eq:sommemirac} but with the additional boundary terms
\begin{equation} \label{eq:finaldeuxbordsb}
\begin{split}
 &  \sum_{{-\imax+1 \leq  i  \leq \imax-1 }, \ i \neq 0} \G_i i \muh_i
    = \dvh^{-1} \norm{ \Dv g}_{\Ldeuxmudvs }^2 - \sep{ \frac{\G_{\imax-1}}{\dvh^2} -
    \frac{\G_{-\imax+1}}{\dvh^2}} \muh_{\imax-1}.
  \end{split}
  \end{equation}
  Now we have by definition of the anti-derivative $\G$,
  \begin{equation*}
  \begin{split}
   \sep{\frac{\G_{\imax-1}}{\dvh^2} -
    \frac{\G_{-\imax+1}}{\dvh^2}} \muh_{\imax-1}
    = \frac{\G_{\imax-1} -\G_{-\imax+1}}{\dvh^2} \muh_{\imax-1} \\
    = \sep{ \sum_{l=-\imax +2}^{\imax-1} (g_l-g_{l-1})^2} \muh_{\imax-1} \geq 0.
    \end{split}
    \end{equation*}
since this term is non-negative we get from \eqref{eq:finaldeuxbordsb}
\begin{equation*}
\begin{split}
   \sum_{{-\imax+1 \leq  i  \leq \imax-1 }, \ i \neq 0} \G_i i \muh_i
    & \leq \dvh^{-1} \norm{ \Dv g}_{\Ldeuxmudvs }^2.
  \end{split}
  \end{equation*}
The proof is complete.
\end{proof}

Before stating the main result of this subsection, we
estimate the norm of the operator $\Dvs+\vs$ from
$\Ldeuxmudvs$ to $\Ldeuxmudvh$.

\begin{lem}
\label{lem:adjDvsplusvs}
Let $\dvh$ be defined in \eqref{eq:defdv}.
We have for all $g \in \Ldeuxmudvs$,
\begin{equation} \label{eq:bornedvsv}
\norm{(-\Dvs+\vs) g}_\Ldeuxmudvh^2 \leq \frac{4(1+ \dvh \vmax)}{\dvh^2} \norm{g}^2_\Ldeuxmudvs.
\end{equation}
\end{lem}

\begin{proof}
The operator $(-\Dvs +v)$ is bounded from $\Ldeuxmudvs$ to $\Ldeuxmudvh$
since it is a linear mapping between finite dimensional normed spaces.
Note that it is equivalent to estimate
the norm of its adjoint
$
\Dv : \Ldeuxmudvh \longrightarrow \Ldeuxmudvs
$.
For this, we consider $1 \leq j \leq \imax$
and recall that $\mus_j = \muh_{j- 1} = (1 + v_j\dvh ) \muh_j$
from definitions \eqref{eq:defMaxwellboundedvelocity}
and \eqref{eq:defmusb}, where
$v_j = j \dvh$ by definition \eqref{eq:defvib}.
By symmetry, we infer that
\begin{equation} 
\forall j\in\iiis,\qquad
0\leq \mus_j \leq (1 + \dvh \abs{v_j}) \muh_j \leq (1 + \dvh \vmax)\muh_j.
\end{equation}
On the other hand, for $j\in\{1,\cdots,\imax\}$,
$$
\abs{(\Dv g)_j}^2 \leq \frac{2}{\dvh^2} \sep{ \abs{g_j}^2 + \abs{g_{j - 1}}^2 }.
$$
Similar estimates hold for $ -\imax  \leq j \leq -1$ with $j-1$
replaced by $j+1$ in the last inequality.
Using these results we get for $g \in \Ldeuxmudvh$ that
\begin{equation*}
\begin{split}
\dvh^{-1}\norm{\Dv g}_\Ldeuxmudvs^2
& = \sum_{i=-\imax+1, \ i\neq 0}^{\imax-1} \abs{(\Dv g)_i}^2 \mus_i \\
& \leq \frac{4}{\dvh^2} \sum_{i=-\imax+1}^{\imax-1} \abs{g_i}^2 (1 + \dvh \vmax) \muh_i,
\end{split}
\end{equation*}
which implies
\begin{equation} \label{eq:bornedv}
\begin{split}
\norm{\Dv g}_\Ldeuxmudvs^2 \leq \frac{4(1+ \dvh \vmax)}{\dvh^2} \norm{g}^2_\Ldeuxmudvh.
\end{split}
\end{equation}
Therefore, by adjunction, we have \eqref{eq:bornedvsv}.
\end{proof}

We give below the result about the exponential trend to the equilibrium
in the  $\Ldeuxmudvh$ and $\Hunmudvh$ norms
of the solution $(f^n)_{n\in\N}$ of the explicit Euler scheme \eqref{eq:DHFPfb}.
As in the continuous and unbounded cases we look at the following two entropies

\begin{equation}
\fffd(g) \defegal \norm{g}_{\Ldeuxmudvh }^2, \qquad \gggd(g) \defegal \norm{g}_{\Ldeuxmudvh
}^2 
+ \norm{\Dv g}_{\Ldeuxmudvs}^2,\label{eq:deffgh}
\end{equation}
defined for $g\in\R^\iii$. The second entropy is called the Fisher information.
 The result is the following.

\begin{thm} \label{thm:mainexplhom}
Let $\dvh>0$ be defined by \eqref{eq:defdv} and set
\begin{equation*} 
\cfl \defegal  \frac{4 (1+ \dvh \vmax)}{\dvh^2}.
\end{equation*}
Suppose that $\dth>0$ is such that the following CLF condition holds
\begin{equation}
  \label{eq:CFLhomog}
  \dth \cfl <1,
\end{equation}
and set $\kappa=1-\dth\cfl$.
For all $f^0 \in \Ldeuxmudvh$ such that $\seq{f^0} = 0$,
we denote by $(f^n)_{n\in\N}$ the solution of \eqref{eq:DHFPfb}
in $ (\Ldeuxmudvh)^\N$ with initial data $f^0$.
We have for all $n\in \N$,
\begin{equation*}
  \fffd(f^n) \leq (1-2\kappa \dth)^n \fffd(f^0),
\end{equation*}
and
$$
\gggd(f^n)  \leq (1-\kappa \dth)^n \gggd(f^0).
$$
\end{thm}

\begin{proof}
The scheme \eqref{eq:DHFPfb} is well-defined and one has for all $n\in\N$,
$\seq{f^n}=0$ by induction.
We look at the explicit scheme for some $n\in\N$
\begin{equation} \label{eq:explib}
f^{n+1} = f^n -\dth (-\Dvs +\vs)\Dv f^n,
\end{equation}
and we prove below the following estimate
\begin{equation} \label{eq:ineqb}
\norm{f^{n+1}}^2_\Ldeuxmudvh \leq \norm{f^n}^2_\Ldeuxmudvh 
-2\dth \norm{\Dv f^n}_\Ldeuxmudvs^2 + 2 \dth^2 \norm{ (-\Dvs +\vs)\Dv f^n}^2_\Ldeuxmudvh.
\end{equation}
For this, we first take the scalar product of \eqref{eq:explib} with $f^{n+1}$.
We get successively
\begin{equation*}
\begin{split}
  & \norm{f^{n+1}}^2_\Ldeuxmudvh \\
  & = \seq{ f^n, f^{n+1} } - \dth \seq{ (-\Dvs +\vs)\Dv f^n, f^{n+1}}_\Ldeuxmudvh \\
  & =  \seq{ f^n, f^{n+1} } - \dth \seq{ \Dv f^n, \Dv f^{n+1}}_\Ldeuxmudvs \\
  & \leq \frac{1}{2} \norm{f^n}^2_\Ldeuxmudvh + \frac{1}{2}
  \norm{f^{n+1}}^2_\Ldeuxmudvh -\dth \norm{\Dv f^{n}}^2_\Ldeuxmudvs
  - \dth \seq{ \Dv f^n, \Dv \sep{f^{n+1}-f^n}}_\Ldeuxmudvs \\
  & \leq \frac{1}{2} \norm{f^n}^2_\Ldeuxmudvh + \frac{1}{2}
  \norm{f^{n+1}}^2_\Ldeuxmudvh -\dth \norm{\Dv f^{n}}^2_\Ldeuxmudvs
  + \dth^2 \seq{ \Dv f^n, \Dv (-\Dvs +\vs)\Dv f^n}_\Ldeuxmudvs \\
  & \leq \frac{1}{2} \norm{f^n}^2_\Ldeuxmudvh + \frac{1}{2}
  \norm{f^{n+1}}^2_\Ldeuxmudvh -\dth \norm{\Dv f^{n}}^2_\Ldeuxmudvs + \dth^2
  \norm{ (-\Dvs +\vs) \Dv f^n}^2_\Ldeuxmudvh.
 \end{split}
 \end{equation*}
where we used again \eqref{eq:explib} to obtain the terms in $\dth^2$,
and we also used \eqref{eq:ofb}.
Multiplying the preceding inequality by $2$ gives then \eqref{eq:ineqb}.
Using Lemma \ref{lem:adjDvsplusvs} with $g = \Dv f^n$
in the last term of \eqref{eq:ineqb}, we obtain
\begin{equation} \label{eq:ineqbter} \norm{f^{n+1}}^2_\Ldeuxmudvh \leq
  \norm{f^n}^2_\Ldeuxmudvh -2\dth \norm{\Dv f^n}_\Ldeuxmudvs^2 + 2 \dth^2
  \frac{4(1+ \dvh \vmax)}{\dvh^2} \norm{ \Dv f^n}^2_\Ldeuxmudvs.
\end{equation}
Using the CFL condition \eqref{eq:CFLhomog} and the definition of $\kappa$
given in the statement of the theorem,
we infer
\begin{equation} \label{eq:ineqbquatinterm}
\norm{f^{n+1}}^2_\Ldeuxmudvh \leq \norm{f^n}^2_\Ldeuxmudvh -2\dth \kappa \norm{\Dv f^n}_\Ldeuxmudvs^2.
\end{equation}
Using the discrete Poincar\'e inequality of Proposition \ref{prop:poindiscreteb},
this implies
\begin{equation*} 
\norm{f^{n+1}}^2_\Ldeuxmudvh \leq \norm{f^n}^2_\Ldeuxmudvh -2\dth \kappa
\norm{ f^n}_\Ldeuxmudvh^2 = (1-2\kappa \dth) \norm{ f^n}^2,
\end{equation*}
so that by induction
$$
\norm{ f^n}_\Ldeuxmudvh^2 =  \fffd(f^n) \leq (1-2\kappa \dth)^n \fffd(f^0).
$$
This proves the result for the first entropy $\fffd$.

\bigskip
For the  second entropy $\gggd$, we fix $n\in\N$ and
we need to get an estimate on
$\norm{\Dv f^{n+1}}_\Ldeuxmudvs^2$.
Therefore, we apply the operator $\Dv$ to \eqref{eq:explib}, which yields
\begin{equation*} 
\Dv f^{n+1} = \Dv f^n -\dth \Dv (-\Dvs +\vs)\Dv f^n.
\end{equation*}
Following exactly the same method as in the proof of \eqref{eq:ineqbter}
with $\Dv f$ instead of $f$ and operator $\Dv (-\Dvs +\vs)$ instead of
$(-\Dvs +\vs)\Dv$, we get
\begin{multline*} 
\norm{\Dv f^{n+1}}^2_\Ldeuxmudvs \leq \norm{\Dv f^n}^2_\Ldeuxmudvs -2\dth \norm{(-\Dvs+\vs)\Dv f^n}_\Ldeuxmudvh^2  \\ + 2 \dth^2 \norm{ \Dv(-\Dvs +\vs)\Dv f^n}^2_\Ldeuxmudvs.
\end{multline*}
Using the explicit bound of $\Dv$ given in \eqref{eq:bornedv}
(at the end of the proof of Lemma \ref{lem:adjDvsplusvs}), we have
\begin{multline*} 
 \norm{\Dv f^{n+1}}^2_\Ldeuxmudvs  \leq \norm{\Dv f^n}^2_\Ldeuxmudvs -2\dth \norm{(-\Dvs+\vs)\Dv f^n}_\Ldeuxmudvh^2 \\
 + 2 \dth^2 \frac{4(1+ \dvh \vmax)}{\dvh^2} \norm{ (-\Dvs+\vs) \Dv f^n}^2_\Ldeuxmudvh,
\end{multline*}
so that under the CFL condition \eqref{eq:CFLhomog}, we get
\begin{equation*}
\norm{\Dv f^{n+1}}^2_\Ldeuxmudvs \leq \norm{\Dv f^n}^2_\Ldeuxmudvs -2\dth \kappa \norm{(-\Dvs+\vs)\Dv f^n}_\Ldeuxmudvh^2.
\end{equation*}
In particular, we have
\begin{equation} \label{eq:ineqbquatdv}
\norm{\Dv f^{n+1}}^2_\Ldeuxmudvs \leq \norm{\Dv f^n}^2_\Ldeuxmudvs.
\end{equation}
Using \eqref{eq:ineqbquatinterm}
and the discrete Poincar\'e inequality of Proposition \ref{prop:poindiscreteb},
we obtain
\begin{equation*}
\begin{split}
\norm{f^{n+1}}^2_\Ldeuxmudvh & \leq \norm{f^n}^2_\Ldeuxmudvh -2\dth \kappa \norm{\Dv f^n}_\Ldeuxmudvs^2 \\
& \leq \norm{f^n}^2_\Ldeuxmudvh -\dth \kappa \norm{\Dv f^n}_\Ldeuxmudvs^2
-\dth \kappa \norm{\Dv f^n}_\Ldeuxmudvs^2 \\
& \leq \norm{f^n}^2_\Ldeuxmudvh -\dth \kappa \norm{\Dv f^n}_\Ldeuxmudvs^2
-\dth \kappa \norm{f^n}_\Ldeuxmudvh^2.
\end{split}
\end{equation*}
Adding this inequality and \eqref{eq:ineqbquatdv} yields
\begin{equation*}
\begin{split}
\gggd(f^{n+1}) & = \norm{f^{n+1}}^2_\Ldeuxmudvh + \norm{\Dv f^{n+1}}^2_\Ldeuxmudvs \\
& \leq \norm{f^n}^2_\Ldeuxmudvh + \norm{\Dv f^{n}}^2_\Ldeuxmudvs
-\dth \kappa \norm{f^n}_\Ldeuxmudvh^2 -\dth \kappa \norm{\Dv f^n}_\Ldeuxmudvs^2 \\
& \leq (1-\dth \kappa) \gggd(f^{n}),
\end{split}
\end{equation*}
so that by induction
$$
\gggd(f^n) \leq (1-\kappa \dth)^n \gggd(f^0).
$$
The proof is complete. \end{proof}

\subsection{Numerical results}
\def\cashom{ind}
\begin{figure}
  \centering
\begin{subfigure}[t]{0.48\linewidth}
 \setlength\fwidth{5.5cm}
\setlength\fheight{5.5cm} 
  \input{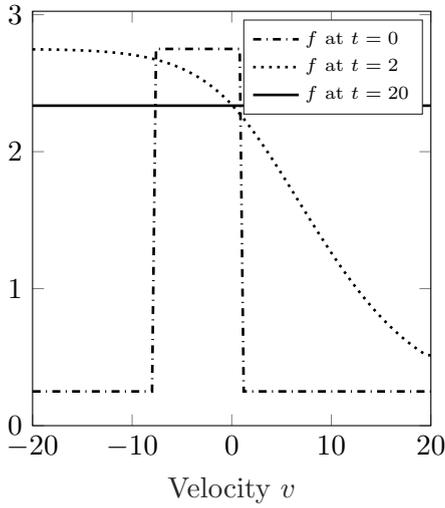}
\end{subfigure}
\hfill
\begin{subfigure}[t]{0.48\linewidth}        
  \centering 
 \setlength\fwidth{5.5cm}
\setlength\fheight{5.5cm}
  \input{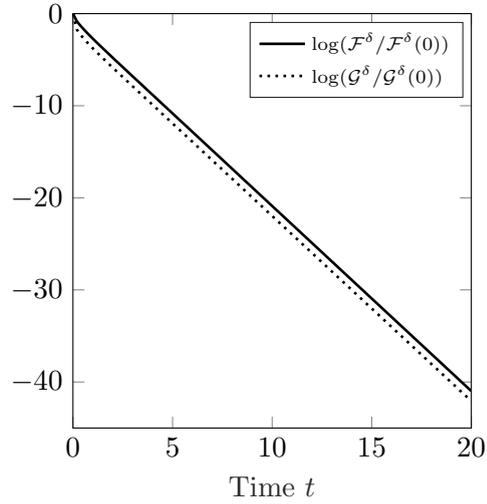}
\end{subfigure} 
\begin{subfigure}[t]{\linewidth}
  \centering   
 \setlength\fwidth{14cm}  
\setlength\fheight{6cm}
  \input{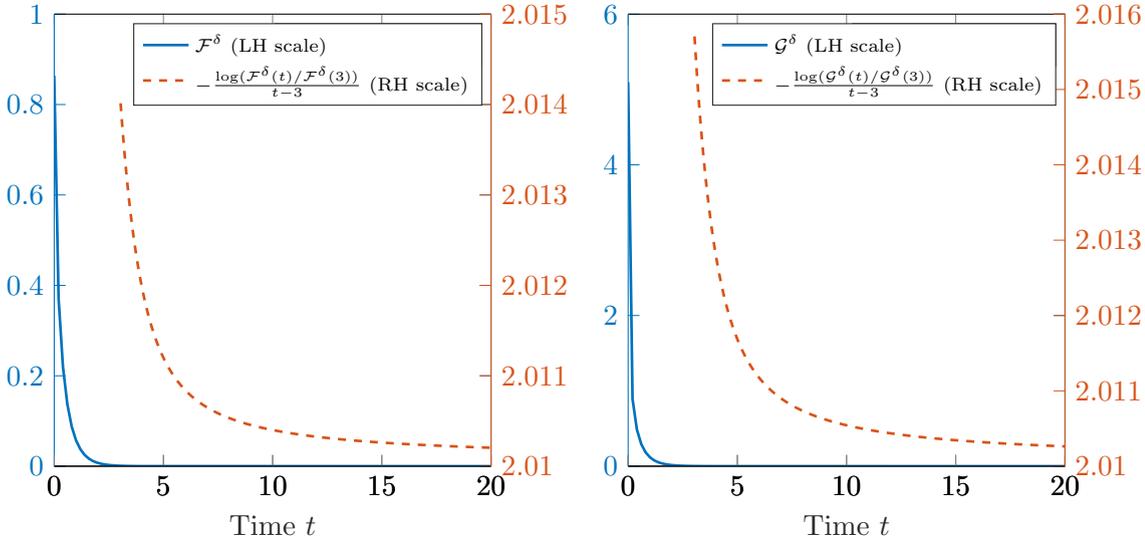}
  \caption{Evolution of the Linearized Entropy $\fffd$, {\em ie} the square of the
    $\Ldeuxmudvh$-norm of $f$ (left) and of the Fisher information $\gggd$
   defined in \eqref{eq:mfi} (right). In each plot, the left-hand scale (plain line) is the linear scale and the
    right-hand scale (dashed line) is the "-log/t" scale that shows the
    numerical rate of convergence in long time. }
\end{subfigure}
\caption{Step function as the initial datum in the homogeneous case}  
 \label{fig:hom1}      
\end{figure}
\def\cashom{rand}
\begin{figure}
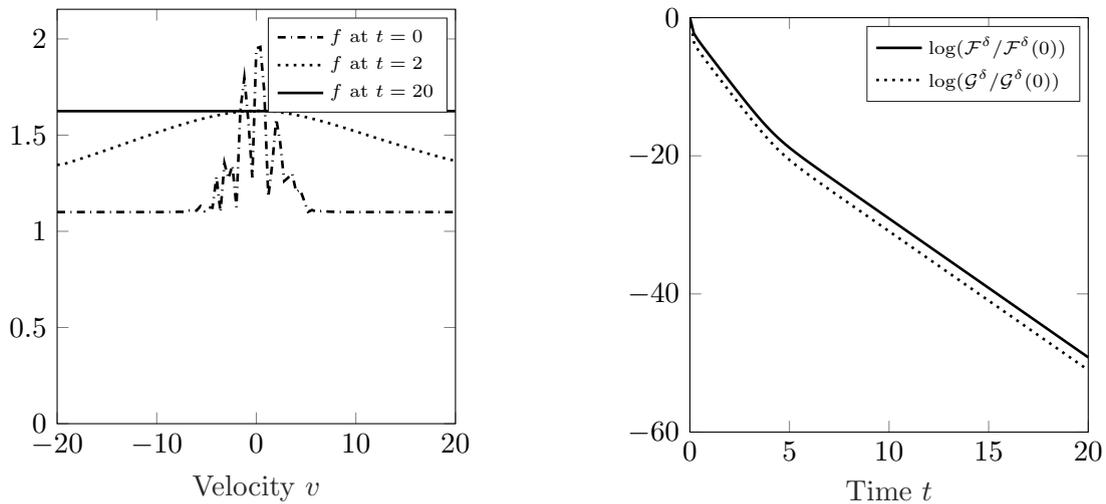
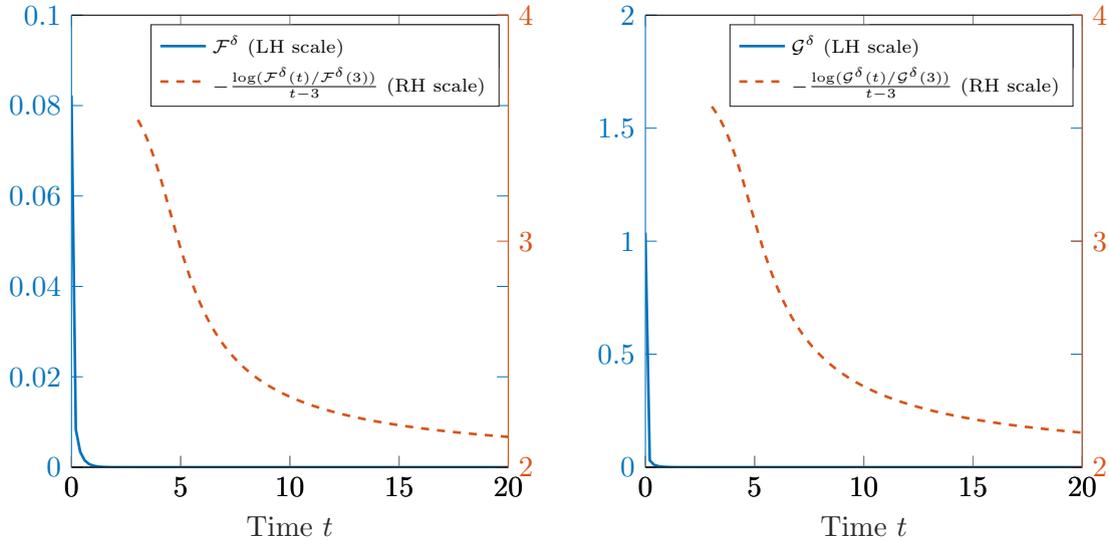

  \centering
\begin{subfigure}[t]{0.48\linewidth}
 \setlength\fwidth{5.5cm}
\setlength\fheight{5.5cm} 
  \input{evolfhom_cas-\cashom.tex}
\end{subfigure}
\hfill 
\begin{subfigure}[t]{0.48\linewidth}        
  \centering 
 \setlength\fwidth{5.5cm}
\setlength\fheight{5.5cm}
  \input{evolentfishhom_cas-\cashom.tex}
\end{subfigure}
\begin{subfigure}[t]{\linewidth}
  \centering   
 \setlength\fwidth{14cm}  
\setlength\fheight{6cm}
  \input{entropyfisherhomtaux_cas-\cashom.tex}
  \caption{Evolution of the Linearized Entropy $\fffd$, {\em ie} the square of the
    $\Ldeuxmudvh$-norm of $f$ (left) and of the Fisher information $\gggd$
   defined in \eqref{eq:mfi} (right). In each plot, the left-hand scale (plain line) is the linear scale and the
    right-hand scale (dashed line) is the "-log/t" scale that shows the
    numerical rate of convergence in long time. } 
\end{subfigure} 
\caption{Random function as the initial datum in the homogeneous case} 
 \label{fig:hom2}     
\end{figure}
This subsection is devoted to the
numerical results obtained through the explicit discretization \eqref{eq:DHFPfb} of
 \eqref{eq:HFPfb}.
 
The quantities of interest here are $\fffd$ and $\gggd$, defined in
\eqref{eq:deffgh}. According to Theorem \ref{thm:mainexplhom}, they are expected
to decrease geometrically fast. The tests that are presented here aim at
illustrating this fact in two cases: 
\begin{itemize}
\item the initial datum is a step function (see Figure \ref{fig:hom1}-(A)). The
  logarithms of the entropy $\fffd$ and of the Fisher information $\gggd$
  decrease linearly fast (see Figure \ref{fig:hom1}-(B)), with a rate that is
  close to $2$, as can be seen in Figures \ref{fig:hom1}-(C). The exponential
  decrease is consistent with Theorem \ref{thm:mainexplhom}, and the rate being
  close to $2$ is consistent with Theorem
  \ref{thm:exponentialtrendtoequilibriumb} for $\fffd$, and shows the bound to
  be optimal, and better than expected for $\gggd$.
\item the initial datum is a random function (see Figure \ref{fig:hom2}-(A)) The
  logarithms of the entropy $\fffd$ and of the Fisher information $\gggd$
  decrease linearly fast (see Figure \ref{fig:hom2}-(B)), with a rate that is
  close to $2$, as can be seen in Figures \ref{fig:hom2}-(C). Again, the exponential
  decrease is consistent with Theorem \ref{thm:mainexplhom}, and the rate being
  going to $2$ is consistent with Theorem
  \ref{thm:exponentialtrendtoequilibriumb} for $\fffd$ and $\gggd$.
\end{itemize}
Comparing the two previous test cases, we get a hint that there is a very fast
regularizing effect in short time, as noted in \cite{PZ16}. The second
initial datum is way less smooth that the first one and the range of the decrease rate 
is a lot larger in the second case. A perspective of our work would be to
investigate the change of slope in Figure \ref{fig:hom2}-(B).

\section{The inhomogeneous equation on bounded velocity domains}
\label{sec:eqinhomoboundedvelocity}

This section is devoted to the analysis of the inhomogeneous
Fokker-Planck equation on bounded velocity domains, in the fully discretized
setting, meaning discretized in velocity, in space and in time.
The main result is the exponential convergence to equilibrium
of numerical solutions stated in Theorem \ref{thm:decrexpeulerexpldiscr}.
We first recall briefly in Section \ref{sec:inhomcb} the statements
for the continuous equation set on a bounded velocity domain.
Next, we study in Section \ref{sec:inhomobounded}
a full discretization by an explicit Euler scheme in time,
by an extension of the operators $\Dv$ and $\Dvs$ introduced in Section
\ref{sec:homobounded} in velocity to this inhomogeneous case,
and a space discretization operator $\Dx$ similar to the one introduced
in the unbounded velocity inhomogeneous case in Section \ref{sec:inhomsd}.
In this context, we prove our main result :
Theorem \ref{thm:decrexpeulerexpldiscr}.
We conclude with numerical simulations carried out using
this numerical scheme.

\subsection{The fully continuous analysis}\label{sec:inhomcb}
In order to prepare the fully discrete inhomogeneous case
in the next subsection,
we briefly show how to extend the results of Section \ref{sec:inhomc}
for the inhomogeneous equation on an {\it unbounded} velocity domain
to the case of a {\it bounded} velocity domain.

In this bounded-velocity setting, we stick to the notations introduced
in Section \ref{sec:homoboundedcontinuous} for the homogeneous case.
In particular the velocity domain is
$I = (-\vmax,\vmax)$ for some $\vmax>0$.
We propose a suitable functional framework for the following
inhomogeneous Fokker--Planck equation with unknown
$F(t,x,v)$ where $(t,x,v) \in \R^+ \times\T\times I$
\begin{equation}\label{eq:IHFPFbinhomo}
  {\D_t} F+v{\D_x}F-{\D_v} ({\D_v}+v)F=0, \qquad F|_{t=0} = F^0, \qquad (({\D_v}+v)F)(\cdot,\cdot, \pm\vmax)=0.
\end{equation}
The initial datum $F^0$ is a non-negative function of $L^1(\T\times I,\dx\dv)$
with $\int_{\T\times I} F^0(x,v)\dx \dv=1$.
The Maxwellian function
\begin{equation*}
  \mu(x,v)=\dfrac{\exp^{-v^2/2}}{\displaystyle\int_{I}\exp^{-w^2/2}\dw},
\end{equation*}
is a continuous equilibrium of \eqref{eq:IHFPFbinhomo},
normalized  in $L^1(\T \times I,\dx \dv)$.
As we did for the unbounded velocity domain case in Section \ref{sec:inhomc},
we pose $F=\mu+\mu f$,
and the rescaled density $f$ solves
\begin{equation}
  \label{eq:IHFPfb}
  {\D_t} f + v{\D_x} f+(- {\D_v}+v) {\D_v} f = 0 , \qquad f|_{t=0} = f^0, \qquad {\D_v} f(\cdot,\cdot, \pm\vmax)=0.
\end{equation}
We introduce the corresponding
functional space $\Ldeuxmudvdxb$ and its subspace
$$
 \Hunmudvdxb \defegal \set{ g \in \Ldeuxmudvdxb, \ {\D_v} g \in \Ldeuxmudvdxb}.
$$
For $g\in L^1(\T \times I, \mu\dx \dv)$,
we denote its $(x,v)$-mean by $\seq{g} = \iint_{\T\times I} g(v) \mudv\dx$.
From now on, the norms and scalar products without subscript are taken
in $\Ldeuxmudvdxb$.
In these spaces, we have again a Poincar\'e inequality (see Lemma
\ref{lem:fullPoincarecontinub} below).
The proof of that inequality follows exactly the lines
of the one for the continuous, inhomogeneous, unbounded-velocity case presented
in Lemma \ref{lem:fullPoincarecontinu}
(but using the homogeneous Poincar\'e inequality on bounded velocity domain
of Lemma \ref{lem:Poincarecontinub} as a tool,
instead of the homogeneous Poincar\'e inequality on unbounded velocity domain
(Lemma \ref{lem:Poincarecontinu})):

\begin{lem}[Inhomogeneous Poincar\'e inequality on bounded velocity domains]
\label{lem:fullPoincarecontinub}
For all $g \in \Hunmudvdx$, we have
 \begin{equation*}
  \norm{g-\seq{g}}^2  \leq  \norm{{\D_v} g}^2 + \norm{{\D_x} g}^2.
\end{equation*}
\end{lem}

In order to state the main result concerning the convergence to the equilibrium
for the solutions of Equation \eqref{eq:IHFPfb} in Theorem \ref{thm:decrexpb},
we introduce a little more functional framework.
We  consider the operator $P = v{\D_x} + (- {\D_v}+v) {\D_v}$ with domain
$$
D(P) = \set{ g \in \Ldeuxmudvdxb,
\ (v{\D_x} + (- {\D_v}+v) {\D_v} )g \in \Ldeuxmudvdxb, \ {\D_v} g (\cdot,\pm \vmax) =0},
$$
which corresponds to the evolution operator in \eqref{eq:IHFPfb} with Neumann
conditions in velocity.  Note that constant functions are in $D(P)$.  Equation
\eqref{eq:IHFPfb} reads then ${\D_t} f+ P f = 0$ with initial condition
$f(0,\cdot,\cdot)=f^0$.

The non-negativity of the operator $P$
is straightforward since $v{\D_x}$ is skew-adjoint in $\Ldeuxmudvdxb$.
The maximal accretivity of this operator in $\Ldeuxmudvdxb$ or $\Hunmudvdxb$
is not so easy and we refer for example to \cite{HelN04}.
As in the unbounded velocity case, using the Hille--Yosida Theorem,
this implies that for an initial datum $f^0 \in \Ldeuxmudvdxb$
(resp. $\Hunmudvdxb$) there exists a unique solution in
$\ccc^0( \R^+, \Ldeuxmudvdxb)$ (resp. $\ccc^0(\R^+, \Hunmudvdxb$).
Moreover, for if $f^0 \in D(P)$
(resp. $D_\Hunmudvdxb(P)$), there exists a unique solution in
$\ccc^1( \R^+, \Ldeuxmudvdxb)$ (resp. $\ccc^1(\R^+, \Hunmudvdxb$).

As a norm in $\Hunmudvdx$ we choose
the standard, the square of which is defined for $g\in \Hunmudvdxb$ by
$$
\norm{g}_\Hunmudvdxb^2 = \norm{g}^2 + \norm{{\D_v} g}^2
+ \norm{{\D_x} g}^2.
$$

As in the unbounded velocity case for Section \ref{sec:eqinhomo},
we shall define a modified entropy
adapted to the $\Hunmudvdxb$ framework.
For $C>D>E>1$ to be precised later, it is
defined for $g\in \Hunmudvdxb$ by
\begin{equation*}
  \hhh(g) =  C\norm{g}^2+D\norm{{\D_v} g}^2+E\seq{{\D_v} g,{\D_x} g}+\norm{{\D_x}g}^2.
\end{equation*}

Following exactly the proof of  Lemma \ref{lem:equiv} we again check that
\begin{lem}
  If $E^2<D$ then for all $g\in \Hunmudvdxb$,
  \begin{equation*}
    \dfrac{1}{2}\norm{g}_{\Hunmudvdxb}^2\leq\hhh(g)\leq 2C\norm{g}_{\Hunmudvdxb}^2.
  \end{equation*}
\end{lem}

The main result is then the following theorem, the proof of which
is exactly the same as that of Theorem \ref{thm:decrexp}

\begin{thm}\label{thm:decrexpb}
  Assume that $C>D>E>1$ satisfy $E^2<D$ and $(2D+E)^2<2C$.
Let $ f^0 \in \Hunmudvdxb$ such that $\seq{f^0} = 0$ and let $f$ be the solution
in $\ccc^0(\R^+, \Hunmudvdxb)$ of Equation \eqref{eq:IHFPfb}.
Then for all $t\geq 0$,
  \begin{equation*}
     \hhh(f(t))\leq \hhh(f^0)\exp^{-2\kappa t}.
  \end{equation*}
  with $2\kappa = \frac{E}{8C}$.
\end{thm}

The following corollary is also similar to the one proposed after the proof
of Theorem \ref{thm:decrexp}.

\begin{cor}\label{cor:decrexpb}
  Let $C>D>E>1$ be chosen as in Theorem \ref{thm:decrexp},
and pose $\kappa = E/(16C)$.
Let $f^0\in \Hunmudvdxb$ and let $f$ be the solution in
$\ccc^0(\R^+, \Hunmudvdxb)$
of Equation \eqref{eq:IHFPfb}.
Then for all $t\geq 0$, we have
  \begin{equation*}
   \norm{f(t) - \seq{f^0}}_{\Hunmudvdxb}\leq 2\sqrt{C} \exp^{-\kappa t} \norm{f^0 - \seq{f^0}}_{\Hunmudvdxb}.
  \end{equation*}
\end{cor} 

\subsection{The full discretization and proof of Theorem \ref{thm:eulerexplicite}}
\label{sec:inhomobounded}

As we did in the unbounded case, we want to discretize the velocity domain
$I = (-\vmax, \vmax)$ and the equation and boundary conditions of
\eqref{eq:IHFPFbinhomo}.

Concerning the discretization of the velocity variable,
we use the very same definitions introduced in Subsection \ref{subsec:homsdb}
in the homogeneous setting
for $\imax$, $\dvh$, the sets $\iii$ and $\iiis$,
the operators $\Dv$, $\Dvs$, $v$ and $\vs$,
the discretized Maxwellians $\muh$ and $\mus$
(see {\it e.g.} Definition \ref{def:dvb}).
For these operators, the space index $j$ plays the role of a parameter.

Concerning the discretization of the space periodic domain $\T$,
we pick from Section \ref{sec:inhomsd} the definitions and notations.
We denote $\dxh >0$ the (uniform) step of discretization of the torus $\T$ into $N$ intervals, and  denote $\jjj = \Z/ N\Z$ the finite set of indices of the discretization in $x \in \T$.
In what follows, the index $i \in \iii$ will always refer to the velocity variable and the index $j \in \jjj$ to the space variable.
In particular, for a  sequence $f = (f_{i,j})_{i\in \Z, j\in \jjj}$ the derivative-in-space $\Dx f$ is then defined by the following centered scheme
$$
\forall i \in \iii, j\in \jjj, \qquad (\Dx f)_{j,i} = \frac{f_{j+1,i} - f_{j-1,i}}{2\dxh}.
$$

Our goal is to introduce a {\it discrete}
functional framework that allows to conclude to qualitatively correct
asymptotic behaviour for the numerical schemes in Theorem
\ref{thm:decrexpeulerexpldiscr}, by mimicking the proofs
of the results recalled in Section \ref{sec:inhomcb} for the {\it continuous}
inhomogeneous equation on bounded velocity domain.
Before introducing the time-discretization,
we equip $\R^{\jjj\times\iii}$ with
the $\ell^1(\jjj\times\iii, \muh \dvh \dxh)$ norm and
we introduce adapted Hilbertian norms.

\begin{defn}
We  denote by $\Ldeuxmudvdxh$ the space $\R^{\jjj\times\iii}$
made of finite sequences $g$ and set
$$
\norm{ g}_{\Ldeuxmudvdxh }^2  \defegal  \dvh \dxh \sum_{j\in \jjj, i\in \iii} (g_{j,i})^2 \muh_i .
$$
This defines a squared Hilbertian norm,
and the related scalar product will be denoted by $\seq{ \cdot, \cdot}$.
For $g \in \Ldeuxmudvdxh$, we also define the mean
$$
\seq{g} \defegal  \dvh \dxh\sum_{j\in \jjj, i\in \iii} g_{j,i} \muh_i = \seq{g, 1}
,
$$
of $g$ (with respect to this weighted scalar product in both velocity and space).
We define the space $\Ldeuxmudvdxs$ to be $\R^{\jjj\times\iiis}$
endowed with the Hilbertian norm defined for $h\in \R^{\jjj\times\iiis}$ by
its square
$$
\norm{ h}_{\Ldeuxmudvdxs }^2  \defegal  \dvh \dxh \sum_{j\in \jjj, i\in \iiis} (h_{j,i})^2 \mus_i .
$$
The related scalar product will be denoted by $\seqs{ \cdot, \cdot}$.
Eventually  we define $\Hunmudvdxh$ to be the space $\Ldeuxmudvdxh=\R^{\jjj\times\iii}$ with the Hilbertian norm defined by its square for $g\in\R^{\iii\times\jjj}$ as
$$
\norm{g}^2_\Hunmudvdxh \defegal \norm{g}^2_\Ldeuxmudvdxh + \norm{\Dv g}^2_\Ldeuxmudvdxs +
\norm{\Dx g}^2_\Ldeuxmudvdxh.
$$
\end{defn}

We define the operator $\Pd$ involved in the discretized rescaled
Fokker--Planck equation by
\begin{equation*} 
 \Pd = \Xzerod + (-\Dvs+\vs)\Dv
 \end{equation*}
with $\Xzerod = v \Dx : \Ldeuxmudvdxh \hookrightarrow \Ldeuxmudvdxh$
defined by for $i\in \iii$ by
 $$
( \Xzerod g)_{j,i} = (v \Dx g)_{i,j} \textrm{ when } i\neq 0, \qquad ( \Xzerod g)_{j,0} =0.
 $$

\noindent
The discretized version of the rescaled equation \eqref{eq:IHFPFbinhomo}
is therefore the linear ODE set in
$\R^{\jjj\times\iii}$ that reads
\begin{equation} 
\label{eq:inhomocontinuoustimediscreteFPbounded}
{\D_t} f+ \Pd f = 0. 
\end{equation}
We now summarize the structural properties of
\ref{eq:inhomocontinuoustimediscreteFPbounded} and of the operator $\Pd$ in the
following Proposition.  From now on and for the rest of this subsection, we work
in $\Ldeuxmudvdxh$ and denote (when no ambiguity happens) the corresponding norm
$\norm{\cdot}$ without subscript. Similarly $\norms{\cdot}$ stands for the norm
in $\Ldeuxmudvdxs$.

\begin{prop} 
 We have
\begin{enumerate}
     \item
The operator $( - \Dvs +\vs) \Dv$
is self-adjoint and
the operator $\Xzerod$ is skew-adjoint
in $\Ldeuxmudvdxh$. Moreover, for all $g \in \Ldeuxmudvdxh$,
$h \in \Ldeuxmudvdxs $, we have
\begin{align}
 \seq{ ( - \Dvs +\vs)  h, g}  = \seq{h, \Dv g}_\sharp, \\
\label{eq:ofinb2}
\seq{\Pd g,g} = \seq{ ( - \Dvs +\vs) \Dv g, g} = \norm{\Dv g}_\sharp^2.
\end{align}
\item Constant functions are the only equilibrium states of equation
  \eqref{eq:inhomocontinuoustimediscreteFPbounded} and we have the conservation
  of mass property : for all $t\geq 0$, $\seq{f(t)} = \seq{f^0}$.
\end{enumerate}
\end{prop}

We pick from Section \ref{sec:inhomsd} the definitions of the operators $S$,
$S^\sharp$ and $S^\flat$ as well as the results and embeddings given in Lemmas
\ref{lem:constcontS} and \ref{lem:constcontSsharp} with the velocity set of
index $\Z$ or $\Z^*$ there replaced here by $\iii$ or $\iiis$ respectively.
Note that the spaces $\Ldeuxmudvdxh$ and $\Ldeuxmudvdxs$ here are exactly
adapted to the inherent shift defining $S$, $S^\sharp$ and $S^\flat$.  Moreover,
it is clear that the commutations lemmas \ref{lem:triple}, \ref{lem:sb} and
\ref{lem:estimdelta} remain true thanks to our choice of indices $\iii$, $\iiis$
and the functional associated spaces of the current section.

We pick from the same section \ref{sec:inhomsd}
the definition of the following modified entropy
defined for $g\in \Hunmudvdxh$ by
\begin{equation}
\hhhd(g) =  C\norm{g}^2+D\norms{\Dv g}^2+E\seqs{\Dv g,S\Dx g}+\norm{\Dx g}^2,
\end{equation}
for well chosen $C>D>E>1$ to be defined later.
Lemma \ref{lem:equivd} remains true in the bounded-velocity discretized context
this section
and we have again with the same proof as there.

\begin{lem}
  If $2E^2<D$ then for all $g\in \Hunmudvdxh$,
  \begin{equation}
    \label{eq:equivdb}
    \dfrac{1}{2}\norm{g}_{\Hunmudvdxh}^2\leq\hhhd(g)\leq 2C\norm{g}_{\Hunmudvdxh}^2.
  \end{equation}
\end{lem}

\noindent
Provided that $2E^2<D$, the modified entropy $\hhhd$ defines
a Hilbertian norm on $\R^{\jjj\times\iii}$, associated with
the following polar form
\begin{equation*}
\phid(g,\gtilde)=C\seq{g,\gtilde}+D\seq{\Dv g,\Dv \gtilde}_\sharp
+ \frac{E}{2}\left(\seq{S\Dx g,\Dv \gtilde}_\sharp+\seq{\Dv g,S \Dx \gtilde}_\sharp\right)
+ \seq{\Dx g,\Dx \gtilde},
\end{equation*}
defined for $g$, $\gtilde \in \R^{\jjj\times\iii}$.
The Cauchy--Schwarz--Young inequality holds true
\begin{equation}
\label{eq:CSYbounded}
  |\phid(g,\gtilde)| \leq \sqrt{\hhhd(g)}\sqrt{\hhhd(\gtilde)}
\leq \frac{1}{2} \hhhd(g) + \frac{1}{2} \hhhd(\gtilde),
\end{equation}
for all $g,\tilde g$.
Moreover, the Poincar\'e inequality in space holds true as well.
First, in the form of \eqref{eq:Poincarediscretxseulement}
in the discretized space variable, and then, following
exactly the lines of the proof of Lemma \ref{lem:fullPoincarediscret},
in the form of the following Lemma.

\begin{lem}[Fully discrete inhomogeneous Poincar\'e inequality for bounded
velocity domains]
\label{lem:fullPoincarediscretb}
For all $g \in \Hunmudvdxh$, we have
 \begin{equation*}
   \norm{g-\seq{g}}^2_\Ldeuxmudvdxh  \leq  \norm{\Dv g}_\Ldeuxmudvdxs^2 + \norm{\Dx g}_\Ldeuxmudvdxh^2.
  \end{equation*}
\end{lem}

The discretization in time of the rescaled inhomogeneous
discretized Fokker--Planck equation
\eqref{eq:inhomocontinuoustimediscreteFPbounded}
that we consider is given by the following explicit scheme

\begin{defn}
We shall say that a sequence $f = (f^n)_{n \in \N} \in (\Ldeuxmudvdxh)^\N$
satisfies
the scaled fully discrete explicit inhomogeneous Fokker-Planck equation if
for some $\dth>0$ and all $n\in\N$,
\begin{equation}
  \label{eq:eulerexpldiscrinhomo}
  f^{n+1} = f^n - \dth (v \Dx f^{n}+(-\Dvs+\vs) \Dv f^{n}).
\end{equation}
\end{defn}

\noindent
As in all the previous cases,
we can check that constant sequences are the only equilibrium states
of this equation, and that the mass conservation property is satisfied:
for all $n\in\N$, $\seq{f^n} = \seq{f^0}$.

Before getting to the main result of this section in Theorem
\ref{thm:decrexpeulerexpldiscr}, we state the following Lemma, which
provides us with explicit bounds on the norms of the linear continuous
operators in the discrete equation \eqref{eq:eulerexpldiscrinhomo}.

\begin{lem}
\label{lem:CFLinhomo}
Let us define
\begin{equation}
\label{eq:defabc}
\quad a^2 = 4 \frac{1+ \dvh \vmax}{\dvh^2}, 
\qquad b^2 = 4 \frac{1+ \dvh \vmax}{\dxh^2}, 
\qquad c^2 = 4\frac{\vmax^2}{\dxh^2},
\end{equation}
and set
$$
\bcfl = \max \set{ 1, a^2,b^2,c^2}.
$$
Then we have for all $g \in \Ldeuxmudvdxh$ and $h \in \Ldeuxmudvdxs$
\begin{equation} \label{eq:estimnorms}
\begin{split}
  \norms{\Dv g} \leq a \norm{g}, \qquad \norms{S \Dx g} & \leq b \norm{g},
  \qquad \norm{ \Dx g} \leq b \norm{g}, \\
  \norm{\Xzerod g} \leq c \norm{g}, \qquad & \norms{\Xzerod h} \leq c \norms{h}.
\end{split}
\end{equation}
\end{lem}

\begin{proof}
Let us first prove now \eqref{eq:estimnorms}.
We first note that the inequality is already proven in \eqref{eq:bornedv}.
The proof of the second one follows exactly the same proof.
For the third one, we directly have by triangular inequality that
$$
\norm{ \Dx g} \leq \frac{2}{\dxh} \norm{g} \leq b \norm{g}.
$$
For the inequalities involving $\Xzerod$, we just note that operator
multiplication by $v$ is bounded with bound $\vmax$
and use the bound for $\Dx$ above, which yields directly the result.

\end{proof}

We can now state the main Theorem of this subsection concerning the exponential
trend to equilibrium of solutions of Equation  \eqref{eq:eulerexpldiscrinhomo}.

\begin{thm}   \label{thm:decrexpeulerexpldiscr}
Assume $C>D>E>1$ and  $\dvh_0\in(0,1)$
are chosen as in Theorem \ref{thm:decrexpd}
and set
$$
\bcfl = \max \set{ 1, 4 \frac{1+ \dvh \vmax}{\dvh^2}, 4 \frac{1+ \dvh
    \vmax}{\dxh^2},4\frac{\vmax^2}{\dxh^2}}.
$$
For all $\dvh\in (0,\dvh_0)$, $\dxh>0$,
$f^0 \in \Hunmudvdxh$ such that $\seq{f^0} = 0$, and $\dth>0$
satisfying the CFL condition
\begin{equation} \label{eq:condcflinhom}
4 (C+ 4D + 9E +2) \dth \bcfl (1+\vmax^2)< 1,
\end{equation}
the solution $(f^n)_{n\in\N}$ of the discretized inhomogeneous Fokker--Planck
equation \eqref{eq:eulerexpldiscrinhomo} in $ (\Hunmudvdxh)^\N$
with initial data $f^0$ satisfies
\begin{equation*}
\forall n\in\N,\qquad  \hhhd(f^n) \leq (1-2\kappa \dth)^n \hhhd(f^0),
\end{equation*}
where $\kappa>0$ is such that
$4C\kappa = 1-4(C+ 4D + 9E +2) (1+\vmax^2)\dth \bcfl$.
\end{thm}

\begin{proof}[of Theorem \ref{thm:decrexpeulerexpldiscr}]
  Fix $\dvh\in(0,\dvh_0)$, $\dxh>0$ and $\dth>0$ as in the hypotheses.  Let
  $f^0\in \Hunmudvdxh$ with zero mean.  Denote by $(f^n)_{n\in\N}$ the sequence
  in $\R^{\jjj\times\iii}$ provided by the explicit Euler scheme
  \eqref{eq:eulerexpldiscrinhomo} for which we recall that $n\in\N$,
  $\seq{f^n}=0$.  We fix $n\in\N$ and as in the proof of Theorem
  \ref{thm:decrexpeulerimpldiscr}, we compute the four terms appearing in the
  definition of $\hhhd(f^{n+1})$ before estimating their sum.  For this, we
  extensively use the computations done there and in the proof of Theorem
  \ref{thm:mainexplhom}.  Our method is the following : bound every term in
  $\hhhd(f^{n+1})$ by a sum of three terms of order $0$, $1$ and $2$ in $\dth$.
  Then, sum up the inequalities after multiplication by $C$, $D$, $E$, and $1$.
  Recognize $\dddd(f^n)$ in the sum of terms of order $1$, then transform the
  sum of the terms of order $2$ into a of order $1$ using the CFL condition
  \eqref{eq:condcflinhom} that can be integrated in the preceding term of order
  $1$ thanks to a version of \eqref{eq:pourplustard} adapted to this bounded
  velocity context. Eventually, conclude using the Cauchy--Schwarz--Young
  inequality \eqref{eq:CSYbounded}.

First, we compute the squared $\Ldeuxmudvdxh$-norm of $f^{n+1}$ using
relation \eqref{eq:eulerexpldiscrinhomo} twice. This yields

\begin{eqnarray}
\lefteqn{\norm{f^{n+1}}^2}\nonumber\\
& = & \seq{f^n,f^{n+1}} -\dth  \seq{\Pd f^n,f^{n+1}}\nonumber\\
 &  = & \seq{f^n,f^{n+1}} -\dth \seq{\Pd f^n,f^{n}}+\dth^2\seq{\Pd f^n,\Pd f^n}\nonumber\\
& = & \seq{f^n,f^{n+1}} - \dth\norm{\Dv f^n}^2_\sharp
+\dth^2\rrrd_1(f^n),
\label{eq:ineqbi}
 \end{eqnarray}
using \eqref{eq:ofinb2} for the term in $\dth$ and defining
\begin{equation*}
  \rrrd_1(f^n)=\norm{\Pd f^n}^2,
\end{equation*}
for the term in $\dth^2$.

For the second term in the definition of the discrete
entropy $\hhhd$, we compute the squared $\Ldeuxmudvdxs$-norm of $\Dv f^{n+1}$
using relation \eqref{eq:eulerexpldiscrinhomo} twice. This yields
\begin{equation} \label{eq:ineqbii}
\begin{split}
\lefteqn{\norm{\Dv f^{n+1}}^2_\sharp }\\
& =  \seq{\Dv f^n,\Dv f^{n+1}}_\sharp-
\dth\seq{\Dv v\Dx f^n,\Dv f^{n+1}}_\sharp - \dth\seq{\Dv
(-\Dvs+\vs)\Dv f^n,\Dv f^{n+1}}_\sharp\\
& =  \seq{\Dv f^n,\Dv f^{n+1}}_\sharp-\dth\seq{S\Dx f^n,\Dv f^{n+1}}_\sharp
-\dth\seq{v\Dv\Dx f^n,\Dv f^{n+1}}_\sharp\\
&  - \dth \seq{(-\Dvs+\vs)\Dv f^{n},(-\Dvs+\vs)\Dv f^{n+1}}\\
& =  \seq{\Dv f^n,\Dv f^{n+1}}_\sharp-\dth\seq{S\Dx f^n,\Dv f^{n}}_\sharp
+\dth^2 \seq{S\Dx f^n,\Dv \Pd f^{n}}_\sharp\\
&  -\dth\underbrace{\seq{v\Dx\Dv f^n,\Dv f^{n}}_\sharp}_{=0}
+\dth^2\seq{v\Dv\Dx f^n,\Dv \Pd f^{n}}_\sharp\\
&  - \dth \seq{(-\Dvs+\vs)\Dv f^{n},(-\Dvs+\vs)\Dv f^{n}}\\
&
+ \dth^2 \seq{(-\Dvs+\vs)\Dv f^{n},(-\Dvs+\vs)\Dv \Pd f^{n}}\\
& =  \seq{\Dv f^n,\Dv f^{n+1}}_\sharp
-\dth\left(
\seq{S\Dx f^n,\Dv f^{n}}_\sharp
+ \norm{(-\Dvs+\vs)\Dv f^n}^2
\right)+\dth^2 \rrrd_2(f^n),
\end{split}
\end{equation}
where we have set
\begin{eqnarray*}
  \rrrd_2(f^n) &= &
\seq{S\Dx f^n,\Dv \Pd f^{n}}_\sharp
+
\seq{v\Dv\Dx f^n,\Dv \Pd f^{n}}_\sharp\\
& & +
\seq{(-\Dvs+\vs)\Dv f^{n},(-\Dvs+\vs)\Dv \Pd f^{n}}.
\end{eqnarray*}

For the third term in $\hhhd(f^{n+1})$,
we take advantage of the computations carried out in Section \ref{sec:eqinhomo}
for the unbounded in velocity, inhomogeneous, semi-discretized
and implicit case.
In particular, we have as in \eqref{eq:estim3discr} the following relation
(with $f^n$ here instead of $f^{n+1}$ there in the right-hand side),
by using the definition \eqref{eq:eulerexpldiscrinhomo} of the explicit
Euler scheme twice
\begin{eqnarray*}
  \lefteqn{2\seq{S\Dx f^{n+1},\Dv f^{n+1}}_\sharp = }\\
 & &  \seq{S\Dx f^{n},\Dv f^{n+1}}_\sharp + \seq{S\Dx f^{n+1},\Dv f^{n}}_\sharp\\
& & -\dth\left(
\seq{S\Dx v \Dx f^{n},\Dv f^{n+1}}_\sharp + \seq{S\Dx f^{n+1},\Dv v \Dx f^{n}}_\sharp
\right)\\
& & -\dth\left(
\seq{S\Dx(-\Dvs+\vs)\Dv f^{n},\Dv f^{n+1}}_\sharp
+ \seq{S\Dx f^{n+1},\Dv(-\Dvs+\vs)\Dv f^{n}}
\right).
\end{eqnarray*}

\noindent
Using again Equation \eqref{eq:eulerexpldiscrinhomo} to replace $f^{n+1}$
in the terms in $\dth$ above, we get
\begin{equation} \label{eq:termdt}
\begin{split}
\lefteqn{  2\seq{S\Dx f^{n+1},\Dv f^{n+1}}_\sharp=}\\
 & \seq{S\Dx f^{n},\Dv f^{n+1}}_\sharp + \seq{S\Dx f^{n+1},\Dv f^{n}}_\sharp\\
&  -\dth\left(
\seq{S\Dx v \Dx f^{n},\Dv f^{n}}_\sharp + \seq{S\Dx f^{n},\Dv v \Dx f^{n}}_\sharp
\right)\\
&  -\dth\left(
\seq{S\Dx(-\Dvs+\vs)\Dv f^{n},\Dv f^{n}}_\sharp
+ \seq{S\Dx f^{n},\Dv(-\Dvs+\vs)\Dv f^{n}}_\sharp \right)\\
& + \dth^2 \rrrd(f^n),
\end{split}
\end{equation}
where $\rrrd_3(f^n) $ is given by
\begin{equation*}
\begin{split}
\rrrd_3(f^n) = &  \seq{S\Dx \Xzerod f^{n},\Dv (\Xzerod+ (-\Dvs+\vs)\Dv))  f^{n}}_\sharp
 \\
& + \seq{S\Dx (\Xzerod+ (-\Dvs+\vs)\Dv)) f^{n},\Dv \Xzerod f^{n}}_\sharp \\
& +
\seq{S\Dx(-\Dvs+\vs)\Dv f^{n},\Dv (\Xzerod+ (-\Dvs+\vs)\Dv)) f^{n}}_\sharp \\
& + \seq{S\Dx (\Xzerod+ (-\Dvs+\vs)\Dv)) f^{n},\Dv(-\Dvs+\vs)\Dv f^{n}}_\sharp.
\end{split}
\end{equation*}

\noindent
The two terms in $\dth$ in \eqref{eq:termdt} can be computed just as terms $(I)$
and $(II)$ in the proof of Theorem \ref{thm:decrexpd} (with $f$ there replaced by
$f^{n}$ here ) and we obtain
\begin{equation}\label{eq:ineqbiii}
\begin{split}
\lefteqn{2\seq{S\Dx f^{n+1},\Dv f^{n+1}}_\sharp =}\\
 & \seq{S\Dx f^{n},\Dv f^{n+1}}_\sharp + \seq{S\Dx f^{n+1},\Dv f^{n}}_\sharp\\
& -\dth\left(
\norm{S\Dx f^{n}}_\sharp^2 - \dvh\seq{S^b \Dx f^{n},\Dx \Dv f^{n}}_\sharp
\right)\\
& + 2 \dth
 \seq{(-\Dvs+v)\Dv f^{n},S^\sharp \Dx \Dv f^{n}}  \\
& -\dth\left( \seq{S\Dx f^{n},\Dv f^{n}}_\sharp+\seq{\sigma \Dx f^{n},\Dv f^{n}}_\sharp
\right)  + \dth^2 \rrrd_3(f^n),
\end{split}
\end{equation}
where we used adapted versions of Lemmas \ref{lem:triple} and \ref{lem:sb}.

Since $\Dx$ commutes with itself and with $(-\Dvs+\vs)\Dv$, the sequence
$(\Dx f^n)_{n\in\N}$ also solves the recursion relation
\eqref{eq:eulerexpldiscrinhomo}.  Adapting our the computation that led to
\eqref{eq:ineqbi} above, we infer that the last term in $\hhhd(f^{n+1})$ satisfies
\begin{eqnarray}
\lefteqn{\norm{\Dx f^{n+1}}^2}\nonumber\\
& \leq & \seq{\Dx f^n,\Dx f^{n+1}} - \dth\norm{\Dv \Dx f^n}^2_\sharp
+\dth^2\rrrd_1(\Dx f^n).
\label{eq:ineqbiiii}
 \end{eqnarray}

Summing up the four identities \eqref{eq:ineqbi}, \eqref{eq:ineqbii},
\eqref{eq:ineqbiii} and \eqref{eq:ineqbiiii}, multiplied
respectively by $C$, $D$, $E/2$ and $1$, we infer that
\begin{equation}
\label{eq:bilanentropy}
\begin{split}
  &  \hhhd(f^{n+1}) = \phid(f^n,f^{n+1})\\
  & -\dth\left[ C \norm{\Dv f^{n}}_\sharp^2 + D \seq{S\Dx f^{n},\Dv
      f^{n}}_\sharp + D \norm{(-\Dvs+v)\Dv f^{n}}^2
    +\frac{E}{2} \norm{S\Dx f^{n}}_\sharp^2 \right.\\
  & \qquad \left.  - \frac{E}{2} \dvh\seq{S^b \Dx f^{n},\Dx \Dv f^{n}}_\sharp
    - E \seq{(-\Dvs+v)\Dv f^{n},S^\sharp \Dx \Dv f^{n}}\right.\\
  & \qquad \left.  +\frac{E}{2} \seq{S\Dx f^{n},\Dv f^{n}}_\sharp
    +\frac{E}{2}\seq{\sigma \Dx f^{n},\Dv f^{n}}_\sharp +\norm{\Dx \Dv
      f^{n}}_\sharp^2
  \right] \\
  & \qquad + \dth^2 \left(C\rrrd_1(f^n)+ D\rrrd_2(f^n) + \frac{E}{2}\rrrd_3(f^n)
    +\rrrd_1(\Dx f^n)\right).
\end{split}
\end{equation}

We recognize here in square brackets in \eqref{eq:bilanentropy} the same term as
the one defining $\dddd(f)$ in \eqref{eq:defdcde} with $f^{n}$ here instead of $f$
there, and in our bounded velocity context.  It remains to show how to handle
the terms in $\dth^2$ in \eqref{eq:bilanentropy} using the CFL condition
\eqref{eq:condcflinhom}.  To do so, we set for all $g\in\Ldeuxmudvdxh$
\begin{equation*}
  M(g) = \norm{g}^2_\Hunmudvdxh + \norm{(-\Dvs+\vs)\Dv g}^2 + \norm{\Dv\Dx g}_\sharp^2.
\end{equation*}
Note that, in view of relation \eqref{eq:pourplustard} adapted to our bounded
velocity setting and of the Poincar\'e inequality of Lemma
\ref{lem:fullPoincarediscretb}, we have for all $g$ with zero mean
\begin{equation}
  \label{eq:propMD}
  M(g) \leq 2 \dddd(g).
\end{equation}
For the rest of the proof, we use the constants $a$, $b$ and $c$ defined
in \eqref{eq:defabc}
in Lemma \eqref{lem:CFLinhomo}.
For the term in \eqref{eq:ineqbi}, we have
\begin{equation}
  \label{eq:estimR1}
  \begin{split}
  |\rrrd_1(f^n)| & \leq  2 (\norm{\Xzerod f^n}^2+\norm{(-\Dvs+\vs)\Dv f^n}^2)\\
  & \leq  2 \left(c^2\norm{f^n}^2+\norm{(-\Dvs+\vs)\Dv f^n}^2\right)\\
  & \leq  2\bcfl \left(\norm{f^n}^2+\norm{(-\Dvs+\vs)\Dv f^n}^2\right) \\
 & \leq  2\bcfl M(f^n),
 \end{split}
\end{equation} 
since $\bcfl$ is greater than 1. For the term in $\dth^2$ in \eqref{eq:ineqbii},
we have first
\begin{eqnarray*}
  \lefteqn{\left| \seq{S\Dx f^n,\Dv \Pd f^{n}}_\sharp \right|}\\
& \leq & \frac12\left(
\norm{S\Dx f^n}_\sharp^2 + \norm{\Dv \Pd f^n}_\sharp^2
\right)\\
& \leq & \frac{b^2}{2}\norm{f^n}^2 + a^2 \left(\norm{v\Dx f^n}^2
+\norm{(-\Dvs+\vs)\Dv f^n}^2 \right)\\
& \leq & (a^2+b^2)(1+\vmax^2)
\left(\norm{f^n}^2 + \norm{\Dx f^n}^2 +\norm{(-\Dvs+\vs)\Dv f^n}^2 \right)\\
& \leq & 2\bcfl (1+\vmax^2) M(f^n).
\end{eqnarray*}
Second, we have
\begin{eqnarray*}
  \lefteqn{\left|\seq{v\Dv\Dx f^n,\Dv \Pd f^{n}}_\sharp\right|}\\
& \leq & \frac12\left(
\vmax^2 \norm{\Dv\Dx f^n}_\sharp^2 + a^2 \norm{(\Xzerod + (-\Dvs+\vs)\Dv) f^n}^2
\right)\\
& \leq & \frac{\vmax^2}{2} \norm{\Dv\Dx f^n}_\sharp^2
+ a^2 \left(\norm{v\Dx f^n}^2 + \norm{(-\Dvs+\vs)\Dv f^n}^2\right)\\
& \leq & (1+a^2)(1+\vmax^2) \left(
\norm{\Dv\Dx f^n}_\sharp^2 + \norm{\Dx f^n}^2 + \norm{(-\Dvs+\vs)\Dv f^n}^2
\right)\\
& \leq & 2\bcfl (1+\vmax^2) M(f^n).
\end{eqnarray*}
Third, we have
\begin{eqnarray*}
\lefteqn{\left|\seq{(-\Dvs+\vs)\Dv f^{n},(-\Dvs+\vs)\Dv \Pd f^{n}}\right|}\\
& \leq &
\norm{(-\Dvs+\vs)\Dv f^{n}} \left(\norm{(-\Dvs+\vs)\Dv \Xzerod f^{n}}
+ \norm{(-\Dvs+\vs)\Dv (-\Dvs+\vs)\Dv f^{n}}\right)\\
& \leq &
a^2
\norm{(-\Dvs+\vs)\Dv f^{n}} \left(
\norm{\Xzerod f^{n}} + \norm{(-\Dvs+\vs)\Dv f^{n}}
\right)\\
& \leq &
a^2
\left(
\vmax \norm{(-\Dvs+\vs)\Dv f^{n}}
\norm{\Dx f^{n}}
+ \norm{(-\Dvs+\vs)\Dv f^{n}}^2
\right)\\
& \leq &
2 (1+a^2)(1+\vmax^2) M(f^n)\\
 & \leq &
4 \bcfl (1+\vmax^2) M(f^n).
\end{eqnarray*}
In the end, we get
\begin{equation}
  \label{eq:estimR2}
  \left|\rrrd_2(f^n)\right| \leq 8 \bcfl (1+\vmax^2) M(f^n).
\end{equation}

\noindent
Let us get now to the third remainder term $\rrrd_3(f^n)$.
One has first
\begin{eqnarray*}
\lefteqn{\left|\seq{S\Dx \Xzerod f^n,\Dv (\Xzerod+ (-\Dvs+v)\Dv))  f^n}_\sharp\right|} \\
& = & \left| \seq{\Xzerod S\Dx f^n+ \dvh \Dx S^\flat \Dx f^n, \Xzerod \Dv f^n 
+ S \Dx  f^n + \Dv (-\Dvs+v)\Dv)  f^n}_\sharp\right|,
\end{eqnarray*}
where we used that $\adf{\Dv, \Xzerod} = S \Dx$ for the second term in the
scalar product and $S \Dx \Xzerod = \Xzerod S\Dx + \dvh \Dx S^\flat \Dx$ for the
first one. Noting that the operator norm of $S$ is equal to the one of $S^\flat$
we therefore get that
\begin{eqnarray*}
\lefteqn{\left|\seq{S\Dx \Xzerod f^n,\Dv (\Xzerod+ (-\Dvs+\vs)\Dv))  f^n}_\sharp\right|} \\
& \leq & \sep{ \norms{\Xzerod S\Dx f^n}+ \dvh \norms{\Dx S^\flat \Dx f^n}} 
\sep{\norms{\Xzerod \Dv f^n}  + \norms{S \Dx  f^n} + \norms{\Dv (-\Dvs+\vs)\Dv  f^n} } \\
& \leq & \sep{ c \norms{ S\Dx f^n}+ b \norms{ S \Dx f^n}} \sep{c \norms{ \Dv
         f^n}  
+ \norms{S \Dx  f^n} + a \norm{(-\Dvs+\vs)\Dv  f^n} } \\
& \leq &(c+b)(c+1+a) \sep{\norms{ \Dv f^n}^2  + \norms{S \Dx  f^n}^2 
+  \norm{(-\Dvs+\vs)\Dv  f^n}^2 }\\
& \leq & 12\bcfl M(f^n).
\end{eqnarray*}
Similarly, we get
\begin{eqnarray*}
\lefteqn{\left|\seq{S\Dx (\Xzerod+ (-\Dvs+v)\Dv)) f^n,\Dv \Xzerod f^n}_\sharp\right|} \\
& = & \left|\seq{\Xzerod S\Dx f^n+ \dvh \Dx S^\flat \Dx f^n + S\Dx (-\Dvs+v)\Dv
      f^n, 
\Xzerod \Dv f^n + S \Dx f^n} \right|\\
& \leq & \sep{ \norms{\Xzerod S\Dx f^n}+ \dvh \norms{S^\flat \Dx \Dx f^n} + 
\norms{S\Dx (-\Dvs+v)\Dv f^n}}\sep{ \norms{\Xzerod \Dv f^n} + \norms{S \Dx f^n}} \\
& \leq & \sep{ c\norms{ S\Dx f^n}+ b \dvh \norms{ S  \Dx f^n} 
+ b \norms{ (-\Dvs+v)\Dv f^n}}\sep{ c \norms{ \Dv f^n} + \norms{S \Dx f^n}} \\
& \leq & (c + 2b)(c+1) \sep{\norms{ \Dv f^n}^2  + \norms{S \Dx  f^n}^2 
+  \norm{(-\Dvs+v)\Dv  f^n}^2 }\\
& \leq & 12\bcfl M(f^n) .
\end{eqnarray*}
The same type of estimates also yields
\begin{eqnarray*}
\lefteqn{\left|\seq{S\Dx(-\Dvs+v)\Dv f^n,\Dv (\Xzerod+ (-\Dvs+v)\Dv)) 
f^n}_\sharp\right|} \\
& \leq & b ( 2b +a) \sep{\norms{ \Dv f^n}^2  + \norms{S \Dx  f^n}^2 
+  \norm{(-\Dvs+v)\Dv  f^n}^2 }\\
& \leq & 6 \bcfl M(f^n),
\end{eqnarray*}
and
\begin{eqnarray*}
\lefteqn{\left|\seq{S\Dx (\Xzerod+ (-\Dvs+v)\Dv)) f^n,\Dv(-\Dvs+v)\Dv f^n}\right|} \\
& \leq & 3 b a \sep{\norms{ \Dv f^n}^2  + \norms{S \Dx  f^n}^2 
+  \norm{(-\Dvs+v)\Dv  f^n}^2 }\\
& \leq & 6 \bcfl M(f^n).
\end{eqnarray*}
Adding the last four inequalities yields by triangle inequality
\begin{equation}
\label{eq:estimR3}
\rrrd_3(f^n) \leq 36 \bcfl M(f^n).
\end{equation}

\noindent
For the last remainder term, one may write
\begin{eqnarray*}
  \label{eq:estimR4}
  |R_1(\Dx f^n)| & \leq & 2 (\norm{\Xzerod \Dx f^n}^2+\norm{(-\Dvs+\vs)\Dv \Dx f^n}^2)\\
  & \leq & 2 \left(c^2\norm{\Dx f^n}^2+a^2\norm{\Dx f^n}^2\right)\\
  & \leq & 2\bcfl \left(\norm{\Dx f^n}^2+\norm{\Dx f^n}^2\right) \\
 & \leq & 4\bcfl M(f^n),
\end{eqnarray*}

From \eqref{eq:estimR1}, \eqref{eq:estimR2}, \eqref{eq:estimR3} and
\eqref{eq:estimR4}, we infer that the term in $\dth^2$ in
\eqref{eq:bilanentropy} can be bounded as follows:
\begin{eqnarray*}
\lefteqn{
\left|C\rrrd_1(f^n)+ D\rrrd_2(f^n) + \frac{E}{2}\rrrd_3(f^n) +\rrrd_1(\Dx f^n)
\right|
}\\
&\leq & \bcfl (1+\vmax^2)\left(2C + 8D+ 18E + 4\right) M(f^n).
\end{eqnarray*}
In view of \eqref{eq:propMD}, since $f^n$ has zero mean, we infer that
\begin{eqnarray*}
\lefteqn{
\left|C\rrrd_1(f^n)+ D\rrrd_2(f^n) + \frac{E}{2}\rrrd_3(f^n) +\rrrd_1(\Dx f^n)
\right|
}\\
&\leq & 4\bcfl (1+\vmax^2)\left(C + 4D+ 9E + 2\right) \ddd(f^n).
\end{eqnarray*}

\noindent
Using the inequality above, we rewrite \eqref{eq:bilanentropy} in the form
\begin{equation*}
  \hhhd(f^{n+1}) \leq \phid(f^n,f^{n+1}) -\dth \left(1-
  \dth 4 (C+ 4D + 9E +2) \bcfl (1+\vmax^2)\right)\dddd(f^{n}).
\end{equation*}

\noindent
Using the CFL condition \eqref{eq:condcflinhom} and the definition of $\kappa$
in the statement of Theorem \ref{thm:decrexpeulerexpldiscr}, we obtain from and
the last inequality that
\begin{equation*}
  \hhhd(f^{n+1}) \leq \phid(f^n,f^{n+1}) - 4C \kappa \dth \dddd(f^{n}).
\end{equation*}
Using a version of Lemma \ref{lem:dissdiss} adapted to our finite velocity context,
we get that for $C$, $D$, $E$ and $\dvh_0\in(0,1)$ chosen as in
\eqref{eq:choixE}-\eqref{eq:choixC}, we have $4C \dddd(f^{n}) \geq \hhh(f^{n})$ so that
\begin{equation*}
  \hhhd(f^{n+1}) \leq \phid(f^n,f^{n+1}) - \kappa \dth \hhhd(f^{n}).
\end{equation*}

\noindent
Using the fact Cauchy--Schwarz--Young inequality for $\phid$,
we infer that for all $n\in\N$,
\begin{equation*}
\hhhd(f^{n+1}) \leq \frac{1}{2} \hhhd(f^{n+1}) + \frac{1}{2} \hhhd (f^n)
-\dth \kappa \hhhd(f^{n}),
\end{equation*}
which yields for all $n\in\N$,
\begin{equation*}
\hhhd(f^{n+1}) \leq (1-2\kappa \dth) \hhhd (f^n),
\end{equation*}
which implies by induction that for all $n\in\N$,
\begin{equation*}
\hhhd(f^{n}) \leq (1- 2\dth \kappa)^{n} \hhhd (f^0).
\end{equation*}
This concludes the proof of Theorem \ref{thm:decrexpeulerexpldiscr}.
\end{proof}

\bigskip As noted for the homogeneous equation in bounded velocity domain at the
beginning of Section \ref{sec:homobounded}, the functional spaces
$\Ldeuxmudvdxh$, $\Ldeuxmudvdxs$ and $\Hunmudvdxh$ associated to the
discretization in space and velocity of the inhomogeneous equation are finite
dimensional in this bounded velocity setting.  Hence, linear operators are
continuous.  The next Lemma provides us with estimates on the norms of the
linear differential operators at hand, that will be helpful to establish the
result (Theorem \ref{thm:decrexpeulerexpldiscr}) on the long time behaviour of
the solutions of the explicit Euler scheme \eqref{eq:eulerexpldiscrinhomo} under
CFL condition.

\subsection{Numerical results}

\def\casinhom{L2L2}
\begin{figure}[!h]
\begin{subfigure}[t]{.4\linewidth} 
 \centering
 \setlength\fwidth{\linewidth}
   \includegraphics[keepaspectratio,height=6cm]{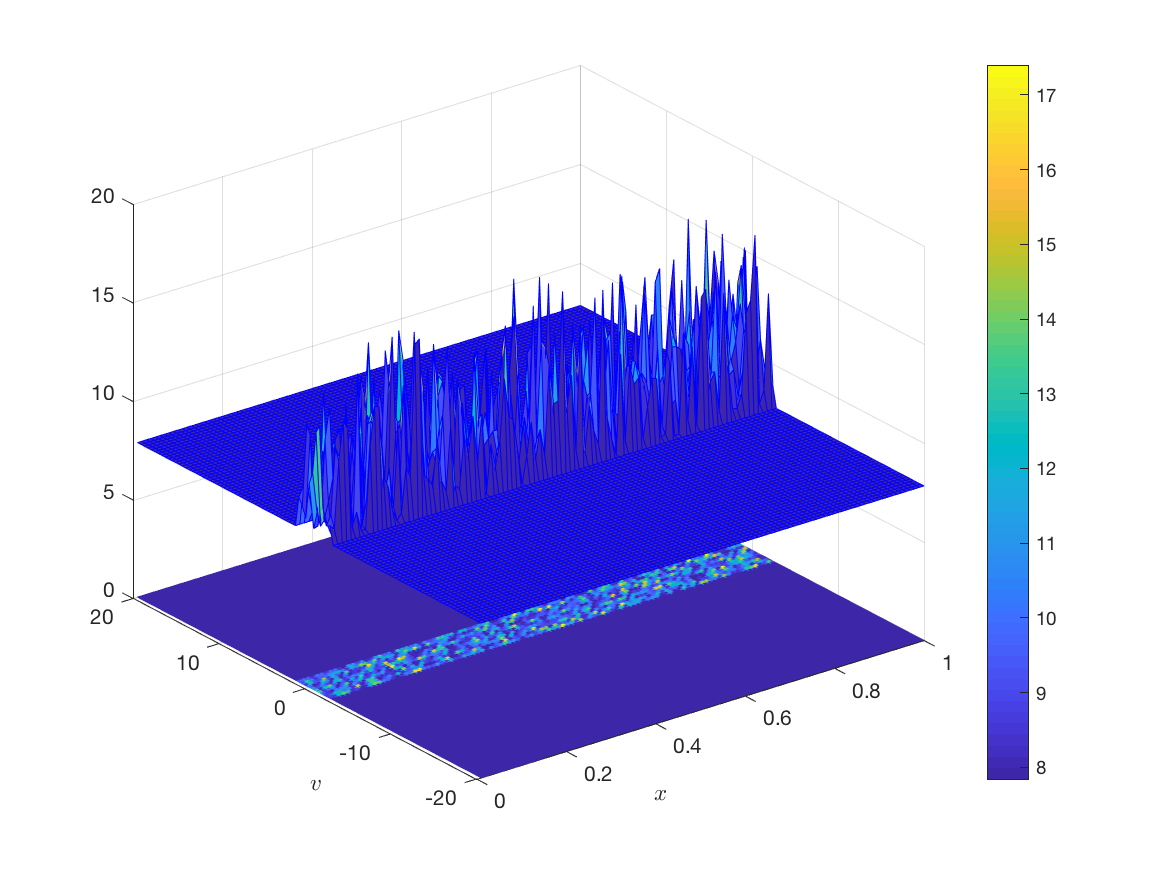}
   \input{evolfinhom_t-0_cas-\casinhom.tex}
 \end{subfigure}\hfill
  \begin{subfigure}[t]{.4\linewidth} 
  \centering  
    \setlength\fwidth{4.5cm} \setlength\fheight{4.5cm}
    \input{evolentfishinhom_cas-\casinhom.tex}
  \end{subfigure}
  \begin{subfigure}[t]{\linewidth}
    \centering \setlength\fwidth{14cm} \setlength\fheight{6cm}
    \input{entropyfisherinhomtaux_cas-\casinhom.tex}
    \caption{Evolution of the Linearized Entropy $\fffd$, {\em ie} the square of
      the $\Ldeuxmudvh$-norm of $f$ (left) and of the Fisher
      information $\gggd$ defined in \eqref{eq:mfi} (right). In each plot, the
      left-hand scale (plain line) is the linear scale and the right-hand scale
      (dashed line) is the "-log/t" scale that shows the numerical rate of
      convergence in long time. }
  \end{subfigure} 
\caption{Numerical simulations of Scheme \eqref{eq:eulerexpldiscrinhomo} with a random function
  as initial datum}
  \label{fig:inhom1}
\end{figure} 
   
\def\casinhom{rings}
\begin{figure}[!h]
\begin{subfigure}[t]{.4\linewidth} 
 \centering
 \setlength\fwidth{\linewidth}
   \includegraphics[keepaspectratio,height=6cm]{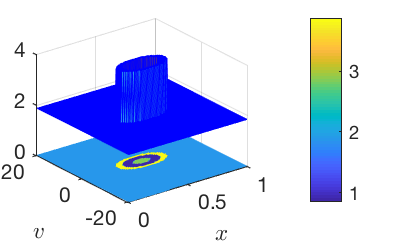}
   \input{evolfinhom_t-0_cas-\casinhom.tex}
 \end{subfigure}\hfill
  \begin{subfigure}[t]{.4\linewidth} 
  \centering  
    \setlength\fwidth{4.5cm} \setlength\fheight{4.5cm}
    \input{evolentfishinhom_cas-\casinhom.tex}
 
  \end{subfigure}
  \begin{subfigure}[t]{\linewidth}
    \centering \setlength\fwidth{14cm} \setlength\fheight{6cm}
    \input{entropyfisherinhomtaux_cas-\casinhom.tex}

    \caption{Evolution of the Linearized Entropy $\fffd$, {\em ie} the square of
      the $\Ldeuxmudvh$-norm of $f$ (left) and of the modified Fisher
      information $\gggd$ defined in \eqref{eq:mfi} (right). In each plot, the
      left-hand scale (plain line) is the linear scale and the right-hand scale
      (dashed line) is the "-log/t" scale that shows the numerical rate of
      convergence in long time. }
  \end{subfigure}
\caption{Numerical simulations of Scheme \eqref{eq:eulerexpldiscrinhomo} with a
  $(x,v)$-radial function
  as initial datum}
\label{fig:inhom2} 
\end{figure} 
We now turn to the implementation of the forward Euler discretization of the
inhomogeneous equation \eqref{eq:eulerexpldiscrinhomo} on a bounded domain in
$v$ and a periodic domain in $x$.

In reference to the homogeneous case, we define the Fisher information as
\begin{equation*}
  \gggd(g)\defegal \norm{g}^2 + \norm{{\Dv} g}^2
+ \norm{{\Dx} g}^2
\end{equation*}
that we know thanks to \eqref{eq:equivdb} to be equivalent to $\hhhd$
and we recall that
\begin{equation*}
  \fffd(g)= \norm{g}^2.
\end{equation*}
According to Theorem \ref{thm:eulerexplicite}, they are expected
to decrease geometrically fast. The tests that are presented here aim at
illustrating this fact in two cases:  
\begin{itemize}
\item the initial datum is a random function in $(x,v)$, with a Gaussian
  envelope in $v$ (see Figure \ref{fig:inhom1}-(A)). The logarithms of the
  entropy $\fffd$ and of the Fisher information $\gggd$ decrease linearly fast
  (see Figure \ref{fig:inhom1}-(B)), with a rate that goes to $2$, as can be
  seen in Figures \ref{fig:inhom1}-(C). The exponential decrease is consistent
  with Theorem \ref{thm:eulerexplicite}, and the rates are consistent with
  Theorem \ref{thm:decrexpb} and Corollary \ref{cor:decrexpb}.  
\item the initial datum is a radial function in $(x,v)$ (see Figure \ref{fig:inhom2}-(A)). The
  logarithms of the entropy $\fffd$ and of the Fisher information $\gggd$
  decrease linearly fast (see Figure \ref{fig:inhom2}-(B)), with a rate that is
  larger than $3$, as can be seen in Figures \ref{fig:inhom2}-(C). The Fisher
  information also seems to decrease in a faster way than the entropy in short time.
\end{itemize}
Again, comparing the two previous test cases, we get a hint that there is a very fast
regularizing effect in short time, as noted in \cite{PZ16}. The second
initial datum is a kind of 1d test case because of its radial nature. A
perspective of our work would be to 
investigate the change of slope at $t=1$ in Figure \ref{fig:inhom2}-(B). Also,
the rate seen on the right-hand side of Figures \ref{fig:inhom2}-(C) is concave,
whereas its behavior as shown to be convex in all three other tests.  We believe
it is also
something worth investigating.
  
\section{Generalizations and Remarks}
\label{sec:ccl}
In Sections \ref{sec:homogeneous} to \ref{sec:eqinhomoboundedvelocity} we
proposed several schemes conserving the basic properties of kinetic
equations. Many direct generalizations are possible, and we list below some of
them among other considerations concerning the proofs and results.

\begin{enumerate}

\item This is clear that the preceding results have their $d$-dimensional
  counterparts, quasi-straightforwardly in the unbounded case or even for
  bounded velocity (tensorized) domains. We did not give the corresponding
  statements in order not to hide the main features of our analysis.

\item Concerning the space variable, direct generalization are also possible,
  since a careful study of the proofs shows that in fact we just need the
  following assumptions concerning the $\Dx$ derivative:
\begin{enumerate} 
\item $\Dx$ is (formally) skew-adjoint,
\item $\norm{\Dx \phi} \geq c_p \norm{ \phi - \seq{\phi}}$ (Poincar\'e inequality).
\end{enumerate}
Note that in particular the full discrete Poincar\'e inequalities presented in
Propositions \ref{lem:fullPoincarecontinu}, \ref{lem:fullPoincarediscret} or
\ref{lem:fullPoincarediscret} remain true.

\item We did not show in details the maximal accretivity of the associated
  operators in the inhomogeneous discrete case (Subsections \ref{sec:inhomsd}
  and \ref{subsec:eqinhomototalementdiscretisee}). We just mention that the
  proof of the continuous case given e.g. in \cite[Proposition 5.5]{HN04} can be
  easily adapted, without even the use of hypoellipticity results since we are
  in a discrete setting. A direct consequence of the maximal accretivity of
  operator $P$ with domain $D(P) \subset H$ in a is that this operator leads to
  a natural semi-group correctly defining the solution $F(t)$ of
  ${\D_t} F + P F = 0$ for initial data even in $H$. This procedure is employed
  many times in this article with $H = \Ldeuxmudv$, $H = \Ldeuxmudvdx$,
  $H=\Hunmudv$, $H=\Hunmudvdx$ etc... and their discrete counterparts (both in
  the unbounded or bounded velocity setting).

\item In this paper, we presented a $H^1$ approach (and not an $L^2$
  one, except in the homogeneous case). Indeed this allows to work only with
  local operators and their finite differences counterparts leading to low
  numerical cost. This could be interesting to see how to extend the result to
  the $L^2$ framework. Anyway, merging the results of \cite{PZ16} in
  short time (to be adapted to our schemes) and the results would give
  indeed the full convergence to the equilibrium in $L^2$ for inhomogeneous
  models.

\item We did not focus on the preservation of the non-negativity of the
  numerical solutions by the schemes we introduced.  However, this preservation
  is straightforward at least in the homogeneous case, for the explicit methods
  (convexity arguments) as well as for implicit methods (monotonicity
  arguments).

\item We did not also prove in details to what extend the Neumann problems of
  Sections \ref{sec:homobounded} and \ref{sec:eqinhomoboundedvelocity} are good
  approximations of the the unbounded ones presented in Sections
  \ref{sec:homogeneous} and \ref{sec:eqinhomo}. This kind of considerations is
  standard in semi-classical analysis and could be done using resolvent identity
  type procedures, as is done e.g. in the study of the tunnelling effect e.g. in
  \cite{DS99}.

\item As a by-product of our analysis, the discrete schemes proposed in the
  preceding sections are naturally asymptotically stable: this is a direct
  consequence of the trend to the equilibrium. They also clearly are consistent
  by construction and therefore convergent.

\end{enumerate}

\bigskip \bigskip

As natural but not straightforward generalizations, we mention the ones below
that are the subject of coming works.

\bigskip We showed in this paper several Poincar\'e inequalities, and perhaps
the first and more surprising one is the one given in Proposition
\ref{prop:poindiscrete}. One interesting direction is to study the corresponding
log-Sobolev inequality in this discrete context, and the consequences on the
exponential decay using standard entropy-entropy dissipation techniques (see
e.g. \cite{Vil09}).

\bigskip In this paper we focused on the Fokker-Planck operator, and the
definition of the velocity derivatives takes deeply into account what
corresponds to incoming and outgoing particles (corresponding to indices
positive or negative in \eqref{eq:defderivintro}).  A natural extension would be to
check how this can be extended to the Landau collision kernel case, which also
involves derivatives, in order to keep positivity and self-adjointness
properties.  In fact it could be also interesting to look at the current
two-direction method also for other collision kernels such as linearized
Boltzmann or BGK ones.

\appendix

\section*{Appendix  : Commutation identities}

\begin{landscape}

\begin{table}[h]
\hspace*{-2cm}
  \centering
\doublespacing
  {\small$\begin{array}[h]{|c||c|c|c|c|c|}\hline
j & -2 & -1 & 0 & 1 & 2\\\hline
    g & g_{-2} & g_{-1} & g_{0} & g_{1} & g_{2} \\
v\Dx  g=\Xzerod  g &v_{-2}\Dx  g_{-2} &v_{-1}\Dx   g_{-1} & 0 &
v_{1}\Dx  g_{1} & v_{2}\Dx  g_{2} \\
\Dv g & \frac{g_{-1}-g_{-2}}{\dvh} & \frac{g_{0}-g_{-1}}{\dvh} & * &
\frac{g_{1}-g_{0}}{\dvh} & \frac{g_{2}-g_{1}}{\dvh}\\
\Dv\Xzerod  g & \frac{v_{-1}\Dx g_{-1}-v_{-2}\Dx g_{-2}}{\dvh} & 
\frac{v_{-1}\Dx g_{-1}}{\dvh} & * & \frac{v_{1}\Dx g_{1}}{\dvh} &
   \frac{v_{2}\Dx g_{2}-v_{1}\Dx g_{1}}{\dvh}\\
\Xzerod \Dv g &  v_{-2}\frac{\Dx g_{-1}-\Dx g_{-2}}{\dvh} & 
v_{-1}\frac{\Dx g_{0}-\Dx g_{-1}}{\dvh} & * & v_1\frac{\Dx g_{1}-\Dx g_{0}}{\dvh}
  & v_2\frac{\Dx g_{2}-\Dx g_{1}}{\dvh}\\
(\Dv\Xzerod -\Xzerod \Dv)g & \Dx g_{-1} & \Dx g_0 & * & \Dx g_0 & \Dx  g_1\\
S g & g_{-1} & g_{0} & * & g_{0} & g_{1} \\
S\Dx g & \Dx g_{-1} & \Dx g_{0} & * & \Dx g_{0} & \Dx g_{1} \\
S\Dx \Xzerod  g &v_{-1}\partial^2_{xx} g_{-1} & 0 & * & 0 & v_{1}\partial^2_{xx} g_{1}  \\
\Xzerod  S\Dx  g & v_{-2}\partial^2_{xx} g_{-1} & v_{-1}\partial^2_{xx} g_{0} 
& * & v_{1}\partial^2_{xx} g_{0}& v_{2}\partial^2_{xx}  g_{1}  \\
(S\Dx \Xzerod -\Xzerod S\Dx )g & h \partial^2_{xx} g_{-1} & h\partial^2_{xx} g_{0} & * &
  h\partial^2_{xx} g_{0}  & h \partial^2_{xx} g_{1}\\
(-\Dvs+v)\Dv g &   -\frac{g_{-3}-2g_{-2}+g_{-1}}{h^2}+v_{-2}\frac{g_{-1}-g_{-2}}{\dvh}
 &-\frac{g_{-2}-2g_{-1}+g_{0}}{h^2}+v_{-1}\frac{g_{0}-g_{-1}}{\dvh} &
  -\frac{g_{-1}-2g_{0}+g_{1}}{h^2} & -\frac{g_{0}-2g_{1}+g_{2}}{h^2}
+v_{1}\frac{g_{1}-g_{0}}{\dvh} &-\frac{g_{3}-2g_{2}+g_{1}}{h^2}+v_{1}\frac{g_{2}-g_{1}}{\dvh}\\
S(-\Dvs+v)\Dv g &-\frac{g_{-2}-2g_{-1}+g_{0}}{h^2}+v_{-1}\frac{g_{0}-g_{-1}}{\dvh}  
&  -\frac{g_{-1}-2g_{0}+g_{1}}{h^2} & * &  -\frac{g_{-1}-2g_{0}+g_{1}}{h^2} 
&  -\frac{g_{0}-2g_{1}+g_{2}}{h^2}+v_{1}\frac{g_{1}-g_{0}}{\dvh} \\
(-\Dvs+v)S g & -\frac{g_{-1}-g_{-2}}{\dvh}+v_{-2}g_{-1} & -\frac{g_{0}-g_{-1}}{\dvh}+v_{-1}g_{0}
  & 0 &  -\frac{g_{1}-g_{0}}{\dvh}+v_{1}g_{0} &   -\frac{g_{2}-g_{1}}{\dvh}+v_{2}g_{1}\\
\Dv(-\Dvs+v)S g &  \frac{g_{-2}-2g_{-1}-g_{0}}{h^2}+\frac{v_{-2}g_{-1}+v_{-1}g_0}{\dvh}  &
 -\frac{g_{0}-g_{-1}}{h^2}+\frac{v_{-1}g_{0}}{\dvh}
 & * &  -\frac{g_{1}-g_{0}}{h^2}+\frac{v_{1}g_{0}}{\dvh} &
    \frac{g_{2}-2g_{1}+g_0}{h^2}+\frac{v_{2}g_{1}+v_{1}g_{0}}{\dvh} \\
(S(-\Dvs+v)\Dv-\Dv(-\Dvs+v)S) g & g_{-1} & g_0+\frac{g_1-g_0}{h^2} 
& * &     g_0-\frac{g_0-g_{-1}}{h^2} &g_1
\\\hline
  \end{array}$ }
  \caption{Summary of the relations}
  \label{tab:comm}
\end{table}
\end{landscape}

\end{document}

%% file: macros.tex
\newtheorem{thm}{Theorem}[section]  \newtheorem{cor}[thm]{Corollary}
\newtheorem{lem}[thm]{Lemma}   \newtheorem{defn}[thm]{Definition}
  \newtheorem{prop}[thm]{Proposition}

\def\ccc{{\mathcal C}}\def\ddd{{\mathcal D}}
\def\eee{{\mathcal E}}\def\fff{{\mathcal F}}\def\ggg{{\mathcal G}} \def\hhh{{\mathcal H}}
\def\iii{{\mathcal I}} \def\jjj{{\mathcal J}}
 
\def\rrr{{\mathcal R}}

\def\norm#1{\left\Vert#1\right\Vert}
\def\norms#1{\left\Vert#1\right\Vert_\sharp}
\def\abs#1{\left\vert#1\right\vert}
\def\set#1{\left\{#1\right\}}
\def\seq#1{\left<#1\right>}
\def\seqs#1{\seq{#1}_{\sharp}}
\def\sep#1{\left(#1\right)}\def\adf#1{\left[#1\right]}
\def\R{\mathbb R}
\def\N{\mathbb N}\def\Z{\mathbb Z}
\def\T{\mathbb T}
\def\D{\partial}
\def\eps{\varepsilon}
\def\phi{\varphi}

\def\defegal{\stackrel{\text{\rm def}}{=}}

\renewcommand{\iint}{ \int \!\!\!\!\! \int}

\def\exp{\textrm{e}}

\newcommand{\dv}{{{\rm d}v}}
\newcommand{\dx}{{{\rm d}x}}
\newcommand{\dw}{{{\rm d}w}}

\newcommand{\dxh}{{\delta \! x}}
\newcommand{\dvh}{{\delta \! v}}
\newcommand{\dth}{{\delta \! t}}
\newcommand{\mus}{\mu^\sharp}
\newcommand{\muh}{\mu^\dvh}
\newcommand{\mudv}{\mu \dv}
\newcommand{\mudvp}{\mu'\dv'}

\newcommand{\mudvdx}{\mu \dv\dx}
\newcommand{\Ldeuxdx}{{L^2(\dx)}}
\newcommand{\Ldeuxdxh}{{\ell^2(\dxh)}}
\newcommand{\Ldeuxmudv}{{L^2(\mu\dv)}}
\newcommand{\Ldeuxmudvb}{{L^2(I,\mu\dv)}}
\newcommand{\Ldeuxmudvdx}{{L^2(\T\times\R,\mu\dv\dx)}}
\newcommand{\Ldeuxmudvdxb}{{L^2(\T\times I,\mu\dv\dx)}}
\newcommand{\Ldeuxmudvh}{{\ell^2( \muh \dvh )}}
\newcommand{\Ldeuxmudvhb}{{\ell^2( \muh \dvh )}}
\newcommand{\Ldeuxmudvs}{{\ell^2(\mus\dvh)}}
\newcommand{\Ldeuxmudvsb}{{\ell^2(\iiis, \mus\dvh)}}
\newcommand{\Ldeuxmudvdxh}{{\ell^2(\muh\dvh\dxh)}}
\newcommand{\Ldeuxmudvdxs}{{\ell^2(\mus\dvh\dxh)}}
\newcommand{\Hunmudv}{{H^1(\mu\dv )}}
\newcommand{\Hunmudvb}{{H^1(I,\mu\dv )}}
\newcommand{\Hunmudvh}{{{\bf h}^1(\muh\dvh)}}
\newcommand{\Hunmudvdx}{{H^1(\T\times\R ,\mu\dv\dx)}}
\newcommand{\Hunmudvdxb}{{H^1(\T\times I,\mu\dv\dx)}}
\newcommand{\Hunmudvdxh}{{{\bf h}^1(\muh\dvh\dxh)}}

\newcommand{\cfl}{{\alpha_{\rm CFL}}}
\newcommand{\bcfl}{{\beta_{\rm CFL}}}
\newcommand{\ddt}{{\frac{{{\rm d}}}{{{\rm d}t}}}}
\newcommand{\Ss}{S^\sharp}
\newcommand{\Dv}{\mathsf{D}_v}
\newcommand{\Dx}{\mathsf{D}_x}
\newcommand{\Dvs}{\mathsf{D}^\sharp_v}
\newcommand{\Xzero}{X_0}
\newcommand{\Xzerod}{X_0^\delta}
\newcommand{\Id}{{\rm I}_d}
\newcommand{\Pd}{{P^\delta}}
\newcommand{\vs}{{v^\sharp}}
\newcommand{\vmax}{v_{\mathrm {max}}}
\newcommand{\imax}{i_{\mathrm {max}}}
\newcommand{\iiis}{{\iii^\sharp}}
\newcommand{\phid}{\varphi^\delta}

\newcommand{\dddd}{\ddd^\delta}
\newcommand{\eeed}{\eee^\delta}
\newcommand{\hhhd}{\hhh^\delta}
\newcommand{\fffd}{\fff^\delta}
\newcommand{\gggd}{\ggg^\delta}
\newcommand{\rrrd}{\rrr^\delta}

\newcommand{\G}{{\bf G}}
\newcommand{\GPLUS}{{G_{\! +}}}
\newcommand{\GMOINS}{{G_{\! -}}}
\newcommand{\gtilde}{\tilde{g}}

%% file: evolfhom_cas-rand.tex
\begin{tikzpicture}

\begin{axis}[%
width=0.954\fwidth,
height=\fheight,
at={(0\fwidth,0\fheight)},
scale only axis,
xmin=-20,
xmax=20,
xlabel style={font=\color{white!15!black}},
xlabel={Velocity $v$},
ymin=-1.1932904672367e-10,
ymax=2.15426642240201,
axis background/.style={fill=white},
legend style={legend cell align=left, align=left, draw=white!15!black},
legend style={font=\tiny}
]
\addplot [color=black, dashdotted, line width=1.0pt]
  table[row sep=crcr]{%
-20	1.1\\
-19.6	1.1\\
-19.2	1.1\\
-18.8	1.1\\
-18.4	1.1\\
-18	1.1\\
-17.6	1.10000000000001\\
-17.2	1.10000000000006\\
-16.8	1.10000000000044\\
-16.4	1.10000000000041\\
-16	1.10000000000394\\
-15.6	1.10000000001001\\
-15.2	1.10000000003701\\
-14.8	1.10000000009655\\
-14.4	1.10000000059264\\
-14	1.10000000290657\\
-13.6	1.10000000386335\\
-13.2	1.1000000089633\\
-12.8	1.10000005786716\\
-12.4	1.10000011591704\\
-12	1.1000001572036\\
-11.6	1.10000083890741\\
-11.2	1.10000315174014\\
-10.8	1.10000289717686\\
-10.4	1.10000555447761\\
-10	1.1000031254981\\
-9.6	1.10006772946851\\
-9.2	1.10002251871312\\
-8.8	1.10036418286973\\
-8.4	1.10049329439464\\
-8	1.10113505856335\\
-7.6	1.10156568235122\\
-7.2	1.10034424556787\\
-6.8	1.10832418492006\\
-6.4	1.10879341873658\\
-6	1.1077979953373\\
-5.6	1.13097951213827\\
-5.2	1.1057521091682\\
-4.8	1.12509129136423\\
-4.4	1.12271172769772\\
-4	1.28353234083319\\
-3.6	1.10007288347515\\
-3.2	1.35327629688072\\
-2.8	1.27434935926281\\
-2.4	1.34295166690542\\
-2	1.12729269692438\\
-1.6	1.63186933907104\\
-1.2	1.7973709648855\\
-0.800000000000001	1.57467930237753\\
-0.399999999999999	1.28116023965472\\
0	1.95131922743089\\
0.399999999999999	1.95842402036546\\
0.800000000000001	1.72664804335174\\
1.2	1.17868308324806\\
1.6	1.34847720442805\\
2	1.58474850852583\\
2.4	1.46089599158167\\
2.8	1.26106135235702\\
3.2	1.27933314754981\\
3.6	1.30119249828813\\
4	1.20504167822497\\
4.4	1.21672680493616\\
4.8	1.15998202361184\\
5.2	1.10128251142807\\
5.6	1.11083429052005\\
6	1.10240520230072\\
6.4	1.10422049956229\\
6.8	1.10745808649424\\
7.2	1.10246568560753\\
7.6	1.10169878349457\\
8	1.1013375343816\\
8.4	1.10046414713285\\
8.8	1.10042401095951\\
9.2	1.1000434853017\\
9.6	1.10009187096912\\
10	1.10000472811293\\
10.4	1.10000266031318\\
10.8	1.1000053582954\\
11.2	1.10000173554426\\
11.6	1.10000138123708\\
12	1.10000040999609\\
12.4	1.10000004857589\\
12.8	1.10000004978978\\
13.2	1.10000000437911\\
13.6	1.10000000026426\\
14	1.10000000266785\\
14.4	1.10000000015962\\
14.8	1.10000000024228\\
15.2	1.10000000000785\\
15.6	1.10000000000649\\
16	1.10000000000136\\
16.4	1.1000000000018\\
16.8	1.10000000000021\\
17.2	1.10000000000008\\
17.6	1.10000000000002\\
18	1.1\\
18.4	1.1\\
18.8	1.1\\
19.2	1.1\\
19.6	1.1\\
20	1.1\\
};
\addlegendentry{$f$ at $t=0$}

\addplot [color=black, dotted, line width=1.0pt]
  table[row sep=crcr]{%
-20	1.3459574398821\\
-19.6	1.3459574398821\\
-19.2	1.35162341712583\\
-18.8	1.35801107087726\\
-18.4	1.36455828422425\\
-18	1.37119738310219\\
-17.6	1.37791716368534\\
-17.2	1.38471269873634\\
-16.8	1.39157947889464\\
-16.4	1.39851271161596\\
-16	1.4055072316324\\
-15.6	1.41255748615668\\
-15.2	1.41965753029455\\
-14.8	1.42680102471457\\
-14.4	1.43398123450412\\
-14	1.44119102897101\\
-13.6	1.44842288223078\\
-13.2	1.45566887440799\\
-12.8	1.46292069325687\\
-12.4	1.47016963598618\\
-12	1.47740661105517\\
-11.6	1.48462213969531\\
-11.2	1.49180635690501\\
-10.8	1.49894901166443\\
-10.4	1.50603946612442\\
-10	1.51306669353885\\
-9.6	1.52001927473328\\
-9.2	1.526885392936\\
-8.8	1.5336528268388\\
-8.4	1.54030894180474\\
-8	1.54684067919689\\
-7.6	1.55323454386336\\
-7.2	1.55947658987644\\
-6.8	1.56555240468223\\
-6.4	1.57144709186383\\
-6	1.57714525274452\\
-5.6	1.58263096704146\\
-5.2	1.58788777269959\\
-4.8	1.5928986448529\\
-4.4	1.59764597351857\\
-4	1.60211153903653\\
-3.6	1.60627648327395\\
-3.2	1.61012127296936\\
-2.8	1.61362564885356\\
-2.4	1.61676854955403\\
-2	1.61952799124434\\
-1.6	1.62188086950805\\
-1.2	1.62380262255551\\
-0.800000000000001	1.62526664041579\\
-0.399999999999999	1.62624318834003\\
0	1.62669734285506\\
0.399999999999999	1.62658474700292\\
0.800000000000001	1.62594280552999\\
1.2	1.62480072970148\\
1.6	1.62318377989175\\
2	1.62111533559756\\
2.4	1.61861796116645\\
2.8	1.61571396102398\\
3.2	1.61242565538566\\
3.6	1.60877549409665\\
4	1.60478607307438\\
4.4	1.60048009104093\\
4.8	1.59588026985952\\
5.2	1.5910092536261\\
5.6	1.58588949677387\\
6	1.58054314836201\\
6.4	1.57499193767285\\
6.8	1.56925706481901\\
7.2	1.56335909902948\\
7.6	1.55731788650453\\
8	1.55115246912404\\
8.4	1.54488101481117\\
8.8	1.5385207599632\\
9.2	1.53208796404219\\
9.6	1.52559787615789\\
10	1.51906471326284\\
10.4	1.51250164940902\\
10.8	1.50592081538087\\
11.2	1.4993333079168\\
11.6	1.4927492076568\\
12	1.48617760490364\\
12.4	1.47962663225768\\
12.8	1.47310350317607\\
13.2	1.46661455551532\\
13.6	1.46016529913784\\
14	1.453760466697\\
14.4	1.44740406675917\\
14.8	1.4410994384723\\
15.2	1.43484930704846\\
15.6	1.42865583938984\\
16	1.42252069925404\\
16.4	1.41644510143416\\
16.8	1.41042986459648\\
17.2	1.4044754632744\\
17.6	1.39858208587421\\
18	1.392749756725\\
18.4	1.38697900333337\\
18.8	1.38127611730435\\
19.2	1.3756976894561\\
19.6	1.3707368981438\\
20	1.3707368981438\\
};
\addlegendentry{$f$ at $t=2$}

\addplot [color=black, line width=1.0pt]
  table[row sep=crcr]{%
-20	1.62481441431722\\
-19.6	1.62481441431722\\
-19.2	1.62481441431918\\
-18.8	1.62481441432137\\
-18.4	1.62481441432359\\
-18	1.62481441432581\\
-17.6	1.62481441432803\\
-17.2	1.62481441433025\\
-16.8	1.62481441433247\\
-16.4	1.62481441433469\\
-16	1.62481441433691\\
-15.6	1.62481441433913\\
-15.2	1.62481441434135\\
-14.8	1.62481441434357\\
-14.4	1.62481441434579\\
-14	1.62481441434801\\
-13.6	1.62481441435023\\
-13.2	1.62481441435245\\
-12.8	1.62481441435467\\
-12.4	1.62481441435689\\
-12	1.62481441435911\\
-11.6	1.62481441436132\\
-11.2	1.62481441436354\\
-10.8	1.62481441436576\\
-10.4	1.62481441436798\\
-10	1.6248144143702\\
-9.6	1.62481441437242\\
-9.2	1.62481441437464\\
-8.8	1.62481441437686\\
-8.4	1.62481441437908\\
-8	1.6248144143813\\
-7.6	1.62481441438352\\
-7.2	1.62481441438574\\
-6.8	1.62481441438796\\
-6.4	1.62481441439018\\
-6	1.6248144143924\\
-5.6	1.62481441439462\\
-5.2	1.62481441439684\\
-4.8	1.62481441439906\\
-4.4	1.62481441440128\\
-4	1.6248144144035\\
-3.6	1.62481441440572\\
-3.2	1.62481441440794\\
-2.8	1.62481441441016\\
-2.4	1.62481441441238\\
-2	1.6248144144146\\
-1.6	1.62481441441682\\
-1.2	1.62481441441904\\
-0.800000000000001	1.62481441442126\\
-0.399999999999999	1.62481441442348\\
0	1.6248144144257\\
0.399999999999999	1.62481441442792\\
0.800000000000001	1.62481441443014\\
1.2	1.62481441443236\\
1.6	1.62481441443458\\
2	1.6248144144368\\
2.4	1.62481441443902\\
2.8	1.62481441444124\\
3.2	1.62481441444346\\
3.6	1.62481441444568\\
4	1.6248144144479\\
4.4	1.62481441445012\\
4.8	1.62481441445234\\
5.2	1.62481441445456\\
5.6	1.62481441445678\\
6	1.624814414459\\
6.4	1.62481441446122\\
6.8	1.62481441446344\\
7.2	1.62481441446566\\
7.6	1.62481441446788\\
8	1.6248144144701\\
8.4	1.62481441447231\\
8.8	1.62481441447453\\
9.2	1.62481441447675\\
9.6	1.62481441447897\\
10	1.62481441448119\\
10.4	1.62481441448341\\
10.8	1.62481441448563\\
11.2	1.62481441448785\\
11.6	1.62481441449007\\
12	1.62481441449229\\
12.4	1.62481441449451\\
12.8	1.62481441449673\\
13.2	1.62481441449895\\
13.6	1.62481441450117\\
14	1.62481441450339\\
14.4	1.62481441450561\\
14.8	1.62481441450783\\
15.2	1.62481441451005\\
15.6	1.62481441451227\\
16	1.62481441451449\\
16.4	1.62481441451671\\
16.8	1.62481441451893\\
17.2	1.62481441452115\\
17.6	1.62481441452337\\
18	1.62481441452559\\
18.4	1.62481441452781\\
18.8	1.62481441453002\\
19.2	1.62481441453221\\
19.6	1.62481441453418\\
20	1.62481441453418\\
};
\addlegendentry{$f$ at $t=20$}

\end{axis}
\end{tikzpicture}%
\caption{Evolution of $f$ in the homogeneous case at three times, the velocity range is $(-20,20)$, the discretization steps are  ${\dvh=0.4}$ and ${\dth=0.01}$.}

%% file: evolentfishhom_cas-rand.tex
\begin{tikzpicture}

\begin{axis}[%
width=0.954\fwidth,
height=\fheight,
at={(0\fwidth,0\fheight)},
scale only axis,
xmin=0,
xmax=20,
xlabel style={font=\color{white!15!black}},
xlabel={Time $t$},
ymin=-60,
ymax=0,
axis background/.style={fill=white},
legend style={legend cell align=left, align=left, draw=white!15!black},
legend style={font=\tiny}
]
\addplot [color=black, line width=1.0pt]
  table[row sep=crcr]{%
0	0\\
0.2	-2.28632170895047\\
0.4	-3.18365237441364\\
0.6	-3.97375472824633\\
0.8	-4.74014951507884\\
1	-5.49729370489057\\
1.2	-6.24994008224321\\
1.4	-6.99996595576233\\
1.6	-7.74804858300974\\
1.8	-8.49424716727599\\
2	-9.23822212836516\\
2.19999999999997	-9.97929926001781\\
2.39999999999994	-10.7164553654408\\
2.59999999999991	-11.448259893699\\
2.79999999999988	-12.1727948309077\\
2.99999999999985	-12.8875774998195\\
3.19999999999981	-13.5895227551459\\
3.39999999999978	-14.2749968374606\\
3.59999999999975	-14.9400230974343\\
3.79999999999972	-15.580680016959\\
3.99999999999969	-16.1936663519614\\
4.19999999999965	-16.7769065798656\\
4.39999999999961	-17.3299910406873\\
4.59999999999957	-17.8542682129456\\
4.79999999999953	-18.3525478650779\\
4.99999999999949	-18.8285443200868\\
5.19999999999945	-19.286271009317\\
5.39999999999941	-19.7295573057012\\
5.59999999999937	-20.1617575104145\\
5.79999999999933	-20.5856358628215\\
5.99999999999929	-21.0033722688935\\
6.19999999999925	-21.4166321063429\\
6.39999999999921	-21.8266589531169\\
6.59999999999917	-22.2343663564838\\
6.79999999999913	-22.6404176311205\\
6.99999999999909	-23.0452904166584\\
7.19999999999905	-23.4493266290027\\
7.39999999999901	-23.8527700056343\\
7.59999999999897	-24.2557937831832\\
7.79999999999893	-24.6585208322017\\
7.9999999999989	-25.0610381714996\\
8.19999999999886	-25.4634073647267\\
8.39999999999882	-25.8656719344839\\
8.59999999999878	-26.2678626331605\\
8.79999999999874	-26.6700011819152\\
8.9999999999987	-27.0721029189434\\
9.19999999999866	-27.474178673183\\
9.39999999999862	-27.8762360889959\\
9.59999999999858	-28.2782805621957\\
9.79999999999854	-28.6803159012004\\
9.9999999999985	-29.0823447939067\\
10.1999999999985	-29.4843691373109\\
10.3999999999984	-29.8863902701975\\
10.5999999999984	-30.2884091373837\\
10.7999999999983	-30.6904264056457\\
10.9999999999983	-31.092442545537\\
11.1999999999983	-31.4944578891319\\
11.3999999999982	-31.8964726707777\\
11.5999999999982	-32.2984870558544\\
11.7999999999981	-32.7005011610717\\
11.9999999999981	-33.1025150687916\\
12.1999999999981	-33.5045288371374\\
12.399999999998	-33.9065425071268\\
12.599999999998	-34.3085561077058\\
12.7999999999979	-34.710569659302\\
12.9999999999979	-35.1125831763309\\
13.1999999999979	-35.5145966689656\\
13.3999999999978	-35.9166101443854\\
13.5999999999978	-36.3186236076566\\
13.7999999999977	-36.7206370623545\\
13.9999999999977	-37.1226505110022\\
14.1999999999977	-37.5246639553803\\
14.3999999999976	-37.9266773967453\\
14.5999999999976	-38.328690835984\\
14.7999999999975	-38.7307042737222\\
14.9999999999975	-39.1327177104014\\
15.1999999999975	-39.5347311463334\\
15.3999999999974	-39.9367445817379\\
15.5999999999974	-40.3387580167703\\
15.7999999999973	-40.7407714515401\\
15.9999999999973	-41.1427848861245\\
16.1999999999973	-41.5447983205782\\
16.3999999999972	-41.9468117549395\\
16.5999999999972	-42.3488251892357\\
16.7999999999971	-42.7508386234859\\
16.9999999999971	-43.1528520577037\\
17.1999999999971	-43.5548654918986\\
17.399999999997	-43.9568789260773\\
17.599999999997	-44.3588923602446\\
17.7999999999969	-44.7609057944039\\
17.9999999999969	-45.1629192285575\\
18.1999999999969	-45.5649326627071\\
18.3999999999968	-45.9669460968539\\
18.5999999999968	-46.3689595309987\\
18.7999999999967	-46.7709729651421\\
18.9999999999967	-47.1729863992844\\
19.1999999999967	-47.5749998334261\\
19.3999999999966	-47.9770132675673\\
19.5999999999966	-48.3790267017081\\
19.7999999999965	-48.7810401358487\\
19.9999999999965	-49.1830535699891\\
};
\addlegendentry{$\log(\fffd/\fffd(0))$}

\addplot [color=black, dotted, line width=1.0pt]
  table[row sep=crcr]{%
0	0\\
0.2	-3.53203987847389\\
0.4	-4.61519124472465\\
0.6	-5.43802340111888\\
0.8	-6.21581662988031\\
1	-6.97815462396889\\
1.2	-7.73359054669822\\
1.4	-8.48545413183514\\
1.6	-9.23511553865241\\
1.8	-9.98303767204413\\
2	-10.7291830267687\\
2.19999999999997	-11.4731686151447\\
2.39999999999994	-12.2143029664042\\
2.59999999999991	-12.9515595143407\\
2.79999999999988	-13.6835125057627\\
2.99999999999985	-14.4082539752611\\
3.19999999999981	-15.1233145124689\\
3.39999999999978	-15.8256231086594\\
3.59999999999975	-16.5115575779875\\
3.79999999999972	-17.1771456131304\\
3.99999999999969	-17.8184578157911\\
4.19999999999965	-18.4321697010267\\
4.39999999999961	-19.0161684461365\\
4.59999999999957	-19.5699997694016\\
4.79999999999953	-20.0949701893274\\
4.99999999999949	-20.5938589601402\\
5.19999999999945	-21.0703653397886\\
5.39999999999941	-21.5285019064565\\
5.59999999999937	-21.972107029173\\
5.79999999999933	-22.4045489644265\\
5.99999999999929	-22.8286070588895\\
6.19999999999925	-23.2464751881783\\
6.39999999999921	-23.6598305347859\\
6.59999999999917	-24.0699260988066\\
6.79999999999913	-24.4776826681866\\
6.99999999999909	-24.8837689800987\\
7.19999999999905	-25.2886666633416\\
7.39999999999901	-25.6927205323798\\
7.59999999999897	-26.0961764126374\\
7.79999999999893	-26.4992090356572\\
7.9999999999989	-26.9019423377481\\
8.19999999999886	-27.3044640952431\\
8.39999999999882	-27.706836409088\\
8.59999999999878	-28.1091031824101\\
8.79999999999874	-28.5112954368131\\
8.9999999999987	-28.9134350837785\\
9.19999999999866	-29.3155375959808\\
9.39999999999862	-29.7176138973441\\
9.59999999999858	-30.1196716993037\\
9.79999999999854	-30.521716445028\\
9.9999999999985	-30.9237519763638\\
10.1999999999985	-31.3257810048033\\
10.3999999999984	-31.727805443997\\
10.5999999999984	-32.1298266444836\\
10.7999999999983	-32.5318455593758\\
10.9999999999983	-32.9338628613043\\
11.1999999999983	-33.3358790249542\\
11.3999999999982	-33.7378943853156\\
11.5999999999982	-34.1399091787936\\
11.7999999999981	-34.5419235722204\\
11.9999999999981	-34.9439376833302\\
12.1999999999981	-35.3459515952086\\
12.399999999998	-35.747965366489\\
12.599999999998	-36.1499790385493\\
12.7999999999979	-36.5519926405898\\
12.9999999999979	-36.9540061932174\\
13.1999999999979	-37.3560197109741\\
13.3999999999978	-37.7580332041225\\
13.5999999999978	-38.1600466799047\\
13.7999999999977	-38.5620601434317\\
13.9999999999977	-38.9640735983101\\
14.1999999999977	-39.3660870470852\\
14.3999999999976	-39.7681004915532\\
14.5999999999976	-40.1701139329817\\
14.7999999999975	-40.5721273722652\\
14.9999999999975	-40.9741408100349\\
15.1999999999975	-41.3761542467364\\
15.3999999999974	-41.7781676826841\\
15.5999999999974	-42.1801811180998\\
15.7999999999973	-42.58219455314\\
15.9999999999973	-42.9842079879153\\
16.1999999999973	-43.3862214225036\\
16.3999999999972	-43.78823485696\\
16.5999999999972	-44.1902482913233\\
16.7999999999971	-44.5922617256209\\
16.9999999999971	-44.9942751598721\\
17.1999999999971	-45.3962885940905\\
17.399999999997	-45.7983020282859\\
17.599999999997	-46.200315462465\\
17.7999999999969	-46.6023288966325\\
17.9999999999969	-47.004342330792\\
18.1999999999969	-47.4063557649457\\
18.3999999999968	-47.8083691990954\\
18.5999999999968	-48.2103826332422\\
18.7999999999967	-48.6123960673871\\
18.9999999999967	-49.0144095015305\\
19.1999999999967	-49.4164229356729\\
19.3999999999966	-49.8184363698146\\
19.5999999999966	-50.2204498039558\\
19.7999999999965	-50.6224632380966\\
19.9999999999965	-51.0244766722372\\
};
\addlegendentry{$\log(\gggd/\gggd(0))$}

\end{axis}
\end{tikzpicture}%
\caption{Normalized linearized entropy $\fffd$ and Fisher information $\gggd$ in logscale}

%% file: entropyfisherhomtaux_cas-rand.tex
\definecolor{mycolor1}{rgb}{0.00000,0.44700,0.74100}%
\definecolor{mycolor2}{rgb}{0.85000,0.32500,0.09800}%
\begin{tikzpicture}

\begin{axis}[%
width=0.411\fwidth,
height=\fheight,
at={(0\fwidth,0\fheight)},
scale only axis,
xmin=0,
xmax=20,
xlabel style={font=\color{white!15!black}},
xlabel={Time $t$},
separate axis lines,
every outer y axis line/.append style={mycolor1},
every y tick label/.append style={font=\color{mycolor1}},
every y tick/.append style={mycolor1},
ymin=0,
ymax=0.1,
axis background/.style={fill=white},
legend style={legend cell align=left, align=left, draw=white!15!black},
legend style={font=\tiny}
]
\addplot [color=mycolor1, line width=1.0pt]
  table[row sep=crcr]{%
0	0.0822450158266734\\
0.2	0.00835935341169584\\
0.4	0.00340774375044749\\
0.6	0.0015464284738146\\
0.8	0.000718602643864506\\
1	0.000337027444273171\\
1.2	0.000158779744427811\\
1.4	7.50002999736644e-05\\
1.6	3.54956262948665e-05\\
1.8	1.68308060662476e-05\\
2	7.99835536023173e-06\\
2.19999999999997	3.81201838581486e-06\\
2.39999999999994	1.82394676158511e-06\\
2.59999999999991	8.77391634461224e-07\\
2.79999999999988	4.25139995990505e-07\\
2.99999999999985	2.08020336821494e-07\\
3.19999999999981	1.03099092904703e-07\\
3.39999999999978	5.19466126024172e-08\\
3.59999999999975	2.67140661955691e-08\\
3.79999999999972	1.40768743086721e-08\\
3.99999999999969	7.62587443532037e-09\\
20	0\\
};
\label{plotyyref:leg3}

\addlegendentry{$\fffd$ (LH scale)}

\end{axis}

\begin{axis}[%
width=0.411\fwidth,
height=\fheight,
at={(0\fwidth,0\fheight)},
scale only axis,
xmin=0,
xmax=20,
every outer y axis line/.append style={mycolor2},
every y tick label/.append style={font=\color{mycolor2}},
every y tick/.append style={mycolor2},
ymin=2,
ymax=4,
ytick={2, 3, 4},
axis x line*=bottom,
axis y line*=right,
legend style={legend cell align=left, align=left, draw=white!15!black},
legend style={font=\tiny}
]
\addlegendimage{/pgfplots/refstyle=plotyyref:leg3}
\addlegendentry{$\fffd$ (LH scale)}
\addplot [color=mycolor2, dashed, line width=1.0pt]
  table[row sep=crcr]{%
3.03999999999984	3.53649563316118\\
3.06999999999983	3.53144062301247\\
3.09999999999983	3.52625281583674\\
3.12999999999983	3.52092999506675\\
3.15999999999982	3.51547003609343\\
3.18999999999982	3.50987091806524\\
3.21999999999981	3.50413073611416\\
3.24999999999981	3.49824771396115\\
3.2799999999998	3.49222021684542\\
3.3099999999998	3.48604676471263\\
3.33999999999979	3.47972604558867\\
3.36999999999979	3.47325692905596\\
3.39999999999978	3.46663847974055\\
3.42999999999978	3.45986997070955\\
3.45999999999977	3.45295089666983\\
3.48999999999977	3.4458809868516\\
3.51999999999977	3.43866021745361\\
3.54999999999976	3.43128882352077\\
3.57999999999976	3.42376731012116\\
3.60999999999975	3.41609646268585\\
3.63999999999975	3.40827735637459\\
3.66999999999974	3.40031136433071\\
3.69999999999974	3.39220016469186\\
3.72999999999973	3.38394574622826\\
3.75999999999973	3.37555041248772\\
3.78999999999972	3.36701678433665\\
3.81999999999972	3.35834780079843\\
3.84999999999971	3.34954671810509\\
3.87999999999971	3.3406171068953\\
3.9099999999997	3.3315628475098\\
3.9399999999997	3.32238812335649\\
3.9699999999997	3.31309741233875\\
3.99999999999969	3.30369547636353\\
4.02999999999968	3.29418734896953\\
4.05999999999968	3.2845783211389\\
4.08999999999967	3.27487392537966\\
4.11999999999967	3.26507991818841\\
4.14999999999966	3.25520226102403\\
4.17999999999966	3.24524709994283\\
4.20999999999965	3.23522074406286\\
4.23999999999964	3.22512964303987\\
4.26999999999964	3.21498036374954\\
4.29999999999963	3.20477956637944\\
4.32999999999963	3.1945339801398\\
4.35999999999962	3.18425037880459\\
4.38999999999961	3.17393555629327\\
4.41999999999961	3.1635963024994\\
4.4499999999996	3.15323937956486\\
4.4799999999996	3.14287149878832\\
4.50999999999959	3.13249929834361\\
4.53999999999958	3.12212932196887\\
4.56999999999958	3.11176799877042\\
4.59999999999957	3.10142162426693\\
4.62999999999957	3.09109634278036\\
4.65999999999956	3.08079813126017\\
4.68999999999955	3.07053278460732\\
4.71999999999955	3.06030590254462\\
4.74999999999954	3.05012287806081\\
4.77999999999954	3.03998888743737\\
4.80999999999953	3.02990888184941\\
4.83999999999952	3.01988758051651\\
4.86999999999952	3.00992946536461\\
4.89999999999951	3.00003877714727\\
4.92999999999951	2.9902195129637\\
4.9599999999995	2.98047542510147\\
4.98999999999949	2.97081002112427\\
5.01999999999949	2.96122656511912\\
5.04999999999948	2.95172808001329\\
5.07999999999948	2.94231735086826\\
5.10999999999947	2.93299692905678\\
5.13999999999946	2.92376913722917\\
5.16999999999946	2.91463607497565\\
5.19999999999945	2.90559962509412\\
5.22999999999945	2.89666146037532\\
5.25999999999944	2.88782305082098\\
5.28999999999943	2.87908567121509\\
5.31999999999943	2.87045040897253\\
5.34999999999942	2.86191817219487\\
5.37999999999942	2.8534896978679\\
5.40999999999941	2.84516556014088\\
5.4399999999994	2.83694617863307\\
5.4699999999994	2.82883182671797\\
5.49999999999939	2.82082263974132\\
5.52999999999939	2.81291862313373\\
5.55999999999938	2.80511966038353\\
5.58999999999937	2.79742552084037\\
5.61999999999937	2.78983586732406\\
5.64999999999936	2.78235026351753\\
5.67999999999936	2.77496818112621\\
5.70999999999935	2.76768900678998\\
5.73999999999934	2.7605120487366\\
5.76999999999934	2.75343654316892\\
5.79999999999933	2.74646166038025\\
5.82999999999933	2.73958651059495\\
5.85999999999932	2.73281014953329\\
5.88999999999931	2.72613158370138\\
5.91999999999931	2.71954977540867\\
5.9499999999993	2.7130636475168\\
5.9799999999993	2.70667208792475\\
6.00999999999929	2.70037395379646\\
6.03999999999929	2.69416807553753\\
6.06999999999928	2.68805326052858\\
6.09999999999927	2.68202829662334\\
6.12999999999927	2.6760919554198\\
6.15999999999926	2.6702429953132\\
6.18999999999926	2.6644801643398\\
6.21999999999925	2.65880220282044\\
6.24999999999924	2.65320784581302\\
6.27999999999924	2.64769582538289\\
6.30999999999923	2.6422648727002\\
6.33999999999923	2.63691371997301\\
6.36999999999922	2.63164110222487\\
6.39999999999921	2.62644575892524\\
6.42999999999921	2.62132643548123\\
6.4599999999992	2.6162818845983\\
6.4899999999992	2.61131086751805\\
6.51999999999919	2.60641215514013\\
6.54999999999918	2.60158452903576\\
6.57999999999918	2.59682678235955\\
6.60999999999917	2.59213772066613\\
6.63999999999917	2.58751616263804\\
6.66999999999916	2.58296094073062\\
6.69999999999915	2.57847090173974\\
6.72999999999915	2.57404490729767\\
6.75999999999914	2.56968183430221\\
6.78999999999914	2.5653805752839\\
6.81999999999913	2.56114003871586\\
6.84999999999912	2.55695914927059\\
6.87999999999912	2.55283684802781\\
6.90999999999911	2.54877209263706\\
6.93999999999911	2.54476385743872\\
6.9699999999991	2.54081113354683\\
6.99999999999909	2.53691292889682\\
7.02999999999909	2.53306826826112\\
7.05999999999908	2.5292761932355\\
7.08999999999908	2.52553576219856\\
7.11999999999907	2.52184605024694\\
7.14999999999906	2.51820614910849\\
7.17999999999906	2.51461516703541\\
7.20999999999905	2.51107222867937\\
7.23999999999905	2.50757647495047\\
7.26999999999904	2.5041270628617\\
7.29999999999903	2.50072316536043\\
7.32999999999903	2.49736397114842\\
7.35999999999902	2.49404868449176\\
7.38999999999902	2.49077652502186\\
7.41999999999901	2.48754672752876\\
7.449999999999	2.48435854174776\\
7.479999999999	2.48121123214037\\
7.50999999999899	2.47810407767047\\
7.53999999999899	2.47503637157659\\
7.56999999999898	2.47200742114096\\
7.59999999999897	2.46901654745607\\
7.62999999999897	2.46606308518946\\
7.65999999999896	2.46314638234724\\
7.68999999999896	2.46026580003685\\
7.71999999999895	2.45742071222971\\
7.74999999999894	2.45461050552402\\
7.77999999999894	2.45183457890818\\
7.80999999999893	2.44909234352528\\
7.83999999999893	2.44638322243887\\
7.86999999999892	2.44370665040033\\
7.89999999999892	2.44106207361822\\
7.92999999999891	2.43844894952967\\
7.9599999999989	2.43586674657415\\
7.9899999999989	2.43331494396982\\
8.01999999999889	2.43079303149249\\
8.04999999999889	2.42830050925749\\
8.07999999999888	2.42583688750453\\
8.10999999999887	2.42340168638555\\
8.13999999999887	2.42099443575588\\
8.16999999999886	2.41861467496852\\
8.19999999999886	2.41626195267188\\
8.22999999999885	2.41393582661076\\
8.25999999999884	2.41163586343086\\
8.28999999999884	2.40936163848668\\
8.31999999999883	2.40711273565289\\
8.34999999999883	2.40488874713924\\
8.37999999999882	2.40268927330888\\
8.40999999999881	2.40051392250022\\
8.43999999999881	2.39836231085229\\
8.4699999999988	2.39623406213353\\
8.4999999999988	2.39412880757402\\
8.52999999999879	2.3920461857012\\
8.55999999999878	2.3899858421789\\
8.58999999999878	2.38794742964982\\
8.61999999999877	2.38593060758127\\
8.64999999999877	2.38393504211427\\
8.67999999999876	2.38196040591586\\
8.70999999999875	2.38000637803461\\
8.73999999999875	2.37807264375934\\
8.76999999999874	2.37615889448091\\
8.79999999999874	2.37426482755714\\
8.82999999999873	2.37239014618066\\
8.85999999999872	2.37053455924985\\
8.88999999999872	2.36869778124257\\
8.91999999999871	2.36687953209288\\
8.94999999999871	2.36507953707046\\
8.9799999999987	2.3632975266629\\
9.00999999999869	2.36153323646063\\
9.03999999999869	2.35978640704452\\
9.06999999999868	2.35805678387609\\
9.09999999999868	2.35634411719029\\
9.12999999999867	2.35464816189073\\
9.15999999999866	2.35296867744742\\
9.18999999999866	2.35130542779683\\
9.21999999999865	2.34965818124438\\
9.24999999999865	2.34802671036912\\
9.27999999999864	2.34641079193076\\
9.30999999999863	2.3448102067788\\
9.33999999999863	2.34322473976385\\
9.36999999999862	2.34165417965105\\
9.39999999999862	2.34009831903548\\
9.42999999999861	2.33855695425968\\
9.4599999999986	2.33702988533295\\
9.4899999999986	2.33551691585278\\
9.51999999999859	2.3340178529279\\
9.54999999999859	2.33253250710336\\
9.57999999999858	2.33106069228726\\
9.60999999999857	2.32960222567925\\
9.63999999999857	2.32815692770072\\
9.66999999999856	2.32672462192664\\
9.69999999999856	2.32530513501903\\
9.72999999999855	2.32389829666194\\
9.75999999999854	2.32250393949806\\
9.78999999999854	2.32112189906673\\
9.81999999999853	2.31975201374348\\
9.84999999999853	2.31839412468094\\
9.87999999999852	2.31704807575121\\
9.90999999999852	2.31571371348949\\
9.93999999999851	2.3143908870391\\
9.9699999999985	2.31307944809777\\
9.9999999999985	2.31177925086521\\
10.0299999999985	2.31049015199183\\
10.0599999999985	2.30921201052873\\
10.0899999999985	2.30794468787887\\
10.1199999999985	2.30668804774926\\
10.1499999999985	2.30544195610439\\
10.1799999999985	2.30420628112068\\
10.2099999999985	2.30298089314196\\
10.2399999999984	2.30176566463604\\
10.2699999999984	2.30056047015222\\
10.2999999999984	2.29936518627983\\
10.3299999999984	2.29817969160766\\
10.3599999999984	2.29700386668439\\
10.3899999999984	2.29583759397986\\
10.4199999999984	2.29468075784727\\
10.4499999999984	2.29353324448622\\
10.4799999999984	2.29239494190657\\
10.5099999999984	2.29126573989319\\
10.5399999999984	2.29014552997139\\
10.5699999999984	2.28903420537324\\
10.5999999999984	2.28793166100462\\
10.6299999999984	2.28683779341292\\
10.6599999999984	2.28575250075563\\
10.6899999999984	2.28467568276947\\
10.7199999999984	2.2836072407403\\
10.7499999999983	2.28254707747369\\
10.7799999999983	2.28149509726609\\
10.8099999999983	2.28045120587673\\
10.8399999999983	2.27941531050002\\
10.8699999999983	2.27838731973868\\
10.8999999999983	2.27736714357737\\
10.9299999999983	2.27635469335692\\
10.9599999999983	2.27534988174918\\
10.9899999999983	2.27435262273232\\
11.0199999999983	2.27336283156673\\
11.0499999999983	2.27238042477145\\
11.0799999999983	2.27140532010104\\
11.1099999999983	2.27043743652305\\
11.1399999999983	2.26947669419587\\
11.1699999999983	2.26852301444714\\
11.1999999999983	2.26757631975256\\
11.2299999999983	2.26663653371518\\
11.2599999999982	2.26570358104514\\
11.2899999999982	2.2647773875398\\
11.3199999999982	2.26385788006434\\
11.3499999999982	2.26294498653273\\
11.3799999999982	2.26203863588911\\
11.4099999999982	2.26113875808955\\
11.4399999999982	2.26024528408426\\
11.4699999999982	2.25935814580003\\
11.4999999999982	2.25847727612319\\
11.5299999999982	2.25760260888281\\
11.5599999999982	2.25673407883431\\
11.5899999999982	2.25587162164337\\
11.6199999999982	2.2550151738702\\
11.6499999999982	2.25416467295409\\
11.6799999999982	2.25332005719835\\
11.7099999999982	2.25248126575547\\
11.7399999999982	2.25164823861263\\
11.7699999999981	2.25082091657751\\
11.7999999999981	2.24999924126434\\
11.8299999999981	2.24918315508029\\
11.8599999999981	2.24837260121208\\
11.8899999999981	2.2475675236129\\
11.9199999999981	2.24676786698955\\
11.9499999999981	2.2459735767899\\
11.9799999999981	2.24518459919049\\
12.0099999999981	2.2444008810845\\
12.0399999999981	2.24362237006987\\
12.0699999999981	2.24284901443767\\
12.0999999999981	2.24208076316075\\
12.1299999999981	2.24131756588254\\
12.1599999999981	2.24055937290609\\
12.1899999999981	2.23980613518335\\
12.2199999999981	2.23905780430466\\
12.249999999998	2.23831433248836\\
12.279999999998	2.23757567257072\\
12.309999999998	2.23684177799598\\
12.339999999998	2.23611260280661\\
12.369999999998	2.23538810163375\\
12.399999999998	2.23466822968784\\
12.429999999998	2.2339529427494\\
12.459999999998	2.23324219716003\\
12.489999999998	2.23253594981354\\
12.519999999998	2.23183415814728\\
12.549999999998	2.23113678013357\\
12.579999999998	2.2304437742714\\
12.609999999998	2.22975509957816\\
12.639999999998	2.22907071558162\\
12.669999999998	2.228390582312\\
12.699999999998	2.2277146602942\\
12.729999999998	2.22704291054022\\
12.7599999999979	2.22637529454165\\
12.7899999999979	2.22571177426233\\
12.8199999999979	2.22505231213113\\
12.8499999999979	2.2243968710349\\
12.8799999999979	2.22374541431151\\
12.9099999999979	2.22309790574302\\
12.9399999999979	2.22245430954899\\
12.9699999999979	2.22181459037989\\
12.9999999999979	2.22117871331063\\
13.0299999999979	2.22054664383426\\
13.0599999999979	2.21991834785568\\
13.0899999999979	2.21929379168555\\
13.1199999999979	2.21867294203428\\
13.1499999999979	2.2180557660061\\
13.1799999999979	2.21744223109328\\
13.2099999999979	2.21683230517043\\
13.2399999999979	2.21622595648888\\
13.2699999999978	2.21562315367119\\
13.2999999999978	2.21502386570573\\
13.3299999999978	2.21442806194142\\
13.3599999999978	2.21383571208243\\
13.3899999999978	2.21324678618311\\
13.4199999999978	2.21266125464291\\
13.4499999999978	2.21207908820148\\
13.4799999999978	2.21150025793372\\
13.5099999999978	2.21092473524507\\
13.5399999999978	2.21035249186676\\
13.5699999999978	2.2097834998512\\
13.5999999999978	2.20921773156744\\
13.6299999999978	2.20865515969667\\
13.6599999999978	2.20809575722787\\
13.6899999999978	2.20753949745346\\
13.7199999999978	2.20698635396506\\
13.7499999999978	2.20643630064932\\
13.7799999999977	2.2058893116838\\
13.8099999999977	2.20534536153295\\
13.8399999999977	2.2048044249441\\
13.8699999999977	2.20426647694362\\
13.8999999999977	2.20373149283301\\
13.9299999999977	2.20319944818517\\
13.9599999999977	2.20267031884066\\
13.9899999999977	2.20214408090403\\
14.0199999999977	2.20162071074029\\
14.0499999999977	2.20110018497129\\
14.0799999999977	2.2005824804723\\
14.1099999999977	2.20006757436855\\
14.1399999999977	2.19955544403188\\
14.1699999999977	2.19904606707745\\
14.1999999999977	2.19853942136042\\
14.2299999999977	2.19803548497279\\
14.2599999999976	2.19753423624023\\
14.2899999999976	2.19703565371898\\
14.3199999999976	2.19653971619276\\
14.3499999999976	2.19604640266982\\
14.3799999999976	2.19555569237992\\
14.4099999999976	2.19506756477145\\
14.4399999999976	2.19458199950856\\
14.4699999999976	2.19409897646832\\
14.4999999999976	2.19361847573796\\
14.5299999999976	2.19314047761212\\
14.5599999999976	2.19266496259015\\
14.5899999999976	2.19219191137349\\
14.6199999999976	2.19172130486303\\
14.6499999999976	2.19125312415656\\
14.6799999999976	2.1907873505462\\
14.7099999999976	2.19032396551598\\
14.7399999999976	2.18986295073931\\
14.7699999999975	2.18940428807662\\
14.7999999999975	2.18894795957296\\
14.8299999999975	2.18849394745565\\
14.8599999999975	2.18804223413198\\
14.8899999999975	2.18759280218694\\
14.9199999999975	2.18714563438099\\
14.9499999999975	2.18670071364783\\
14.9799999999975	2.18625802309228\\
15.0099999999975	2.18581754598806\\
15.0399999999975	2.18537926577577\\
15.0699999999975	2.18494316606077\\
15.0999999999975	2.18450923061112\\
15.1299999999975	2.18407744335561\\
15.1599999999975	2.18364778838175\\
15.1899999999975	2.18322024993383\\
15.2199999999975	2.18279481241095\\
15.2499999999975	2.18237146036521\\
15.2799999999974	2.18195017849974\\
15.3099999999974	2.18153095166693\\
15.3399999999974	2.18111376486658\\
15.3699999999974	2.18069860324413\\
15.3999999999974	2.18028545208888\\
15.4299999999974	2.17987429683226\\
15.4599999999974	2.17946512304613\\
15.4899999999974	2.17905791644106\\
15.5199999999974	2.17865266286471\\
15.5499999999974	2.17824934830014\\
15.5799999999974	2.17784795886426\\
15.6099999999974	2.17744848080616\\
15.6399999999974	2.1770509005056\\
15.6699999999974	2.17665520447144\\
15.6999999999974	2.17626137934011\\
15.7299999999974	2.1758694118741\\
15.7599999999974	2.17547928896051\\
15.7899999999973	2.17509099760954\\
15.8199999999973	2.1747045249531\\
15.8499999999973	2.17431985824335\\
15.8799999999973	2.17393698485129\\
15.9099999999973	2.17355589226543\\
15.9399999999973	2.17317656809037\\
15.9699999999973	2.17279900004551\\
15.9999999999973	2.17242317596367\\
16.0299999999973	2.17204908378981\\
16.0599999999973	2.17167671157976\\
16.0899999999973	2.17130604749891\\
16.1199999999973	2.17093707982099\\
16.1499999999973	2.17056979692681\\
16.1799999999973	2.17020418730304\\
16.2099999999973	2.16984023954104\\
16.2399999999973	2.16947794233563\\
16.2699999999972	2.16911728448395\\
16.2999999999972	2.16875825488428\\
16.3299999999972	2.16840084253493\\
16.3599999999972	2.16804503653309\\
16.3899999999972	2.16769082607375\\
16.4199999999972	2.16733820044857\\
16.4499999999972	2.16698714904483\\
16.4799999999972	2.16663766134434\\
16.5099999999972	2.16628972692243\\
16.5399999999972	2.16594333544687\\
16.5699999999972	2.16559847667687\\
16.5999999999972	2.16525514046206\\
16.6299999999972	2.16491331674152\\
16.6599999999972	2.16457299554276\\
16.6899999999972	2.16423416698076\\
16.7199999999972	2.16389682125705\\
16.7499999999972	2.16356094865871\\
16.7799999999971	2.16322653955747\\
16.8099999999971	2.16289358440879\\
16.8399999999971	2.16256207375094\\
16.8699999999971	2.16223199820412\\
16.8999999999971	2.16190334846955\\
16.9299999999971	2.16157611532862\\
16.9599999999971	2.16125028964201\\
16.9899999999971	2.16092586234888\\
17.0199999999971	2.16060282446595\\
17.0499999999971	2.16028116708676\\
17.0799999999971	2.1599608813808\\
17.1099999999971	2.15964195859271\\
17.1399999999971	2.15932439004147\\
17.1699999999971	2.15900816711967\\
17.1999999999971	2.15869328129264\\
17.2299999999971	2.15837972409777\\
17.2599999999971	2.15806748714369\\
17.289999999997	2.15775656210957\\
17.319999999997	2.15744694074434\\
17.349999999997	2.15713861486596\\
17.379999999997	2.15683157636076\\
17.409999999997	2.15652581718264\\
17.439999999997	2.15622132935244\\
17.469999999997	2.1559181049572\\
17.499999999997	2.1556161361495\\
17.529999999997	2.15531541514678\\
17.559999999997	2.15501593423066\\
17.589999999997	2.15471768574627\\
17.619999999997	2.15442066210165\\
17.649999999997	2.15412485576703\\
17.679999999997	2.15383025927426\\
17.709999999997	2.15353686521614\\
17.739999999997	2.15324466624583\\
17.769999999997	2.1529536550762\\
17.7999999999969	2.15266382447925\\
17.8299999999969	2.15237516728553\\
17.8599999999969	2.15208767638349\\
17.8899999999969	2.15180134471894\\
17.9199999999969	2.1515161652945\\
17.9499999999969	2.15123213116893\\
17.9799999999969	2.15094923545669\\
18.0099999999969	2.15066747132729\\
18.0399999999969	2.15038683200479\\
18.0699999999969	2.15010731076724\\
18.0999999999969	2.14982890094613\\
18.1299999999969	2.14955159592589\\
18.1599999999969	2.14927538914336\\
18.1899999999969	2.14900027408723\\
18.2199999999969	2.14872624429759\\
18.2499999999969	2.14845329336538\\
18.2799999999968	2.14818141493193\\
18.3099999999968	2.1479106026884\\
18.3399999999968	2.14764085037538\\
18.3699999999968	2.14737215178232\\
18.3999999999968	2.14710450074713\\
18.4299999999968	2.14683789115566\\
18.4599999999968	2.14657231694123\\
18.4899999999968	2.14630777208422\\
18.5199999999968	2.14604425061157\\
18.5499999999968	2.14578174659634\\
18.5799999999968	2.14552025415728\\
18.6099999999968	2.14525976745836\\
18.6399999999968	2.14500028070838\\
18.6699999999968	2.14474178816051\\
18.6999999999968	2.14448428411186\\
18.7299999999968	2.14422776290309\\
18.7599999999968	2.14397221891797\\
18.7899999999967	2.14371764658299\\
18.8199999999967	2.14346404036692\\
18.8499999999967	2.14321139478045\\
18.8799999999967	2.14295970437579\\
18.9099999999967	2.14270896374624\\
18.9399999999967	2.14245916752583\\
18.9699999999967	2.14221031038896\\
18.9999999999967	2.14196238704998\\
19.0299999999967	2.14171539226282\\
19.0599999999967	2.14146932082067\\
19.0899999999967	2.14122416755554\\
19.1199999999967	2.14097992733794\\
19.1499999999967	2.14073659507655\\
19.1799999999967	2.1404941657178\\
19.2099999999967	2.14025263424555\\
19.2399999999967	2.14001199568079\\
19.2699999999967	2.13977224508121\\
19.2999999999966	2.13953337754092\\
19.3299999999966	2.13929538819012\\
19.3599999999966	2.13905827219474\\
19.3899999999966	2.13882202475611\\
19.4199999999966	2.13858664111068\\
19.4499999999966	2.13835211652965\\
19.4799999999966	2.13811844631867\\
19.5099999999966	2.13788562581756\\
19.5399999999966	2.13765365039993\\
19.5699999999966	2.13742251547294\\
19.5999999999966	2.13719221647698\\
19.6299999999966	2.13696274888533\\
19.6599999999966	2.13673410820392\\
19.6899999999966	2.13650628997101\\
19.7199999999966	2.13627928975688\\
19.7499999999966	2.1360531031636\\
19.7799999999966	2.13582772582467\\
19.8099999999965	2.13560315340481\\
19.8399999999965	2.13537938159963\\
19.8699999999965	2.13515640613541\\
19.8999999999965	2.13493422276874\\
19.9299999999965	2.13471282728635\\
19.9599999999965	2.13449221550479\\
19.9899999999965	2.13427238327015\\
};
\addlegendentry{$-\frac{\log(\fffd(t)/\fffd(3))}{t-3}$ (RH scale)}

\end{axis}

\begin{axis}[%
width=0.411\fwidth,
height=\fheight,
at={(0.54\fwidth,0\fheight)},
scale only axis,
xmin=0,
xmax=20,
xlabel style={font=\color{white!15!black}},
xlabel={Time $t$},
separate axis lines,
every outer y axis line/.append style={mycolor1},
every y tick label/.append style={font=\color{mycolor1}},
every y tick/.append style={mycolor1},
ymin=0,
ymax=2,
axis background/.style={fill=white},
legend style={legend cell align=left, align=left, draw=white!15!black},
legend style={font=\tiny}
]
\addplot [color=mycolor1, line width=1.0pt]
  table[row sep=crcr]{%
0	1.03719275943712\\
0.2	0.0303329079550919\\
0.4	0.0102685089463639\\
0.6	0.00450978586186067\\
0.8	0.00207188007297845\\
1	0.000966686000650169\\
1.2	0.000454154661919769\\
1.4	0.000214128054144148\\
1.6	0.000101181184241096\\
1.8	4.78940212693983e-05\\
2	2.27109077109896e-05\\
2.19999999999997	1.07925894502914e-05\\
2.39999999999994	5.1434568140938e-06\\
2.59999999999991	2.46075653546554e-06\\
2.79999999999988	1.18354698116423e-06\\
2.99999999999985	5.73369095730959e-07\\
3.19999999999981	2.80470659354847e-07\\
3.39999999999978	1.3895644268868e-07\\
3.59999999999975	6.99811592844952e-08\\
3.79999999999972	3.59683008134562e-08\\
3.99999999999969	1.89409418562261e-08\\
20	0\\
};
\label{plotyyref:leg4}

\addlegendentry{$\gggd$ (LH scale)}

\end{axis}

\begin{axis}[%
width=0.411\fwidth,
height=\fheight,
at={(0.54\fwidth,0\fheight)},
scale only axis,
xmin=0,
xmax=20,
every outer y axis line/.append style={mycolor2},
every y tick label/.append style={font=\color{mycolor2}},
every y tick/.append style={mycolor2},
ymin=2,
ymax=4,
ytick={2, 3, 4},
axis x line*=bottom,
axis y line*=right,
legend style={legend cell align=left, align=left, draw=white!15!black},
legend style={font=\tiny}
]
\addlegendimage{/pgfplots/refstyle=plotyyref:leg4}
\addlegendentry{$\gggd$ (LH scale)}
\addplot [color=mycolor2, dashed, line width=1.0pt]
  table[row sep=crcr]{%
3.03999999999984	3.59567116232324\\
3.06999999999983	3.59184731808928\\
3.09999999999983	3.58791330903797\\
3.12999999999983	3.58386653928755\\
3.15999999999982	3.57970442729926\\
3.18999999999982	3.57542441346817\\
3.21999999999981	3.57102396826727\\
3.24999999999981	3.56650060094664\\
3.2799999999998	3.56185186878549\\
3.3099999999998	3.55707538688926\\
3.33999999999979	3.5521688385186\\
3.36999999999979	3.54712998593049\\
3.39999999999978	3.54195668170497\\
3.42999999999978	3.5366468805237\\
3.45999999999977	3.53119865135836\\
3.48999999999977	3.52561019001894\\
3.51999999999977	3.51987983200301\\
3.54999999999976	3.51400606557848\\
3.57999999999976	3.50798754502307\\
3.60999999999975	3.50182310393502\\
3.63999999999975	3.49551176852049\\
3.66999999999974	3.48905277075502\\
3.69999999999974	3.4824455613083\\
3.72999999999973	3.47568982211453\\
3.75999999999973	3.46878547846454\\
3.78999999999972	3.46173271049114\\
3.81999999999972	3.45453196391519\\
3.84999999999971	3.44718395991877\\
3.87999999999971	3.43968970401119\\
3.9099999999997	3.43205049375609\\
3.9399999999997	3.42426792523152\\
3.9699999999997	3.41634389810191\\
3.99999999999969	3.40828061918882\\
4.02999999999968	3.4000806044391\\
4.05999999999968	3.39174667920167\\
4.08999999999967	3.38328197674019\\
4.11999999999967	3.37468993492572\\
4.14999999999966	3.36597429107313\\
4.17999999999966	3.35713907490517\\
4.20999999999965	3.34818859965036\\
4.23999999999964	3.33912745130268\\
4.26999999999964	3.32996047609452\\
4.29999999999963	3.32069276625629\\
4.32999999999963	3.31132964415857\\
4.35999999999962	3.3018766449534\\
4.38999999999961	3.29233949785062\\
4.41999999999961	3.28272410618293\\
4.4499999999996	3.27303652642813\\
4.4799999999996	3.26328294636985\\
4.50999999999959	3.25346966258749\\
4.53999999999958	3.24360305747287\\
4.56999999999958	3.23368957597425\\
4.59999999999957	3.22373570226904\\
4.62999999999957	3.21374793656319\\
4.65999999999956	3.20373277220966\\
4.68999999999955	3.19369667332967\\
4.71999999999955	3.18364605310917\\
4.74999999999954	3.17358725292961\\
4.77999999999954	3.16352652247688\\
4.80999999999953	3.15347000095544\\
4.83999999999952	3.14342369951695\\
4.86999999999952	3.13339348499415\\
4.89999999999951	3.12338506501186\\
4.92999999999951	3.1134039745283\\
4.9599999999995	3.10345556384147\\
4.98999999999949	3.09354498807787\\
5.01999999999949	3.08367719816375\\
5.04999999999948	3.07385693326411\\
5.07999999999948	3.06408871466006\\
5.10999999999947	3.05437684102274\\
5.13999999999946	3.0447253850309\\
5.16999999999946	3.03513819126964\\
5.19999999999945	3.02561887534021\\
5.22999999999945	3.01617082410427\\
5.25999999999944	3.00679719698149\\
5.28999999999943	2.99750092821595\\
5.31999999999943	2.98828473002514\\
5.34999999999942	2.97915109654432\\
5.37999999999942	2.97010230847992\\
5.40999999999941	2.96114043838655\\
5.4399999999994	2.95226735648497\\
5.4699999999994	2.943484736941\\
5.49999999999939	2.93479406452904\\
5.52999999999939	2.92619664160796\\
5.55999999999938	2.91769359534148\\
5.58999999999937	2.90928588509975\\
5.61999999999937	2.90097430998365\\
5.64999999999936	2.89275951641839\\
5.67999999999936	2.88464200576745\\
5.70999999999935	2.87662214192318\\
5.73999999999934	2.86870015883463\\
5.76999999999934	2.86087616793813\\
5.79999999999933	2.85315016545994\\
5.82999999999933	2.84552203956498\\
5.85999999999932	2.83799157732898\\
5.88999999999931	2.83055847151541\\
5.91999999999931	2.82322232714154\\
5.9499999999993	2.81598266782136\\
5.9799999999993	2.80883894187568\\
6.00999999999929	2.80179052820235\\
6.03999999999929	2.79483674190192\\
6.06999999999928	2.78797683965592\\
6.09999999999927	2.78121002485701\\
6.12999999999927	2.77453545249167\\
6.15999999999926	2.76795223377755\\
6.18999999999926	2.76145944055891\\
6.21999999999925	2.7550561094644\\
6.24999999999924	2.74874124583266\\
6.27999999999924	2.7425138274115\\
6.30999999999923	2.73637280783744\\
6.33999999999923	2.73031711990254\\
6.36999999999922	2.72434567861608\\
6.39999999999921	2.71845738406869\\
6.42999999999921	2.71265112410684\\
6.4599999999992	2.70692577682572\\
6.4899999999992	2.70128021288849\\
6.51999999999919	2.69571329768002\\
6.54999999999918	2.69022389330288\\
6.57999999999918	2.68481086042365\\
6.60999999999917	2.6794730599771\\
6.63999999999917	2.67420935473582\\
6.66999999999916	2.66901861075267\\
6.69999999999915	2.66389969868309\\
6.72999999999915	2.65885149499418\\
6.75999999999914	2.65387288306732\\
6.78999999999914	2.64896275420043\\
6.81999999999913	2.6441200085164\\
6.84999999999912	2.6393435557832\\
6.87999999999912	2.63463231615147\\
6.90999999999911	2.62998522081502\\
6.93999999999911	2.62540121259914\\
6.9699999999991	2.62087924648168\\
6.99999999999909	2.61641829005163\\
7.02999999999909	2.61201732390924\\
7.05999999999908	2.60767534201219\\
7.08999999999908	2.60339135197144\\
7.11999999999907	2.59916437530049\\
7.14999999999906	2.59499344762173\\
7.17999999999906	2.59087761883283\\
7.20999999999905	2.58681595323647\\
7.23999999999905	2.58280752963631\\
7.26999999999904	2.57885144140171\\
7.29999999999903	2.57494679650403\\
7.32999999999903	2.57109271752666\\
7.35999999999902	2.56728834165108\\
7.38999999999902	2.5635328206212\\
7.41999999999901	2.55982532068759\\
7.449999999999	2.55616502253379\\
7.479999999999	2.55255112118612\\
7.50999999999899	2.54898282590876\\
7.53999999999899	2.54545936008537\\
7.56999999999898	2.54197996108889\\
7.59999999999897	2.53854388014051\\
7.62999999999897	2.53515038215918\\
7.65999999999896	2.53179874560269\\
7.68999999999896	2.52848826230131\\
7.71999999999895	2.52521823728498\\
7.74999999999894	2.52198798860487\\
7.77999999999894	2.51879684715018\\
7.80999999999893	2.51564415646079\\
7.83999999999893	2.51252927253661\\
7.86999999999892	2.50945156364411\\
7.89999999999892	2.50641041012066\\
7.92999999999891	2.50340520417729\\
7.9599999999989	2.50043534970018\\
7.9899999999989	2.49750026205146\\
8.01999999999889	2.49459936786972\\
8.04999999999889	2.49173210487052\\
8.07999999999888	2.48889792164726\\
8.10999999999887	2.48609627747283\\
8.13999999999887	2.48332664210209\\
8.16999999999886	2.48058849557569\\
8.19999999999886	2.47788132802521\\
8.22999999999885	2.47520463948005\\
8.25999999999884	2.47255793967604\\
8.28999999999884	2.46994074786609\\
8.31999999999883	2.46735259263294\\
8.34999999999883	2.46479301170419\\
8.37999999999882	2.4622615517696\\
8.40999999999881	2.45975776830102\\
8.43999999999881	2.45728122537464\\
8.4699999999988	2.45483149549605\\
8.4999999999988	2.45240815942784\\
8.52999999999879	2.45001080601996\\
8.55999999999878	2.44763903204291\\
8.58999999999878	2.44529244202358\\
8.61999999999877	2.44297064808404\\
8.64999999999877	2.44067326978308\\
8.67999999999876	2.43839993396067\\
8.70999999999875	2.43615027458517\\
8.73999999999875	2.43392393260344\\
8.76999999999874	2.43172055579384\\
8.79999999999874	2.42953979862194\\
8.82999999999873	2.42738132209914\\
8.85999999999872	2.42524479364403\\
8.88999999999872	2.42312988694659\\
8.91999999999871	2.42103628183501\\
8.94999999999871	2.4189636641454\\
8.9799999999987	2.41691172559404\\
9.00999999999869	2.41488016365237\\
9.03999999999869	2.41286868142457\\
9.06999999999868	2.41087698752778\\
9.09999999999868	2.40890479597476\\
9.12999999999867	2.40695182605923\\
9.15999999999866	2.40501780224353\\
9.18999999999866	2.40310245404878\\
9.21999999999865	2.40120551594748\\
9.24999999999865	2.39932672725837\\
9.27999999999864	2.39746583204368\\
9.30999999999863	2.39562257900863\\
9.33999999999863	2.39379672140312\\
9.36999999999862	2.39198801692567\\
9.39999999999862	2.39019622762946\\
9.42999999999861	2.38842111983052\\
9.4599999999986	2.38666246401789\\
9.4899999999986	2.38492003476594\\
9.51999999999859	2.38319361064853\\
9.54999999999859	2.38148297415521\\
9.57999999999858	2.37978791160925\\
9.60999999999857	2.37810821308757\\
9.63999999999857	2.37644367234246\\
9.66999999999856	2.37479408672512\\
9.69999999999856	2.37315925711087\\
9.72999999999855	2.37153898782614\\
9.75999999999854	2.36993308657706\\
9.78999999999854	2.3683413643797\\
9.81999999999853	2.36676363549192\\
9.84999999999853	2.36519971734676\\
9.87999999999852	2.3636494304873\\
9.90999999999852	2.36211259850311\\
9.93999999999851	2.36058904796803\\
9.9699999999985	2.35907860837948\\
9.9999999999985	2.35758111209903\\
10.0299999999985	2.35609639429447\\
10.0599999999985	2.35462429288302\\
10.0899999999985	2.35316464847603\\
10.1199999999985	2.35171730432476\\
10.1499999999985	2.35028210626749\\
10.1799999999985	2.3488589026778\\
10.2099999999985	2.34744754441403\\
10.2399999999984	2.34604788476986\\
10.2699999999984	2.34465977942606\\
10.2999999999984	2.34328308640323\\
10.3299999999984	2.34191766601573\\
10.3599999999984	2.34056338082654\\
10.3899999999984	2.3392200956032\\
10.4199999999984	2.33788767727471\\
10.4499999999984	2.33656599488939\\
10.4799999999984	2.33525491957369\\
10.5099999999984	2.33395432449196\\
10.5399999999984	2.33266408480701\\
10.5699999999984	2.33138407764168\\
10.5999999999984	2.33011418204115\\
10.6299999999984	2.32885427893616\\
10.6599999999984	2.32760425110701\\
10.6899999999984	2.32636398314834\\
10.7199999999984	2.32513336143472\\
10.7499999999983	2.323912274087\\
10.7799999999983	2.32270061093933\\
10.8099999999983	2.321498263507\\
10.8399999999983	2.32030512495488\\
10.8699999999983	2.31912109006665\\
10.8999999999983	2.31794605521461\\
10.9299999999983	2.31677991833022\\
10.9599999999983	2.31562257887519\\
10.9899999999983	2.31447393781327\\
11.0199999999983	2.31333389758264\\
11.0499999999983	2.31220236206882\\
11.0799999999983	2.31107923657827\\
11.1099999999983	2.30996442781246\\
11.1399999999983	2.30885784384254\\
11.1699999999983	2.30775939408454\\
11.1999999999983	2.30666898927511\\
11.2299999999983	2.30558654144772\\
11.2599999999982	2.30451196390943\\
11.2899999999982	2.30344517121811\\
11.3199999999982	2.30238607916016\\
11.3499999999982	2.30133460472866\\
11.3799999999982	2.30029066610201\\
11.4099999999982	2.299254182623\\
11.4399999999982	2.29822507477835\\
11.4699999999982	2.29720326417857\\
11.4999999999982	2.2961886735384\\
11.5299999999982	2.2951812266575\\
11.5599999999982	2.2941808484016\\
11.5899999999982	2.29318746468409\\
11.6199999999982	2.29220100244784\\
11.6499999999982	2.29122138964758\\
11.6799999999982	2.29024855523246\\
11.7099999999982	2.28928242912909\\
11.7399999999982	2.28832294222485\\
11.7699999999981	2.28737002635158\\
11.7999999999981	2.28642361426957\\
11.8299999999981	2.28548363965188\\
11.8599999999981	2.28455003706897\\
11.8899999999981	2.28362274197365\\
11.9199999999981	2.28270169068632\\
11.9499999999981	2.28178682038048\\
11.9799999999981	2.28087806906858\\
12.0099999999981	2.27997537558808\\
12.0399999999981	2.27907867958786\\
12.0699999999981	2.27818792151483\\
12.0999999999981	2.27730304260082\\
12.1299999999981	2.27642398484977\\
12.1599999999981	2.27555069102509\\
12.1899999999981	2.27468310463733\\
12.2199999999981	2.27382116993209\\
12.249999999998	2.27296483187806\\
12.279999999998	2.27211403615545\\
12.309999999998	2.27126872914454\\
12.339999999998	2.27042885791441\\
12.369999999998	2.26959437021202\\
12.399999999998	2.26876521445138\\
12.429999999998	2.26794133970296\\
12.459999999998	2.26712269568333\\
12.489999999998	2.26630923274494\\
12.519999999998	2.26550090186614\\
12.549999999998	2.26469765464136\\
12.579999999998	2.2638994432715\\
12.609999999998	2.26310622055444\\
12.639999999998	2.2623179398758\\
12.669999999998	2.26153455519984\\
12.699999999998	2.26075602106049\\
12.729999999998	2.25998229255261\\
12.7599999999979	2.25921332532338\\
12.7899999999979	2.25844907556382\\
12.8199999999979	2.25768950000053\\
12.8499999999979	2.2569345558875\\
12.8799999999979	2.25618420099817\\
12.9099999999979	2.25543839361749\\
12.9399999999979	2.25469709253431\\
12.9699999999979	2.2539602570337\\
12.9999999999979	2.25322784688959\\
13.0299999999979	2.25249982235744\\
13.0599999999979	2.25177614416704\\
13.0899999999979	2.2510567735155\\
13.1199999999979	2.2503416720603\\
13.1499999999979	2.2496308019125\\
13.1799999999979	2.24892412563006\\
13.2099999999979	2.24822160621127\\
13.2399999999979	2.24752320708832\\
13.2699999999978	2.24682889212093\\
13.2999999999978	2.24613862559015\\
13.3299999999978	2.24545237219224\\
13.3599999999978	2.24477009703268\\
13.3899999999978	2.2440917656202\\
13.4199999999978	2.24341734386104\\
13.4499999999978	2.24274679805321\\
13.4799999999978	2.2420800948809\\
13.5099999999978	2.24141720140895\\
13.5399999999978	2.24075808507744\\
13.5699999999978	2.24010271369639\\
13.5999999999978	2.23945105544048\\
13.6299999999978	2.23880307884391\\
13.6599999999978	2.23815875279539\\
13.6899999999978	2.23751804653311\\
13.7199999999978	2.23688092963986\\
13.7499999999978	2.23624737203823\\
13.7799999999977	2.2356173439859\\
13.8099999999977	2.23499081607094\\
13.8399999999977	2.23436775920728\\
13.8699999999977	2.23374814463019\\
13.8999999999977	2.23313194389185\\
13.9299999999977	2.23251912885703\\
13.9599999999977	2.23190967169879\\
13.9899999999977	2.23130354489425\\
14.0199999999977	2.2307007212205\\
14.0499999999977	2.23010117375047\\
14.0799999999977	2.22950487584896\\
14.1099999999977	2.2289118011687\\
14.1399999999977	2.22832192364644\\
14.1699999999977	2.22773521749917\\
14.1999999999977	2.22715165722032\\
14.2299999999977	2.22657121757612\\
14.2599999999976	2.22599387360193\\
14.2899999999976	2.22541960059867\\
14.3199999999976	2.22484837412928\\
14.3499999999976	2.2242801700153\\
14.3799999999976	2.22371496433342\\
14.4099999999976	2.22315273341213\\
14.4399999999976	2.22259345382845\\
14.4699999999976	2.22203710240461\\
14.4999999999976	2.22148365620492\\
14.5299999999976	2.22093309253258\\
14.5599999999976	2.22038538892659\\
14.5899999999976	2.2198405231587\\
14.6199999999976	2.21929847323038\\
14.6499999999976	2.21875921736989\\
14.6799999999976	2.21822273402935\\
14.7099999999976	2.21768900188188\\
14.7399999999976	2.21715799981875\\
14.7699999999975	2.21662970694664\\
14.7999999999975	2.21610410258487\\
14.8299999999975	2.21558116626268\\
14.8599999999975	2.21506087771664\\
14.8899999999975	2.21454321688797\\
14.9199999999975	2.21402816392\\
14.9499999999975	2.2135156991556\\
14.9799999999975	2.21300580313472\\
15.0099999999975	2.2124984565919\\
15.0399999999975	2.21199364045382\\
15.0699999999975	2.21149133583699\\
15.0999999999975	2.21099152404529\\
15.1299999999975	2.21049418656774\\
15.1599999999975	2.20999930507617\\
15.1899999999975	2.20950686142299\\
15.2199999999975	2.20901683763894\\
15.2499999999975	2.20852921593093\\
15.2799999999974	2.2080439786799\\
15.3099999999974	2.20756110843867\\
15.3399999999974	2.20708058792984\\
15.3699999999974	2.20660240004378\\
15.3999999999974	2.20612652783654\\
15.4299999999974	2.2056529545279\\
15.4599999999974	2.20518166349937\\
15.4899999999974	2.20471263829223\\
15.5199999999974	2.20424586260566\\
15.5499999999974	2.20378132029481\\
15.5799999999974	2.20331899536897\\
15.6099999999974	2.20285887198971\\
15.6399999999974	2.20240093446908\\
15.6699999999974	2.20194516726784\\
15.6999999999974	2.20149155499368\\
15.7299999999974	2.2010400823995\\
15.7599999999974	2.20059073438172\\
15.7899999999973	2.20014349597855\\
15.8199999999973	2.19969835236837\\
15.8499999999973	2.19925528886807\\
15.8799999999973	2.19881429093145\\
15.9099999999973	2.19837534414763\\
15.9399999999973	2.19793843423946\\
15.9699999999973	2.19750354706199\\
15.9999999999973	2.19707066860094\\
16.0299999999973	2.19663978497118\\
16.0599999999973	2.19621088241528\\
16.0899999999973	2.19578394730202\\
16.1199999999973	2.19535896612492\\
16.1499999999973	2.19493592550089\\
16.1799999999973	2.19451481216875\\
16.2099999999973	2.19409561298789\\
16.2399999999973	2.19367831493689\\
16.2699999999972	2.19326290511215\\
16.2999999999972	2.1928493707266\\
16.3299999999972	2.19243769910838\\
16.3599999999972	2.1920278776995\\
16.3899999999972	2.19161989405461\\
16.4199999999972	2.19121373583975\\
16.4499999999972	2.19080939083106\\
16.4799999999972	2.19040684691357\\
16.5099999999972	2.19000609208003\\
16.5399999999972	2.18960711442967\\
16.5699999999972	2.18920990216701\\
16.5999999999972	2.18881444360077\\
16.6299999999972	2.18842072714264\\
16.6599999999972	2.1880287413062\\
16.6899999999972	2.18763847470579\\
16.7199999999972	2.1872499160554\\
16.7499999999972	2.18686305416759\\
16.7799999999971	2.18647787795241\\
16.8099999999971	2.18609437641635\\
16.8399999999971	2.18571253866128\\
16.8699999999971	2.18533235388344\\
16.8999999999971	2.18495381137239\\
16.9299999999971	2.18457690051004\\
16.9599999999971	2.18420161076961\\
16.9899999999971	2.18382793171471\\
17.0199999999971	2.18345585299832\\
17.0499999999971	2.18308536436187\\
17.0799999999971	2.18271645563426\\
17.1099999999971	2.18234911673098\\
17.1399999999971	2.18198333765316\\
17.1699999999971	2.18161910848664\\
17.1999999999971	2.18125641940112\\
17.2299999999971	2.18089526064926\\
17.2599999999971	2.18053562256579\\
17.289999999997	2.18017749556668\\
17.319999999997	2.17982087014825\\
17.349999999997	2.17946573688635\\
17.379999999997	2.17911208643553\\
17.409999999997	2.17875990952824\\
17.439999999997	2.17840919697396\\
17.469999999997	2.17805993965848\\
17.499999999997	2.17771212854304\\
17.529999999997	2.17736575466359\\
17.559999999997	2.17702080913001\\
17.589999999997	2.17667728312537\\
17.619999999997	2.17633516790511\\
17.649999999997	2.1759944547964\\
17.679999999997	2.17565513519732\\
17.709999999997	2.17531720057618\\
17.739999999997	2.1749806424708\\
17.769999999997	2.17464545248778\\
17.7999999999969	2.17431162230183\\
17.8299999999969	2.1739791436551\\
17.8599999999969	2.17364800835641\\
17.8899999999969	2.1733182082807\\
17.9199999999969	2.17298973536825\\
17.9499999999969	2.17266258162412\\
17.9799999999969	2.17233673911743\\
18.0099999999969	2.17201219998076\\
18.0399999999969	2.1716889564095\\
18.0699999999969	2.17136700066123\\
18.0999999999969	2.17104632505508\\
18.1299999999969	2.17072692197118\\
18.1599999999969	2.17040878384997\\
18.1899999999969	2.17009190319169\\
18.2199999999969	2.16977627255573\\
18.2499999999969	2.16946188456006\\
18.2799999999968	2.16914873188067\\
18.3099999999968	2.168836807251\\
18.3399999999968	2.16852610346136\\
18.3699999999968	2.1682166133584\\
18.3999999999968	2.16790832984452\\
18.4299999999968	2.1676012458774\\
18.4599999999968	2.16729535446936\\
18.4899999999968	2.16699064868693\\
18.5199999999968	2.16668712165027\\
18.5499999999968	2.16638476653267\\
18.5799999999968	2.16608357656002\\
18.6099999999968	2.16578354501035\\
18.6399999999968	2.16548466521326\\
18.6699999999968	2.16518693054949\\
18.6999999999968	2.1648903344504\\
18.7299999999968	2.16459487039749\\
18.7599999999968	2.16430053192192\\
18.7899999999967	2.16400731260405\\
18.8199999999967	2.16371520607296\\
18.8499999999967	2.163424206006\\
18.8799999999967	2.16313430612833\\
18.9099999999967	2.16284550021246\\
18.9399999999967	2.16255778207782\\
18.9699999999967	2.16227114559029\\
18.9999999999967	2.16198558466182\\
19.0299999999967	2.16170109324994\\
19.0599999999967	2.16141766535734\\
19.0899999999967	2.16113529503152\\
19.1199999999967	2.16085397636425\\
19.1499999999967	2.16057370349129\\
19.1799999999967	2.16029447059189\\
19.2099999999967	2.1600162718884\\
19.2399999999967	2.15973910164593\\
19.2699999999967	2.15946295417188\\
19.2999999999966	2.15918782381559\\
19.3299999999966	2.15891370496797\\
19.3599999999966	2.15864059206107\\
19.3899999999966	2.15836847956775\\
19.4199999999966	2.15809736200128\\
19.4499999999966	2.15782723391497\\
19.4799999999966	2.15755808990184\\
19.5099999999966	2.15728992459421\\
19.5399999999966	2.15702273266338\\
19.5699999999966	2.15675650881926\\
19.5999999999966	2.15649124781001\\
19.6299999999966	2.15622694442173\\
19.6599999999966	2.1559635934781\\
19.6899999999966	2.15570118984002\\
19.7199999999966	2.15543972840531\\
19.7499999999966	2.15517920410835\\
19.7799999999966	2.15491961191979\\
19.8099999999965	2.15466094684619\\
19.8399999999965	2.15440320392972\\
19.8699999999965	2.15414637824782\\
19.8999999999965	2.15389046491293\\
19.9299999999965	2.15363545907213\\
19.9599999999965	2.15338135590688\\
19.9899999999965	2.15312815063266\\
};
\addlegendentry{$-\frac{\log(\gggd(t)/\gggd(3))}{t-3}$ (RH scale)}

\end{axis}
\end{tikzpicture}%

%% file: evolfinhom_t-0_cas-rings.tex
\begin{tikzpicture}
\end{tikzpicture}%
\caption{Initial datum $f^0$ in the inhomogeneous case, the velocity range is $(-20,20)$, the space range is $(0,1)$ the discretization steps are  $\dvh=0.4$, $\dxh=0.01$ and $\dth=0.0005$.}

%% file: evolentfishinhom_cas-rings.tex
\begin{tikzpicture}

\begin{axis}[%
width=0.954\fwidth,
height=\fheight,
at={(0\fwidth,0\fheight)},
scale only axis,
xmin=0.000,
xmax=20.000,
xlabel style={font=\color{white!15!black}},
xlabel={Time $t$},
ymin=-80.000,
ymax=0.000,
axis background/.style={fill=white},
legend style={legend cell align=left, align=left, draw=white!15!black},
legend style={font=\tiny}
]
\addplot [color=black, line width=1.0pt]
  table[row sep=crcr]{%
0.000	0.000\\
0.182	-3.115\\
0.363	-4.751\\
0.544	-5.741\\
0.726	-7.373\\
0.907	-9.380\\
1.089	-10.720\\
1.270	-11.593\\
1.452	-12.341\\
1.633	-13.045\\
1.815	-13.732\\
1.996	-14.412\\
2.178	-15.089\\
2.359	-15.765\\
2.541	-16.441\\
2.723	-17.116\\
2.904	-17.791\\
3.086	-18.466\\
3.267	-19.142\\
3.449	-19.817\\
3.630	-20.492\\
3.812	-21.167\\
3.993	-21.842\\
4.175	-22.517\\
4.356	-23.192\\
4.538	-23.867\\
4.719	-24.542\\
4.901	-25.218\\
5.082	-25.893\\
5.263	-26.568\\
5.445	-27.243\\
5.626	-27.918\\
5.808	-28.593\\
5.989	-29.268\\
6.171	-29.943\\
6.352	-30.618\\
6.534	-31.294\\
6.715	-31.969\\
6.897	-32.644\\
7.078	-33.319\\
7.260	-33.994\\
7.441	-34.669\\
7.623	-35.344\\
7.804	-36.019\\
7.986	-36.695\\
8.167	-37.370\\
8.349	-38.045\\
8.530	-38.720\\
8.712	-39.395\\
8.893	-40.070\\
9.075	-40.745\\
9.256	-41.420\\
9.438	-42.095\\
9.620	-42.771\\
9.801	-43.446\\
9.983	-44.121\\
10.164	-44.796\\
10.346	-45.471\\
10.527	-46.146\\
10.709	-46.821\\
10.890	-47.496\\
11.072	-48.171\\
11.253	-48.847\\
11.435	-49.522\\
11.616	-50.197\\
11.798	-50.872\\
11.979	-51.547\\
12.161	-52.222\\
12.342	-52.897\\
12.524	-53.572\\
12.705	-54.247\\
12.887	-54.923\\
13.068	-55.598\\
13.250	-56.273\\
13.431	-56.948\\
13.613	-57.623\\
13.794	-58.298\\
13.976	-58.973\\
14.157	-59.648\\
14.339	-60.323\\
14.520	-60.998\\
14.702	-61.673\\
14.883	-62.348\\
15.065	-63.022\\
15.246	-63.696\\
15.428	-64.368\\
15.609	-65.038\\
15.791	-65.703\\
15.972	-66.357\\
16.154	-66.993\\
16.335	-67.596\\
16.517	-68.142\\
16.698	-68.604\\
16.880	-68.959\\
17.061	-69.204\\
17.243	-69.356\\
17.424	-69.444\\
17.606	-69.491\\
17.787	-69.517\\
17.969	-69.530\\
18.150	-69.536\\
};
\addlegendentry{$\log(\fffd/\fffd(0))$}

\addplot [color=black, dotted, line width=1.0pt]
  table[row sep=crcr]{%
0.000	0.000\\
0.182	-5.218\\
0.363	-8.611\\
0.544	-9.856\\
0.726	-11.538\\
0.907	-14.030\\
1.089	-16.230\\
1.270	-17.725\\
1.452	-18.866\\
1.633	-19.767\\
1.815	-20.527\\
1.996	-21.230\\
2.178	-21.914\\
2.359	-22.592\\
2.541	-23.269\\
2.723	-23.944\\
2.904	-24.620\\
3.086	-25.295\\
3.267	-25.970\\
3.449	-26.645\\
3.630	-27.320\\
3.812	-27.996\\
3.993	-28.671\\
4.175	-29.346\\
4.356	-30.021\\
4.538	-30.696\\
4.719	-31.371\\
4.901	-32.046\\
5.082	-32.721\\
5.263	-33.396\\
5.445	-34.072\\
5.626	-34.747\\
5.808	-35.422\\
5.989	-36.097\\
6.171	-36.772\\
6.352	-37.447\\
6.534	-38.122\\
6.715	-38.797\\
6.897	-39.473\\
7.078	-40.148\\
7.260	-40.823\\
7.441	-41.498\\
7.623	-42.173\\
7.804	-42.848\\
7.986	-43.523\\
8.167	-44.198\\
8.349	-44.873\\
8.530	-45.549\\
8.712	-46.224\\
8.893	-46.899\\
9.075	-47.574\\
9.256	-48.249\\
9.438	-48.924\\
9.620	-49.599\\
9.801	-50.274\\
9.983	-50.949\\
10.164	-51.625\\
10.346	-52.300\\
10.527	-52.975\\
10.709	-53.650\\
10.890	-54.325\\
11.072	-55.000\\
11.253	-55.675\\
11.435	-56.350\\
11.616	-57.025\\
11.798	-57.701\\
11.979	-58.376\\
12.161	-59.051\\
12.342	-59.726\\
12.524	-60.401\\
12.705	-61.076\\
12.887	-61.751\\
13.068	-62.426\\
13.250	-63.101\\
13.431	-63.777\\
13.613	-64.452\\
13.794	-65.127\\
13.976	-65.802\\
14.157	-66.477\\
14.339	-67.152\\
14.520	-67.827\\
14.702	-68.502\\
14.883	-69.177\\
15.065	-69.852\\
15.246	-70.527\\
15.428	-71.201\\
15.609	-71.874\\
15.791	-72.545\\
15.972	-73.213\\
16.154	-73.874\\
16.335	-74.522\\
16.517	-75.145\\
16.698	-75.726\\
16.880	-76.238\\
17.061	-76.655\\
17.243	-76.960\\
17.424	-77.161\\
17.606	-77.281\\
17.787	-77.348\\
17.969	-77.384\\
18.150	-77.403\\
};
\addlegendentry{$\log(\gggd/\gggd(0))$}

\end{axis}
\end{tikzpicture}%
\caption{Normalized linearized entropy $\fffd$ (plain) and Fisher information $\gggd$ (dotted) in logscale}

%% file: entropyfisherinhomtaux_cas-rings.tex
\definecolor{mycolor1}{rgb}{0.00000,0.44700,0.74100}%
\definecolor{mycolor2}{rgb}{0.85000,0.32500,0.09800}%
\begin{tikzpicture}

\begin{axis}[%
width=0.411\fwidth,
height=\fheight,
at={(0\fwidth,0\fheight)},
scale only axis,
xmin=0,
xmax=20,
xlabel style={font=\color{white!15!black}},
xlabel={Time $t$},
separate axis lines,
every outer y axis line/.append style={mycolor1},
every y tick label/.append style={font=\color{mycolor1}},
every y tick/.append style={mycolor1},
ymin=0,
ymax=1,
axis background/.style={fill=white},
legend style={legend cell align=left, align=left, draw=white!15!black},
legend style={font=\tiny}
]
\addplot [color=mycolor1, line width=1.0pt]
  table[row sep=crcr]{%
0	0.693624846142672\\
0.04	0.373788201361643\\
0.0800000000000001	0.20717442058422\\
0.12	0.0986897018495456\\
0.16	0.0451272664451047\\
0.2	0.0230001009200171\\
0.24	0.0139748019533147\\
0.28	0.00983152698572283\\
0.32	0.00756171364818913\\
0.36	0.00608866156447981\\
0.4	0.00499241132192569\\
0.44	0.00409312398498056\\
0.48	0.00331109740563036\\
0.519999999999998	0.00261469901028606\\
0.559999999999994	0.00199897321922425\\
0.599999999999989	0.00147152563472809\\
0.639999999999985	0.00104063626648108\\
0.679999999999981	0.000707662621807431\\
0.719999999999976	0.000464881904352651\\
0.759999999999972	0.000297576415060016\\
0.799999999999967	0.000188021624618269\\
0.839999999999963	0.000119227026573908\\
0.879999999999958	7.72467501990315e-05\\
0.919999999999954	5.19152214172387e-05\\
0.95999999999995	3.64904208939126e-05\\
0.999999999999945	2.68212714219873e-05\\
1.03999999999994	2.0487947967779e-05\\
1.07999999999994	1.61265515842908e-05\\
1.11999999999993	1.29766980991873e-05\\
1.15999999999993	1.06092544852155e-05\\
1.19999999999992	8.77399631510678e-06\\
1.23999999999992	7.31806366391494e-06\\
1.27999999999991	6.14310141900075e-06\\
1.31999999999991	5.18260293624019e-06\\
1.35999999999991	4.38962064861712e-06\\
1.3999999999999	3.72982890441867e-06\\
1.4399999999999	3.1774198419295e-06\\
1.47999999999989	2.71255628115034e-06\\
1.51999999999989	2.31971878388411e-06\\
1.55999999999988	1.98659000749115e-06\\
1.59999999999988	1.70327622776907e-06\\
1.63999999999987	1.46174901060576e-06\\
1.67999999999987	1.25543574183114e-06\\
1.71999999999987	1.07891384956542e-06\\
1.75999999999986	9.27679043342672e-07\\
1.79999999999986	7.97967426188034e-07\\
1.83999999999985	6.86617422616776e-07\\
1.87999999999985	5.90961493723666e-07\\
1.91999999999984	5.08740364000182e-07\\
1.95999999999984	4.38034417917731e-07\\
1.99999999999984	3.77208309620311e-07\\
2.03999999999985	3.2486583529632e-07\\
2.07999999999986	2.79812854701817e-07\\
2.11999999999988	2.41026590388038e-07\\
2.15999999999989	2.07630032961504e-07\\
2.1999999999999	1.78870476185823e-07\\
2.23999999999992	1.54101424823696e-07\\
2.27999999999993	1.32767281381245e-07\\
2.31999999999994	1.1439034046964e-07\\
2.35999999999996	9.85597123789013e-08\\
2.39999999999997	8.49218686664989e-08\\
2.43999999999998	7.31725578614347e-08\\
2.48	6.30498829060171e-08\\
2.52000000000001	5.4328366664367e-08\\
2.56000000000002	4.68138598474102e-08\\
2.60000000000004	4.03391685785191e-08\\
2.64000000000005	3.47602976774568e-08\\
2.68000000000006	2.99532214062162e-08\\
2.72000000000008	2.58111065253216e-08\\
2.76000000000009	2.22419235320135e-08\\
2.8000000000001	1.916639126439e-08\\
2.84000000000012	1.65162079514465e-08\\
2.88000000000013	1.42325285027519e-08\\
2.92000000000014	1.22646535533752e-08\\
2.96000000000016	1.05689006652509e-08\\
3.00000000000017	9.10763226400709e-09\\
3.04000000000018	7.84841846730308e-09\\
3.0800000000002	6.76331602645549e-09\\
3.12000000000021	5.82824723285405e-09\\
3.16000000000022	5.02246489796312e-09\\
3.20000000000024	4.32809145444584e-09\\
3.24000000000025	3.72972189189851e-09\\
3.28000000000026	3.21408167281689e-09\\
3.32000000000028	2.7697320059761e-09\\
3.36000000000029	2.38681591379396e-09\\
3.4000000000003	2.05683944176217e-09\\
3.44000000000032	1.77248314247872e-09\\
3.48000000000033	1.52743964203354e-09\\
3.52000000000034	1.31627367778973e-09\\
3.56000000000036	1.13430149709642e-09\\
3.60000000000037	9.77486937447631e-10\\
3.64000000000038	8.42351879755858e-10\\
3.6800000000004	7.25899086071423e-10\\
3.72000000000041	6.25545708412967e-10\\
3.76000000000042	5.39065992542928e-10\\
3.80000000000044	4.64541904826396e-10\\
3.84000000000045	4.00320586314668e-10\\
3.88000000000046	3.44977689823013e-10\\
3.92000000000048	2.97285786406018e-10\\
3.96000000000049	2.5618714018391e-10\\
18.1500000000031	0\\
};
\label{plotyyref:leg7}

\addlegendentry{$\fffd$ (LH scale)}

\end{axis}

\begin{axis}[%
width=0.411\fwidth,
height=\fheight,
at={(0\fwidth,0\fheight)},
scale only axis,
xmin=0,
xmax=20,
every outer y axis line/.append style={mycolor2},
every y tick label/.append style={font=\color{mycolor2}},
every y tick/.append style={mycolor2},
ymin=3,
ymax=4,
ytick={  3, 3.5,   4},
axis x line*=bottom,
axis y line*=right,
legend style={legend cell align=left, align=left, draw=white!15!black},
legend style={font=\tiny}
]
\addlegendimage{/pgfplots/refstyle=plotyyref:leg7}
\addlegendentry{$\fffd$ (LH scale)}
\addplot [color=mycolor2, dashed, line width=1.0pt]
  table[row sep=crcr]{%
3.03000000000018	3.7200248385179\\
3.06000000000019	3.72000550825155\\
3.0900000000002	3.71998743585986\\
3.12000000000021	3.71997052782113\\
3.15000000000022	3.71995469824251\\
3.18000000000023	3.71993986817997\\
3.21000000000024	3.71992596502533\\
3.24000000000025	3.71991292195297\\
3.27000000000026	3.71990067741954\\
3.30000000000027	3.71988917471128\\
3.33000000000028	3.71987836153362\\
3.36000000000029	3.71986818963886\\
3.3900000000003	3.71985861448798\\
3.42000000000031	3.71984959494309\\
3.45000000000032	3.71984109298743\\
3.48000000000033	3.71983307347036\\
3.51000000000034	3.71982550387467\\
3.54000000000035	3.71981835410431\\
3.57000000000036	3.7198115962904\\
3.60000000000037	3.71980520461395\\
3.63000000000038	3.71979915514362\\
3.66000000000039	3.71979342568721\\
3.6900000000004	3.71978799565569\\
3.72000000000041	3.71978284593841\\
3.75000000000042	3.71977795878881\\
3.78000000000043	3.71977331771939\\
3.81000000000044	3.71976890740542\\
3.84000000000045	3.71976471359633\\
3.87000000000046	3.71976072303443\\
3.90000000000047	3.71975692337999\\
3.93000000000048	3.7197533031425\\
3.96000000000049	3.71974985161726\\
3.9900000000005	3.71974655882707\\
4.02000000000049	3.71974341546857\\
4.05000000000047	3.71974041286256\\
4.08000000000046	3.71973754290845\\
4.11000000000044	3.71973479804222\\
4.14000000000043	3.71973217119753\\
4.17000000000041	3.71972965576987\\
4.20000000000039	3.71972724558352\\
4.23000000000038	3.71972493486092\\
4.26000000000036	3.7197227181945\\
4.29000000000034	3.71972059052052\\
4.32000000000033	3.71971854709498\\
4.35000000000031	3.7197165834713\\
4.38000000000029	3.71971469547966\\
4.41000000000028	3.71971287920787\\
4.44000000000026	3.71971113098374\\
4.47000000000024	3.71970944735859\\
4.50000000000023	3.71970782509217\\
4.53000000000021	3.71970626113851\\
4.56000000000019	3.71970475263292\\
4.59000000000018	3.71970329687986\\
4.62000000000016	3.71970189134169\\
4.65000000000014	3.71970053362832\\
4.68000000000013	3.71969922148746\\
4.71000000000011	3.71969795279567\\
4.74000000000009	3.71969672555\\
4.77000000000008	3.71969553786023\\
4.80000000000006	3.71969438794165\\
4.83000000000004	3.71969327410834\\
4.86000000000003	3.71969219476696\\
4.89000000000001	3.71969114841088\\
4.91999999999999	3.71969013361482\\
4.94999999999998	3.71968914902977\\
4.97999999999996	3.7196881933783\\
5.00999999999994	3.71968726545021\\
5.03999999999993	3.71968636409839\\
5.06999999999991	3.71968548823507\\
5.09999999999989	3.7196846368282\\
5.12999999999988	3.71968380889819\\
5.15999999999986	3.71968300351477\\
5.18999999999984	3.7196822197941\\
5.21999999999983	3.7196814568961\\
5.24999999999981	3.71968071402184\\
5.27999999999979	3.71967999041124\\
5.30999999999978	3.71967928534084\\
5.33999999999976	3.71967859812171\\
5.36999999999974	3.71967792809749\\
5.39999999999973	3.71967727464264\\
5.42999999999971	3.71967663716064\\
5.45999999999969	3.71967601508248\\
5.48999999999968	3.71967540786509\\
5.51999999999966	3.71967481498996\\
5.54999999999964	3.71967423596184\\
5.57999999999963	3.71967367030747\\
5.60999999999961	3.71967311757441\\
5.63999999999959	3.71967257732996\\
5.66999999999958	3.71967204916014\\
5.69999999999956	3.71967153266869\\
5.72999999999954	3.71967102747619\\
5.75999999999953	3.71967053321918\\
5.78999999999951	3.71967004954937\\
5.81999999999949	3.71966957613284\\
5.84999999999948	3.71966911264939\\
5.87999999999946	3.71966865879178\\
5.90999999999944	3.71966821426517\\
5.93999999999943	3.71966777878648\\
5.96999999999941	3.7196673520838\\
5.99999999999939	3.71966693389591\\
6.02999999999938	3.71966652397172\\
6.05999999999936	3.71966612206982\\
6.08999999999934	3.71966572795802\\
6.11999999999933	3.71966534141291\\
6.14999999999931	3.71966496221949\\
6.17999999999929	3.71966459017074\\
6.20999999999928	3.71966422506728\\
6.23999999999926	3.71966386671704\\
6.26999999999924	3.71966351493491\\
6.29999999999923	3.71966316954243\\
6.32999999999921	3.71966283036751\\
6.35999999999919	3.71966249724416\\
6.38999999999918	3.71966217001221\\
6.41999999999916	3.71966184851706\\
6.44999999999914	3.71966153260944\\
6.47999999999913	3.71966122214519\\
6.50999999999911	3.71966091698506\\
6.53999999999909	3.71966061699449\\
6.56999999999908	3.71966032204339\\
6.59999999999906	3.71966003200601\\
6.62999999999905	3.71965974676072\\
6.65999999999903	3.71965946618987\\
6.68999999999901	3.71965919017961\\
6.719999999999	3.71965891861974\\
6.74999999999898	3.7196586514036\\
6.77999999999896	3.71965838842787\\
6.80999999999895	3.7196581295925\\
6.83999999999893	3.71965787480054\\
6.86999999999891	3.71965762395805\\
6.8999999999989	3.71965737697395\\
6.92999999999888	3.71965713375995\\
6.95999999999886	3.71965689423044\\
6.98999999999885	3.71965665830235\\
7.01999999999883	3.71965642589512\\
7.04999999999881	3.71965619693054\\
7.0799999999988	3.71965597133271\\
7.10999999999878	3.71965574902794\\
7.13999999999876	3.71965552994468\\
7.16999999999875	3.71965531401342\\
7.19999999999873	3.71965510116666\\
7.22999999999871	3.71965489133878\\
7.2599999999987	3.71965468446602\\
7.28999999999868	3.71965448048642\\
7.31999999999866	3.71965427933971\\
7.34999999999865	3.71965408096729\\
7.37999999999863	3.71965388531217\\
7.40999999999861	3.71965369231892\\
7.4399999999986	3.71965350193358\\
7.46999999999858	3.71965331410366\\
7.49999999999856	3.71965312877805\\
7.52999999999855	3.71965294590701\\
7.55999999999853	3.71965276544211\\
7.58999999999851	3.71965258733616\\
7.6199999999985	3.71965241154322\\
7.64999999999848	3.71965223801853\\
7.67999999999846	3.71965206671847\\
7.70999999999845	3.71965189760054\\
7.73999999999843	3.71965173062331\\
7.76999999999841	3.71965156574639\\
7.7999999999984	3.7196514029304\\
7.82999999999838	3.71965124213694\\
7.85999999999836	3.71965108332857\\
7.88999999999835	3.71965092646875\\
7.91999999999833	3.71965077152183\\
7.94999999999831	3.71965061845304\\
7.9799999999983	3.71965046722843\\
8.0099999999983	3.71965031781488\\
8.03999999999833	3.71965017018002\\
8.06999999999837	3.7196500242923\\
8.09999999999841	3.7196498801209\\
8.12999999999844	3.71964973763571\\
8.15999999999848	3.71964959680732\\
8.18999999999852	3.71964945760699\\
8.21999999999855	3.71964932000666\\
8.24999999999859	3.7196491839789\\
8.27999999999863	3.7196490494969\\
8.30999999999866	3.71964891653447\\
8.3399999999987	3.71964878506599\\
8.36999999999874	3.71964865506644\\
8.39999999999877	3.71964852651132\\
8.42999999999881	3.71964839937669\\
8.45999999999885	3.71964827363915\\
8.48999999999888	3.71964814927578\\
8.51999999999892	3.71964802626419\\
8.54999999999896	3.71964790458245\\
8.57999999999899	3.71964778420911\\
8.60999999999903	3.71964766512319\\
8.63999999999907	3.71964754730414\\
8.6699999999991	3.71964743073185\\
8.69999999999914	3.71964731538663\\
8.72999999999918	3.71964720124922\\
8.75999999999921	3.71964708830073\\
8.78999999999925	3.7196469765227\\
8.81999999999929	3.71964686589702\\
8.84999999999932	3.71964675640595\\
8.87999999999936	3.71964664803215\\
8.9099999999994	3.71964654075858\\
8.93999999999943	3.71964643456858\\
8.96999999999947	3.71964632944582\\
8.99999999999951	3.71964622537429\\
9.02999999999954	3.71964612233829\\
9.05999999999958	3.71964602032245\\
9.08999999999962	3.71964591931169\\
9.11999999999965	3.71964581929124\\
9.14999999999969	3.71964572024659\\
9.17999999999973	3.71964562216353\\
9.20999999999976	3.71964552502814\\
9.2399999999998	3.71964542882675\\
9.26999999999984	3.71964533354594\\
9.29999999999987	3.71964523917256\\
9.32999999999991	3.71964514569371\\
9.35999999999995	3.71964505309675\\
9.38999999999998	3.71964496136923\\
9.42000000000002	3.71964487049898\\
9.45000000000006	3.71964478047403\\
9.48000000000009	3.71964469128264\\
9.51000000000013	3.71964460291329\\
9.54000000000017	3.71964451535466\\
9.5700000000002	3.71964442859566\\
9.60000000000024	3.71964434262537\\
9.63000000000028	3.71964425743308\\
9.66000000000031	3.71964417300829\\
9.69000000000035	3.71964408934067\\
9.72000000000039	3.71964400642007\\
9.75000000000042	3.71964392423654\\
9.78000000000046	3.71964384278029\\
9.8100000000005	3.71964376204171\\
9.84000000000053	3.71964368201135\\
9.87000000000057	3.71964360267994\\
9.90000000000061	3.71964352403836\\
9.93000000000064	3.71964344607764\\
9.96000000000068	3.71964336878898\\
9.99000000000072	3.71964329216374\\
10.0200000000008	3.71964321619339\\
10.0500000000008	3.71964314086958\\
10.0800000000008	3.71964306618408\\
10.1100000000009	3.71964299212883\\
10.1400000000009	3.71964291869585\\
10.1700000000009	3.71964284587736\\
10.200000000001	3.71964277366565\\
10.230000000001	3.71964270205317\\
10.260000000001	3.71964263103249\\
10.2900000000011	3.7196425605963\\
10.3200000000011	3.71964249073741\\
10.3500000000012	3.71964242144874\\
10.3800000000012	3.71964235272333\\
10.4100000000012	3.71964228455434\\
10.4400000000013	3.71964221693502\\
10.4700000000013	3.71964214985875\\
10.5000000000013	3.71964208331899\\
10.5300000000014	3.71964201730933\\
10.5600000000014	3.71964195182344\\
10.5900000000015	3.71964188685509\\
10.6200000000015	3.71964182239817\\
10.6500000000015	3.71964175844663\\
10.6800000000016	3.71964169499453\\
10.7100000000016	3.71964163203603\\
10.7400000000016	3.71964156956536\\
10.7700000000017	3.71964150757684\\
10.8000000000017	3.71964144606489\\
10.8300000000017	3.71964138502399\\
10.8600000000018	3.71964132444871\\
10.8900000000018	3.71964126433371\\
10.9200000000019	3.7196412046737\\
10.9500000000019	3.71964114546349\\
10.9800000000019	3.71964108669795\\
11.010000000002	3.71964102837202\\
11.040000000002	3.71964097048072\\
11.070000000002	3.71964091301912\\
11.1000000000021	3.71964085598235\\
11.1300000000021	3.71964079936564\\
11.1600000000021	3.71964074316423\\
11.1900000000022	3.71964068737344\\
11.2200000000022	3.71964063198866\\
11.2500000000023	3.7196405770053\\
11.2800000000023	3.71964052241885\\
11.3100000000023	3.71964046822481\\
11.3400000000024	3.71964041441877\\
11.3700000000024	3.71964036099631\\
11.4000000000024	3.7196403079531\\
11.4300000000025	3.71964025528479\\
11.4600000000025	3.71964020298709\\
11.4900000000026	3.71964015105573\\
11.5200000000026	3.71964009948646\\
11.5500000000026	3.71964004827504\\
11.5800000000027	3.71963999741723\\
11.6100000000027	3.71963994690882\\
11.6400000000027	3.71963989674558\\
11.6700000000028	3.71963984692325\\
11.7000000000028	3.71963979743759\\
11.7300000000028	3.71963974828431\\
11.7600000000029	3.71963969945907\\
11.7900000000029	3.71963965095752\\
11.820000000003	3.71963960277521\\
11.850000000003	3.71963955490765\\
11.880000000003	3.71963950735024\\
11.9100000000031	3.71963946009828\\
11.9400000000031	3.71963941314696\\
11.9700000000031	3.71963936649133\\
12.0000000000032	3.71963932012626\\
12.0300000000032	3.71963927404645\\
12.0600000000032	3.71963922824638\\
12.0900000000033	3.71963918272029\\
12.1200000000033	3.71963913746214\\
12.1500000000034	3.71963909246559\\
12.1800000000034	3.71963904772392\\
12.2100000000034	3.71963900323004\\
12.2400000000035	3.71963895897639\\
12.2700000000035	3.71963891495491\\
12.3000000000035	3.71963887115696\\
12.3300000000036	3.71963882757328\\
12.3600000000036	3.71963878419387\\
12.3900000000037	3.71963874100793\\
12.4200000000037	3.71963869800378\\
12.4500000000037	3.71963865516868\\
12.4800000000038	3.71963861248882\\
12.5100000000038	3.71963856994907\\
12.5400000000038	3.71963852753292\\
12.5700000000039	3.71963848522227\\
12.6000000000039	3.71963844299725\\
12.6300000000039	3.71963840083601\\
12.660000000004	3.7196383587145\\
12.690000000004	3.71963831660621\\
12.7200000000041	3.71963827448189\\
12.7500000000041	3.71963823230919\\
12.7800000000041	3.71963819005237\\
12.8100000000042	3.71963814767184\\
12.8400000000042	3.71963810512376\\
12.8700000000042	3.71963806235952\\
12.9000000000043	3.71963801932521\\
12.9300000000043	3.71963797596098\\
12.9600000000043	3.71963793220037\\
12.9900000000044	3.71963788796957\\
13.0200000000044	3.71963784318651\\
13.0500000000045	3.71963779775997\\
13.0800000000045	3.71963775158849\\
13.1100000000045	3.71963770455918\\
13.1400000000046	3.71963765654645\\
13.1700000000046	3.71963760741048\\
13.2000000000046	3.71963755699564\\
13.2300000000047	3.71963750512862\\
13.2600000000047	3.71963745161643\\
13.2900000000048	3.71963739624409\\
13.3200000000048	3.71963733877216\\
13.3500000000048	3.71963727893384\\
13.3800000000049	3.71963721643189\\
13.4100000000049	3.71963715093508\\
13.4400000000049	3.71963708207428\\
13.470000000005	3.71963700943812\\
13.500000000005	3.71963693256809\\
13.530000000005	3.71963685095312\\
13.5600000000051	3.71963676402353\\
13.5900000000051	3.71963667114429\\
13.6200000000052	3.71963657160744\\
13.6500000000052	3.71963646462377\\
13.6800000000052	3.71963634931336\\
13.7100000000053	3.71963622469521\\
13.7400000000053	3.71963608967557\\
13.7700000000053	3.7196359430349\\
13.8000000000054	3.71963578341345\\
13.8300000000054	3.71963560929505\\
13.8600000000054	3.71963541898911\\
13.8900000000055	3.71963521061049\\
13.9200000000055	3.71963498205714\\
13.9500000000056	3.719634730985\\
13.9800000000056	3.71963445478022\\
14.0100000000056	3.71963415052797\\
14.0400000000057	3.71963381497779\\
14.0700000000057	3.71963344450494\\
14.1000000000057	3.71963303506716\\
14.1300000000058	3.71963258215664\\
14.1600000000058	3.71963208074631\\
14.1900000000059	3.71963152522995\\
14.2200000000059	3.71963090935551\\
14.2500000000059	3.7196302261506\\
14.280000000006	3.71962946783945\\
14.310000000006	3.71962862575033\\
14.340000000006	3.71962769021228\\
14.3700000000061	3.71962665043988\\
14.4000000000061	3.71962549440482\\
14.4300000000061	3.71962420869267\\
14.4600000000062	3.71962277834309\\
14.4900000000062	3.71962118667167\\
14.5200000000063	3.71961941507129\\
14.5500000000063	3.71961744279054\\
14.5800000000063	3.71961524668672\\
14.6100000000064	3.71961280095035\\
14.6400000000064	3.71961007679806\\
14.6700000000064	3.71960704213011\\
14.7000000000065	3.71960366114862\\
14.7300000000065	3.71959989393193\\
14.7600000000065	3.7195956959601\\
14.7900000000066	3.71959101758602\\
14.8200000000066	3.71958580344586\\
14.8500000000067	3.71957999180208\\
14.8800000000067	3.71957351381125\\
14.9100000000067	3.71956629270819\\
14.9400000000068	3.71955824289708\\
14.9700000000068	3.71954926893892\\
15.0000000000068	3.71953926442387\\
15.0300000000069	3.71952811071545\\
15.0600000000069	3.71951567555246\\
15.090000000007	3.71950181149274\\
15.120000000007	3.7194863541815\\
15.150000000007	3.71946912042477\\
15.1800000000071	3.71944990604695\\
15.2100000000071	3.71942848350896\\
15.2400000000071	3.71940459926152\\
15.2700000000072	3.71937797080524\\
15.3000000000072	3.71934828342696\\
15.3300000000072	3.71931518657856\\
15.3600000000073	3.71927828986198\\
15.3900000000073	3.71923715858055\\
15.4200000000074	3.71919130881424\\
15.4500000000074	3.71914020197268\\
15.4800000000074	3.71908323877695\\
15.5100000000075	3.71901975261821\\
15.5400000000075	3.71894900223829\\
15.5700000000075	3.71887016367528\\
15.6000000000076	3.71878232141529\\
15.6300000000076	3.71868445869117\\
15.6600000000076	3.71857544686899\\
15.6900000000077	3.71845403386572\\
15.7200000000077	3.7183188315455\\
15.7500000000078	3.71816830204921\\
15.7800000000078	3.7180007430222\\
15.8100000000078	3.71781427172032\\
15.8400000000079	3.71760680799408\\
15.8700000000079	3.71737605617748\\
15.9000000000079	3.7171194859422\\
15.930000000008	3.71683431222072\\
15.960000000008	3.71651747435588\\
15.9900000000081	3.71616561469922\\
16.020000000008	3.71577505695948\\
16.0500000000079	3.71534178469467\\
16.0800000000079	3.71486142044926\\
16.1100000000078	3.7143292061597\\
16.1400000000077	3.71373998558744\\
16.1700000000077	3.71308818968529\\
16.2000000000076	3.71236782595618\\
16.2300000000075	3.71157247301602\\
16.2600000000075	3.71069528171548\\
16.2900000000074	3.70972898429517\\
16.3200000000073	3.7086659131295\\
16.3500000000072	3.70749803063704\\
16.3800000000072	3.70621697187744\\
16.4100000000071	3.70481410119371\\
16.440000000007	3.70328058397314\\
16.470000000007	3.70160747417137\\
16.5000000000069	3.6997858176626\\
16.5300000000068	3.69780677074673\\
16.5600000000068	3.69566173227959\\
16.5900000000067	3.69334248693196\\
16.6200000000066	3.69084135608697\\
16.6500000000065	3.68815135193018\\
16.6800000000065	3.68526632946594\\
16.7100000000064	3.68218113060573\\
16.7400000000063	3.67889171421164\\
16.7700000000063	3.67539526611278\\
16.8000000000062	3.67169028368326\\
16.8300000000061	3.66777663056949\\
16.8600000000061	3.66365555852838\\
16.890000000006	3.65932969498241\\
16.9200000000059	3.6548029966766\\
16.9500000000058	3.65008067157808\\
16.9800000000058	3.64516907273762\\
17.0100000000057	3.64007556910297\\
17.0400000000056	3.63480839914187\\
17.0700000000056	3.6293765135556\\
17.1000000000055	3.62378941334566\\
17.1300000000054	3.61805698908801\\
17.1600000000054	3.61218936655033\\
17.1900000000053	3.60619676285986\\
17.2200000000052	3.6000893563926\\
17.2500000000051	3.59387717250506\\
17.2800000000051	3.58756998624291\\
17.310000000005	3.58117724229251\\
17.3400000000049	3.57470799172541\\
17.3700000000049	3.5681708445349\\
17.4000000000048	3.56157393657535\\
17.4300000000047	3.55492490927534\\
17.4600000000047	3.54823090038215\\
17.4900000000046	3.54149854398387\\
17.5200000000045	3.53473397812154\\
17.5500000000044	3.52794285842292\\
17.5800000000044	3.52113037634332\\
17.6100000000043	3.51430128076982\\
17.6400000000042	3.50745990192052\\
17.6700000000042	3.50061017664151\\
17.7000000000041	3.4937556743645\\
17.730000000004	3.4868996231329\\
17.760000000004	3.48004493523285\\
17.7900000000039	3.473194232077\\
17.8200000000038	3.46634986808278\\
17.8500000000037	3.45951395336592\\
17.8800000000037	3.45268837513385\\
17.9100000000036	3.44587481771556\\
17.9400000000035	3.43907478120491\\
17.9700000000035	3.43228959872599\\
18.0000000000034	3.42552045235237\\
18.0300000000033	3.4187683877296\\
18.0600000000033	3.41203432746193\\
18.0900000000032	3.40531908333203\\
18.1200000000031	3.39862336742698\\
18.1500000000031	3.39194780224565\\
};
\addlegendentry{$-\frac{\log(\fffd(t)/\fffd(3))}{t-3}$ (RH scale)}

\end{axis}

\begin{axis}[%
width=0.411\fwidth,
height=\fheight,
at={(0.54\fwidth,0\fheight)},
scale only axis,
xmin=0,
xmax=20,
xlabel style={font=\color{white!15!black}},
xlabel={Time $t$},
separate axis lines,
every outer y axis line/.append style={mycolor1},
every y tick label/.append style={font=\color{mycolor1}},
every y tick/.append style={mycolor1},
ymin=0,
ymax=2000,
axis background/.style={fill=white},
legend style={legend cell align=left, align=left, draw=white!15!black},
legend style={font=\tiny}
]
\addplot [color=mycolor1, line width=1.0pt]
  table[row sep=crcr]{%
0	1832.25195970768\\
0.04	457.230869364375\\
0.0800000000000001	186.843375591967\\
0.12	64.0838490485069\\
0.16	19.1088372367367\\
0.2	5.80022354102056\\
0.24	2.11209835594141\\
0.28	0.976295708517352\\
0.32	0.543140343006395\\
0.36	0.343749043945174\\
0.4	0.241037379780462\\
0.44	0.182145807282033\\
0.48	0.142958823846039\\
0.519999999999998	0.112477648547238\\
0.559999999999994	0.0864302085104626\\
0.599999999999989	0.0638474267701678\\
0.639999999999985	0.0449756513227927\\
0.679999999999981	0.0301205821376652\\
0.719999999999976	0.0191908088162747\\
0.759999999999972	0.0116781790898374\\
0.799999999999967	0.0068376444819676\\
0.839999999999963	0.00389525199522849\\
0.879999999999958	0.00219176303743151\\
0.919999999999954	0.00124014946911917\\
0.95999999999995	0.000718544024230524\\
0.999999999999945	0.000432416892805694\\
1.03999999999994	0.000272142882803537\\
1.07999999999994	0.00017896282247743\\
1.11999999999993	0.00012225794876875\\
1.15999999999993	8.61418363656255e-05\\
1.19999999999992	6.22104253884462e-05\\
1.23999999999992	4.58426731457662e-05\\
1.27999999999991	3.43693832090829e-05\\
1.31999999999991	2.61695649042958e-05\\
1.35999999999991	2.02145383387908e-05\\
1.3999999999999	1.58281931962442e-05\\
1.4399999999999	1.25542189159084e-05\\
1.47999999999989	1.00785848810496e-05\\
1.51999999999989	8.18206643770607e-06\\
1.55999999999988	6.70998443095266e-06\\
1.59999999999988	5.55228864326458e-06\\
1.63999999999987	4.63012095146744e-06\\
1.67999999999987	3.88658312028437e-06\\
1.71999999999987	3.28031307609345e-06\\
1.75999999999986	2.78097996108782e-06\\
1.79999999999986	2.36611159097447e-06\\
1.83999999999985	2.01885668686156e-06\\
1.87999999999985	1.72640625922956e-06\\
1.91999999999984	1.47888025886407e-06\\
1.95999999999984	1.26854204547393e-06\\
1.99999999999984	1.08924305396079e-06\\
2.03999999999985	9.36028506095351e-07\\
2.07999999999986	8.04855448960418e-07\\
2.11999999999988	6.92389033371216e-07\\
2.15999999999989	5.95853350776008e-07\\
2.1999999999999	5.1292047692767e-07\\
2.23999999999992	4.41626480259041e-07\\
2.27999999999993	3.80306678919316e-07\\
2.31999999999994	3.27544842255186e-07\\
2.35999999999996	2.82132671197594e-07\\
2.39999999999997	2.43037001027114e-07\\
2.43999999999998	2.09372920092008e-07\\
2.48	1.80381506996e-07\\
2.52000000000001	1.55411236561867e-07\\
2.56000000000002	1.33902345199505e-07\\
2.60000000000004	1.15373614850887e-07\\
2.64000000000005	9.94111550221747e-08\\
2.68000000000006	8.56588500712718e-08\\
2.72000000000008	7.38102041301274e-08\\
2.76000000000009	6.36013655789035e-08\\
2.8000000000001	5.48051513738933e-08\\
2.84000000000012	4.72259218143165e-08\\
2.88000000000013	4.06951806054303e-08\\
2.92000000000014	3.50677947942707e-08\\
2.96000000000016	3.02187453557276e-08\\
3.00000000000017	2.60403326420944e-08\\
3.04000000000018	2.24397721353141e-08\\
3.0800000000002	1.9337125378467e-08\\
3.12000000000021	1.66635189355056e-08\\
3.16000000000022	1.43596109896196e-08\\
3.20000000000024	1.23742709428405e-08\\
3.24000000000025	1.06634422854147e-08\\
3.28000000000026	9.18916319595224e-09\\
3.32000000000028	7.91872292185613e-09\\
3.36000000000029	6.82393506496761e-09\\
3.4000000000003	5.8805115357823e-09\\
3.44000000000032	5.06752320493331e-09\\
3.48000000000033	4.36693522694218e-09\\
3.52000000000034	3.76320668431639e-09\\
3.56000000000036	3.24294563889575e-09\\
3.60000000000037	2.79461191514693e-09\\
3.64000000000038	2.4082610052618e-09\\
3.6800000000004	2.07532340272232e-09\\
3.72000000000041	1.78841446026193e-09\\
3.76000000000042	1.54117054773016e-09\\
3.80000000000044	1.32810787057324e-09\\
3.84000000000045	1.1445008136376e-09\\
3.88000000000046	9.86277109095812e-10\\
3.92000000000048	8.4992750120716e-10\\
3.96000000000049	7.32427902722958e-10\\
18.1500000000031	0\\
};
\label{plotyyref:leg8}

\addlegendentry{$\gggd$ (LH scale)}

\end{axis}

\begin{axis}[%
width=0.411\fwidth,
height=\fheight,
at={(0.54\fwidth,0\fheight)},
scale only axis,
xmin=0,
xmax=20,
every outer y axis line/.append style={mycolor2},
every y tick label/.append style={font=\color{mycolor2}},
every y tick/.append style={mycolor2},
ymin=3.4,
ymax=3.8,
ytick={3.4, 3.5, 3.6, 3.7, 3.8},
axis x line*=bottom,
axis y line*=right,
legend style={legend cell align=left, align=left, draw=white!15!black},
legend style={font=\tiny}
]
\addlegendimage{/pgfplots/refstyle=plotyyref:leg8}
\addlegendentry{$\gggd$ (LH scale)}
\addplot [color=mycolor2, dashed, line width=1.0pt]
  table[row sep=crcr]{%
3.03000000000018	3.72030355051483\\
3.06000000000019	3.7202686846899\\
3.0900000000002	3.7202363436574\\
3.12000000000021	3.72020630399383\\
3.15000000000022	3.72017836570872\\
3.18000000000023	3.72015234939326\\
3.21000000000024	3.72012809375814\\
3.24000000000025	3.72010545350317\\
3.27000000000026	3.72008429747026\\
3.30000000000027	3.72006450703877\\
3.33000000000028	3.72004597472842\\
3.36000000000029	3.72002860298029\\
3.3900000000003	3.72001230309077\\
3.42000000000031	3.7199969942771\\
3.45000000000032	3.71998260285649\\
3.48000000000033	3.7199690615233\\
3.51000000000034	3.71995630871097\\
3.54000000000035	3.71994428802766\\
3.57000000000036	3.71993294775584\\
3.60000000000037	3.71992224040759\\
3.63000000000038	3.71991212232855\\
3.66000000000039	3.7199025533445\\
3.6900000000004	3.71989349644523\\
3.72000000000041	3.71988491750126\\
3.75000000000042	3.71987678500954\\
3.78000000000043	3.71986906986465\\
3.81000000000044	3.71986174515263\\
3.84000000000045	3.7198547859649\\
3.87000000000046	3.71984816922999\\
3.90000000000047	3.7198418735612\\
3.93000000000048	3.71983587911844\\
3.96000000000049	3.71983016748276\\
3.9900000000005	3.71982472154221\\
4.02000000000049	3.71981952538811\\
4.05000000000047	3.71981456422009\\
4.08000000000046	3.71980982425987\\
4.11000000000044	3.71980529267225\\
4.14000000000043	3.71980095749296\\
4.17000000000041	3.71979680756248\\
4.20000000000039	3.71979283246551\\
4.23000000000038	3.71978902247545\\
4.26000000000036	3.71978536850336\\
4.29000000000034	3.71978186205117\\
4.32000000000033	3.71977849516859\\
4.35000000000031	3.71977526041353\\
4.38000000000029	3.7197721508156\\
4.41000000000028	3.71976915984252\\
4.44000000000026	3.71976628136915\\
4.47000000000024	3.71976350964889\\
4.50000000000023	3.71976083928732\\
4.53000000000021	3.71975826521781\\
4.56000000000019	3.71975578267905\\
4.59000000000018	3.71975338719415\\
4.62000000000016	3.71975107455147\\
4.65000000000014	3.7197488407867\\
4.68000000000013	3.7197466821664\\
4.71000000000011	3.71974459517265\\
4.74000000000009	3.71974257648891\\
4.77000000000008	3.71974062298682\\
4.80000000000006	3.71973873171396\\
4.83000000000004	3.71973689988259\\
4.86000000000003	3.71973512485901\\
4.89000000000001	3.71973340415385\\
4.91999999999999	3.71973173541291\\
4.94999999999998	3.71973011640875\\
4.97999999999996	3.71972854503271\\
5.00999999999994	3.71972701928769\\
5.03999999999993	3.7197255372812\\
5.06999999999991	3.71972409721908\\
5.09999999999989	3.7197226973995\\
5.12999999999988	3.71972133620744\\
5.15999999999986	3.71972001210955\\
5.18999999999984	3.71971872364924\\
5.21999999999983	3.71971746944228\\
5.24999999999981	3.7197162481725\\
5.27999999999979	3.71971505858791\\
5.30999999999978	3.71971389949698\\
5.33999999999976	3.71971276976523\\
5.36999999999974	3.71971166831197\\
5.39999999999973	3.7197105941073\\
5.42999999999971	3.71970954616931\\
5.45999999999969	3.71970852356136\\
5.48999999999968	3.71970752538968\\
5.51999999999966	3.71970655080098\\
5.54999999999964	3.71970559898026\\
5.57999999999963	3.71970466914882\\
5.60999999999961	3.71970376056227\\
5.63999999999959	3.71970287250872\\
5.66999999999958	3.71970200430714\\
5.69999999999956	3.7197011553057\\
5.72999999999954	3.71970032488027\\
5.75999999999953	3.71969951243302\\
5.78999999999951	3.71969871739109\\
5.81999999999949	3.71969793920533\\
5.84999999999948	3.7196971773491\\
5.87999999999946	3.71969643131716\\
5.90999999999944	3.71969570062464\\
5.93999999999943	3.71969498480601\\
5.96999999999941	3.71969428341419\\
5.99999999999939	3.71969359601961\\
6.02999999999938	3.71969292220943\\
6.05999999999936	3.7196922615867\\
6.08999999999934	3.71969161376966\\
6.11999999999933	3.71969097839099\\
6.14999999999931	3.71969035509718\\
6.17999999999929	3.71968974354788\\
6.20999999999928	3.7196891434153\\
6.23999999999926	3.71968855438366\\
6.26999999999924	3.71968797614861\\
6.29999999999923	3.7196874084168\\
6.32999999999921	3.7196868509053\\
6.35999999999919	3.71968630334123\\
6.38999999999918	3.71968576546126\\
6.41999999999916	3.71968523701125\\
6.44999999999914	3.71968471774583\\
6.47999999999913	3.71968420742801\\
6.50999999999911	3.71968370582886\\
6.53999999999909	3.71968321272719\\
6.56999999999908	3.71968272790915\\
6.59999999999906	3.71968225116803\\
6.62999999999905	3.71968178230389\\
6.65999999999903	3.71968132112332\\
6.68999999999901	3.71968086743919\\
6.719999999999	3.71968042107035\\
6.74999999999898	3.71967998184144\\
6.77999999999896	3.71967954958267\\
6.80999999999895	3.71967912412953\\
6.83999999999893	3.71967870532269\\
6.86999999999891	3.71967829300771\\
6.8999999999989	3.71967788703489\\
6.92999999999888	3.71967748725911\\
6.95999999999886	3.71967709353962\\
6.98999999999885	3.7196767057399\\
7.01999999999883	3.7196763237275\\
7.04999999999881	3.71967594737387\\
7.0799999999988	3.71967557655427\\
7.10999999999878	3.71967521114756\\
7.13999999999876	3.71967485103611\\
7.16999999999875	3.7196744961057\\
7.19999999999873	3.71967414624533\\
7.22999999999871	3.71967380134718\\
7.2599999999987	3.71967346130644\\
7.28999999999868	3.71967312602123\\
7.31999999999866	3.71967279539251\\
7.34999999999865	3.71967246932396\\
7.37999999999863	3.7196721477219\\
7.40999999999861	3.7196718304952\\
7.4399999999986	3.71967151755518\\
7.46999999999858	3.71967120881556\\
7.49999999999856	3.71967090419233\\
7.52999999999855	3.71967060360372\\
7.55999999999853	3.71967030697013\\
7.58999999999851	3.719670014214\\
7.6199999999985	3.71966972525982\\
7.64999999999848	3.719669440034\\
7.67999999999846	3.71966915846485\\
7.70999999999845	3.71966888048251\\
7.73999999999843	3.71966860601888\\
7.76999999999841	3.71966833500756\\
7.7999999999984	3.71966806738385\\
7.82999999999838	3.71966780308461\\
7.85999999999836	3.7196675420483\\
7.88999999999835	3.71966728421484\\
7.91999999999833	3.71966702952567\\
7.94999999999831	3.71966677792362\\
7.9799999999983	3.71966652935289\\
8.0099999999983	3.71966628375902\\
8.03999999999833	3.71966604108884\\
8.06999999999837	3.71966580129049\\
8.09999999999841	3.71966556431327\\
8.12999999999844	3.7196653301077\\
8.15999999999848	3.71966509862544\\
8.18999999999852	3.71966486981927\\
8.21999999999855	3.71966464364304\\
8.24999999999859	3.71966442005167\\
8.27999999999863	3.71966419900111\\
8.30999999999866	3.71966398044829\\
8.3399999999987	3.71966376435111\\
8.36999999999874	3.71966355066841\\
8.39999999999877	3.71966333935997\\
8.42999999999881	3.71966313038642\\
8.45999999999885	3.71966292370927\\
8.48999999999888	3.71966271929089\\
8.51999999999892	3.71966251709445\\
8.54999999999896	3.71966231708391\\
8.57999999999899	3.71966211922401\\
8.60999999999903	3.71966192348027\\
8.63999999999907	3.7196617298189\\
8.6699999999991	3.71966153820686\\
8.69999999999914	3.71966134861178\\
8.72999999999918	3.71966116100199\\
8.75999999999921	3.71966097534647\\
8.78999999999925	3.71966079161483\\
8.81999999999929	3.71966060977734\\
8.84999999999932	3.71966042980485\\
8.87999999999936	3.71966025166881\\
8.9099999999994	3.71966007534125\\
8.93999999999943	3.71965990079478\\
8.96999999999947	3.71965972800255\\
8.99999999999951	3.71965955693823\\
9.02999999999954	3.71965938757605\\
9.05999999999958	3.71965921989072\\
9.08999999999962	3.71965905385747\\
9.11999999999965	3.71965888945199\\
9.14999999999969	3.71965872665046\\
9.17999999999973	3.71965856542953\\
9.20999999999976	3.71965840576629\\
9.2399999999998	3.71965824763828\\
9.26999999999984	3.71965809102345\\
9.29999999999987	3.71965793590019\\
9.32999999999991	3.71965778224729\\
9.35999999999995	3.71965763004394\\
9.38999999999998	3.71965747926974\\
9.42000000000002	3.71965732990464\\
9.45000000000006	3.71965718192898\\
9.48000000000009	3.71965703532346\\
9.51000000000013	3.71965689006915\\
9.54000000000017	3.71965674614744\\
9.5700000000002	3.71965660354009\\
9.60000000000024	3.71965646222916\\
9.63000000000028	3.71965632219707\\
9.66000000000031	3.71965618342652\\
9.69000000000035	3.71965604590055\\
9.72000000000039	3.71965590960249\\
9.75000000000042	3.71965577451597\\
9.78000000000046	3.7196556406249\\
9.8100000000005	3.71965550791348\\
9.84000000000053	3.71965537636619\\
9.87000000000057	3.71965524596779\\
9.90000000000061	3.71965511670328\\
9.93000000000064	3.71965498855794\\
9.96000000000068	3.7196548615173\\
9.99000000000072	3.71965473556713\\
10.0200000000008	3.71965461069345\\
10.0500000000008	3.71965448688252\\
10.0800000000008	3.71965436412083\\
10.1100000000009	3.71965424239509\\
10.1400000000009	3.71965412169225\\
10.1700000000009	3.71965400199946\\
10.200000000001	3.71965388330411\\
10.230000000001	3.71965376559376\\
10.260000000001	3.71965364885622\\
10.2900000000011	3.71965353307946\\
10.3200000000011	3.71965341825167\\
10.3500000000012	3.71965330436124\\
10.3800000000012	3.71965319139672\\
10.4100000000012	3.71965307934687\\
10.4400000000013	3.71965296820063\\
10.4700000000013	3.71965285794709\\
10.5000000000013	3.71965274857555\\
10.5300000000014	3.71965264007546\\
10.5600000000014	3.71965253243645\\
10.5900000000015	3.71965242564828\\
10.6200000000015	3.71965231970093\\
10.6500000000015	3.71965221458447\\
10.6800000000016	3.71965211028918\\
10.7100000000016	3.71965200680545\\
10.7400000000016	3.71965190412384\\
10.7700000000017	3.71965180223506\\
10.8000000000017	3.71965170112995\\
10.8300000000017	3.71965160079948\\
10.8600000000018	3.71965150123477\\
10.8900000000018	3.71965140242707\\
10.9200000000019	3.71965130436778\\
10.9500000000019	3.71965120704839\\
10.9800000000019	3.71965111046054\\
11.010000000002	3.719651014596\\
11.040000000002	3.71965091944663\\
11.070000000002	3.71965082500445\\
11.1000000000021	3.71965073126156\\
11.1300000000021	3.71965063821018\\
11.1600000000021	3.71965054584266\\
11.1900000000022	3.71965045415144\\
11.2200000000022	3.71965036312907\\
11.2500000000023	3.7196502727682\\
11.2800000000023	3.71965018306158\\
11.3100000000023	3.71965009400206\\
11.3400000000024	3.7196500055826\\
11.3700000000024	3.71964991779623\\
11.4000000000024	3.71964983063609\\
11.4300000000025	3.71964974409538\\
11.4600000000025	3.71964965816742\\
11.4900000000026	3.71964957284558\\
11.5200000000026	3.71964948812333\\
11.5500000000026	3.7196494039942\\
11.5800000000027	3.71964932045182\\
11.6100000000027	3.71964923748986\\
11.6400000000027	3.71964915510207\\
11.6700000000028	3.71964907328226\\
11.7000000000028	3.7196489920243\\
11.7300000000028	3.71964891132211\\
11.7600000000029	3.71964883116966\\
11.7900000000029	3.71964875156096\\
11.820000000003	3.71964867249009\\
11.850000000003	3.71964859395111\\
11.880000000003	3.71964851593817\\
11.9100000000031	3.71964843844538\\
11.9400000000031	3.71964836146692\\
11.9700000000031	3.71964828499693\\
12.0000000000032	3.71964820902958\\
12.0300000000032	3.71964813355902\\
12.0600000000032	3.71964805857937\\
12.0900000000033	3.71964798408473\\
12.1200000000033	3.71964791006914\\
12.1500000000034	3.7196478365266\\
12.1800000000034	3.71964776345101\\
12.2100000000034	3.7196476908362\\
12.2400000000035	3.71964761867588\\
12.2700000000035	3.71964754696364\\
12.3000000000035	3.71964747569291\\
12.3300000000036	3.71964740485694\\
12.3600000000036	3.71964733444877\\
12.3900000000037	3.71964726446121\\
12.4200000000037	3.71964719488681\\
12.4500000000037	3.71964712571777\\
12.4800000000038	3.71964705694599\\
12.5100000000038	3.71964698856292\\
12.5400000000038	3.71964692055958\\
12.5700000000039	3.71964685292647\\
12.6000000000039	3.71964678565353\\
12.6300000000039	3.71964671873003\\
12.660000000004	3.71964665214449\\
12.690000000004	3.71964658588465\\
12.7200000000041	3.7196465199373\\
12.7500000000041	3.71964645428818\\
12.7800000000041	3.71964638892189\\
12.8100000000042	3.71964632382173\\
12.8400000000042	3.71964625896953\\
12.8700000000042	3.71964619434548\\
12.9000000000043	3.71964612992797\\
12.9300000000043	3.71964606569335\\
12.9600000000043	3.71964600161566\\
12.9900000000044	3.71964593766643\\
13.0200000000044	3.71964587381432\\
13.0500000000045	3.71964581002482\\
13.0800000000045	3.71964574625987\\
13.1100000000045	3.71964568247746\\
13.1400000000046	3.71964561863114\\
13.1700000000046	3.71964555466953\\
13.2000000000046	3.71964549053575\\
13.2300000000047	3.71964542616675\\
13.2600000000047	3.71964536149263\\
13.2900000000048	3.71964529643584\\
13.3200000000048	3.71964523091028\\
13.3500000000048	3.71964516482034\\
13.3800000000049	3.71964509805974\\
13.4100000000049	3.71964503051039\\
13.4400000000049	3.71964496204093\\
13.470000000005	3.71964489250526\\
13.500000000005	3.71964482174084\\
13.530000000005	3.71964474956673\\
13.5600000000051	3.71964467578155\\
13.5900000000051	3.71964460016107\\
13.6200000000052	3.71964452245556\\
13.6500000000052	3.71964444238692\\
13.6800000000052	3.71964435964533\\
13.7100000000053	3.71964427388563\\
13.7400000000053	3.71964418472324\\
13.7700000000053	3.71964409172963\\
13.8000000000054	3.71964399442722\\
13.8300000000054	3.71964389228373\\
13.8600000000054	3.71964378470592\\
13.8900000000055	3.71964367103249\\
13.9200000000055	3.71964355052627\\
13.9500000000056	3.7196434223655\\
13.9800000000056	3.71964328563402\\
14.0100000000056	3.71964313931041\\
14.0400000000057	3.71964298225587\\
14.0700000000057	3.71964281320066\\
14.1000000000057	3.71964263072903\\
14.1300000000058	3.71964243326233\\
14.1600000000058	3.71964221904029\\
14.1900000000059	3.71964198610003\\
14.2200000000059	3.71964173225273\\
14.2500000000059	3.71964145505753\\
14.280000000006	3.71964115179252\\
14.310000000006	3.71964081942226\\
14.340000000006	3.71964045456168\\
14.3700000000061	3.71964005343572\\
14.4000000000061	3.71963961183438\\
14.4300000000061	3.71963912506251\\
14.4600000000062	3.71963858788391\\
14.4900000000062	3.71963799445888\\
14.5200000000063	3.71963733827465\\
14.5500000000063	3.7196366120677\\
14.5800000000063	3.71963580773723\\
14.6100000000064	3.71963491624856\\
14.6400000000064	3.71963392752543\\
14.6700000000064	3.71963283032986\\
14.7000000000065	3.71963161212822\\
14.7300000000065	3.71963025894178\\
14.7600000000065	3.71962875518011\\
14.7900000000066	3.71962708345522\\
14.8200000000066	3.71962522437435\\
14.8500000000067	3.7196231563089\\
14.8800000000067	3.71962085513665\\
14.9100000000067	3.7196182939545\\
14.9400000000068	3.71961544275799\\
14.9700000000068	3.71961226808407\\
15.0000000000068	3.71960873261282\\
15.0300000000069	3.71960479472333\\
15.0600000000069	3.71960040799876\\
15.090000000007	3.71959552067445\\
15.120000000007	3.71959007502297\\
15.150000000007	3.71958400666862\\
15.1800000000071	3.71957724382363\\
15.2100000000071	3.71956970643702\\
15.2400000000071	3.71956130524646\\
15.2700000000072	3.71955194072206\\
15.3000000000072	3.7195415018901\\
15.3300000000072	3.71952986502328\\
15.3600000000073	3.71951689218254\\
15.3900000000073	3.71950242959431\\
15.4200000000074	3.71948630584479\\
15.4500000000074	3.71946832987148\\
15.4800000000074	3.7194482887299\\
15.5100000000075	3.71942594511129\\
15.5400000000075	3.71940103458476\\
15.5700000000075	3.71937326253482\\
15.6000000000076	3.71934230076246\\
15.6300000000076	3.71930778371534\\
15.6600000000076	3.71926930430926\\
15.6900000000077	3.71922640930057\\
15.7200000000077	3.71917859416549\\
15.7500000000078	3.71912529743973\\
15.7800000000078	3.71906589446833\\
15.8100000000078	3.71899969051313\\
15.8400000000079	3.71892591316279\\
15.8700000000079	3.7188437039881\\
15.9000000000079	3.71875210938467\\
15.930000000008	3.71865007054442\\
15.960000000008	3.71853641249921\\
15.9900000000081	3.71840983218278\\
16.020000000008	3.71826888546319\\
16.0500000000079	3.71811197310637\\
16.0800000000079	3.71793732564444\\
16.1100000000078	3.71774298713951\\
16.1400000000077	3.71752679785705\\
16.1700000000077	3.71728637589298\\
16.2000000000076	3.7170190978372\\
16.2300000000075	3.71672207860412\\
16.2600000000075	3.71639215062002\\
16.2900000000074	3.7160258426283\\
16.3200000000073	3.71561935845874\\
16.3500000000072	3.71516855620628\\
16.3800000000072	3.71466892837841\\
16.4100000000071	3.71411558369856\\
16.440000000007	3.71350323139215\\
16.470000000007	3.71282616893086\\
16.5000000000069	3.71207827436126\\
16.5300000000068	3.71125300449002\\
16.5600000000068	3.71034340032645\\
16.5900000000067	3.70934210128094\\
16.6200000000066	3.70824136966624\\
16.6500000000065	3.7070331270279\\
16.6800000000065	3.70570900371721\\
16.7100000000064	3.70426040289389\\
16.7400000000063	3.7026785797854\\
16.7700000000063	3.7009547365234\\
16.8000000000062	3.6990801322201\\
16.8300000000061	3.69704620715065\\
16.8600000000061	3.69484471899937\\
16.890000000006	3.69246788815701\\
16.9200000000059	3.68990854809263\\
16.9500000000058	3.68716029595146\\
16.9800000000058	3.68421763784691\\
17.0100000000057	3.68107612291341\\
17.0400000000056	3.67773246014919\\
17.0700000000056	3.67418461245665\\
17.1000000000055	3.67043186309548\\
17.1300000000054	3.66647485096628\\
17.1600000000054	3.66231557265804\\
17.1900000000053	3.65795735089942\\
17.2200000000052	3.65340477080259\\
17.2500000000051	3.64866358692629\\
17.2800000000051	3.64374060556908\\
17.310000000005	3.63864354772786\\
17.3400000000049	3.63338089875587\\
17.3700000000049	3.62796175091391\\
17.4000000000048	3.62239564476022\\
17.4300000000047	3.61669241473667\\
17.4600000000047	3.61086204347327\\
17.4900000000046	3.60491452834998\\
17.5200000000045	3.59885976282212\\
17.5500000000044	3.59270743401487\\
17.5800000000044	3.58646693718852\\
17.6100000000043	3.58014730690938\\
17.6400000000042	3.57375716415345\\
17.6700000000042	3.56730467812392\\
17.7000000000041	3.56079754127011\\
17.730000000004	3.55424295583485\\
17.760000000004	3.54764763020764\\
17.7900000000039	3.54101778339643\\
17.8200000000038	3.53435915602813\\
17.8500000000037	3.52767702642699\\
17.8800000000037	3.52097623048121\\
17.9100000000036	3.51426118417979\\
17.9400000000035	3.50753590787142\\
17.9700000000035	3.50080405145927\\
18.0000000000034	3.49406891989377\\
18.0300000000033	3.48733349845859\\
18.0600000000033	3.48060047746073\\
18.0900000000032	3.47387227603503\\
18.1200000000031	3.46715106485648\\
18.1500000000031	3.46043878762254\\
};
\addlegendentry{$-\frac{\log(\gggd(t)/\gggd(3))}{t-3}$ (RH scale)}

\end{axis}
\end{tikzpicture}%